\documentclass{amsart}
\usepackage{amsmath,amsthm,latexsym,amssymb,mathrsfs}
\usepackage{mathabx}
\usepackage{mathtools}
\usepackage{dsfont}
\usepackage[english]{babel}
\usepackage[utf8]{inputenc}
\usepackage{color}
\usepackage{physics}
\usepackage{faktor}
\usepackage{overpic}
\usepackage{enumerate}
\usepackage[shortlabels]{enumitem}
\usepackage{pdflscape}
\usepackage{float}
\usepackage[alphabetic,nobysame]{amsrefs}
\usepackage{hyperref}
\usepackage[all,cmtip]{xy}
\usepackage{todonotes}
\usepackage{csquotes}
\usepackage{url}

\newtheoremstyle{mythm}% ⟨name⟩
{3pt}% ⟨Space above⟩
{3pt}% ⟨Space below⟩
{\itshape}% ⟨Body font⟩
{}% ⟨Indent amount⟩
{\bfseries}% ⟨Theorem head font⟩
{.}% ⟨Punctuation after theorem head⟩
{.5em}% ⟨Space after theorem head⟩
{\thmnote{#1 }#3}% ⟨Theorem head spec (can be left empty, meaning ‘normal’)⟩

\newtheorem{thm}{Theorem}[section]
\newtheorem*{thmx}{Theorem} %unlabeled theorem
\newtheorem{alphathm}{Theorem}
\newtheorem{mthminner}{Theorem}  % manually labelled theorem
\newenvironment{mthm}[1]{%
	\renewcommand\themthminner{#1}%
	\mthminner
}{\endmthminner}
\newtheorem{lem}[thm]{Lemma}
\newtheorem{cor}[thm]{Corollary}
\newtheorem{pro}[thm]{Proposition}
\newtheorem{alphapro}[alphathm]{Proposition} % "Letter numbered" propositions

\newtheorem*{qn*}{Question}
\newtheorem*{lem*}{Lemma}

\newtheorem{thm*}{Theorem}
\newtheorem{pro*}[thm*]{Proposition}
\newtheorem{cor*}[thm*]{Corollary}

\newtheorem*{thm**}{Theorem}

\theoremstyle{mythm}

\theoremstyle{definition}
\newtheorem{dfn}[thm]{Definition}

\theoremstyle{remark}
\newtheorem{rmk}[thm]{Remark}
\newtheorem*{claim*}{Claim}
\newtheorem*{fact*}{Fact}
\newtheorem{claiminner}{Claim}  % manually labelled claim
\newenvironment{claim}[1]{%
    \renewcommand\theclaiminner{#1}%
    \claiminner
}{\endclaiminner}

\newtheorem{propertyinner}{Property}  % manually labelled property
\newenvironment{property}[1]{%
    \renewcommand\thepropertyinner{#1}%
    \propertyinner
}{\endpropertyinner}

\newcommand{\ep}{
    \epsilon
}
\newcommand{\mc}[1]{
    \mathcal{#1}
}
\newcommand{\mb}[1]{
    \mathbb{#1}
}
\newcommand{\T}{
    \mc{T}
}

\newcommand{\psl}{
   {\rm PSL}(2,\mb{R})
}

\newcommand{\rp}{
	\mb{R}{\rm P}
}

\newcommand{\SOtwon}{
    {{\rm{SO}}_0{(2,n\!+\!1)}}
}

\newcommand{\hyp}{\mathbb{H}}

\newcommand{\R}{\mathbb{R}}
\newcommand{\N}{\mathbb{N}}
\newcommand{\Z}{\mathbb{Z}}

\DeclarePairedDelimiterX{\scal}[2]{\langle}{\rangle}{#1, #2}
\DeclarePairedDelimiterX{\scall}[2]{(}{)}{#1, #2}

\DeclareMathOperator{\supp}{supp}

\begin{document}
    
\title{$\SOtwon$-Maximal representations and hyperbolic surfaces}
%Teichmüller

\author[Filippo Mazzoli]{Filippo Mazzoli}
\address{Filippo Mazzoli: Department of Mathematics, University of Virginia, 
    141 Cabell Drive, 
    22904-4137 Charlottesville VA, USA.} \email{filippomazzoli@me.com}

\author[Gabriele Viaggi]{Gabriele Viaggi}
\address{Department of Mathematics, Heidelberg University, Heidelberg, Germany.}
\email{gviaggi@mathi.uni-heidelberg.de}

\thanks{DFG 427903332}
    
\begin{abstract}
	We study maximal representations of surface groups $\rho:\pi_1(\Sigma)\to\SOtwon$ via the introduction of $\rho$-invariant pleated surfaces inside the pseudo-Riemannian space $\mathbb{H}^{2,n}$ associated to maximal geodesic laminations of $\Sigma$.
	
	We prove that $\rho$-invariant pleated surfaces are always embedded, acausal, and possess an intrinsic pseudo-metric and a hyperbolic structure. We describe the latter by constructing a shear cocycle from the cross ratio naturally associated to $\rho$. The process developed to this purpose applies to a wide class of cross ratios, including examples arising from Hitchin and $\Theta$-positive representations in $\mathrm{SO}(p,q)$.  We also show that the length spectrum of $\rho$ dominates the ones of $\rho$-invariant pleated surfaces, with strict inequality exactly on curves that intersect the bending locus.
	
	We observe that the canonical decomposition of a $\rho$-invariant pleated surface into leaves and plaques corresponds to a decomposition of the Guichard-Wienhard domain of discontinuity of $\rho$ into standard fibered blocks, namely triangles and lines of photons. Conversely, we give a concrete construction of photon manifolds fibering over hyperbolic surfaces by gluing together triangles of photons. 
	
	The tools we develop allow to recover various results by Collier, Tholozan, and Toulisse on the (pseudo-Riemannian) geometry of $\rho$ and on the correspondence between maximal representations and fibered photon manifolds through a constructive and geometric approach, bypassing the use of Higgs bundles.
\end{abstract}
    
\maketitle

\tableofcontents

\section{Introduction}

The notion of maximal representations of the fundamental group $\Gamma$ of a compact hyperbolic surface into a semi-simple Lie group of Hermitian type $G$ was introduced by Burger, Iozzi, and Wienhard in their groundbreaking work \cite{BIW10}. It provides a vast generalization of the notion of \emph{Fuchsian representations}, namely discrete and faithful homomorphisms of $\Gamma$ into $\psl$, which naturally arise as holonomies of complete hyperbolic structures on surfaces. As already observed in \cite{BIW10}, multiple dynamical and geometric properties of Fuchsian representations extend to this wider context: Every maximal representation $\rho : \Gamma \to G$ is faithful, its image $\rho(\Gamma)$ is a discrete subgroup of $G$ acting freely and properly discontinuously on the Riemannian symmetric space associated to $G$, and the set of conjugacy classes of maximal representations constitutes a union of connected components of the character variety $\mathfrak{X}(\Gamma,G)$. 

In recent years, a great variety of results have further investigated and strengthened the relations between Fuchsian representations and geometric structures that naturally arise from maximal representations, and this article is no exception. In our exposition we will consider maximal representations of the fundamental group $\Gamma$ of a closed orientable surface $\Sigma$ of genus $g\ge 2$ into the connected Lie group $G = \SOtwon$. Moreover, rather than investigating the properties of the action of $\Gamma$ on the Riemannian symmetric space of $\SOtwon$, we will focus our attention on a class of pseudo-Riemannian and photon structures naturally associated to $\rho : \Gamma \to \SOtwon$, as previously done by Collier, Tholozan, and Toulisse in \cite{CTT19}.

The main aim of this paper is to provide a purely geometric approach to the study of $\SOtwon$-maximal representations, and establish a direct and explicit link with hyperbolic surfaces and classical Teichm\"uller theory. This gives a possible answer to the question addressed in \cite{CTT19}*{Remark~4.13}, and a suitable framework for generalizations to open surfaces.
In particular, inspired by Thurston's and Mess' works in the study of hyperbolic $3$-manifolds (see e.g. \cite{ThNotes}*{Chapter~8}, or Canary, Epstein, and Green \cite{CEG}*{Chapter~I.5} for a detailed exposition), and of constant curvature Lorentzian $3$-manifolds (see \cite{M07}), respectively, we will pursue this goal by introducing a notion of $\rho$-equivariant pleated surfaces inside $\hyp^{2,n}$, and we will investigate their topological, causal, and geometric properties.

We start by introducing the pseudo-Riemannian and photon spaces that we will be interested in. First, we recall that the group $\SOtwon$ is the identity component of the group of isometries of $\mb{R}^{2,n+1}$, which denotes the vector space $\mb{R}^{n+3}$ endowed with the quadratic form
\[
\langle\bullet,\bullet\rangle_{2,n+1}:=x_1^2+x_2^2-y_1^2-\ldots-y_{n+1}^2.
\]

There are multiple homogeneous spaces $X$ naturally associated with the Lie group $G=\SOtwon$, and each of them leads to a different class of $(G,X)$-structures in the sense of Thurston (see \cite{ThNotes}*{Chapter~3}). Here we will consider:
\begin{itemize}
	\item{The pseudo-Riemannian symmetric space $\mb{H}^{2,n}$ of negative lines of $\mb{R}^{2,n+1}$.}
	\item{The Photon space ${\rm Pho}^{2,n}$ of isotropic $2$-planes of $\mb{R}^{2,n+1}$.}
\end{itemize}

In both cases, every maximal representation $\rho:\Gamma\to\SOtwon$ has a natural domain of discontinuity $\Omega_\rho(X)\subset X$, as a consequence of the work of Guichard and Wienhard \cite{GW12} when $X={\rm Pho}^{2,n}$, and of Danciger, Guéritaud, and Kassel \cite{DGK17} when $X=\mb{H}^{2,n}$. Accordingly, any maximal representation $\rho$ gives rise to: 
\begin{itemize}
	\item{A pseudo-Riemannian manifold $M_\rho=\Omega_\rho(\mb{H}^{2,n})/\rho(\Gamma)$ of signature $(2,n)$.}
	\item{A closed photon manifold $E_\rho=\Omega_\rho({\rm Pho}^{2,n})/\rho(\Gamma)$.}
\end{itemize}

The geometries of these objects are strictly tied, as described by Collier, Tholozan, and Toulisse \cite{CTT19}. Our work parallels in many aspects the article \cite{CTT19} with a central difference: While in \cite{CTT19} the geometric and topological information is extracted by relating the theory of Higgs bundles to the immersion data of equivariant maximal surfaces in $\hyp^{2,n}$, our techniques rely on the study of specific $1$- and $2$-dimensional subsets of the pseudo-Riemannian manifold $M_\rho$, namely {\em geodesic laminations} and {\em pleated surfaces}, in analogy with the tools originally developed by Thurston in his investigation of the structure of the ends of hyperbolic 3-manifolds (see Chapters 8 and 9 of \cite{ThNotes}).  

A valuable feature of this approach, which is in many aspects explicit and constructive, is that it determines a concrete connection between maximal representations and hyperbolic structures on surfaces. Notice also that, if on the one hand the notion of equivariant pleated surfaces is well suited for generalizations to finite-type surfaces, the analytical techniques required for the study of Higgs bundles do not easily extend outside of the realm of closed orientable surface groups. 

For convenience of the reader, we now summarize the main results of the paper. We will then provide a detailed description of each of them, together with the techniques developed for their proof, in the remainder of the introduction:

\begin{enumerate}[(a)]
	\item{For any maximal representation $\rho$ and for any maximal geodesic lamination $\lambda$ of $\Sigma$, there exists a $\rho(\Gamma)$-invariant, {\em acausal}, and properly embedded Lipschitz disk, the {\em pleated set} of $\lambda$, 
		\[
		{\widehat{S}}_\lambda\cup\partial{\widehat{S}}_\lambda\subset\mb{H}^{2,n}\cup\partial\mb{H}^{2,n}.
		\]
		(see Theorem \ref{topology pleated surfaces h2n}). It is naturally decomposed as a union of spacelike geodesics and spacelike ideal triangles of $\mb{H}^{2,n}$ and is contained in the $\rho$-domain of discontinuity in $\mb{H}^{2,n}$. In particular, $S_\lambda={\widehat{S}}_\lambda/\rho(\Gamma)$ is a properly embedded subsurface of the pseudo-Riemannian manifold $M_\rho$. The decomposition of ${\widehat{S}}_\lambda$ into lines and triangles corresponds to an analogue decomposition of the $\rho$-domain of discontinuity in ${\rm Pho}^{2,n}$ into {\em lines} and {\em triangles of photons} and there exists a natural fibration $E_\rho\to S_\lambda$ (see Proposition \ref{decomposition domain}).}
	\item{The pleated set ${\widehat{S}}_\lambda$ has a natural intrinsic $\rho(\Gamma)$-invariant hyperbolic structure and a natural pseudo-metric induced by the pseudo-Riemannian metric of $\mb{H}^{2,n}$. The developing map ${\widehat{S}}_\lambda\to\mb{H}^2$ is 1-Lipschitz with respect to the pseudo-metric on ${\widehat{S}}_\lambda$ and the hyperbolic metric on $\mb{H}^2$ (see Theorem \ref{geometry pleated surfaces h2n}). This implies that the length spectrum of the hyperbolic surface $S_\lambda$ is {\em dominated} by the pseudo-Riemannian length spectrum of $\rho$, that is
		\[
		L_\rho(\bullet)\ge L_{S_\lambda}(\bullet).
		\]
		There is a simple characterization of those curves $\gamma\in\Gamma$ for which the strict inequality holds: They are exactly the curves that intersect essentially the {\em bending locus} of $S_\lambda$.}
	\item{The intrinsic hyperbolic structure on the pleated set $\widehat{S}_\lambda$ is described by a \emph{shear cocycle} $\sigma^\rho_\lambda$ through Bonahon's shear parametrization of Teichm\"uller space (see \cite{Bo96}). The definition of the cocycle $\sigma^\rho_\lambda$ uniquely relies on the data of the lamination $\lambda$ and of a $\Gamma$-invariant cross ratio on the Gromov boundary of $\Gamma$, naturally associated to the representation $\rho$. In fact, the construction applies in great generality, and associates to any \emph{(strictly) positive} and \emph{locally bounded} cross ratio $\beta$ on $\partial \Gamma$, and to any maximal lamination $\lambda$, an intrinsic hyperbolic structure $X_\lambda$ whose length spectrum coincides with the length spectrum of $\beta$ on all measured laminations with support contained in $\lambda$.}
	\item{The set of hyperbolic surfaces $S_\lambda$ arising as intrinsic hyperbolic structures on pleated surfaces lie on the boundary of the {\em dominated set} of $\rho$, namely the subset of Teichmüller space defined by
		\[
		\mc{P}_\rho:=\{Z\in\T\left|L_Z(\bullet)\le L_\rho(\bullet)\right.\}.
		\]
		The set $\mc{P}_\rho$ is convex for the Weil-Petersson metric and with respect to shear paths. Its interior ${\rm int}(\mc{P}_\rho)$ corresponds to those hyperbolic surfaces $Z$ that are {\em strictly dominated} by $\rho$, that is $L_\rho(\bullet)>c L_Z(\bullet)$ for some $c>1$. Combining a geometric construction in $\hyp^{2,n}$ with the convexity of length functions along Weil-Petersson geodesics, we observe that $\rho$ is not Fuchsian if and only if ${\rm int}(\mc{P}_\rho)$ is non-empty (Theorem \ref{structure pleated}). This allows to recover part of the results described in \cite{CTT19}.}
	\item{ We describe an elementary process to construct \emph{fibered photon manifolds} (see \cite{CTT19}*{\S~4.2}), namely photon manifolds $E$ that fiber over a closed hyperbolic surface $S$ homeomorphic to $\Sigma$. Any such fibered photon manifold $E \to S$ has maximal holonomy $\rho:\pi_1(E)\to\SOtwon$ and determines a natural $\rho$-equivariant pleated acausal embedding ${\widehat{S}}\to\mb{H}^{2,n}$ of the universal cover ${\widehat{S}}\to S$, with bending locus lying in some maximal geodesic lamination $\lambda$ of $S$.
		The construction of a fibered photon manifold $E\to S$ is completely analogous to the process by which a closed hyperbolic surface is obtained by first gluing together ideal hyperbolic triangles to form (incomplete) pair of pants, and then pasting the completions of the pair of pants along their boundaries. Here, instead of gluing ideal hyperbolic triangles, we will glue together triangles of photons forming (incomplete) fibered pairs of pants of photons, find suitable completions, and combine them to form closed manifolds (Theorem \ref{pants of photons} and Proposition \ref{maximal holonomy}).}
\end{enumerate}

We now describe more in detail each of the previous points.

\subsection{Topology and acausality of pleated surfaces}

Our discussion will heavily rely on the existence of equivariant boundary maps naturally associated to $\SOtwon$-maximal representations, which is guaranteed by the following result of Burger, Iozzi, Labourie, and Wienhard: We recall that the boundary at infinity $\partial\mb{H}^{2,n}$ of the pseudo-Riemannian symmetric space $\mb{H}^{2,n}$ identifies with the space of isotropic lines of $\mb{R}^{2,n+1}$.

\begin{thmx}[{\cite{BILW}*{\S~6}, see also \cite{CTT19}*{Theorem~2.5}}]
	\label{acausal limit curve}
	If $\rho:\Gamma\to\SOtwon$ is a maximal representation, then there exists a unique $\rho$-equivariant, continuous, and dynamics preserving embedding
	\[
	\xi:\partial\Gamma\to\partial\mb{H}^{2,n}
	\]
	such that the image of $\xi$ is an {\rm acausal} curve, meaning that for every triple of distinct points $a,b,c\in\partial\Gamma$, the subspace of $\mb{R}^{2,n+1}$ generated by the isotropic lines $\xi(a),\xi(b),\xi(c)$ has signature $(2,1)$.
\end{thmx}		

Theorem \ref{acausal limit curve} has a simple interpretation in terms of the geometry of $\mb{H}^{2,n}$: Every pair of distinct points $a, b \in \partial \Gamma$ is sent by $\xi$ into the pair of endpoints of a unique spacelike geodesic of $\hyp^{2,n}$, and for every triple of distinct points $a, b, c \in \partial \Gamma$, the images $\xi(a), \xi(b), \xi(c)$ are the vertices of a unique ideal totally geodesic {\em spacelike} triangle in $\mb{H}^{2,n}$. In light of this phenomenon, the boundary map $\xi$ allows us to naturally realize geodesic laminations on the surface $\Sigma$ as $\rho$-invariant closed subsets of $\mb{H}^{2,n}$, and consequently in the pseudo-Riemannian manifold $M_\rho$.

To see this, we start by briefly recalling the notion of geodesic lamination, and the related terminology that will be used throughout our exposition. We will think of a geodesic $\ell$ in the universal cover $\widetilde{\Sigma}$ of $\Sigma$ as an element of the space
\[
\mc{G}:=(\partial\Gamma\times\partial\Gamma-\Delta)/(x,y)\sim(y,x),
\]
simply by identifying $\ell$ with the unordered pair of its endpoints. We say that two geodesics $\ell$ and $\ell'$ with endpoints $a,b$ and $a',b'$, respectively, are \emph{crossing} if $a'$ and $b'$ lie in distinct connected components of $\partial \Gamma - \{a,b\}$ (recall that $\partial \Gamma$ is a topological circle). Two geodesics that are not crossing will be said to be \emph{disjoint}. Within this framework, a geodesic lamination of $\Sigma$ is a $\Gamma$-invariant closed subset $\lambda\subset\mc{G}$ made of pairwise disjoint geodesics, and it is said to be {\em maximal} if every geodesic $\ell$ outside $\lambda$ crosses some $\ell'$ in $\lambda$. The elements of a lamination will be also called its \emph{leaves}, and the connected components $P$ of $\widetilde{\Sigma} - \lambda$ will be called its {\em plaques}. 

Let now $\rho$ be a maximal representation, and let $\xi : \partial \Gamma \to \partial\mb{H}^{2,n}$ be its associated boundary map. For any leaf $\ell = [a,b]$ in $\lambda$, we can find a unique spacelike geodesic $\hat{\ell}$ in $\hyp^{2,n}$ with endpoints $\xi(a),\xi(b)$, and similarly for any plaque $P=\Delta(a,b,c)$, we have a unique spacelike ideal triangle $\hat{P}$ with endpoints $\xi(a), \xi(b), \xi(c) \in \partial\mb{H}^{2,n}$. We then define the \emph{geometric realization of $\lambda$ in $\hyp^{2,n}$} to be
\[
\hat{\lambda} : = \bigcup_{\text{$\ell$ leaf of $\lambda$}} \hat{\ell} ,
\]
and its associated \emph{pleated set} as
\[
\widehat{S}_\lambda : = \hat{\lambda} \cup \bigcup_{\text{$P$ plaque of $\lambda$}} \hat{P} .
\]

Our first result establishes some structural properties about the topology and the causal features of these sets:

\begin{alphathm}
	\label{topology pleated surfaces h2n}
	Let $\rho:\Gamma\to\SOtwon$ be a maximal representation. For every maximal lamination $\lambda$, the pleated set ${\widehat{S}}_\lambda\subset\mb{H}^{2,n}$ is an embedded Lipschitz disk which is also {\rm acausal}, that is, every pair of points $x,y\in {\widehat{S}}_\lambda$ is joined by a spacelike geodesic.   
\end{alphathm}

The basic idea behind Theorem \ref{topology pleated surfaces h2n} is the following: A pair of geodesics $\hat{\ell},\hat{\ell}'$ with endpoints on the limit curve $\Lambda_\rho=\xi(\partial\Gamma)$ form an acausal set $\hat{\ell} \cup \hat{\ell}'$ inside $\mb{H}^{2,n}$ if and only if the corresponding leaves $\ell, \ell'$ of $\lambda$ are disjoint. 

This property immediately implies that the geometric realization ${\hat \lambda}\subset\mb{H}^{2,n}$ of any lamination $\lambda$ is an acausal subset. In turn, working in the Poincaré model of $\mb{H}^{2,n}$, we show that the acausality property is preserved once the complementary triangles are added to ${\hat \lambda}$. By general properties of acausal subsets of $\mb{H}^{2,n}$, we deduce that ${\widehat{S}}_\lambda\cup\Lambda_\rho\subset\mb{H}^{2,n}\cup\partial\mb{H}^{2,n}$ is a properly embedded Lipschitz disk. 

The surface $S_\lambda={\widehat{S}}_\lambda/\rho(\Gamma)\subset M_\rho$ carries two natural geometries: It has an intrinsic {\em hyperbolic structure} and a {\em pseudo-metric} induced by the ambient space $\mb{H}^{2,n}$. We now focus our attention of the description of the former. 

\subsection{Cross ratio and shear cocycles}
As in the case of pleated surfaces in hyperbolic 3-space $\mb{H}^3$, the hyperbolic structure on the pleated set ${\widehat{S}}_\lambda$ can be recorded by a {\em shear cocycle} \cite{Bo96}. 

In order to define it, we again rely on the properties of the boundary map: The acausality condition satisfied by $\xi$ and the pseudo-Riemannian structure of the boundary $\partial\mb{H}^{2,n}$ uniquely determine a $\Gamma$-invariant {\em cross ratio} $\beta^\rho$ on $\partial\Gamma$, satisfying the following properties:
\begin{itemize}
	\item{It is {\em (strictly) positive} on positively ordered quadruples in $\partial \Gamma$. This follows from the acausal properties of the boundary map $\xi$ and implies, via general results of Martone and Zhang \cite{MZ19}, and Hamenstädt \cite{H99}, that $\beta^\rho$ induces a length function $L_{\rho}$ on the space of geodesic currents $\mc{C}$.}
	\item{It is {\em locally bounded}, meaning that there exists a hyperbolic structure $X$ on $\Sigma$ such that, for every compact subset $K$ in the space of distinct 4-tuples in $\partial \Gamma$, we can find constants $C, \alpha > 0$ such that 
		\[
		\abs{\log \abs{\beta^\rho(a,b,c,d)}} \leq C \abs{\log{\abs{\beta^X(a,b,c,d)}}}^\alpha
		\]	
		for every cyclically ordered $4$-tuples $(a,b,c,d) \in K$, where $\beta^X$ is the cross ratio on $\partial \Gamma$ determined by the structure $X \in \T$. This property is a consequence of the explicit definition of $\beta^\rho$ and of the H\"older continuity of the limit map $\xi$.}
\end{itemize}

Notice that examples of (strictly) positive and locally bounded cross ratios naturally arise also from other interesting contexts related to pseudo-Riemannian symmetric spaces $\mb{H}^{p,q}$ such as Hitchin representations in ${{\rm SO}{(p,p\!+\!1)}}$ or $\Theta$-positive representations in ${\rm SO}(p,q)$ where similar pleated surface construction might be possible (see also Appendix \ref{other cross ratios}). 

We also remark that positive cross ratios have been used by Martone and Zhang in \cite{MZ19}, by Labourie \cite{Lab08}, and Burger, Iozzi, Parreau, and Pozzetti \cite{BIPP21} to study common features of Higher Teichmüller Theories.

Making use of the cross ratio $\beta^\rho$, we then describe the intrinsic hyperbolic structure of a pleated set $S_\lambda$ through the data of a so-called \emph{H\"older cocycle} $\sigma^\rho_{\lambda}$ transverse to the maximal lamination $\lambda$, in the sense of \cite{Bo96}. 

Tranverse H\"older cocycles were introduced by Bonahon (see \cites{Bo97transv,Bo97geo}), who deployed them for instance to provide a parametrization of Teichmüller space $\T$ of a closed orientable surface $\Sigma$ in \cite{Bo96}, following ideas of Thurston \cite{T86}. Heuristically speaking, if $\lambda_Z$ is the geometric realization of a maximal lamination $\lambda$ on the hyperbolic surface $Z$, the shear cocycle $\sigma^Z_{\lambda}$ records how the ideal triangles in $Z-\lambda_Z$ are glued together along the leaves of $\lambda_Z$. The space $\mathcal{H}(\lambda;\R)$ of H\"older cocycles transverse to $\lambda$ has a natural structure of vector space of dimension $3|\chi(\Sigma)|$, and the map that associates to any hyperbolic structure $Z \in \T$ its shear cocycle $\sigma^Z_\lambda \in \mc{H}(\lambda;\mb{R})$ embeds Teichmüller space as an open convex cone with finitely many faces inside $\mc{H}(\lambda;\mb{R})$. The resulting set of coordinates is usually referred to as {\em shear coordinates} with respect to the maximal lamination $\lambda$. 

This point of view on Teichmüller space has proved to be fruitful also in the setting of Hitchin representations and, more generally, to analyze Anosov representations as witnessed by work of Bonahon and Dreyer \cite{BD17}, Alessandrini, Guichard, Rogozinnikov, and Wienhard \cite{AGRW22}, and Pfeil \cite{P21}.

The underlying principle for the construction of a shear cocycle starting from a cross ratio is very elementary: The classical shear between two adjacent ideal triangles $\Delta$ and $\Delta'$ in the hyperbolic plane is an explicit function of the $\rp^1$-cross ratio of the four ideal vertices of $\Delta \cup \Delta'$, and shears between triangles separated by finitely many leaves of $\lambda$ can be expressed as a finite sum of shears between adjacent plaques. One can then define the $\rho$-shear between two adjacent plaques $P, Q$ of $\lambda$ simply by replacing the role of the $\rp^1$-cross ratio with $\beta^\rho$. In fact with some additional (but elementary) work, this allows to introduce a natural notion of $\rho$-shear cocycle $\sigma^\rho_\lambda$ for a large class of maximal laminations, namely laminations on $\Sigma$ obtained by adding finitely many isolated leaves to a collection of disjoint simple closed curves (see Section \ref{subsec:shear finite case}).

The construction of the shear cocycle $\sigma^\rho_\lambda$ that we describe relies only on the properties of the cross ratio $\beta^\rho$ that we mentioned above, namely that $\beta^\rho$ is strictly positive and locally bounded. Consequently, our techniques allow to deduce the following general statement:

\begin{alphathm}
	\label{shear of cross ratio}
	Let $\beta:\partial\Gamma^4\to\mb{R}$ be a strictly positive locally bounded cross ratio. For every maximal lamination $\lambda$ there exists a transverse H\"older cocycle $\sigma^\beta_{\lambda}\in\mc{H}(\lambda;\mb{R})$ with the following properties:
	\begin{enumerate}[(i)]
		\item{The cocycle $\sigma^\beta_{\lambda}$ is the shear cocycle of a unique hyperbolic metric $X_\lambda$ on $\Sigma$.}
		\item{For every transverse measure $\mu$ on $\lambda$ we have $L_{X_\lambda}(\mu)=L_\beta(\mu)$.}
		\item{The map $\lambda \mapsto X_{\lambda}$ is continuous with respect to the Hausdorff topology on the space of maximal geodesic laminations.}
	\end{enumerate}
\end{alphathm}

The process to construct $\beta$-shear cocycles $\sigma^\beta_\lambda$ for a generic maximal lamination is technically quite involved, and our strategy will heavily rely on multiple tools developed by Bonahon \cite{Bo96} in his construction of shear coordinates for Teichm\"uller space, such as the notion of \emph{divergence radius function} associated to the choice of a train track carrying $\lambda$ (see also Bonahon and Dreyer \cite{BD17}*{\S~8.2}). However, if $\lambda$ is a finitely leaved lamination and $\beta = \beta^\rho$ is the cross ratio associated to some maximal representation $\rho$, then the shear cocycle $\sigma_\lambda^\rho$ has a simple interpretation in terms of horocycle foliations on the plaques of the pleated set $S_\lambda$, in direct analogy with Bonahon's original description of shear coordinates (see e.g. \cite{Bo96}*{\S~2}).

We call the cocycle $\sigma^\rho_{\lambda}$ the {\em intrinsic shear cocycle} associated to $\lambda$ and $\rho$.

\subsection{Geometry of pleated surfaces}
The other intrinsic geometric structure carried by the pleated set ${\widehat{S}}_\lambda$ is a $\rho$-invariant pseudo-metric: Since any two points $x,y\in \widehat{S}_\lambda$ are connected by a unique spacelike geodesic segment $[x,y]$ (see Theorem \ref{topology pleated surfaces h2n}), we can define 
\[
d_{\mb{H}^{2,n}}(x,y):=\ell[x,y].
\]
It is worth to mention that the function 
\[
d_{\mb{H}^{2,n}}:{\widehat{S}}_\lambda\times{\widehat{S}}_\lambda\to[0,\infty)
\]
is not a distance in the traditional sense as it does not satisfy the triangle inequality nor its inverse (see also \cite{GM19} and \cite{CTT19}). However, it is continuous, it vanishes exactly on the diagonal, and its metric balls $B(x,r)=\{y\in{\widehat{S}_\lambda}\left|d_{\mb{H}^{2,n}}(x,y)\le r\right.\}$ form a fundamental system of neighborhoods.

Nevertheless, the pseudo-distance $d_{\hyp^{2,n}}$ naturally relates to the hyperbolic structure $X_\lambda$ associated to $\rho$ and the maximal lamination $\lambda$. To see this, let us introduce the following notion: We say that a function $f : S_\lambda={\widehat{S}}_\lambda/\rho(\Gamma)\to X$ with values in a hyperbolic surface $X$ is {\em $K$-Lipschitz} with respect to the intrinsic pseudo-metric if it lifts to a map ${\hat f}:{\widehat{S}}_\lambda\to\mb{H}^2$ that satisfies
\[
d_{\mb{H}^2}({\hat f}(x),{\hat f}(y))\le Kd_{\mb{H}^{2,n}}(x,y)
\]
for any $x, y \in \widehat{S}_\lambda$. Then we have:

\begin{alphathm}
	\label{geometry pleated surfaces h2n}
	Let $\rho:\Gamma\to\SOtwon$ be a maximal representation, and $\lambda$ be a maximal lamination. If $X_\lambda$ denotes the hyperbolic surface with {\rm intrinsic shear cocycle} $\sigma_\lambda^\rho$, then the pleated surface $S_\lambda\subset M_\rho$ admits a unique developing homeomorphism $f:S_\lambda\to X_\lambda$ which is $1$-Lipschitz with respect to the intrinsic pseudo-metric on $S_\lambda$. Furthermore, we have
	\[
	L_{X_\lambda}(\gamma)\le L_\rho(\gamma)
	\]
	for every $\gamma\in\Gamma$, where $L_{X_\lambda}, L_\rho:\Gamma\to(0,\infty)$ denote the length functions of the hyperbolic surface $X_\lambda$ and the representation $\rho$, respectively, with strict inequality if and only if $\gamma$ intersects the {\rm bending locus} of $S_\lambda$.
\end{alphathm}

The heuristic idea of Theorem \ref{geometry pleated surfaces h2n} is the following: The pleated set ${\widehat{S}}_\lambda$ has an intrinsic hyperbolic path metric (whose shear cocycle is exactly the intrinsic shear cocycle $\sigma^\rho_\lambda$). Using the fact that ${\widehat{S}}_\lambda$ is an acausal subset that can be represented as a graph in the Poincaré model of $\mb{H}^{2,n}$, one can show that for every pair of points $x,y\in{\widehat{S}}_\lambda$ there exists a path inside ${\widehat{S}}_\lambda$ joining them whose length is bounded by $d_{\hyp^{2,n}}(x,y)$. This immediately implies that the path metric on ${\widehat{S}}_\lambda$ is dominated by the intrinsic pseudo-metric $d_{\hyp^{2,n}}$.

We show that this picture is accurate in the case of finite leaved maximal laminations. The proof here is elementary and uses a cut-and-paste argument in the spirit of \cite{CEG}*{Theorem~I.5.3.6}. In order to deduce the statement of Theorem \ref{geometry pleated surfaces h2n} from the finite leaved case, we exploit continuity properties of pleated surfaces. 

We conclude here our discussion on the existence of $\rho$-equivariant pleated surfaces and the study of their topology and geometry. In what follows, we deploy the results just described to extract information on the maximal representation $\rho$.

\subsection{Length spectra of maximal representations}

We now focus on the study of the set of pleated surfaces $\{X_\lambda\}_{\lambda}$ associated to a given maximal representation $\rho$, considered as a subset of Teichmüller space $\T$. As it turns out, it can be described as a subset of the boundary of a set that is {\em convex} with respect to multiple natural structures on $\T$. More precisely, given a maximal representation $\rho$ let us define the {\em dominated set} of $\rho$ as 
\[
\mc{P}_\rho:=\{Z\in\T\left| L_Z(\gamma) \le L_\rho(\gamma)\text{ \rm for every }\gamma\in\Gamma\right.\}.
\]
We also define the companion $\mc{P}_\rho^{{\rm simple}}$ consisting of those hyperbolic surfaces whose {\em simple} length spectrum is dominated by the simple length spectrum of $\rho$. 

By Theorem \ref{geometry pleated surfaces h2n}, the set $\mc{P}_\rho$ is always non-empty as it contains all the hyperbolic structures of the pleated surfaces associated to $\rho$. Furthermore, it is convex with respect to the Weil-Petersson metric, by work of Wolpert \cite{W87}, and with respect to {\em shear paths}, by work of Bestvina, Bromberg, Fujiwara and Souto \cite{BBFS13} and Th\'eret \cite{The14} generalizing a result of Kerckhoff \cite{K83} (see also \cite{MV} for a different approach). 

We prove the following: 

\begin{alphathm}
	\label{structure pleated}
	Let $\rho:\Gamma\to\SOtwon$ be a maximal representation. For every $Z\in\T$ define
	\[
	\kappa(Z):=\sup_{\gamma\in\Gamma-\{1\}}{\frac{L_Z(\gamma)}{L_\rho(\gamma)}}.
	\]
	
	We have:
	\begin{enumerate}
		\item{$Z\in{\rm int}(\mc{P}_\rho)$ if and only if $\kappa(Z)<1$.}
		\item{If $\rho$ is not Fuchsian, then ${\rm int}(\mc{P}_\rho)\neq\emptyset$.}
		\item{If $X_\lambda$ is the hyperbolic structure with shear cocycle $\sigma^\rho_\lambda$, then $X_\lambda\in\partial\mc{P}_\rho$.}
	\end{enumerate}
	
	Furthermore,
	\begin{enumerate}
		\setcounter{enumi}{3}
		\item{If $Z\not\in{\rm int}(\mc{P}_\rho)$, then there exists $\mu\in\mc{ML}$ such that $\kappa(Z)=L_Z(\mu)/L_\rho(\mu)$.}
		\item{$\mc{P}_\rho=\mc{P}_\rho^{{\rm simple}}$.}
	\end{enumerate}
\end{alphathm}

As a consequence of properties (1) and (2), we obtain that if $\rho$ is not Fuchsian, then there exist hyperbolic structures $Z\in\T$ whose length spectrum $L_Z(\bullet)$ is strictly dominated by the length spectrum of the maximal representation $L_\rho(\bullet)$. Thus, we recover the following:

\begin{thmx}[{Collier, Tholozan, and Toulisse \cite{CTT19}}]
	Let $\rho:\Gamma\to\SOtwon$ be a maximal representation with $n\ge 1$. We have the following: Either $\rho$ is Fuchsian or there exists a hyperbolic surface $Z$ and a constant $c>1$ such that $L_\rho(\bullet)>cL_Z(\bullet)$.
\end{thmx}

In particular, the inequality $L_\rho(\bullet)>cL_Z(\bullet)$ immediately implies that the {\em entropy} of $\rho$, defined by
\[
\delta_\rho:=\limsup_{R\to\infty}{\frac{\log\left|\{[\gamma]\in[\Gamma]\left|L_\rho(\gamma)\le R\right.\}\right|}{R}} ,
\]
where $[\Gamma]$ denotes the set of conjugacy classes of elements in $\Gamma$, is bounded by 
\[
\delta_\rho\le 1/c\le 1
\]
and the equality $\delta_\rho=1$ holds if and only if $\rho$ is Fuchsian (compare with \cite{CTT19}*{Corollary~5}). 

Let us briefly comment on properties (1), (2), and (3). Property (1) characterizes interior points of $\mc{P}_\rho$ as those points $Z$ whose length spectrum $L_Z(\bullet)$ is {\em strictly} dominated by the length spectrum $L_\rho(\bullet)$. The proof proceeds as follows. On the one hand, being strictly dominated is an open condition: For $Z\in\T$ and for every $K>1$ there is a neighborhood $U$ of $Z$ such that every surface $Z'\in U$ is $K$-biLipschitz to $Z$ and in particular $1/K\le L_Z/L_{Z'}\le K$. Therefore, if $\kappa(Z)<1$ and $K<1/\kappa(Z)$, then $\kappa(Z')<1$. On the other hand, interior points are strictly dominated due to the strict convexity of length functions along Weil-Petersson geodesics.

The idea of (2) is the following: In order to prove that ${\rm int}(\mc{P}_\rho)$ is non empty, it is enough to show that $\mc{P}_\rho$ contains at least two distinct points $X,Y$. Indeed, by the strict convexity of length functions with respect to the Weil-Petersson metric (see Wolpert \cite{W87} and \cite{W04}), the midpoint $Z \in \mathcal{P}_\rho$ of the Weil-Petersson segment $[X,Y]$ is strictly dominated and, hence, by property (1), is an interior point.

If $\rho$ is not Fuchsian, such pair of points $X,Y\in\mc{P}_\rho$ can be produced by taking two pleated surfaces $S_\alpha$ and $S_\beta$ realizing simple closed curves $\alpha$ and $\beta$ that intersect (completed to maximal laminations $\lambda_\alpha,\lambda_\beta$ by adding finitely many leaves spiraling around them). On the one hand, Theorem \ref{geometry pleated surfaces h2n} tells us that $L_{S_\alpha}(\alpha)=L_\rho(\alpha)$ and $L_{S_\beta}(\beta)=L_\rho(\beta)$. On the other hand, as $\rho$ is not Fuchsian, the bending locus of $S_\alpha,S_\beta$ is not empty. Since the bending loci are sublaminations of the maximal extensions $\lambda_\alpha,\lambda_\beta$, they contain the curves $\alpha,\beta$ respectively. Since $\alpha, \beta$ intersect essentially, the curve $\alpha$ intersects the bending locus of $S_\beta$ and $\beta$ intersects the bending locus of $S_\alpha$, therefore $L_{S_\alpha}(\beta)<L_\rho(\beta)$ and $L_{S_\beta}(\alpha)<L_\rho(\alpha)$, again by Theorem \ref{geometry pleated surfaces h2n}. In any case $S_\alpha\neq S_\beta$.

Property (3) follows from the fact that every measured lamination $\mu$ whose support does not intersect essentially the bending locus of $S_\lambda$ realizes $L_{S_\lambda}(\mu)=L_\rho(\mu)$ which implies $\kappa(S_\lambda)=1$ and, hence, by Property (1), $S_\lambda\in\partial\mc{P}_\rho$.

Lastly, let us also spend a couple of words on the simple length spectrum of $\rho$: It follows from (5) that the simple length spectrum alone completely determines the dominated set. This is an indication that there might be simple length spectrum rigidity for ${\SOtwon}$-maximal representations. 

The proof of (5) depends on (4): In fact, on the one hand, we always have $\mc{P}_\rho\subset\mc{P}_\rho^{{\rm simple}}$, by definition. On the other hand, properties (1) and (4) imply together that $\partial\mc{P}_\rho\subset\partial\mc{P}_\rho^{{\rm simple}}$. As both sets are topological disks, by convexity with respect to Weil-Petersson geometry or with respect to shear paths, we can conclude that they are equal. The proof of (4) follows arguments of Thurston \cite{T86} on the existence of a maximally stretched laminations. 

\subsection{Photon structures fibering over hyperbolic surfaces}
We now describe the picture from the perspective of photon structures. 

By work of Guichard and Wienhard \cite{GW12}, maximal representations $\rho:\Gamma\to\SOtwon$ parametrize certain geometric structures, in the sense of Thurston, on appropriate closed manifolds. More precisely, every representation $\rho$ determines a $(\SOtwon,{\rm Pho}^{2,n})$-structure, called a \emph{photon structure}, on a closed manifold $E_\rho$. By the Ehresmann-Thurston principle (see \cite{ThNotes}*{Chapter~3}), the topology of $E_\rho$ does not change over each connected component of the character variety $\mathfrak{X}(\Gamma,\SOtwon)$. However, different components can correspond to different topological types. 

Using maximal surfaces in $\mb{H}^{2,n}$, Collier, Tholozan, and Toulisse showed in \cite{CTT19} that $E_\rho$ can always be seen as a {\em fibered photon bundle} over $\Sigma$ with geometric fibers which are copies of ${\rm Pho}^{2,n-1}$ and, furthermore, its topology can be computed from some characteristic classes of $\rho$. 

A fibered photon bundle over a surface $\pi:E\to \Sigma$ is an object that comes together with a developing map $\delta:{\hat E}\to{\rm Pho}^{2,n}$ and a natural associated map $\iota:{\widetilde{\Sigma}}\to\mb{H}^{2,n}$, where ${\widetilde{\Sigma}}\to \Sigma$ is the universal covering and $\widetilde{E}\to{\widetilde{\Sigma}}$ is the pull-back bundle on the universal covering, with the property that $\delta(\pi^{-1}(x))={\rm Pho}(\iota(x)^\perp)$. Here $\iota(x)^\perp\subset\mb{R}^{2,n+1}$ is the orthogonal subspace of the negative line $\iota(x)\in\mb{H}^{2,n}$. In particular, the fibered photon bundle $E\to \Sigma$ has an associated {\em underlying vector bundle} $V_E\to \Sigma$ where the fiber over $x$ is the vector space $\iota(x)^\perp$.

In a similar spirit, using pleated surfaces, we show that every maximal lamination $\lambda\subset\Sigma$ induces a geometric decomposition of $E_\rho$ into standard fibered blocks called {\em triangles of photons} which we briefly describe: Let $\Delta\subset\mb{H}^{2,n}$ be an ideal spacelike triangle with vertices $a,b,c\in\partial\mb{H}^{2,n}$. The standard triangle of photons $E(\Delta)\subset{\rm Pho}^{2,n}$ is the codimension 0 submanifold with boundary 
\[
E(\Delta):=\{V\in{\rm Pho}^{2,n}\left|V\perp x\text{ {\rm for some }}x\in\Delta\right.\}.
\] 

The boundary $\partial E(\Delta)$ is a union of {\em lines of photons} $\partial E(\Delta)=E(\ell_a)\cup E(\ell_b)\cup E(\ell_c)$ where $\ell_a,\ell_b,\ell_c$ are the boundary spacelike geodesics of $\Delta$ opposite to the ideal vertices $a,b,c$ and 
\[
E(\ell):=\{V\in{\rm Pho}^{2,n}\left|V\perp x\text{ {\rm for some }}x\in\ell\right.\}
\]
if $\ell\subset\mb{H}^{2,n}$ is a spacelike geodesic. 

The triangle of photons $E(\Delta)$ naturally fibers over the ideal hyperbolic triangle $\Delta$. We denote by $\pi:E(\Delta)\to\Delta$ the natural fibration. The fiber $\pi^{-1}(x)$ over the point $x\in\Delta$ is given by ${\rm Pho}(x^\perp)\simeq{\rm Pho}^{2,n-1}$. 

We have:

\begin{alphapro}
	\label{decomposition domain}
	Let $\rho:\Gamma\to\SOtwon$ be a maximal representation. Let ${\widehat{S}}_\lambda\subset\mb{H}^{2,n}$ be the pleated set associated to the maximal lamination $\lambda$. Then the Guichard-Wienhard domain of discontinuity $\Omega_\rho\subset{\rm Pho}^{2,n}$ naturally decomposes as
	\[
	\Omega_\rho=\bigsqcup_{\ell\subset{\hat \lambda}}{E(\ell)}\sqcup\bigsqcup_{\Delta\subset {\widehat{S}}_\lambda-{\hat \lambda}}{E(\Delta)}
	\]
	and we have an equivariant bundle projection $\Omega_\rho\to {\widehat{S}}_\lambda$ induced by the standard projections $E(\Delta)\to\Delta$ and $E(\ell)\to\ell$.
\end{alphapro}

Conversely, we are also able to describe a procedure to abstractly assemble triangles of photons and explicitly build fibered photon structures on a fiber bundle $E\to \Sigma$ with maximal holonomy $\rho:\Gamma\to\SOtwon$. 

Our approach is completely analogous to the procedure that constructs a closed hyperbolic surface by gluing ideal triangles. First we construct {\em pair of pants of photons} $E_j\to S_j$ by gluing two copies of a standard triangle of photons $E(\Delta)$. As it happens for hyperbolic surfaces, if the holonomy around the boundary curves of $S_j$ is {\em loxodromic} (with respect to a suitable notion of loxodromic), then $E_j\to S_j$ is the interior of a fibered photon structure with {\em totally geodesic boundary} $E_j'\to S_j'$.

For us, loxodromic means {\em bi-proximal}, a property that is equivalent to a suitable north-south dynamics on $\partial\mb{H}^{2,n}$: The set $\mc{L}\subset\SOtwon$ of loxodromic elements is an open subset with two connected components $\mc{L}=\mc{L}^+\cup\mc{L}^-$ distinguished by the sign of the leading eigenvalue. 

We remark that the condition of being loxodromic as well as the topology of the resulting bundles $E_j'\to S_j'$ can be read off the gluing maps without difficulties. We prove the following: Denote by ${\rm PStab}(\ell),{\rm PStab}(\Delta)$ the stabilizers of the spacelike geodesic $\ell$ and the ideal spacelike triangle $\Delta$ that fix the endpoints of $\ell$ and the vertices of $\Delta$ in $\partial\mb{H}^{2,n}$, respectively.

\begin{alphathm}
	\label{pants of photons}
	Let $\Delta=\Delta(a,b,c)\subset\mb{H}^{2,n}$ be an ideal spacelike triangle with $\partial\Delta=\ell_a\cup\ell_b\cup\ell_c$ where $\ell_u$ is the side opposite to the vertex $u\in\{a,b,c\}$. For every equivalence class of triples
	\[
	\psi\in\left\{
	[\psi_a,\psi_b,\psi_c]\in\left(\prod_{j\in\{a,b,c\}}{{\rm PStab}(\ell_j)}\right)\left/{\rm PStab}(\Delta)^2\right.
	\right\}
	\] 
	there is a fibered photon structure
	\[
	E_\psi=E(\Delta)\sqcup E(\Delta)/\sim_\psi
	\]
	fibering over a (possibly incomplete) hyperbolic pair of pants
	\[
	S_\psi=\Delta\sqcup\Delta/\sim_\psi
	\]
	such that the holonomy around the peripheral simple closed curve $\gamma_u$ surrounding the puncture of $S$ corresponding to the vertex $u\in\{a,b,c\}$ is given by
	\[
	\rho_u=\psi_w \nu_{vw}^2 \psi_v^{-1},
	\]
	where $\nu_{v w}\in\SOtwon$ denotes the unique unipotent isometry of $\R^{2,n+1}$ that restricts to the identity on ${\rm Span}\{a,b,c\}^\perp$ and to the parabolic transformation of ${\rm Span}\{a,b,c\}$ that fixes $u$ and sends $w$ into $v$, for any cyclic permutation $(u,v,w)$ of $(a,b,c)$.
	
	If for every $u\in\{a,b,c\}$ the holonomy $\rho_u$ is loxodromic, then $S_\psi,E_\psi$ are respectively the interior of a hyperbolic pair of pants $S'_\psi$ with totally geodesic boundary and the interior of a fibered photon structure $E'_\psi$ with totally geodesic boundary fibering over $S'_\psi$. The fibration $E'_\psi\to S'_\psi$ extends $E_\psi\to S_\psi$. For every $n \geq 1$, the topology of $E'_\psi$ is determined by the first Stiefel-Whitney class $w_1(V_\psi)\in H^1(S_\psi,\mb{Z}/2\mb{Z})$ of the underlying vector bundle $V_\psi \to S_\psi$. The class $w_1(V_\psi)$ can be computed as follows: Let $\gamma_a,\gamma_b,\gamma_c\subset S_\psi$ be the peripheral curves corresponding to the vertices $a,b,c$ respectively. Then
	\[
	w_1(V_\psi)[\gamma_u]=\left\{
	\begin{array}{l l}
		0 &\text{if $\rho_u\in\mc{L}^+$},\\
		1 &\text{if $\rho_u\in\mc{L}^-$},\\
	\end{array}
	\right. .
	\]
\end{alphathm}

As a second step, we take several pair of pants of photons with totally geodesic boundary $E_j'\to S_j'$ and glue them together. Again, some compatibility conditions must be fulfilled by the gluing maps in order to perform the gluing. As a result, we get a photon structure on a manifold $E$ that naturally fibers over a hyperbolic surface $S$ with geometric fibers and we also obtain a maximal geodesic lamination $\lambda$ on $S$.

In analogy with \cite{CTT19}*{Proposition~3.13}, we have:  

\begin{alphapro}
	\label{maximal holonomy}
	The holonomy $\rho:\pi_1(E)\to\SOtwon$ of the fibered photon structure $E\to S$ descends to a maximal representation $\rho:\Gamma\to\SOtwon$. The hyperbolic surface $S$ is the pleated surface that realizes $\lambda$ in $M_\rho$.
\end{alphapro}

In combination with Theorems~\ref{topology pleated surfaces h2n} and \ref{geometry pleated surfaces h2n}, this provides an analogue of \cite{CTT19}*{Corollary~4.12}. In fact, in order to prove Proposition \ref{maximal holonomy}, we generalize results of \cite{CTT19} about smooth spacelike surfaces to purely topological versions. This allows us to treat pleated surfaces.

\subsection{The anti-de Sitter case}

When $n=1$ much of the above picture on pleated surfaces can be made explicit and quantitative. Due to the fact that ${\rm SO}_0(2,2)$ is a 2-fold cover of ${\rm PSL}(2,\mb{R})\times{\rm PSL}(2,\mb{R})$, maximal representations in ${\rm SO}_0(2,2)$ naturally correspond to pairs of maximal representations in ${\rm PSL}(2,\mb{R})$, which, by Goldman's work \cite{G80}, are precisely the holonomies of hyperbolic structures on $\Sigma$. As such, this low dimensional case has a special connection with classical Teichmüller theory. This is highlighted by groundbreaking work of Mess \cite{M07}, who connected the study of globally hyperbolic maximal Cauchy compact (GHMC) anti-de Sitter $3$-manifolds with maximal representations inside $\mathrm{SO}_0(2,2)$, and gave a proof of Thurston's Earthquake Theorem based on the pseudo-Riemannian geometry of the manifold $M_\rho$. (Thurston's original approach is outlined in work of Kerckhoff \cite{K83}.) Since Mess' seminal paper, the study of GHMC anti-de Sitter $3$-manifolds has propagated in multiple directions and has produced further connections with Teichm\"uller theory, as for example described in \cites{notes,bonsante2010maximal,bsduke,benedetti2009canonical_wick,bbz}, among other works. We refer to Bonsante and Seppi \cite{MR4264588} for a detailed exposition of the current state-of-art and for further references.

We will explore the geometry of pleated surfaces in anti-de Sitter $3$-space in a separate paper \cite{MV} where we obtain, among other results, a "Lorentzian proof" of (strict) convexity of length functions in shear coordinates for Teichmüller space (recovering the work of Bestvina, Bromberg, Fujiwara, and Suoto \cite{BBFS13}, and Théret \cite{The14}) and a shear-bend parametrization of globally hyperbolic maximal Cauchy compact anti-de Sitter 3-manifolds. 

\subsection*{Outline}

This article is structured as follows:

In Section \ref{sec:preliminaries} we cover the background material that we need. More precisely: The geometry of the pseudo-hyperbolic space $\mb{H}^{2,n}$ and its boundary $\partial\mb{H}^{2,n}$, acausality, and the Poincaré model. The dynamical and geometric characterization of maximal representations $\rho:\Gamma\to\SOtwon$. Some classical Teichmüller theory: Geodesic laminations, geodesic currents, measured laminations, shear coordinates. Positive cross ratios and their Liouville currents.

In Section \ref{sec:laminations} we first discuss the geometric realizations ${\hat \lambda}$ of maximal laminations $\lambda$ and the associated pleated sets ${\widehat{S}}_\lambda\subset\mb{H}^{2,n}$. Then, we relate the acausal properties of the limit curve $\xi(\partial\Gamma)\subset\partial\mb{H}^{2,n}$ to the topology and the causal structure of ${\hat \lambda}$ and ${\widehat{S}}_\lambda$ (see Propositions \ref{pro:lamination acausal} and \ref{pro:existence pleated sets}). The fact that the pleated set ${\widehat{S}}_\lambda$ is acausal implies that it can be represented as a graph in the Poincaré model, and we show that the graph depends continuously on the lamination (see Proposition \ref{pro:continuity pleated sets}). Lastly, we analyze more in detail the locus where ${\widehat{S}}_\lambda$ is folded and define the bending locus (see Proposition \ref{pro:bending locus}). The bending locus will control how the geometry of ${\widehat{S}}_\lambda$ is distorted in $\mb{H}^{2,n}$. This will play a role in Sections \ref{sec:geometry} and \ref{sec:teichmuller}. 

In Sections \ref{sec:cross ratio} and \ref{sec:shear cocycles} we explain how to attach a natural H\"older cocycle $\sigma^\beta_\lambda\in\mc{H}(\lambda;\mb{R})$ to every positive and locally bounded cross ratio $\beta$ and every maximal lamination $\lambda$ (see Theorem \ref{shear of cross ratio improved}). Section \ref{sec:cross ratio} mainly focuses on the study finite leaved laminations, setting that can be treated with elementary techniques (see Propositions \ref{pro:shear finite case} and \ref{pro:length and thurston sympl finite}). In Section \ref{sec:shear cocycles} we extend the construction to the case of a general maximal lamination. The procedure here is analytic: We define the shear cocycle as a limit of elementary finite approximations. The process needed to establish the convergence of finite approximations is quite delicate as it depends on the geometry of $\lambda$ on a fine scale. In the end, we show that the shear cocycle $\sigma^\beta_\lambda$ is contained in the closure of Teichmüller space $\T$  (considered as an open subset of $\mc{H}(\lambda;\mb{R})$ via shear coordinates with respect to $\lambda$), and it coincides with the shear cocycle of a hyperbolic structure if $\beta$ is strictly positive.

In Section \ref{sec:geometry} we formally define pleated surfaces and study their intrinsic geometric properties. Our analysis here is based on a precise understanding of the case of finite leaved maximal laminations and on continuity properties of pleated surfaces. The pleated set ${\widehat{S}}_\lambda$ has an intrinsic length space structure that makes it locally isometric to $\mb{H}^2$ and the local isometry ${\hat f}:{\widehat{S}}_\lambda\to\mb{H}^2$ is 1-Lipschitz with respect to the intrinsic pseudo-metric (see Proposition \ref{pro:developing finite}). We check that the shear cocycle of the intrinsic path metric on ${\widehat{S}}_\lambda$ coincides with $\sigma_\lambda^\rho$ (Proposition \ref{pro:intrinsic shear finite case}). In both cases, the proofs are elementary. The general case (Proposition \ref{pro:developing general}) follows from the finite leaved case by continuity arguments. As a consequence, we derive a precise comparison between the length spectrum of $S_\lambda$ and the length spectrum of $\rho$ (Proposition \ref{pro:strict inequality}).

In Section \ref{sec:teichmuller} we link the geometry of the maximal representation $\rho:\Gamma\to\SOtwon$ to the geometry of the dominated set $\mc{P}_\rho\subset\T$ consisting of those hyperbolic surfaces $Z$ whose length spectrum $L_Z(\bullet)$ is strictly dominated by $L_\rho(\bullet)$. Such a subset is non-empty, as it contains all pleated surfaces associated to $\rho$, and is convex with respect to the Weil-Petersson metric. We describe the structure of interior and boundary points (see Lemma \ref{lem:interior points} and Proposition \ref{pro:outside max is lamination}) and show that if the representation $\rho$ is not Fuchsian then the interior of $\mc{P}_\rho$ is never empty (Proposition \ref{pro:non empty}). 

In Section \ref{sec:photons} we discuss the point of view of fibered photon structures. We introduce \emph{triangles and lines of photons} $E(\Delta)$ and $E(\ell)$. We show that given a maximal lamination $\lambda$ the Guichard-Wienhard domain of discontinuity admits a fibration $\pi:\Omega_\rho\to{\widehat{S}}_\lambda$ where $\pi^{-1}(\ell)=E(\ell)$ and $\pi^{-1}(\Delta)=E(\Delta)$ for every leaf $\ell\subset{\hat \lambda}$ and plaque $\Delta\subset{\widehat{S}}_\lambda-{\hat \lambda}$ (see Proposition \ref{pro:decomposition}). In the opposite direction, we construct fibered photon structures by gluing together triangles of photons along lines of photons. We provide a detailed description of \emph{pants of photons}, obtained by gluing two triangles of photons. Provided that the holonomy along the boundary curves is loxodromic such pants of photons are the interior of fibered photon structures with totally geodesic boundary (see Lemma \ref{lem:completion}). We completely classify those (see Theorem \ref{thm:classification pair of pants}). Finally, we obtain fibered photon structures over closed hyperbolic surfaces $E\to S$ by gluing pants of photons with totally geodesic boundary, and prove that holonomy of the total space $E$ is always maximal (see Lemma \ref{lem:spacelike implies maximal}).

\subsection*{Acknowledgments}
We are very happy to thank Francesco Bonsante, Colin Davalo, Mitul Islam, Sara Maloni, and Beatrice Pozzetti, with whom we had several useful discussions and received generous feedback during the realization of this project. 

We are also grateful to Francesco Bonsante, Brian Collier, Ursula Hamenstädt, Sara Maloni, Beatrice Pozzetti, Nicholas Tholozan, Jérémy Toulisse, Jean-Marc Schlenker, Andrea Seppi, and Anna Wienhard for precious feedback on the first draft of the article.

Gabriele gratefully acknowledges the financial support of the DFG 427903332.

%%%

\section{Preliminaries}
\label{sec:preliminaries}

In this section we review some basic facts that we will need in our exposition.

We start by discussing the geometry and causal structure of the pseudo-Riemannian space $\mb{H}^{2,n}$ and its boundary $\partial\mb{H}^{2,n}$. In particular, we discuss the Poincaré model of $\mb{H}^{2,n}$ (see Proposition \ref{pro:poincare H2n}) which is a useful device to examine the structure of acausal sets (see Lemmas \ref{lem:projection} and \ref{lem:acausal graph}).

Then, we introduce maximal representations $\rho:\Gamma\to\SOtwon$ and describe the acausal and dynamical properties of their associated limit curve $\xi:\partial\Gamma\to\partial\mb{H}^{2,n}$ (see Theorem \ref{thm:maximal limit curve}). This is the starting point of our constructions in Section \ref{sec:laminations}. 

Afterwards, we recall some background material from classical Teichmüller theory and introduce geodesic laminations, geodesic currents, and shear coordinates. Geodesic laminations are the objects that provide us a direct link between maximal representations and hyperbolic surfaces. We explain how to associate to a maximal lamination $\lambda$ a pleated set $\widehat{S}_\lambda\subset\mb{H}^{2,n}$ and investigate its topology and causal properties in Section \ref{sec:laminations}. We then focus on the study of the geometry of the pleated sets $\widehat{S}_\lambda$ in Sections \ref{sec:shear cocycles} and \ref{sec:geometry}. 

Geodesic currents and Teichmüller geometry are the main tools that we will use to analyze the length spectrum of maximal representations in Section \ref{sec:teichmuller}.

We end the section by describing (positive) cross ratios on $\partial\Gamma$ and their associated Liouville currents (see Theorem \ref{thm:current of cross-ratio}). As we will see in Section \ref{sec:cross ratio}, every maximal representation $\rho$ has a natural strictly positive cross ratio $\beta^\rho$ induced by the boundary map. As a consequence, its length spectrum can be represented by a Liouville current $\mathscr{L}_\rho$. Our use of cross ratios will be twofold: In Sections \ref{sec:cross ratio} and \ref{sec:shear cocycles}, we use $\beta^\rho$ to define the shear cocycle of a pleated surface. In Section \ref{sec:teichmuller}, we use the Liouville current $\mathscr{L}_\rho$ to study the structure of the set of $\rho$-equivariant pleated surfaces inside Teichm\"uller space.

\subsection{The pseudo-Riemannian space \texorpdfstring{$\mathbb{H}^{2,n}$}{H2n}}\label{subset:pseudo riemann}
We first introduce the linear and projective models of $\mb{H}^{2,n}$ and $\partial\mb{H}^{2,n}$: Let $\mb{R}^{2,n+1}$ denote the vector space $\mb{R}^{n+3}$ endowed with the quadratic form
\[
\langle x,y\rangle_{2,n+1}:=x_1y_1+x_2y_2-x_3y_3-\ldots-x_{n+3}y_{n+3}
\]
of signature $(2,n+1)$. Consider the hyperboloid
\[
\widehat{\hyp}^{2,n}:=\{x\in\mb{R}^{2,n}\left|\;\langle x,x\rangle_{2,n+1}=-1\right.\}.
\]
The restriction of the quadratic form $\langle\bullet,\bullet\rangle_{2,n+1}$ to each tangent space 
\[
T_x\widehat{\hyp}^{2,n}=x^\perp
\] 
has signature $(2,n)$ and, therefore, endows $\widehat{\hyp}^{2,n}$ with a pseudo-Riemannian structure of the same signature. The group $\SOtwon$ acts transitively and by orientation preserving isometries on $\widehat{\hyp}^{2,n}$. However, notice that this action is not proper, since point stabilizers are non-compact.

Tangent vectors $v\in T_x\widehat{\hyp}^{2,n}$ split into three types: 
\[
\text{ $v$ is }\left\{
\begin{array}{l l}
	\text{spacelike} &\text{ if $\langle v,v\rangle_{2,n+1}>0$},\\
	\text{lightlike} &\text{ if $\langle v,v\rangle_{2,n+1}=0$},\\
	\text{timelike} &\text{ if $\langle v,v\rangle_{2,n+1}<0$}.\\
\end{array}
\right..
\]
Similarly, we call a curve $\alpha:I\to\widehat{\hyp}^{2,n}$ spacelike, lightlike, or timelike if ${\dot \alpha}$ is always spacelike, lightlike, or timelike. 

Geodesics in the linear model $\widehat{\hyp}^{2,n}$ are easy to describe: Let $x\in\widehat{\hyp}^{2,n}$ be a point and $v\in T_x\widehat{\hyp}^{2,n}$ a tangent vector. Let $\gamma:\mb{R}\to\widehat{\hyp}^{2,n}$ be the geodesic starting at $x$ with velocity $v$. Then
\[
\gamma(t)=\left\{
\begin{array}{l l}
	\cosh(t)x+\sinh(t)v &\text{ if $\langle v,v\rangle_{2,n+1}=1$},\\
	x+tv &\text{ if $\langle v,v\rangle_{2,n+1}=0$},\\
	\cos(t)x+\sin(t)v &\text{ if $\langle v,v\rangle_{2,n+1}=-1$}.\\
\end{array}
\right..
\]

The pseudo-Riemannian space $\mb{H}^{2,n}$ is the quotient 
\[
\mb{H}^{2,n}:=\widehat{\hyp}^{2,n}/(x\sim-x)
\]
and can be realized as an open subset of the projective space $\mb{RP}^{n+2}$. The projection $\mb{R}^{2,n+1}-\{0\}\to\mb{RP}^{n+2}$ induces the 2-to-1 covering projection $\widehat{\hyp}^{2,n}\to\mb{H}^{2,n}$. In the projective model, the geodesic starting at $x$ with velocity $v$ is just the intersection of the projective line corresponding to the 2-plane ${\rm Span}\{x,v\}$ with $\hyp^{2,n}$. Given two points $x,y\in\mb{H}^{2,n}$, they are always connected by a geodesic, namely, the projective line corresponding to ${\rm Span}\{x,y\}$. The type of the line can be determined using the following simple criterion:

\begin{lem}[{\cite{GM19}*{Proposition 3.2}}] 
	\label{lem:spacelike segments}
	Two distinct points $x,y\in\hyp^{2,n}$ are joined by: 
	\begin{itemize}
		\item{A spacelike geodesic if and only if $|\scal{x}{y}_{2,n+1}|>1$.}
		\item{A lightlike geodesic if and only if $|\scal{x}{y}_{2,n+1}|=1$.}
		\item{A timelike geodesic if and only if $|\scal{x}{y}_{2,n+1}|<1$.}
	\end{itemize}
\end{lem}

An analogous characterization holds also in the linear model $\widehat{\hyp}^{2,n}$:

\begin{lem} 
	\label{lem:spacelike segments linear}
	Two distinct points $x,y\in\widehat{\hyp}^{2,n}$ are joined by:
	\begin{itemize}
		\item{A spacelike geodesic in $\widehat{\hyp}^{2,n}$ if and only if $\scal{x}{y}_{2,n+1} < -1$.}
		\item{A lightlike geodesic in $\widehat{\hyp}^{2,n}$ if and only if $\scal{x}{y}_{2,n+1}= -1$.}
		\item{A timelike geodesic in $\widehat{\hyp}^{2,n}$ if and only if $\abs*{\scal{x}{y}_{2,n+1}} < 1$ or $y = -x$.}
	\end{itemize}
\end{lem}

Similarly to what happens for geodesics, the intersection of a linear space 
\[
\mb{P}\left(V\subset\mb{R}^{2,n+1}\right)\subset\mb{RP}^{n+2}
\]
with $\mb{H}^{2,n}$ also gives a totally geodesic subspace. In particular, every 3-dimensional subspace $V$ on which the restriction of the quadratic form has signature $(2,1)$ provides a totally geodesic subspace of $\mb{H}^{2,n}$ isometric to $\mb{H}^2$. 

The space $\mb{H}^{2,n}$ has a natural boundary at infinity which can be described in the projective model as the projectivization of the cone of isotropic vectors $\partial\mb{H}^{2,n}=\mb{P}(C)$, where 
\[
C:=\{x\in\mb{R}^{2,n}\left|\;\langle x,x\rangle_{2,n+1}=0\right.\}.
\]
In the linear model $\widehat{\mb{H}}^{2,n}$, the boundary at infinity $\partial\widehat{\mb{H}}^{2,n}$ is a two-fold covering of $\partial\mb{H}^{2,n}$, which can be described as $C/x\sim\lambda^2 x$. We can endow $\widehat{\mb{H}}^{2,n}\cup\partial\widehat{\mb{H}}^{2,n}$
with a topology by simultaneously embedding them in the sphere of rays $(\mb{R}^{2,n+1}-\{0\})/x\sim\lambda^2 x$. 

Similarly to what we observed in Lemmas \ref{lem:spacelike segments} and \ref{lem:spacelike segments linear}, we have:

\begin{lem}
	\label{lem:spacelike lines}
	Two distinct points $x\in\partial\hyp^{2,n}, y \in \hyp^{2,n} \cup \partial\hyp^{2,n}$ are joined by: 
	\begin{itemize}
		\item{A spacelike geodesic if and only if $\scal{x}{y}_{2,n+1}\neq 0$.}
		\item{A lightlike geodesic if and only if $\scal{x}{y}_{2,n+1}=0$.}
	\end{itemize}
	
	Similarly, two distinct points $x\in\partial\widehat{\hyp}^{2,n}, y \in \widehat{\hyp}^{2,n} \cup \partial\widehat{\hyp}^{2,n}$ are joined by: 
	\begin{itemize}
		\item{A spacelike geodesic inside $\widehat{\hyp}^{2,n} \cup \partial\widehat{\hyp}^{2,n}$ if and only if $\scal{x}{y}_{2,n+1} < 0$.}
		\item{A lightlike geodesic inside $\widehat{\hyp}^{2,n} \cup \partial\widehat{\hyp}^{2,n}$ if and only if $\scal{x}{y}_{2,n+1} = 0$.}
	\end{itemize}
\end{lem}

Centered at each point $a \in\partial\widehat{\mb{H}}^{2,n}$ we have a family of {\em horoballs} 
\[
O=\{x\in\widehat{\hyp}^{2,n}\left|\; - c <\scal{x}{v}_{2,n+1}< 0 \right.\}\subset\widehat{\mb{H}}^{2,n}
\]
where $v\in\mb{R}^{2,n+1}$ is a representative of $a$ and $c>0$ is a positive constant. The boundary $\partial O \subset \widehat{\hyp}^{2,n}$ is a {\em horosphere} centered at $a$. A \emph{horoball} (resp. a \emph{horosphere}) of $\hyp^{2,n}$ is a subset of the form $\mathrm{P}(O)$, for some horoball $O$ of $\widehat{\hyp}^{2,n}$. 

The terminology is justified by the fact that every spacelike $2$-plane $H \cong \hyp^2$ of $\hyp^{2,n}$ whose boundary at infinity contains $a \in\partial\mb{H}^{2,n}$ intersects $O$ and $\partial O$ in a (usual) horoball or horocycle based at $a \in \partial H$. 

For more material on the geometry of $\mb{H}^{2,n}$ and $\widehat{\mb{H}}^{2,n}$, we also refer to \cite{GM19}*{\S~2} and \cite{CTT19}*{\S~3}.

\subsection{Acausal sets and Poincaré model} \label{subsec:acausal and poincare}
Certain subsets of $\mb{H}^{2,n}$ display some features that make them similar to metric spaces and, in particular, their geometry can be compared to the one of subsets of $\mb{H}^2$.  

\begin{dfn}[Acausal Set]\label{def:acausal}
	Let $X$ be a subset of $\mb{H}^{2,n}\cup\partial\mb{H}^{2,n}$. We say that $X$ is \emph{acausal} if any pair of distinct points in $X$ is joined by a spacelike geodesic inside $\mb{H}^{2,n}\cup\partial\mb{H}^{2,n}$.  Similarly, we say that a subset $\widehat{X}$ of $\widehat{\mb{H}}^{2,n}\cup\partial\widehat{\mb{H}}^{2,n}$ is \emph{acausal} if any pair of distinct points in $\widehat{X}$ is joined by a spacelike geodesic inside $\widehat{\mb{H}}^{2,n}\cup\partial\widehat{\mb{H}}^{2,n}$.
\end{dfn}

\subsubsection*{A natural pseudo-metric}

Acausal subsets are naturally endowed with a pseudo-metric determined by the pseudo-Riemannian structure of $\mb{H}^{2,n}$ (or $\widehat{\mb{H}}^{2,n}$):

\begin{dfn}[Pseudo-Metric]\label{def:pseudo-distance}
	Let $X$ be an acausal subset of $\mb{H}^{2,n}$. For any $x, y \in X$, we define the \emph{pseudo-distance} $d_{\mb{H}^{2,n}}(x,y)$ to be equal to the length of the unique spacelike geodesic segment $[x,y]$ joining them. We will refer to the function $d_{\mb{H}^{2,n}} : X \times X \to [0,\infty)$ as the \emph{pseudo-metric} of $X$. We similarly define the pseudo-metric $d_{\widehat{\hyp}^{2,n}}$ for acausal subsets of $\widehat{\hyp}^{2,n}$.
\end{dfn}

The pseudo-distance between two points $x, y \in \mb{H}^{2,n}$ joined by a spacelike geodesic satisfies the identity
\[
\cosh (d_{\mb{H}^{2,n}}(x,y))= \abs{\langle x,y\rangle_{2,n+1}} ,
\]
and similarly, the pseudo-distance between two points $\hat{x}, \hat{y} \in \widehat{\mb{H}}^{2,n}$ joined by a spacelike geodesic satisfies the identity
\[
\cosh (d_{\widehat{\mb{H}}^{2,n}}(\hat{x},\hat{y}))=- \langle \hat{x},\hat{y}\rangle_{2,n+1} ,
\]

Using this formula, it is simple to check that $d_{\mb{H}^{2,n}}:X\times X\to[0,\infty)$ is continuous, and vanishes precisely on the diagonal $\Delta_X\subset X\times X$. However, we emphasize that the function $d_{\mb{H}^{2,n}}$ (or $d_{\widehat{\mb{H}}^{2,n}}$) does not satisfy the triangle inequality nor its inverse. Nevertheless, such pseudo-metric is compatible with the subset topology of any \emph{closed} acausal subset.

\begin{lem}\label{lem:same topology}
	Let $X$ be a closed acausal subset of $\mb{H}^{2,n} \cup \partial\mb{H}^{2,n}$. Then for any $x \in X \cap \mb{H}^{2,n}$ the family of open sets
	\[
	B_r(x) : = \{ y \in X \cap \mb{H}^{2,n} \mid d_{\mb{H}^{2,n}}(x,y)< r\} ,
	\]
	for $r \leq r_0$, form a fundamental system of neighborhoods for the subset topology of $X \cap \mb{H}^{2,n} \subset \mb{H}^{2,n}$. Up to replacing the role of $\mb{H}^{2,n}$ with $\widehat{\mb{H}}^{2,n}$, the same holds for acausal subsets of $\widehat{\mb{H}}^{2,n} \cup \partial \widehat{\mb{H}}^{2,n}$.
\end{lem}

\begin{proof}
	We will consider the case of a acausal subset $X$ of $\widehat{\mb{H}}^{2,n}$, the same argument applies to acausal subsets of $\hyp^{2,n}$.
	
	Since $d_{\widehat{\mb{H}}^{2,n}}$ is continuous with respect to the subset topology $\mathscr{T}$ of $X \cap \widehat{\mb{H}}^{2,n}$ and vanishes on the diagonal $\Delta \subset (X \cap \widehat{\mb{H}}^{2,n})^2$, every ball $B_r(x)$ contains a small $\mathscr{T}$-neighborhood of $x$. Hence, it is enough to show that every $\mathscr{T}$-neighborhood of $x$ contains a pseudo-ball $B_r(x)$ for a sufficiently small value of $r$.
	
	Assume by contradiction that there exist a $\mathscr{T}$-open set $U$ containing $x$ and a sequence $(x_n)_n$ in $(X \cap \widehat{\mb{H}}^{2,n}) - U$ such that $x_m \in B_{1/m}(x)$ for every $m \in \N$. Since $X$ is a closed subset of $\widehat{\mb{H}}^{2,n} \cup \partial\widehat{\mb{H}}^{2,n}$ (and hence compact), up to subsequence the sequence $(x_m)_m$ converges to some $y \in X - U \subset \widehat{\mb{H}}^{2,n} \cup \partial\widehat{\mb{H}}^{2,n}$. The limit point $y$ cannot lie inside $X \cap \widehat{\mb{H}}^{2,n}$, otherwise $y \neq x$ would satisfy $d_{\widehat{\mb{H}}^{2,n}}(y, x) = 0$, which contradicts the properties of the pseudo-metric $d_{\widehat{\mb{H}}^{2,n}}$. 
	
	If $y \in X \cap \partial \widehat{\mb{H}}^{2,n}$, then there exists a sequence of positive real numbers $(t_m)_m$ converging to $0$ and such that $\lim_m t_m x_m = v \in \R^{2,n+1}$ is an isotropic vector lying in the positive projective class $y \in X \cap \partial {\widehat{\mb{H}}^{2,n}}$. Therefore we must have
	\[
	\scal{x}{v}_{2,n+1} = \lim_{m \to \infty} \scal{x}{t_m x_m}_{2,n+1} = \lim_{m \to \infty} (- t_m \, \cosh(d_{\widehat{\mb{H}}^{2,n}}(x,x_m))) = 0 .
	\]
	Lemma \ref{lem:spacelike lines} implies that the points $x \in X \cap \widehat{\mb{H}}^{2,n}$ and $y = [v] \in X \cap \partial \widehat{\mb{H}}^{2,n}$ are joined by a lightlike geodesic, phenomenon that contradicts the acausality of $X$.
\end{proof}

\subsubsection*{A model for \texorpdfstring{$\widehat{\mb{H}}^{2,n}$}{H2n}}

We now describe the \emph{Poincaré model} of $\widehat{\hyp}^{2,n}$, previously introduced in a similar form by Bonsante and Schlenker \cite{bonsante2010maximal} for anti-de Sitter space (here $n = 1$), and by Collier, Tholozan, and Toulisse \cite{CTT19} for any $n \geq 1$. The Poincaré model will be particularly well suited for our study of acausal subsets (see e.g. Lemma \ref{lem:acausal graph}), and will be used extensively in our computations. 

We start by selecting a spacelike $2$-plane $E\subset\mb{R}^{2,n+1}$. Notice that the bilinear form $\langle\bullet,\bullet\rangle_{2,n+1}$ is negative definite on the orthogonal complement $E^\perp$. We then consider the Euclidean disk
\[
\mb{D}^2:=\{u\in E\mid\scal{u}{u}_{2,n+1}<1\} ,
\]
with closure $\overline{\mathbb{D}}{}^2 = \mb{D}^2 \cup \partial\mb{D}^2 \subset E$, and the negative definite round sphere
\[
\mb{S}^n:=\{v\in E^\perp\mid\scal{v}{v}_{2,n+1}=-1\}.
\]
The Poincaré model of $\hyp^{2,n}$ associated to $E$ is described by the map
\[
\begin{matrix}
	\Psi = \Psi_E: &\mb{D}^2\times \mb{S}^n & \longrightarrow & \widehat{\hyp}^{2,n} \\
	&(u,v) &\longmapsto &\frac{2}{1-\norm{u}^2}u + \frac{1+\norm{u}^2}{1-\norm{u}^2}v.
\end{matrix}
\]
We also introduce
\[
\begin{matrix}
	\partial\Psi: &\partial\mb{D}^2\times\mb{S}^n &\to &\partial\widehat{\mb{H}}^{2,n}\\
	&(u,v) &\longmapsto &u+v ,
\end{matrix}
\]
The main properties of $\Psi$ and $\partial \Psi$ are summarized in the following statement:

\begin{pro}[{\cite{CTT19}*{Proposition~3.5}}] 
	\label{pro:poincare H2n}
	For any spacelike 2-plane $E\subset\mb{R}^{2,n+1}$, we have the following
	\begin{enumerate}[(a)]
		\item{The map $\Psi=\Psi_E$ is a diffeomorphism.}
		\item{The pull-back pseudo-Riemannian metric can be written as
			\[
			\Psi^*g_{\widehat{\hyp}^{2,n}}=\left(\frac{2}{1-\norm{u}^2}\right)^2\abs{\dd u}^2-\left(\frac{1+\norm{u}^2}{1-\norm{u}^2}\right)^2g_{\mb{S}^n}.
			\]
		}
		\item{The map $\partial\Psi$ is a diffeomorphism and extends continuously $\Psi$.}
	\end{enumerate}
\end{pro}

The Poincaré model is especially useful when dealing with acausal subsets for the following reasons: Firstly, acausal subsets $X\subset\widehat{\mb{H}}^{2,n}$ can always be written as graphs of functions $g:U\subset\mb{D}^2\to\mb{S}^n$ that are $1$-Lipschitz with respect to the hemispherical metric
\[
g_{\mb{S}^2}:=\left(\frac{2}{1+\norm{u}^2}\right)^2\abs{\dd u}^2
\]
on $\mb{D}^2$ and the spherical metric on $\mb{S}^n$. Secondly, the graph map $u:U\subset\mb{D}^2\to X\subset\widehat{\mb{H}}^{2,n}$ is $1$-Lipschitz with respect to the hyperbolic metric
\[
g_{\mb{H}^2}:=\left(\frac{2}{1-\norm{u}^2}\right)^2\abs{\dd u}^2
\]
on $\mb{D}^2$ and the intrinsic pseudo-metric on $X\subset\widehat{\mb{H}}^{2,n}$. In both cases, for us, compactness properties of $1$-Lipschitz maps will translate in compactness properties of acausal subsets.

We start with the following lemma:

\begin{lem}
	\label{lem:projection}
	Let $E\subset\mb{R}^{2,n+1}$ be a spacelike $2$-plane, and let $$\overline{\Psi} = \Psi \cup \partial \Psi: \overline{\mathbb{D}}{}^2 \times \mathbb{S}^n \longrightarrow \widehat{\hyp}^{2,n} \cup \partial\widehat{\hyp}^{2,n}$$
	denote the associated Poincaré model. For any $x:=\overline{\Psi}(u,v), x':=\overline{\Psi}(u',v')$, with $x \neq x'$, we have the following:
	\begin{enumerate}
		\item{$x,x'$ are joined by a spacelike segment if and only if
			\[
			d_{\mb{S}^n}(v,v')<d_{\mb{S}^2}(u,u') ,
			\]
			where $d_{\mb{S}^2}$ stands for the hemispherical distance of $\mb{D}^2$.}
		\item{$x,x'$ are joined by a lightlike segment if and only if
			\[
			d_{\mb{S}^n}(v,v')=d_{\mb{S}^2}(u,u') .
			\]
		}
		\item{If $x,x' \in \widehat{\hyp}^{2,n}$ are joined by a spacelike geodesic, then
			\[
			d_{\mb{H}^2}(u,u')\ge d_{\widehat{\mb{H}}^{2,n}}(x,x'),
			\]
			where $d_{\mb{H}^2}$ stands for the hyperbolic distance of $\mb{D}^2$.}
	\end{enumerate}
\end{lem}

\begin{proof}
	The spherical distance between two points $v,v'\in\mb{S}^n$ is computed as follows
	\[
	\cos\left(d_{\mb{S}^n}(v,v')\right)=\scal{v}{v'}_{\mb{S}^n}=-\scal{v}{v'}_{2,n+1}. 
	\]
	Similarly, the hemispherical distance between $u,u'\in\overline{\mb{D}}{}^2$ is given by
	\[
	\cos\left(d_{\mb{S}^2}(u,u')\right)=\frac{2}{1+\norm{u}^2}\frac{2}{1+\norm{u'}^2}\scal{u}{u'}_{2,n+1}+\frac{1-\norm{u}^2}{1+\norm{u}^2}\frac{1-\norm{u'}^2}{1+\norm{u'}^2}.
	\]
	We start by providing expressions for $\scal{x}{x'}_{2,n+1} = \scal{\overline{\Psi}(u,v)}{\overline{\Psi}(u',v')}_{2,n+1}$, depending on whether $x, x'$ belong to $\widehat{\hyp}^{2,n}$ or $\partial\widehat{\hyp}^{2,n}$: If both $x, x' \in \widehat{\hyp}^{2,n}$, then
	\begin{align*}
		\scal{x}{x'}_{2,n+1} & = \frac{1+\norm{u}^2}{1-\norm{u}^2} \frac{1+\norm{u'}^2}{1-\norm{u'}^2}
		\left( \frac{2}{1+\norm{u}^2} \frac{2}{1+\norm{u'}^2} \scal{u}{u'}_{2,n+1}+\scal{v}{v'}_{2,n+1} \right) \\
		& = \frac{1+\norm{u}^2}{1-\norm{u}^2} \frac{1+\norm{u'}^2}{1-\norm{u'}^2} \left( \cos\left(d_{\mb{S}^2}(u,u')\right) - \cos\left(d_{\mb{S}^n}(v,v')\right) \right) - 1 .
	\end{align*}
	If $x \in \widehat{\hyp}^{2,n}$ and $x' \in \partial\widehat{\hyp}^{2,n}$, then
	\begin{align*}
		\scal{x}{x'}_{2,n+1} & = \frac{1+\norm{u}^2}{1-\norm{u}^2} 
		\left( \frac{2}{1+\norm{u}^2} \scal{u}{u'}_{2,n+1}+\scal{v}{v'}_{2,n+1} \right) \\
		& = \frac{1+\norm{u}^2}{1-\norm{u}^2} \left( \cos\left(d_{\mb{S}^2}(u,u')\right) - \cos\left(d_{\mb{S}^n}(v,v')\right) \right) .
	\end{align*}
	Finally, if both $x, x' \in \partial\widehat{\hyp}^{2,n}$, then
	\begin{align*}
		\scal{x}{x'}_{2,n+1} & = \scal{u}{u'}_{2,n+1}+\scal{v}{v'}_{2,n+1} \\
		& = \cos\left(d_{\mb{S}^2}(u,u')\right) - \cos\left(d_{\mb{S}^n}(v,v')\right) .
	\end{align*}

	\begin{property}{\it (1)}
		If both $x$ and $x'$ lie in $\widehat{\hyp}^{2,n}$, then there exists a spacelike geodesic segment joining them if and only if $\scal{x}{x'}_{2,n+1} < - 1$ by Lemma \ref{lem:spacelike segments linear}. By the explicit expression for $\scal{x}{x'}_{2,n+1}$ that is provided above, we see that $\scal{x}{x'}_{2,n+1} < - 1$ if and only if
		\[
		\cos\left(d_{\mb{S}^2}(u,u')\right) < \cos\left(d_{\mb{S}^n}(v,v')\right) .
		\]
		
		If $x \in \widehat{\hyp}^{2,n}$ and $x' \in \partial \widehat{\hyp}^{2,n}$, then $x$ and $x'$ are joined by a spacelike geodesic if and only if $\scal{x}{x'}_{2,n+1} < 0$, again by Lemma \ref{lem:spacelike segments linear}. By the identity given above, this holds if and only if $\cos\left(d_{\mb{S}^2}(u,u')\right) < \cos\left(d_{\mb{S}^n}(v,v')\right)$. The same equivalence holds in the case of $x ,x' \in \partial\widehat{\hyp}^{2,n}$. 
		
		Since $d_{\mb{S}^2}(u,u') , d_{\mb{S}^n}(v,v') \in [0,\pi]$, in each case we conclude that $x$ and $x'$ are joined by a spacelike geodesic if and only if
		\[
		d_{\mb{S}^n}(v,v') < d_{\mb{S}^2}(u,u') .
		\]
	\end{property}
	
	\begin{property}{\it (2)}
		If $x, x' \in \widehat{\hyp}^{2,n}$, then $x,x'$ are joined by a lightlike geodesic if and only if $\langle x,x'\rangle_{2,n+1} = - 1$ by Lemma \ref{lem:spacelike segments linear}. On the other hand, if $x \in \partial \widehat{\hyp}^{2,n}$ and $x' \in \widehat{\hyp}^{2,n} \cup \partial\widehat{\hyp}^{2,n}$, then $x$ and $x'$ are joined by a lightlike geodesic if and only if $\scal{x}{x'}_{2,n+1} = 0$. The conclusion then follows from a computation analogous to the one of Property (1).
	\end{property}
	
	\begin{property}{\it (3)}
		If $x,x' \in \widehat{\hyp}^{2,n}$ are joined by a spacelike geodesic, then $\langle x,x'\rangle_{2,n+1}<-1$, and their pseudo-distance in $\widehat{\hyp}^{2,n}$ is given by $$\cosh\left(d_{\widehat{\mb{H}}^{2,n}}(x,x')\right) =- \scal{x}{x'}_{2,n+1}.$$
		
		On the other hand, the hyperbolic distance in $\mb{D}^2$ can be expressed as follows:
		\begin{align*}
			\cosh\left(d_{\mb{H}^2}(u,u')\right)=-\frac{2}{1-\norm{u}^2}\frac{2}{1-\norm{u'}^2}\scal{u}{u'}_{2,n+1}+\frac{1+\norm{u}^2}{1-\norm{u}^2}\frac{1+\norm{u'}^2}{1-\norm{u'}^2}.
		\end{align*}
		Since $v,v'\in\mb{S}^n$, we have $\abs{\scal{v}{v'}}\le 1$. Therefore, we conclude that
		\begin{align*}
			\cosh\left(d_{\widehat{\mb{H}}^{2,n}}(x,x')\right) & = - \scal{x}{x'}_{2,n+1} \\
			& = \cosh(d_{\mb{H}^2}(u,u')) - \frac{1+\norm{u}^2}{1-\norm{u}^2}\frac{1+\norm{u'}^2}{1-\norm{u'}^2} \left( 1 + \scal{v}{v'}_{2,n+1}\right) \\
			& \leq \cosh\left(d_{\mb{H}^2}(u,u')\right) ,
		\end{align*}
		with equality if and only if $v=v'$.
	\end{property}
	This concludes the proof of the statement.
\end{proof}

We will now show that acausal subsets can be described as graphs of $1$-Lipschitz functions. This is the content of the next lemma (appeared for the anti-de Sitter $3$-space case in \cite{MR4264588}*{Lemma~4.1.2}, and in the case of smooth spacelike surfaces inside $\widehat{\mathbb{H}}^{2,n}$ in \cite{CTT19}*{Proposition~3.8}):

\begin{lem}
	\label{lem:acausal graph}
	Let $E\subset\mb{R}^{2,n+1}$ be a spacelike 2-plane, and let $\overline{\Psi} : \overline{\mathbb{D}}{}^2 \times \mathbb{S}^n \to \widehat{\hyp}^{2,n} \cup \partial\widehat{\hyp}^{2,n}$
	be the associated Poincaré model. If $$\pi : \widehat{\hyp}^{2,n} \cup \partial\widehat{\hyp}^{2,n} \longrightarrow \overline{\mb{D}}{}^2$$
	denotes the composition of $\overline{\Psi}^{-1}$ with the projection onto the first factor, then for every acausal subset $X \subset \widehat{\hyp}^{2,n} \cup \partial\widehat{\hyp}^{2,n}$:
	\begin{enumerate}
		\item{The projection $\pi:X\to\overline{\mb{D}}{}^2$ is injective. In particular, we can write $X$ as the graph of a function $g:\pi(X)\subset\overline{\mb{D}}{}^2\to\mb{S}^n$.}
		\item{The function $g$ is strictly $1$-Lipschitz with respect to the hemispherical metric on $\overline{\mb{D}}^2$ and the standard spherical metric on $\mb{S}^n$, that is,
			\[
			d_{\mb{S}^n}(g(u), g(u'))<d_{\mb{S}^2}(u,u')
			\]
			for any distinct $u,u'\in\pi(X)$.}
		\item{Vice versa, the graph of any strictly $1$-Lipschitz function $g:U\subset\mb{D}^2\to\mb{S}^n$ defined on a connected subset $U\subset\mb{D}^2$ is an acausal subset.}
	\end{enumerate}
\end{lem}

\begin{proof}
	Any pair of distinct points $x = \overline{\Psi}(u,v), x' = \overline{\Psi}(u',v') \in X$ is joined by a spacelike geodesic. Hence, by Property (1) of Lemma \ref{lem:projection}, we must have
	\[
	d_{\mathbb{S}^2}(\pi(x), \pi(x')) > d_{\mathbb{S}^n}(v, v') \geq 0 ,
	\]
	and so the restriction of $\pi$ to $X$ is injective. In particular, for any $u\in\pi(X )\subseteq \overline{\mb{D}}{}^2$, there exists a unique $g(u)\in\mb{S}^n$ such that $\overline{\Psi}(u,g(u))\in X$.
	Again by property (1) of Lemma \ref{lem:projection}, the function $g : \overline{\mathbb{D}}{}^2 \to \mathbb{S}^n$ is strictly $1$-Lipschitz.
	
	Vice versa, if $g : \overline{\mathbb{D}}{}^2 \to \mathbb{S}^n$ is strictly $1$-Lipschitz function, then the expressions for $\scal{x}{x'}_{2,n+1}$ provided in the proof of Lemma \ref{lem:projection}, combined with Lemma \ref{lem:spacelike segments linear}, show that the graph of $g$ is an acausal subset of $\widehat{\mathbb{H}}^{2,n} \cup \partial\widehat{\mathbb{H}}^{2,n}$.
\end{proof}

We now restrict our attention to some special acausal subsets, namely spacelike geodesics and spacelike planes, and prove a couple of topological properties that will be useful later on:

\begin{lem}
	\label{lem:planes and lines}
	Let $E\subset\mb{R}^{2,n+1}$ be a spacelike 2-plane with associated Poincaré model $\Psi:\mb{D}^2\times\mb{S}^n\to\widehat{\hyp}^{2,n}$, and let $\pi:\widehat{\hyp}^{2,n}\to\mb{D}^2$ be the composition of $\Psi^{-1}$ with the projection onto the first factor. Then
	\begin{enumerate}
		\item{If $H\subset\widehat{\mb{H}}^{2,n}$ is a spacelike plane, then the restriction of $\pi:H\to\mb{D}^2$ is a diffeomorphism and extends continuously to $\partial H\to\partial\mb{D}^2$.}
		\item{If $\ell\subset\widehat{\mb{H}}^{2,n}$ is a spacelike geodesic, then $\pi(\ell)$ is a smooth properly embedded curve. Either it is a diameter in $\mathbb{D}^2$ or it intersects every diameter at most once.}
	\end{enumerate}
\end{lem}

\begin{proof}
	We prove the properties in order.
	
	\begin{property}{\it (1)}
		Since $H$ is a spacelike surface, it is transverse to the (negative definite) fibers $\Psi(\{x\}\times\mb{S}^n)$ of $\pi$, and the restriction $\pi:H\to\mb{D}^2$ is a local diffeomorphism. By Property (3) of Lemma \ref{lem:projection}, we also have that $\pi$ is distance non-decreasing when we endow $\mb{D}^2$ with the hyperbolic metric, since the pseudo-distance $d_{\widehat{\mb{H}}^{2,n}}$ restricts to the hyperbolic distance on $H$. In particular, $\pi$ is proper and injective. Together, the two facts imply that $\pi:H\to\mb{D}^2$ is a diffeomorphism. 
	\end{property}
	
	\begin{property}{\it (2)}
		By the previous point, $\pi:\ell\to\mb{D}^2$ is a smooth proper embedding. We claim that the projection $\pi(\ell)$ is either a diameter of $\mb{D}^2$, or it intersects every diameter of $\mb{D}^2$ at most once. In order to see this, parametrize $\ell$ as $\ell(t)=e^ta+e^{-t}b$ for some $a, b \in \partial \hyp^{2,n}$, and write $a=u_a+v_a$ and $b=u_b+v_b$ with $u_a,u_b\in\partial\mb{D}^2$ and $v_a,v_b\in\mb{S}^n$. The projection of $\ell(t)$ to $\mb{D}^2$ is a curve $u = u(t)$ satisfying 
		\[
		e^tu_a+e^{-t}u_b=\frac{2}{1-\norm{u(t)}^2}u(t).
		\]
		In particular, $u(t)$ intersects a line $pu^1+qu^2=0$ if and only if $p(e^tu_a^1+e^{-t}u_b^2)+q(e^tu_a^1+e^{-t}u_b^2)=0$. This equation has at most one solution unless $\pi(\ell)=\{pu^1+qu^2=0\}$, i.e. if $\pi(\ell)$ is a diameter.
	\end{property}
	This concludes the proof of the statement.
\end{proof}

We conclude our analysis by studying extension and convergence properties of families of continuous functions on acausal subsets of $\widehat{\hyp}^{2,n}$. We start with the following definition:

\begin{dfn}[Uniform Continuity on Acausal Sets] \label{def:uniformly continuous}
	Let $X$ be an acausal subset of $\widehat{\hyp}^{2,n}$, and let $(\mathcal{Y},d_{\mathcal{Y}})$ be a complete metric space. A function $f : X \to \mathcal{Y}$ is \emph{uniformly continuous} if for every $\varepsilon > 0$ there exists a $\delta > 0$ such that
	\[
	d_{\widehat{\mb{H}}^{2,n}}(x,y) < \delta \Rightarrow d_{\mathcal{Y}}(f(x), f(y)) < \varepsilon
	\] 
	for every $x, y \in X$. 
	
	A family of functions $\{f_i : X \to \mathcal{Y}\}_{i \in I}$ is \emph{equicontinuous} if for every $\varepsilon > 0$ there exists a $\delta > 0$ such that
	\[
	d_{\widehat{\mb{H}}^{2,n}}(x,y) < \delta \Rightarrow d_{\mathcal{Y}}(f_i(x), f_i(y)) < \varepsilon
	\] 
	for every $x, y \in X$ and $i \in I$.
\end{dfn}

With this notion, we can formulate and prove the following statement:

\begin{lem}
	\label{lem:extension and convergence}
	Let $X$ be a closed acausal subset of $\widehat{\mathbb{H}}^{2,n}$. Let $\mathcal{Y}$ be a complete metric space. 
	\begin{enumerate}
		\item{Let $f:D\subset X\to \mathcal{Y}$ be a uniformly continuous function defined on a dense subset $D\subset X$, then $f$ extends continuously to $X$.}
		\item{Let $f_n:X\to \mathcal{Y}$ be a sequence of equicontinuous functions that converges at some point $x\in X$. Then, up to subsequences, $f_n$ converges to $f:X\to \mathcal{Y}$ with the same modulus of continuity of $(f_n)_n$.}
	\end{enumerate}
\end{lem}

\begin{proof}
	We work in a fixed Poincaré model of $\widehat{\mb{H}}^{2,n}$. Represent $X$ as the graph of a 1-Lipschitz function $g:U\subset\mb{D}^2\to\mb{S}^n$ defined over the projection $U:=\pi(X)$. Observe that $\pi:X\to U$ is a homeomorphism with inverse given by the graph map $u:U\to X$ given by $u(x)=(x,g(x))$. Endow $\mb{D}^2$ with the hyperbolic metric and recall that, by property (3) of Lemma \ref{lem:projection}, $u$ is 1-Lipschitz with respect to the hyperbolic metric on $\mb{D}^2$ and the pseudo-metric on $X$.
	
	\begin{property}{\it (1)}
		Consider $h:=fu:\pi(D)\subset U\to \mathcal{Y}$. Since $f$ is uniformly continuous and $u$ is 1-Lipschitz, we have that $h$ is uniformly continuous with respect to the hyperbolic metric of $\mathbb{D}^2$ and the metric of $\mathcal{Y}$. Therefore $h$ extends continuously to the closure of $U$ inside $\mathbb{D}^2$ and, hence, $f=h\pi$ extends continuously to $X$.
	\end{property}
	
	\begin{property}{\it (2)}
		Consider $h_n:=f_nu:U\to \mathcal{Y}$. Since the family $f_n$ is equicontinuous and $u$ is 1-Lipschitz, we have that the family $h_n$ is equicontinuous as well (with respect to the hyperbolic metric), and therefore extends to an equicontinuous family $\bar{h}_n$ defined on the closure $\overline{U}$ of $U$ in $\mathbb{D}^2$. As $(h_n)_n$ converges on $\pi(x)$, by Ascoli-Arzelà the sequence $\bar{h}_n$ converges to $h:\overline{U}\to \mathcal{Y}$ uniformly on all compact subsets of $\overline{U}$, up to subsequences. Hence, $f_n=h_n\pi$ converges up to subsequences to $f=h\pi$. 
	\end{property}
	
	This concludes the proof of the statement.
\end{proof}

\subsection{Maximal representations}\label{subsec:max reprs}
We now introduce maximal representations in $\SOtwon$ and a couple of geometric objects that are naturally associated to them. 

The first geometric object which one can attach to every representation $\rho:\Gamma\to\SOtwon$ is a flat vector bundle $V_\rho\to\Sigma$. The total space $V_\rho$ is defined as follows: 
\[
V_\rho:=\widetilde{\Sigma}\times\mb{R}^{2,n+1}/(x,v)\sim(\gamma x,\rho(\gamma)v).
\]

Here $\widetilde{\Sigma}$ is the universal covering of $\Sigma$ and $\gamma$ acts on it as a deck transformation. The bundle projection $V_\rho\to\Sigma$ is just the one induced by the the universal covering projection $\widetilde{\Sigma}\to\Sigma$.

The vector bundle $V_\rho\to\Sigma$ has an associated cohomological invariant $T(\rho)\in\mb{Z}$, called the {\em Toledo invariant} (see \cite{BIW10}). The number $T(\rho)$ always satisfies a Milnor-Wood inequality $|T(\rho)|\le 2|\chi(\Sigma)|$. 

\begin{dfn}[Maximal Representation]
	A representation $\rho$ is called {\em maximal} if it satisfies $|T(\rho)|=2|\chi(\Sigma)|$.
\end{dfn}

By the work of Burger, Iozzi, Labourie, and Wienhard \cite{BILW}, we can equivalently describe maximal representations in terms of equivariant limit maps. Here we will mostly adopt this more geometric perspective.

\begin{thm}[Burger, Iozzi, Labourie, and Wienhard \cite{BILW}] 
	\label{thm:maximal limit curve}
	A representation $\rho:\Gamma\to\SOtwon$ is maximal if and only if there exists a $\rho$-equivariant H\"older continuous embedding
	\[
	\xi:\partial\Gamma\longrightarrow\partial\mb{H}^{2,n}
	\]
	such that $\Lambda_\rho:=\xi(\partial\Gamma)$ is an {\rm acausal curve}, meaning that for every triple of distinct points $u,v,w\in\partial\Gamma$, the subspace of $\mb{R}^{2,n+1}$ spanned by the lines $\xi(u),\xi(v),\xi(w)$ has signature $(2,1)$.
\end{thm}

The second geometric object that we associate to a maximal representation $\rho$ is a pseudo-Riemannian manifold $M_\rho$ locally modeled on $\mb{H}^{2,n}$: Even though the group $\rho(\Gamma)$ does not act properly discontinuously on the whole $\mb{H}^{2,n}$, it admits a nice domain of discontinuity. In fact, every maximal representation $\rho$ in $\SOtwon$ is convex cocompact in the sense of \cite{DGK18} and \cite{DGK17}: In the projective model $\mb{H}^{2,n}\subset\mb{P}(\mb{R}^{2,n+1})$, the representation $\rho$ preserves a properly convex open subset $\Omega_\rho\subset\mb{H}^{2,n}$, whose $\mc{C}^1$-boundary $\partial\Omega_\rho$ contains the limit curve $\Lambda_\rho = \partial\Omega_\rho \cap \partial {\hyp}^{2,n}$, and its action is cocompact on the convex hull $\mc{CH}(\Lambda_\rho)\cap {\hyp}^{2,n} \subset \Omega_\rho$ of the limit curve.

As the representation acts by projective transformations on $\Omega_\rho$, it preserves the natural Hilbert metric on the convex domain and, hence, the action on $\Omega_\rho$ is properly discontinuous. Furthermore, since every $\rho(\gamma)$ acts by isometries on $\Omega_\rho$ and has an attracting fixed point on $\Lambda_\rho\subset\partial\Omega_\rho$, it follows that $\rho(\gamma)$ cannot have fixed points in $\Omega_\rho$, so  the action is also free. In conclusion, since the action is free and properly discontinuous, we can associate to $\rho$ the pseudo-Riemannian manifold $M_\rho:=\Omega_\rho/\rho(\Gamma)$. The quotient 
\[
\mc{CC}(M_\rho):=(\mc{CH}(\Lambda_\rho)\cap{\hyp}^{2,n})/\rho(\Gamma) \subset M_\rho
\]
is the convex core of $M_\rho$.

The convex set $\Omega_\rho$ is by no means unique. However, the convex hull $\mc{CH}(\Lambda_\rho) \cap \hyp^{2,n}\subset\Omega_\rho$ does not depend on the choice of $\Omega_\rho$. Therefore, the geometry of the convex core $\mc{CC}(M_\rho)$ is also independent of the choice of the domain $\Omega_\rho$. 

\begin{rmk}\label{rmk:limit curves lift}
	Let us observe that, since the set
	\[
	\Omega_\rho\cup\Lambda_\rho
	\subset\mb{H}^{2,n}\cup\partial\mb{H}^{2,n}
	\]
	is simply connected, it admits a lift $\widehat{\Omega}_\rho \cup \widehat{\Lambda}_\rho \subset \widehat{\mb{H}}^{2,n}\cup\partial\widehat{\mb{H}}^{2,n}$. By Theorem \ref{thm:maximal limit curve}, such topological $\widehat{\Lambda}_\rho \subset \partial\widehat{\hyp}^{2,n}$ is a $\rho$-invariant acausal subset as in Definition \ref{def:acausal}.
\end{rmk}

\subsection{Hyperbolic surfaces and Teichmüller space}
When $n=0$, Goldman \cite{G80} has shown that maximal representations in ${\rm SO}_0(2,1)$ correspond exactly to holonomies of hyperbolic structures on $\Sigma$. 

We will denote by $\T$ the classical Teichmüller space that parametrizes such hyperbolic structures on $\Sigma$ up to isotopy. We recall that our goal is to relate the geometry of maximal representations to the one of hyperbolic surfaces.  

In this section we collect some facts from classical Teichmüller theory that will be needed later on starting from geodesic laminations which are one of our main tools.

\subsubsection{Geodesic laminations} \label{subsec:geo lam}
We start with some familiar properties of the hyperbolic plane: Every geodesic on $\mb{H}^2$ can be uniquely identified by its pair of endpoints on $\partial\mb{H}^2$. 

\begin{dfn}[Space of Geodesics]
	The \emph{space of (unoriented) geodesics} of $\mb{H}^2$ is
	\[
	\mc{G}:= (\partial\hyp^2 \times \partial\hyp^2)/(x,y)\sim(y,x).
	\]
\end{dfn}

Given two geodesics $\ell,\ell'\in \mathcal{G}$ we can also describe their relative position by looking at the configuration of their endpoints at infinity. More precisely:

\begin{dfn}[Crossing and Disjoint]\label{def:crossing and disjoint}
	Let $a,b,a',b'\in S^1$ be four points on a circle such that $a\neq b$ and $a'\neq b'$. We say that the pairs $(a,b)$ and $(a',b')$ are {\em crossing} if $a',b'$ are contained in distinct components of $S^1-\{a,b\}$ and {\em disjoint} otherwise.
\end{dfn}

We now fix once and for all a reference hyperbolic structure on $\Sigma$ and identify $\partial\mb{H}^2$ with the Gromov boundary $\partial\Gamma$.

\begin{dfn}[Geodesic Lamination]\label{def:geodesic laminations}
	A \emph{geodesic lamination of $\hyp^2$} is a closed subset $\lambda$ of $\mc{G}$ such that every pair of geodesics $\ell,\ell'\in\lambda$ is disjoint, as in Definition \ref{def:crossing and disjoint}. If $\Sigma$ is a closed orientable hyperbolic surface of genus $>1$, then a \emph{geodesic lamination of $\Sigma$} is a $\Gamma$-invariant geodesic lamination of $\hyp^2$, where $\Gamma = \pi_1(\Sigma)$.
	
	The elements of a lamination $\lambda$ are called the {\em leaves} of $\lambda$. We denote by $\lambda_0$ the \emph{geometric realization of $\lambda$}, namely the subset of $\hyp^{2}$ obtained as the union of all leaves of $\lambda$. In addition, every connected component of $\mb{H}^2-\lambda_0$ is called a \emph{plaque} of $\lambda$. We say that a geodesic lamination $\lambda$ is \emph{maximal} if every plaque of $\lambda$ is equal to the interior of an ideal triangle in $\hyp^2$.
\end{dfn}

\begin{rmk}
	We recall that, while geodesic laminations of a closed surface are uniquely determined by their geometric realizations (see e.g. \cite{CEG}*{Chapter~I.4}), the same does not hold in general for geodesic laminations of $\hyp^2$. For instance, there exist distinct geodesic laminations whose geometric realizations coincides with the entire hyperbolic plane $\hyp^2$ (see e.g. \cite{CEG}*{Definition~II.2.4.1}).
\end{rmk}

We denote by $\mc{GL}$ the space of geodesic laminations on $\Sigma$. As geodesic laminations are closed subsets of $\mc{G}$, the space $\mc{GL}$ is naturally endowed with the Chabauty (or Hausdorff) topology. It is a standard fact (see \cite{CEG}*{Proposition~I.4.1.7}) that $\mc{GL}$ is compact with respect to this topology. For a more detailed exposition on geodesic laminations in hyperbolic surfaces we refer to \cite{CEG}*{Chapter~I.4}. 

\subsubsection{Geodesic currents and measured laminations} \label{subsec:geodesic currents}
Geodesic currents were introduced by Bonahon \cite{Bo88}. They are defined as follows: 

\begin{dfn}[Geodesic Currents]
	A {\em geodesic current} is a $\Gamma$-invariant locally finite Borel measure $\mu$ on $\mc{G}$. We denote by $\mc{C}$ the space of geodesic currents.
\end{dfn}

The space $\mc{C}$ has the structure of a cone and it possesses a natural weak-$\star$ topology. Furthermore, it is endowed with a natural continuous symmetric bilinear form
\[
i(\bullet,\bullet):\mc{C}\times\mc{C}\to[0,\infty),
\]
called the {\em intersection form}. We briefly recall its definition, as we will use it later on: Let $\mc{J}\subset\mc{G}\times\mc{G}$ be the space of pairs of crossing geodesics $(\ell,\ell')$. The group $\Gamma$ acts properly discontinuously on $\mc{J}$. Any pair of geodesic currents $\alpha,\beta\in\mc{C}$ induces a $\Gamma$-invariant measure $\alpha\times\beta$ on $\mc{J}$ and, hence, a well defined measure $\alpha\times\beta$ on the quotient $\mc{J}/\Gamma$ which it is possible to show to be always finite. The intersection between $\alpha$ and $\beta$ is defined as 
\[
i(\alpha,\beta):=(\alpha\times\beta)(\mc{J}/\Gamma).
\]

\begin{dfn}[Measured Lamination]
	A {\em measured lamination} on $\Sigma$ is a geodesic current $\mu\in\mc{C}$ such that $i(\mu,\mu)=0$. 
\end{dfn}

It is a standard fact, that, with this definition, the support of a measured lamination is a geodesic lamination of $\Sigma$ (see Bonahon \cite{Bo88}*{Proposition~17}). We denote by $\mc{ML}$ the space of measured laminations on $\Sigma$.

Bonahon shows that the following natural objects associated to $\Sigma$ embed canonically in $\mc{C}$:  
\begin{itemize}
	\item{The space $\mc{S}$ of free homotopy class of closed curves of $\Sigma$.} 
	\item{The space $\T$ of isotopy classes of hyperbolic metrics on $\Sigma$.}
\end{itemize}
We will make no distinction between a point in these spaces and its image in the space of currents $\mc{C}$. Bonahon also proves that, with respect to the intersection form $i(\bullet,\bullet)$ we have the following relations: 
\begin{itemize}
	\item{If $\alpha,\beta\in\mc{S}$, then $i(\alpha,\beta)$ is the geometric intersection number between $\alpha,\beta$.} 
	\item{If $\alpha\in\mc{S}$ and $X\in\T$, then $i(X,\alpha)=L_X(\alpha)$ is the length of $\alpha$ on $X$.}
\end{itemize}
In particular, the intersection form provides a continuous extension of the length function $L_X(\bullet):\mc{S}\to(0,\infty)$ to a continuous positive function on the space of geodesic currents as $i(X,\bullet):\mc{C}\to(0,\infty)$. For more details on such properties, we refer to Bonahon \cite{Bo88}.

\subsubsection{Shear coordinates}\label{subsec:shear coordinates}
Let $\lambda$ be a maximal lamination of $\Sigma$. Following Bonahon \cite{Bo96}, we have the following definition:

\begin{dfn}[H\"older Cocycle]\label{def:holder cocycle}
	A {\em H\"older cocycle} transverse to $\lambda$ is a real-valued function on the set of pairs of distinct plaques of ${\hat \lambda}$ that satisfies:
	\begin{enumerate}
		\item{Symmetry: For every pair of distinct plaques $P,Q$, we have $\sigma(P,Q)=\sigma(Q,P)$.}
		\item{Additivity: For every pair of distinct plaques $P,Q$, and for every plaque $R$ that separates $P$ from $Q$, we have $\sigma(P,Q)=\sigma(P,R)+\sigma(R,Q)$.}
		\item{Invariance: For every pair of distinct plaques $P,Q$ and for every $\gamma\in\Gamma$, we have $\sigma(P,Q)=\sigma(\gamma P, \gamma Q)$.}
	\end{enumerate}
	We denote by $\mathcal{H}(\lambda;\R)$ the space of H\"older cocycles transverse to $\lambda$, which has a natural structure of real vector space of dimension $3\vert\chi(\Sigma)\vert$ by \cite{Bo97transv}*{Theorem~15}.
\end{dfn}

H\"older cocycles are a useful device that allow to encode among other things the following data:
\begin{itemize}
	\item{Every hyperbolic metric $X\in\T$ has an associated shear cocycle $\sigma_\lambda^X\in\mc{H}(\lambda;\mb{R})$ that describes the relative position of the plaques of $X-\lambda$.}
	\item{Every measured lamination $\mu\in\mc{ML}$ with support contained in $\lambda$ has an associated transverse cocycle $\mu\in\mc{H}(\lambda;\mb{R})$ and a length functional $L_\mu:\mc{H}(\lambda;\mb{R})\to\mb{R}$ whose evaluation on shear cocycles $\sigma_\lambda^X$ coming from hyperbolic metrics $X\in\T$ equals $L_X(\mu)$.}
\end{itemize}
We refer to Bonahon \cite{Bo96} for the details of the construction.

The space of H\"older cocycles transverse to $\lambda$ is naturally endowed with a symplectic form $\omega_\lambda(\bullet,\bullet)$, called the \emph{Thurston's symplectic form}, which essentially generalizes the notion of intersection between geodesic currents to transverse H\"older distributions, in the sense of \cite{RS75}. The form $\omega_\lambda$ can be also described concretely in terms of the classical algebraic intersection between $1$-chains on a surface, interpretation that will be recalled in Section \ref{subsubsec:thurston sympl} from the work of Bonahon \cite{Bo96}.

The Thurston symplectic form was deployed by Bonahon \cite{Bo96} to relate the notion of shear cocycle $\sigma^X_\lambda$ associated to a hyperbolic structure $X \in \T$ with the notion of hyperbolic length for measured laminations. Concretely, we have that for every $X\in\T$ and $\mu\in\mc{ML}$ with support contained in $\lambda$, the following relation holds:
\[
\omega(\sigma_\lambda^X,\mu)=L_X(\mu)
\]
(see in particular \cite{Bo96}*{Theorem~E}). 

The Thurston symplectic form is particularly relevant in the study of shear cocycles because it provides a complete characterization of the set of transverse H\"older cocycles that can be realized as shears of hyperbolic metrics. Inspired by ideas of Thurston \cite{T86}, Bonahon proved the following parametrization result:

\begin{thm}[Bonahon {\cite{Bo96}*{Theorems A, B}}]
	\label{thm:thurston bonahon}
	For any maximal geodesic lamination $\lambda$ of $\Sigma$, the map
	\[
	\begin{matrix}
		\T &\longrightarrow & \mathcal{H}(\lambda;\R) \\
		X & \longmapsto & \sigma^X_\lambda
	\end{matrix}
	\]
	is a real analytic diffeomorphism. The image of the map is the open convex cone 
	\[
	C:=\{\sigma\in\mc{H}(\lambda,\mb{R})\left|\;\omega(\mu,\sigma)>0\text{ \rm for every }\mu\in\mc{ML}\text{ \rm with }{\rm supp}(\mu)\subset\lambda\right.\}
	\]
	where $\omega(\bullet,\bullet)$ is the Thurston's symplectic form on $\mc{H}(\lambda;\mb{R})$.
\end{thm}

The resulting set of coordinates for Teichmüller space are called {\em shear coordinates} relative to $\lambda$. 

Let us mention that, Bonahon and Sözen \cite{SB01} proved that the pullback of the Thurston's symplectic form $\omega$ via the above diffeomorphism is (a multiple of) the Weil-Petersson symplectic form on Teichmüller space. In Section \ref{sec:teichmuller}, we will use the Weil-Petersson geometry of Teichmüller to study the length spectrum of maximal representations $\rho:\Gamma\to\SOtwon$.

When dealing with different spaces of H\"older cocycles $\mc{H}(\lambda;\mb{R})$ relative to nearby laminations $\lambda\in\mc{GL}$, it is useful to identify all such spaces with the space of real weights $\mc{W}(\tau;\mb{R})$ of a suitable train track $\tau$ carrying all the laminations considered. This is particularly convenient when studying continuity properties of maps $\lambda\in\mc{GL}\to\sigma_\lambda\in\mc{H}(\lambda;\mb{R})$ as we will need later on.

Thus, we now briefly introduce train tracks and systems of real weights. 

\subsubsection{Train tracks}
\label{subsubsec:train tracks}

We recall the necessary terminology (see e.g. \cites{PH92,Bo97transv,Bo97geo,BD17}). We define a \emph{branch} inside $\Sigma$ to be a homeomorphism $\varphi : [0,1] \times [0,1] \to B$ (which we abusively identify with its image $B$). We refer to: (the images of) the curves $t \mapsto \varphi(t, \bullet)$ as the \emph{ties} of the branch $B$, to $\partial_v B : = \varphi(\{0,1\} \times [0,1])$ and $\partial_h B : = \varphi([0,1] \times \{0,1\})$ as its \emph{vertical} and \emph{horizontal} boundaries, respectively, and to the images of the points in $\{0,1\} \times \{0,1\}$ through the map $\varphi$ as its \emph{vertices}.

We then define a (trivalent) \emph{train track} $\tau$ as a closed subset of $\Sigma$ that can be decomposed into the union of a finite number of branches $(B_i)_i$ satisfying the following conditions:
\begin{enumerate}[\it i)]
	\item every connected component of the intersection $B_i \cap B_j$ between two distinct branches coincides with a component of $\partial_v B_i$, it is strictly contained in a component of $\partial_v B_j$, and it contains exactly one vertex of $B_j$ (up to exchanging the roles of $i$ and $j$);
	\item for every $i$, each vertex of $B_i$ is contained in the vertical boundary of some branch $B_j$, with $i \neq j$;
	\item no complementary region of the interior of $\tau$ is homeomorphic to a disc that intersects 0, 1 or 2 distinct components of the vertical boundaries of the branches $(B_i)_i$.
\end{enumerate}

Any tie of a branch $B_i$ of $\tau$ that is not strictly contained inside a connected component of the vertical of some (possibly different) branch $B_j$ will be simply called a \emph{tie} of the train track $\tau$. The \emph{horizontal boundary} $\partial_h \tau$ of $\tau$ is defined as the union of the horizontal boundaries of its branches, and the closure of $\partial \tau - \partial_h \tau$ is called the \emph{vertical boundary} $\partial_v \tau$ of $\tau$. The ties of $\tau$ that contain a component of $\partial_v \tau$ are called \emph{switches}. A switch coincides with a connected component of the vertical boundary of some branch $B_i$ in $\tau$, and strictly contains two components of the vertical boundary of some branches $B_j, B_k$ of $\tau$ (possibly two of the three branches $B_i, B_j, B_k$ coincide). Moreover, every switch contains exactly one connected component $c$ of the vertical boundary of $\tau$.

If $\tilde{\tau}$ is the preimage of $\tau$ in the universal cover of $\Sigma$, then a branch of $\tilde{\tau}$ is simply the lift of some branch of $\tau$. Similarly we define the ties, the switches, the vertical and horizontal boundary of $\tilde{\tau}$ and of its branches. We say that a train track $\tau$ carries a lamination $\lambda$ if $\lambda$ is contained in the interior of $\tau$ and every tie of $\tau$ is transverse to the leaves of $\lambda$.

Train tracks come naturally together with a vector space of real weights as we now describe. 

\subsubsection{Systems of real weights}
\label{subsubsec:real weights}

Given $\tau$ a trivalent train track of $\Sigma$, a {\em system of real weights} $\eta = (\eta_i)_i$ of $\tau$ is a real-valued function on the set of branches $(B_i)_i$ of $\tau$ that satisfies a natural linear constraint for every switch of $\tau$ (compare with \cite{Bo97geo}, or \cite{Bo96}*{\S~3}): For any switch $s$, let $B^s_i, B^s_j, B^s_k$ be the branches of $\tau$ adjacent to $s$, and assume that $s$ coincides with a connected component of the vertical boundary of the branch $B^s_i$. If $\eta^s_i, \eta^s_j, \eta^s_k$ denote the weights associated by $\eta$ with $B^s_i, B^s_j, B^s_k$, respectively, then we require $\eta$ to satisfy $\eta^s_i = \eta^s_j + \eta^s_k$, for any switch $s$ of $\tau$.

We denote by $\mathcal{W}(\tau;\R)$ the space of systems of real weights of $\tau$. Observe that $\mathcal{W}(\tau;\R)$ is naturally endowed with a real vector space structure, and its dimension is completely determined by the topology of $\tau$ (see \cite{Bo97transv}*{Theorem~15}). In particular, if $\tau$ carries a maximal lamination, which will be the only case we will be interested in, then $\mathcal{W}(\tau;\R) \cong \R^{- 3 \chi(S)}$. 

For any maximal lamination $\lambda'$ carried by $\tau$, there exists a natural isomorphism $\mathcal{H}(\lambda';\R) \cong \mathcal{W}(\tau;\R)$, which can be described as follows: let $\alpha$ be a H\"older cocycle transverse to $\lambda'$, and let $B_i$ be a branch of $\tau$. Select arbitrarily a lift $\widetilde{B}_i$ of $B_i$ to the universal cover $\widetilde{\Sigma}$, and select a tie $k_i$ of $\widetilde{B}_i$ disjoint from its vertical boundary. Since $\tau$ carries $\lambda'$, there exist two distinct plaques $P_i', Q_i'$ of $\lambda'$ whose interior contain the endpoints of $k_i$. Then we define the real weight of $\alpha$ associated with $B_i$ to be $\alpha_i : = \alpha(P_i', Q_i') \in \R$. By the properties of H\"older cocycles (see Definition~\ref{def:holder cocycle}), it is easy to check that the weight $\alpha_i$ does not depend on the choice of the lift of $k_i$, and the weights $(\alpha_i)_i$ satisfy the switch conditions described above. The corresponding map $\mathcal{H}(\lambda;\R) \to \mathcal{W}(\tau;\R)$ is a linear isomorphism, as shown in \cite{Bo97transv}*{Theorem~11}.

The space of real weights $\mathcal{W}(\tau;\R)$ provides us a way to compare shear cocycles associated with distinct maximal geodesic laminations that are close with respect to the Hausdorff topology. Indeed, if $(\lambda_m)_m$ is a sequence of maximal geodesic laminations that converges to $\lambda$, and $\lambda$ is carried by a train track $\tau$, then for $m$ sufficiently large $\tau$ carries $\lambda_m$. In particular, we have isomorphisms $\mathcal{H}(\lambda_m;\R) \cong \mathcal{W}(\tau;\R) \cong \mathcal{H}(\lambda;\R)$.

\subsection{Cross ratios}
\label{subsec:cross ratios definitions}

Our use of cross ratios will be twofold: On the one hand, we will use them to abstractly define the shear cocycles of our pleated surfaces (the basic computation will be exploited in Remark \ref{rmk:shear_is_shear}). On the other hand, they will also help us in the study of the length spectrum of a maximal representation $\rho$ as they provide a natural Liouville current $\mathscr{L}_\rho$ such that $i(\mathscr{L}_\rho,\bullet)$ extends continuously the length spectrum $L_\rho(\bullet)$ from the space of closed geodesics $\mc{S}$ to the space of geodesic currents $\mc{C}$.  

Let us remark that cross ratios are also objects of interests in their own and have been widely used to study maximal and Hitchin representations \cites{Lab08,MZ19,BIPP21}, and questions about length spectrum rigidity of negatively curved manifolds \cites{Ot90,ledrappier,H99,H97}.  

We now introduce these objects formally. Observe that the Gromov boundary $\partial \Gamma$ admits a natural H\"older structure. To see this, recall that the choice of a Fuchsian representation $\hat{\rho} : \Gamma \to \psl$ determines a unique $\hat{\rho}$-equivariant homeomorphism $\phi_{\hat{\rho}} : \partial \Gamma \to \partial \hyp^2$. Different choices of Fuchsian representations $\hat{\rho}, \hat{\rho}'$ provide homeomorphisms $\phi_{\hat{\rho}}, \phi_{\hat{\rho}'}$ that differ by post-composition with a quasi-symmetric homeomorphism of $\partial \hyp^2 \cong \rp^1$. Since quasi-symmetric homeomorphisms are bi-H\"older continuous with respect to any choice of a Riemannian distance on $\partial \hyp^2$, the notion of H\"older continuous functions $f : \partial \Gamma \to \R$ is independent of the choice of the Fuchsian representation $\hat{\rho}$, and therefore intrinsic of the $\Gamma$-space $\partial \Gamma$.

\begin{dfn}[Cross Ratio]\label{def:crossratio}
	Let $\partial \Gamma^{(4)}$ denote the space of 4-tuples $(u,v,w,z) \in (\partial \Gamma)^4$ satisfying $u \neq z$ and $v \neq w$. A \emph{cross ratio} is a H\"older continuous function $\beta : \partial \Gamma^{(4)} \to \R$ that satisfies the following properties:
	\begin{enumerate}[\it i)]
		\item $\beta$ is $\Gamma$-invariant with respect to the diagonal action of $\Gamma$ on $\partial \Gamma^{(4)}$, i.e. $\gamma \cdot (u,v,w,z) = (\gamma u, \gamma v, \gamma w, \gamma z)$ for any $(u,v,w,z) \in \partial \Gamma^{(4)}$;
		\item For every $u,v,w,z,x \in \partial \Gamma$ we have
		\begin{align}
			\begin{split}
				\beta(u,v,w,z) & = 0 \quad \Leftrightarrow \text{$u = w$ or $v = z$}, \\
				\beta(u,u,w,z) & = \beta(u,v,w,w) = 1, \\
				\beta(u,v,w,z) & = \beta(w,z,u,v) , \\
				\beta(u,v,w,z) & = \beta(u,v,x,z) \beta(u,v,w,x) , \\
				\abs{\beta(u,v,w,z)} & = \abs{\beta(u,w,v,z) \beta(u,z,w,v)} ,
			\end{split} \label{eq:crossshear}
		\end{align}
		whenever the 4-tuples appearing above belong to $\partial \Gamma^{(4)}$.
	\end{enumerate}
\end{dfn}

\begin{rmk}
	Observe that the second and fourth relations in \eqref{eq:crossshear} imply that for any 4-tuple of pairwise distinct points $u, v, w, z \in \partial \Gamma$ we have
	\begin{equation}\label{eq:inverse crossratio}
		\beta(u,v,z,w) = \beta(u,v,w,z)^{-1} .
	\end{equation}
	In turn, relation \eqref{eq:inverse crossratio} and the third symmetry in \eqref{eq:crossshear} imply that
	\begin{equation} \label{eq:symmetry for shear}
		\beta(u,v,w,z) = \beta(v,u,z,w) .
	\end{equation}
\end{rmk}

We alert the reader of the existence of multiple non-equivalent definitions of cross ratios in the literature. For the reader's convenience, we summarize the relations between Definition~\ref{def:crossratio} and other notions in the literature in Appendix \ref{other cross ratios}.

We now recall the notion of positive cross ratios from \cite{H99} (see also \cite{MZ19}):

\begin{dfn}[Positive Cross Ratio]\label{def:positive cross ratio}
	A cross ratio $\beta : \partial \Gamma^{(4)} \to \R$ is said to be \emph{positive} if for every 4-tuple of pairwise distinct cyclically ordered points $x, y, w, z \in \partial \Gamma$ it satisfies $\beta(x,y,z,w) \geq 1$. We say that $\beta$ is \emph{strictly positive} if for every 4-tuple $x, y, w, z \in \partial \Gamma$ as above we have $\beta(x,y,z,w) > 1$.
\end{dfn}

A positive cross ratio has a natural notion of length functions associated to any non-trivial element $\gamma \in \Gamma$. We briefly recall its definition:

\begin{dfn}[Period of a Cross Ratio] \label{def:period}
	Let $\beta : \partial \Gamma^{(4)} \to \R$ be a cross ratio. For any $\gamma \in \Gamma - \{e\}$ we define the \emph{$\beta$-period of $\gamma$} to be
	\[
	L_\beta(\gamma) : = \log \abs{\beta(\gamma^+, \gamma^-, x, \gamma x)} ,
	\]
	for some $x \in \partial \Gamma - \{\gamma^+, \gamma^-\}$, where $\gamma^+$ and $\gamma^-$ denote the attracting and repelling fixed points of $\gamma$ in $\partial \Gamma$.
\end{dfn}

It is simple to deduce from the symmetries of a cross ratio (see in particular \eqref{eq:crossshear}, \eqref{eq:inverse crossratio}, \eqref{eq:symmetry for shear}) that the quantity $L_\beta(\gamma)$ does not depend on the choice of $x \in \partial \Gamma - \{\gamma^+, \gamma^-\}$, and it satisfies $L_\beta(\gamma) = L_\beta(\gamma^{-1}) = L_\beta(\delta \gamma \delta^{-1})$ for any $\gamma, \delta \in \Gamma$, with $\gamma \neq e$.

As observed by Hamenstädt \cite{H97} (see also Martone-Zhang \cite{MZ19}), any positive cross ratio $\beta$ uniquely determines a geodesic current compatible with its period functions, as described by the following result:

\begin{thm}[{\cite{H97}*{Lemma~1.10}, \cite{MZ19}*{Appendix~A}}]
	\label{thm:current of cross-ratio}
	Every positive cross ratio $\beta:\partial\Gamma^{(4)}\to\mb{R}$ is represented by a geodesic current $\mathscr{L}_\beta \in \mc{C}$, that is, for every $\gamma\in\Gamma - \{e\}$ we have
	\[
	L_\beta(\gamma) = i(\mathscr{L}_\beta,\gamma) ,
	\]
	where $L_\beta(\gamma)$ denotes the $\beta$-period of $\gamma$.
\end{thm}

The geodesic current $\mathscr{L}_\beta$ will be called the \emph{Liouville current} of $\beta$, in analogy with the terminology introduced by Bonahon \cite{Bo88}*{\S~2} in the case of hyperbolic structures on closed surfaces. By Theorem \ref{thm:current of cross-ratio}, the non-negative function
\[
\begin{matrix}
	L_\beta : & \mathcal{C} & \longrightarrow & \R \\
	& c & \longmapsto & i(\mathscr{L}_\beta, c) 
\end{matrix}
\]
naturally extends the $\beta$-period functions to the entire space of geodesic currents. Moreover we have:

\begin{lem}\label{lem: positive lengths}
	If $\beta$ is a strictly positive cross ratio, then $L_\beta(c) > 0$ for any non-trivial geodesic current $c \in \mathcal{C}$. 
\end{lem}

\begin{proof}
	The first part of the assertion follows immediately from the definition of the intersection form on geodesic currents (see Section \ref{subsec:geodesic currents}). Consider now a non-trivial geodesic current $c \in \mathcal{C}$, and select a leaf $\ell'$ in the support of $c$, namely a point inside $\mathcal{G}$ for which all neighborhoods have positive $c$-measure. We now choose a geodesic $\ell \in \mathcal{G}$ that crosses $\ell'$. We can find small intervals $I, J$ and $I', J'$ inside $\partial \Gamma$ around the endpoints of $\ell$ and $\ell'$, respectively, so that every pair of geodesics $h \in I \tilde{\times} J : = (I \times J)/\sim$ and $h' \in I' \tilde{\times} J' : = (I' \times J')/\sim$ is crossing (here $\sim$ denotes the equivalence relation $(x,y) \sim (y,x)$ on $(\partial \Gamma \times \partial \Gamma)/\sim$, compare with Section \ref{subsec:geo lam}). Accordingly with the notation introduced in Section \ref{subsec:geodesic currents}, we have
	\[
	i(\mathscr{L}_\beta, c) = (\mathscr{L}_\beta \times c)(\mathcal{J}/\Gamma) ,
	\]
	where $\mathcal{J}$ denotes the set of pairs of crossing geodesics in $\mathcal{G} \times \mathcal{G}$. Since $\Gamma$ acts freely and properly discontinuously on $\mathcal{J}$, up to choosing smaller intervals $I, J, I', J'$, we can assume that the Borel-measurable set
	\[
	\mathcal{K} : = (I \tilde{\times} J) \times (I' \tilde{\times} J') \subset \mathcal{J}
	\]
	projects injectively inside $\mathcal{J}/\Gamma$. In particular we have
	\begin{align*}
		(\mathscr{L}_\beta \times c)(\mathcal{J}/\Gamma) & \geq (\mathscr{L}_\beta \times c)(\mathcal{K}) \\
		& = (\mathscr{L}_\beta)(I \tilde{\times} J) \cdot c(I' \tilde{\times} J') .
	\end{align*}
	By construction $c(I' \tilde{\times} J') > 0$, so it is enough to show that $\mathscr{L}_\beta(I \tilde{\times} J) > 0$. This is in fact a direct consequence of the definition of the Liouville current $\mathscr{L}_\beta$, and the fact that $\beta$ is strictly positive: indeed the measure $\mathscr{L}_\beta$ satisfies
	\[
	\mathscr{L}_\beta([a,b] \tilde{\times} [c,d]) = \log \beta(a,b,c,d)
	\]
	for any pair of disjoint intervals $[a,b], [c,d]$ in $\partial\Gamma$, where $a, b, c, d$ are cyclically ordered (compare with \cite{H97}*{Lemma~1.10}, \cite{MZ19}*{Appendix~A}). Therefore, being $\beta$ strictly positive, we immediately conclude that $\mathscr{L}_\beta(I \tilde{\times} J) > 0$, and consequently $i(\mathscr{L}_\beta, c) > 0$, as desired.
\end{proof}

%%%

\section{Laminations and pleated sets}
\label{sec:laminations}

In order to understand the geometry of maximal representations and relate it to Teichmüller space $\T$, we study certain $1$- and $2$-dimensional invariant objects contained in the pseudo-Riemannian symmetric space $\mb{H}^{2,n}$, namely {\em geodesic laminations} and {\em pleated sets}. 

Recall that, by Theorem~\ref{thm:maximal limit curve}, the image of the limit map $\xi : \partial\Gamma \to \partial\hyp^{2,n}$ of every maximal representation $\rho$ describes a $\rho$-invariant topological circle $\Lambda_\rho : = \xi(\partial\Gamma) \subset \partial \hyp^{2,n}$ satisfying the following condition: Every triple of distinct points in $\Lambda_\rho$ generates a subspace of $\R^{2,n+1}$ of signature $(2,1)$. Since a consistent part of the results on geodesic laminations and pleated sets in $\mb{H}^{2,n}$ that we focus on in the current section rely only on this specific property of the set $\Lambda_\rho$, we introduce the following notion:

\begin{dfn}[Acausal Curve]
	Let $\Lambda$ be a topological circle embedded in $\partial\hyp^{2,n}$. We say that $\Lambda$ is an \emph{acausal curve} if every triple of distinct points of $\Lambda$ generates a subspace of $\R^{2,n+1}$ with signature $(2,1)$.
\end{dfn}

Notice that acausal curves provide specific examples of acausal sets inside $\partial\widehat{\mb{H}}^{2,n}$ in the sense of Definition \ref{def:acausal}. Indeed we have:

\begin{lem}[{\cite{DGK17}}] \label{lem:acausal lifts}
	Every acausal curve $\Lambda$ is entirely contained in an affine chart of $\mathbb{H}^{2,n} \subset \rp^{n+2}$. In particular, there exists an acausal subset $\widehat{\Lambda}$ inside $\partial\widehat{\hyp}^{2,n}$ such that the natural projection $\widehat{\Lambda} \to \Lambda$ is a homeomorphism with respect to their respective subset topologies.
\end{lem}

\begin{proof}
	This assertion is a corollary of \cite{DGK17}*{Proposition~1.10}. Notice that the notion of acausal subsets that we use and the notion of negative subsets of Danciger, Guéritaud, and Kassel \cite{DGK17} agree, as proved in \cite{DGK17}*{Lemma~3.2}.
\end{proof}

In the same way that geodesic laminations in $\hyp^2$ can be identified with suitable sets of unoriented pairs of points in $\partial\Gamma$ (see Section~\ref{subsec:geo lam}), we define:

\begin{dfn}[$\Lambda$-Lamination]
	Let $\Lambda$ be an acausal curve inside $\partial \hyp^{2,n}$. A \emph{$\Lambda$-lamination} is a closed subset $\lambda\subset \mathcal{G}_\Lambda = (\Lambda\times\Lambda-\Delta)/(x,y)\sim(y,x)$  such that every pair of points $(a,b),(a',b')\in\lambda$ gives disjoint pairs on $\Lambda$. 
	
	A $\Lambda$-lamination $\lambda$ is {\em maximal} if there is no geodesic $\ell\in\mc{G}_\Lambda-\lambda$ which is disjoint from all the leaves of $\lambda$. We denote by $\mc{G}_\Lambda^m$ the space of maximal $\Lambda$-laminations.
	
	As $\Lambda$ is an acausal curve, every $(a,b)\in\Lambda\times\Lambda-\Delta$ represents a spacelike line $[a,b]\subset\mb{H}^{2,n}$. We define the \emph{geometric realization of a $\Lambda$-lamination $\lambda$} in $\mb{H}^{2,n}$ to be ${\hat \lambda}=\cup_{(a,b)\in\lambda}{[a,b]}$. 
\end{dfn}

Notice that the geometric realization ${\hat \lambda}$ of $\lambda$ is a closed subset of $\mb{H}^{2,n}$ contained in the convex hull $\mc{CH}(\Lambda)\subset\mb{H}^{2,n}$ of the acausal curve $\Lambda\subset\partial\mb{H}^{2,n}$ (which is well defined as, by Lemma \ref{lem:acausal lifts}, $\Lambda$ is contained in a properly convex set).  

The main class of geodesic laminations and pleated sets that we will focus on in the rest of the exposition are those arising from maximal representations $\rho$ and their associated pseudo-Riemannian manifolds $M_\rho$:

\begin{dfn}[$\rho$-Lamination]
	Let $\rho:\Gamma\to\SOtwon$ be a maximal representation with associated acausal curve $\Lambda_\rho\subset\partial\mb{H}^{2,n}$. A \emph{$\rho$-lamination} is a $\rho(\Gamma)$-invariant $\Lambda_\rho$-lamination. 
\end{dfn}

Whenever a $\Lambda$-lamination $\lambda$ is maximal, it is possible to "fill" the geometric realization ${\hat \lambda}$ of $\lambda$ with totally geodesic spacelike ideal triangles of $\hyp^{2,n}$, one for each complementary triangle of $\lambda_0$ inside $\hyp^2$. Together, the geometric realizations of $\lambda$ and of its complementary regions form the \emph{pleated set} $\widehat{S}_\lambda$ of $\lambda$. 

\begin{dfn}[Pleated Set]
	Let $\Lambda\subset\partial\mb{H}^{2,n}$ be an acausal curve and let $\lambda$ be a maximal $\Lambda$-lamination. The \emph{pleated set associated to $\lambda$} is the set $\widehat{S}_\lambda$ obtained as the union of ${\hat \lambda}$ with all spacelike triangles $\Delta$ bounded by leaves of ${\hat \lambda}$.
\end{dfn}

Let us now give a brief outline of the content of Section \ref{sec:laminations}: 
\begin{description}
	\item[\S~\ref{subsec:crossing and acausality}] We start by establishing the existence of geodesic laminations and discuss their causal structure and topological features: As it turns out, geometric realizations of $\Lambda$-laminations are always acausal subsets (see Proposition \ref{pro:lamination acausal}). To see this, it is in fact sufficient to investigate the causal structure of a set of the form $\ell \cup \ell'$, where $\ell,\ell'\subset\mb{H}^{2,n}$ are two geodesics with distinct endpoints on $\Lambda\subset\partial\mb{H}^{2,n}$. In particular, we observe in Lemma \ref{lem:four points} that the set  $\ell\cup\ell'$ is acausal if and only if their endpoints $\partial\ell,\partial\ell'$ are not crossing, as in Definition \ref{def:crossing and disjoint}.
	
	\item[\S~\ref{subsec:pleated sets}] We prove the acausality of pleated sets in Proposition \ref{pro:existence pleated sets}. The proof relies on the acausality of ${\hat \lambda}$ in combination with purely topological arguments. Let $\mb{D}^2\times\mb{S}^n$ be a Poincaré model of $\widehat{\mb{H}}^{2,n}$ and let $\pi : \widehat{\mb{H}}^{2,n} \to \mb{D}^2$ denote its associated projection. By analyzing the complementary regions of $\pi(\hat{\lambda})$ inside $\mathbb{D}^2$, we show that the restriction of the induced projection $\pi:\widehat{\mb{H}}^{2,n}\to\mb{D}^2$ to the pleated set $\widehat{S}_\lambda$ is bijective, and in particular $\widehat{S}_\lambda$ is equal to the graph of some function $g_\lambda:\mb{D}^2\to\mb{S}^n$. This fact, in combination with the acausality properties of $\hat{\lambda}$, implies that $\widehat{S}_\lambda$ is acausal and, hence, that $g_\lambda$ is a strictly $1$-Lipschitz function with respect to the spherical metrics. In particular, we deduce that pleated sets are always nicely embedded Lipschitz subsurfaces of $\mb{H}^{2,n}$.
	
	\item[\S~ \ref{subsec:continuity pleated sets}] We establish the continuous dependence of pleated sets with respect to the choice of maximal laminations. More precisely, we show that the map
	\[
	\begin{matrix}
		\mc{G}_\Lambda^m & \longrightarrow & {\rm Lip}_1(\mb{D}^2,\mb{S}^n) \\
		\lambda & \longmapsto & g_\lambda
	\end{matrix}
	\]
	is continuous with respect to the Chabauty topology on the source, and the topology of uniform convergence over compact subsets of $\mb{D}^2$ on the target (see Proposition \ref{pro:continuity pleated sets}). This property will play an important role in the study of the geometry of pleated sets, as described in Section~\ref{sec:geometry}.
	
	\item[\S~\ref{subsec:bending locus}] We conclude the first part of our analysis on pleated sets with the study of their \emph{bending locus}, namely the subset on which a pleated set $\widehat{S}_\lambda$ is folded. The bending locus does not necessarily coincide with the entire maximal lamination $\lambda$, but it always describes a sublamination of it (see Proposition \ref{pro:bending locus}). In the $\rho$-invariant setting, the notion of bending locus will allow us to characterize the set of (homotopy classes of) curves inside a pleated set $S_\lambda = \widehat{S}_\lambda/\rho(\Gamma) \subset M_\rho$ whose lengths are strictly dominated by the lengths of their geodesic representatives inside $M_\rho$.
\end{description}

We emphasize that every statement appearing in Sections \ref{subsec:crossing and acausality}, \ref{subsec:pleated sets}, and \ref{subsec:continuity pleated sets} apply to general $\Lambda$-laminations $\lambda$, and in particular no $\rho$-invariance of $\Lambda$ or $\lambda$ are required. On the other hand, the content of Section \ref{subsec:bending locus} applies to pleated sets that are invariant by the action of some maximal representation.

\subsection{Crossing geodesics and acausality}\label{subsec:crossing and acausality}

We start our analysis by showing that the topological property of spacelike geodesics with endpoints in a acausal curve $\Lambda$ of being crossing or disjoint has an immediate interpretation in terms of their acausal structure inside the pseudo-Riemannian space $\mb{H}^{2,n}$. More precisely, we have:

\begin{lem}
	\label{lem:four points}	
	Let $a,b,a',b'\in\Lambda$ be four distinct points on an acausal curve $\Lambda\subset\partial\mb{H}^{2,n}$ such that the geodesics $[c,d]$ with $c,d\in\{a,b,a',b'\}$ are all spacelike. Then:
	\begin{enumerate}[(i)]
		\item{The pairs $(a,b)$ and $(a',b')$ are {\em disjoint} if and only if the geodesics $\ell=[a,b],\ell'=[a',b']$ are disjoint and the subset $\ell\cup\ell'\subset\mb{H}^{2,n}$ is acausal.}
		\item{The pairs $(a,b)$ and $(a',b')$ are {\em crossing} if and only if there is a timelike geodesic which is orthogonal to both geodesics $\ell=[a,b]$ and $\ell'=[a',b']$.}
	\end{enumerate} 
\end{lem}

\begin{proof}
	By Lemma \ref{lem:acausal lifts}, we can find a lift $\widehat{\Lambda}$ of the acausal curve $\Lambda$ in $\widehat{\mb{H}}^{2,n}$, so that $\widehat{\Lambda}$ is an acausal subset of $\widehat{\mb{H}}^{2,n}$, in the sense of Definition \ref{def:acausal}. In particular, we can find representatives of $a,b,a',b'$ (which we continue to denote by $a,b,a',b'$ with abuse) inside the isotropic cone of $\langle\bullet,\bullet\rangle$, so that their pairwise scalar products are all negative. We can then parametrize $\ell$ by $\ell(t)=(e^ta+e^{-t}b)/\sqrt{-2\langle a,b\rangle}$ and $\ell'$ by $\ell'(s)=(e^sa'+e^{-s}b')/\sqrt{-2\langle a',b'\rangle}$. With these parametrizations we have: 
	\begin{align*}
		-2\sqrt{\langle a,b\rangle\langle a',b'\rangle}\cdot\langle\ell(t)&,\ell'(s)\rangle = \\
		&=-\langle e^ta+e^{-t}b,e^sa'+e^{-s}b'\rangle\\
		&=-e^{t+s}\langle a,a'\rangle-e^{t-s}\langle a,b'\rangle-e^{-t+s}\langle b,a'\rangle-e^{-t-s}\langle b,b'\rangle.
	\end{align*}
	
	Since the four products $\langle a,a'\rangle,\langle a,b'\rangle,\langle b,a'\rangle,\langle b,b'\rangle$ are all negative, the function $(t,s)\mapsto-\langle\ell(t),\ell'(s)\rangle$ is proper. As a consequence, it has a global minimum $m>0$. A small computation shows that the function has a unique critical point $(t_0,s_0)$, which must coincide with the point of minimum, where it assumes the value 
	\begin{align*}
		m &=-\langle\ell(t_0),\ell'(s_0)\rangle\\
		&=\sqrt{\frac{\langle a,a'\rangle\langle b,b'\rangle}{\langle a,b\rangle\langle a',b'\rangle}}+\sqrt{\frac{\langle a,b'\rangle\langle b,a'\rangle}{\langle a,b\rangle\langle a',b'\rangle}} .
	\end{align*}
	
	Let $e_1,\cdots e_{n+3}$ be the canonical basis of $\mb{R}^{2,n+1}$. Up to isometries and rescaling, it is enough to consider the following setting:
	\begin{align*}
		a & =-e_2+e_3, & a'&=e_1+e_3, \\
		b & =e_2+e_3, & b'&=\alpha e_1+\beta e_2+\gamma e_3+u,
	\end{align*}
	with $u$ in the linear span of $e_4,\cdots,e_{n+2}$ and $\alpha^2+\beta^2=1$. Furthermore, by acausality we must have $\gamma>\alpha,\beta$ and $\beta+\gamma>0$. We now identify $\widehat{\mb{H}}^{2,n}$ with $\mb{D}^2\times\mb{S}^n$ via the chart $\Psi_E$ induced by the orthogonal decomposition $\mb{R}^{2,n+1}=E\oplus F$ where $E={\rm Span}\{e_1,e_2\}$ and $F={\rm Span}\{e_3,\cdots,e_{n+3}\}$ (see Section~\ref{subsec:acausal and poincare}). Under such identification we have 
	\[
	\pi_{\mb{D}^2}(a)=(0,-1),\pi_{\mb{D}^2}(b)=(0,1),\pi_{\mb{D}^2}(a')=(0,1),\pi_{\mb{D}^2}(b')=(\alpha,\beta)\in\partial\mb{D}^2,
	\]
	and the above expression becomes
	\[
	m = \sqrt{\frac{\langle a,a'\rangle\langle b,b'\rangle}{\langle a,b\rangle\langle a',b'\rangle}}+\sqrt{\frac{\langle a,b'\rangle\langle b,a'\rangle}{\langle a,b\rangle\langle a',b'\rangle}}=\sqrt{\frac{\gamma-\beta}{2(\gamma-\alpha)}}+\sqrt{\frac{\beta+\gamma}{2(\gamma-\alpha)}}.
	\]
	Notice that $\pi_{\mb{D}^2}$ induces a homeomorphism between $\widehat{\Lambda}$ and $\partial\mb{D}^2$ so that the cyclic order of $a,b,a',b'$ on $\widehat{\Lambda}$ is the same as the cyclic order of their projections to $\partial\mb{D}^2$.
	
	A little algebraic manipulation shows that $m \le 1$ if and only if
	\[
	\sqrt{\gamma^2-\beta^2}\le-\alpha.
	\]
	Also notice that, as $b'$ is isotropic, we always have $0=|b'|^2=\alpha^2+\beta^2-\gamma^2-|u|^2$ with $|u|^2>0$ and, hence, $\sqrt{\gamma^2-\beta^2}\le|\alpha|$. 
	
	Thus, we have the following two cases: If $\alpha>0$, that is, if the pairs $(a,b)$ and $(a',b')$ are disjoint, then the inequality $\sqrt{\gamma^2-\beta^2}<-\alpha$ is never satisfied and, hence, $m>1$. Therefore, $\ell(t),\ell'(s)$ are always distinct and the geodesic segments $[\ell(t),\ell'(s)]$ are always spacelike which implies that the subset $\ell\cup\ell'$ is acausal.
	
	If $\alpha<0$, that is, if the pairs $(a,b)$ and $(a',b')$ are crossing, then the inequality $\sqrt{\gamma^2-\beta^2}\le-\alpha$ is always satisfied and, hence $m\le 1$. In this case, the geodesic segment $[\ell(t_0),\ell'(s_0)]$ is timelike and, as $\langle\ell(t_0),\ell'(s_0)\rangle$ realizes the minimum of $\langle\ell(t),\ell'(s)\rangle$, it is also orthogonal to $\ell,\ell'$ at $\ell(t_0)$ and $\ell'(s_0)$. 
\end{proof}

As a consequence of Lemma \ref{lem:four points}, we immediately get:

\begin{pro}
	\label{pro:lamination acausal}
	Let $\Lambda\subset\partial\mb{H}^{2,n}$ be an acausal curve and let ${\hat \lambda}$ be the geometric realization of a $\Lambda$-lamination $\lambda\in\mc{G}_\Lambda$. Then ${\hat \lambda}$ is a proper acausal subset.
\end{pro}

\subsection{Pleated sets}\label{subsec:pleated sets}
We now study the topology and causal structure of a pleated set associated to a maximal lamination. A priori, a pleated set can be a very complicated topological subspace of $\mb{H}^{2,n}$. We now show that, instead, Proposition \ref{pro:lamination acausal} forces a good topological behavior:

\begin{pro} 
	\label{pro:existence pleated sets}
	Let $\Lambda\subset\partial\mb{H}^{2,n}$ be an acausal curve and let $\lambda\in\mc{G}_\Lambda^m$ be a maximal $\Lambda$-lamination. Then its associated pleated set $\widehat{S}_\lambda \subset \mathbb{H}^{2,n}$ is a topological Lipschitz acausal subsurface.
\end{pro}

\begin{proof}
	We lift $\Lambda$ to an acausal curve $\widehat{\Lambda}\subset\partial\widehat{\mb{H}}^{2,n}$. The proof strategy is as follows: We show that in every Poincaré model $\mb{D}^2\times\mb{S}^n$ of $\widehat{\mb{H}}^{2,n}$ the pleated set $\widehat{S}_\lambda$ is the graph of a function $g:\mb{D}^2\to\mb{S}^n$. We deduce that $\widehat{S}_\lambda$ is achronal which implies that the function $g$ is $1$-Lipschitz with respect to the spherical metrics on $\mb{D}^2$ and $\mb{S}^n$. Therefore, being a graph of a Lipschitz function, the pleated set $\widehat{S}_\lambda$ is a Lipschitz subsurface. Acausality will follow from the fact that $\widehat{S}_\lambda$ does not contain lightlike segments. 
	
	Let 
	\[
	\Psi:\overline{\mb{D}}{}^2\times\mb{S}^n\longrightarrow\widehat{\mb{H}}^{2,n}\cup\partial\widehat{\mb{H}}^{2,n}
	\]
	be the Poincaré model associated to a splitting $\mb{R}^{2,n+1}=E\oplus E^\perp$ with $E$ spacelike 2-plane, and let $\pi:\widehat{\mb{H}}^{2,n}\cup\partial\widehat{\mb{H}}^{2,n}\to\overline{\mb{D}}{}^2$ denote the composition of $\Psi^{-1}$ with the projection onto the first factor. Recall that, as $\widehat{\Lambda}$ is an acausal subset of $\partial\widehat{\mb{H}}^{2,n}$, the projection $\pi$ restricts to a homeomorphism $\pi:\widehat{\Lambda}\to\partial\mb{D}^2$. 
	
	Observe that the restriction of $\pi$ to ${\hat \lambda}$ is proper, being $\pi$ a fibration with compact fibers and ${\hat \lambda}$ a closed subset of $\widehat{\mb{H}}^{2,n}$, and injective, by Proposition \ref{pro:lamination acausal} and Lemma \ref{lem:acausal graph}. Hence, the map $\pi : \hat{\lambda} \to \mathbb{D}^2$ is a homeomorphism onto its image $\pi({\hat \lambda})$, which is a closed subset of $\mathbb{D}^2$. Notice also that Lemma \ref{lem:planes and lines} implies that the image of any leaf $\hat{\ell}$ of $\hat{\lambda}$ by $\pi$ is a smooth proper arc in $\mathbb{D}^2$ joining the projections of the endpoints of $\hat{\ell}$ in $\widehat{\Lambda}$.
	
	We now show that the connected components of $\mb{D}^2-\pi({\hat \lambda})$ correspond to the triangles associated to ${\hat \lambda}$. This comes from the fact that both triangles and connected components can be characterized in terms of cyclic order of the endpoints of the leaves of ${\hat \lambda}$ and $\pi({\hat \lambda})$ and $\pi$ induces a homeomorphism between $\widehat{\Lambda}$ and $\partial\mb{D}^2$. 
	
	First notice that $\pi$ maps triangles to connected components: Let $\Delta=\Delta(a,b,c)$ be a triangle bounded by the leaves $[a,b],[b,c],[c,a]$ of ${\hat \lambda}$. Since $\Delta$ is contained in a totally geodesic spacelike plane, by Lemma \ref{lem:planes and lines}, the restriction of $\pi$ to $\Delta$ is a homeomorphism onto the image. The image $\pi({\rm int}(\Delta))$ must be disjoint from the other leaves $\pi(\ell)$ of $\pi({\hat \lambda})$, otherwise $\pi(\ell)$ would intersect one of the sides $\pi(\partial\Delta)=\pi[a,b]\cup\pi[b,c]\cup\pi[c,a]$. Thus $\pi({\rm int}(\Delta))$ is a connected component of $\mb{D}^2-\pi({\hat \lambda})$. 
	
	Then we show that every connected component of $\mb{D}^2-\pi({\hat \lambda})$ arises as a projection of a triangle.
	
	Let $U\subset\mb{D}^2-\pi({\hat \lambda})$ be a connected component. The set $U$ is open inside $\mb{D}^2$, and its boundary $\partial U$ is contained in $\pi({\hat \lambda})$. In fact, we can be more precise:
	
	\begin{claim}{\it 1}
		The boundary $\partial U$ consists of a union of projections of leaves $\pi(\ell)$.
	\end{claim}
	
	\begin{proof}[Proof of the claim]
		Suppose that there exists a leaf $\ell$ of ${\hat \lambda}$ such that $\pi(\ell)\cap\partial U\neq\emptyset$. Our aim is to show that the subset $\pi(\ell)\cap\partial U$ is both open and closed inside $\pi(\ell)$ which, by connectedness, implies that the projection $\pi(\ell)$ is entirely contained inside $\partial U$. 
		
		It is clear that the set $\pi(\ell)\cap\partial U$ is closed inside $\pi(\ell)$, as $\partial U$ is a closed subset of $\mb{D}^2$. Hence, it is enough to prove that $\pi(\ell)\cap\partial U$ is an open subset of $\pi(\ell)$. Notice that $\pi(\ell)$ divides $\mb{D}^2$ in two half planes $\mb{D}^2-\pi(\ell)=A\cup A'$ and $U$, being connected, must lie inside one of them. Up to relabeling we assume that $U\subset A$. 
		
		We now claim that there exists a small neighborhood $B$ of $x \in \mathbb{D}^2$ such that $B\cap A\subset\mb{D}^2-\pi({\hat \lambda})$. If this is the case, then the set $B\cap\pi(\ell)$ contains an open subsegment of $\pi(\ell)$ around $x$ that lies entirely inside $\partial U$, proving that $\partial U\cap\pi(\ell)$ is an open subset of $\pi(\ell)$.
		
		Suppose that there exist no neighborhood $B$ satisfying such requirements. We can then find a sequence of distinct leaves $\ell_n$ and points $x_n\in\pi(\ell_n)$ converging to $x$ such that $x_n\in A$ for every $n \in \N$. The sequence of leaves $\ell_n$ converges to $\ell$ in the Hausdorff topology of $\mb{H}^{2,n}\cup\partial\mb{H}^{2,n}$ and, hence, their projections $\pi(\ell_n)$ converge to $\pi(\ell)$ in the Hausdorff topology of $\overline{\mathbb{D}}{}^2$. As a consequence, for any point $y$ in $U$ there exists a sufficiently large $n$ for which the curve $\pi(\ell_n)$ separates $y$ from $\pi(\ell)$. Being $U$ a connected component of $\mathbb{D}{}^2 - \pi(\hat{\lambda})$, it must lie inside the complementary region of $\pi(\ell_n)$ that contains $y$, and hence that does not contain $x$. However, this contradicts the fact that $x$ lies in the boundary $\partial U$, and hence proves the existence of a neighborhood $B$ satisfying the requirements.
	\end{proof}
	
	Using the fact that $\lambda$ is maximal, we now prove:
	
	\begin{claim}{\it 2}
		The boundary $\partial U$ consists of the projection of exactly three leaves of the form $\pi[a,b],\pi[b,c],\pi[c,a]$.
	\end{claim}
	
	\begin{proof}[Proof of the claim]
		If $\ell,\ell'$ are distinct leaves of $\lambda$ that share no endpoint, then we can find a leaf $\ell^* \in \lambda$ whose endpoints separate the endpoints of $\ell$ from the ones of $\ell'$, by maximality of $\lambda$. As a consequence, the curve $\pi(\ell^*)$ must separate $\pi(\ell)$ from $\pi(\ell')$. This shows in particular that $\pi(\ell),\pi(\ell')$ cannot be boundary components of a single complementary region of $\mb{D}^2-\pi({\hat \lambda})$. We deduce that every pair of boundary components of $U$ shares exactly one endpoint, and that $\partial U$ consists of the union of at most three projected leaves.
		
		On the other hand, again by the maximality of $\lambda$, 
		every complementary region of the union of two projected leaves $\pi(\ell), \pi(\ell')$ must contain at least another leaf of $\pi({\hat \lambda})$. This tells us that $\partial U$ must consist of the union of at least three projected leaves. Combining these observations, we deduce that the boundary $\partial U$ of every complementary region of $\pi(\hat{\lambda})$ corresponds to a triangle bounded by leaves of ${\hat \lambda}$.
	\end{proof}
	
	We have now established a bijective correspondence between spacelike ideal triangles bounded by leaves of $\hat{\lambda}$ and complementary regions of $\pi(\hat{\lambda})$ inside $\mathbb{D}^2$.
	Combining this with the fact that the restriction of $\pi$ to every spacelike triangle (or to the geometric relation $\hat{\lambda}$) is a homeomorphism onto its image, we deduce that the restriction of the projection $\pi$ to the pleated set $\widehat{S}_\lambda$ is injective and has image equal to the entire disk $\mb{D}^2$.
	
	We can now show that $\widehat{S}_\lambda$ is a topological \emph{achronal} subsurface, namely:
	
	\begin{claim}{\it 3}
		Every pair of distinct points in $\widehat{S}_\lambda$ is joined by a spacelike or a lightlike segment of $\widehat{\hyp}^{2,n}$. Consequently, $\widehat{S}_\lambda$ is equal to the graph of a $1$-Lipschitz function $g : \mathbb{D}^2 \to \mathbb{S}^n$. 
	\end{claim}
	
	\begin{proof}[Proof of the claim]
		Suppose there there are two points on $\widehat{S}_\lambda$ connected by a timelike geodesic $\alpha$. Consider suitable local coordinates $\Psi:\mb{D}^2\times\mb{S}^n\to\widehat{\mb{H}}^{2,n}$ adapted to a timelike sphere $T$ containing $\alpha$ and an orthogonal spacelike plane $H$, i.e. satisfying $$\Psi(\mb{D}^2\times\{v\})=H, \quad \Psi(\{0\}\times\mb{S}^n)=T,$$
		for some fixed $v \in \mathbb{S}^n$. As the projection $\pi:\widehat{\mb{H}}^{2,n}\to\mb{D}^2$ collapses $\alpha$ to a point, its restriction to $\widehat{S}_\lambda$ cannot be injective, which contradicts the previous claim. As a consequence $\widehat{S}_\lambda$ is achronal and the function $g:\mb{D}^2\to\mb{S}^n$ describing it as a graph is $1$-Lipschitz by Lemma \ref{lem:acausal graph}. In particular $\widehat{S}_\lambda$ is a topological Lipschitz subsurface of $\widehat{\mb{H}}^{2,n}$.
	\end{proof}
	
	The last property left to prove is the following:
	
	\begin{claim}{\it 4}
		The surface $\widehat{S}_\lambda$ is an acausal subset of $\widehat{\hyp}^{2,n}$.
	\end{claim}
	
	\begin{proof}[Proof of the claim]
		Again, we work in local coordinates adapted to a timelike sphere $T$ and an orthogonal spacelike plane $H$, so that we can write $\widehat{S}_\lambda$ as the graph of a $1$-Lipschitz function $g:\mb{D}^2\to\mb{S}^n$. By Lemma \ref{lem:projection}, the points $p=(x,g(x))$ and $q=(y,g(y))$ on $\widehat{S}_\lambda$ are connected by a lightlike geodesic if and only if $d_{\mb{S}^n}(g(x),g(y))=d_{\mb{S}^2}(x,y)$. As $g$ is $1$-Lipschitz, this means that $d_{\mb{S}^n}(g(t),g(s))=d_{\mb{D}^2}(t,s)$ for every $t,s$ on the geodesic arc $[x,y]$ (in the hemispherical metric). Therefore the lightlike geodesic $[p,q]$ is entirely contained in $\widehat{S}_\lambda$. In particular, either $[p,q]$ is contained in a leaf of ${\hat \lambda}$, or $[p,q]$ meets two different leaves, or it meets the interior of a complementary triangle. However a leaf, a pair of distinct leaves, and a complementary region are all acausal subsets. Therefore, all these cases are not possible and we conclude that $\widehat{S}_\lambda$ must be acausal.
	\end{proof}
	
	This finishes the proof.
\end{proof}

In Sections \ref{sec:cross ratio}, \ref{sec:shear cocycles}, and \ref{sec:geometry}, we will prove that every pleated set $\widehat{S}_\lambda\subset\mb{H}^{2,n}$ associated to a maximal representation $\rho:\Gamma\to\SOtwon$ and a maximal $\rho$-lamination $\lambda$ has a natural associated $\rho(\Gamma)$-invariant hyperbolic structure and admits a developing map $f:\widehat{S}\to\mb{H}^2$ which is $1$-Lipschitz with respect to the intrinsic pseudo-metric on $\widehat{S}_\lambda$ and the hyperbolic metric. 

The data of the pleated set $S_\lambda=\widehat{S}_\lambda/\rho(\Gamma)$ together with the intrinsic pseudo-Riemannian metric, the intrinsic hyperbolic structure, and the $1$-Lipschitz developing map $f:\widehat{S}_\lambda\to\mb{H}^2$ is what we will call a {\em pleated surface} (see in particular Definition \ref{def:pleated surface}).

\subsection{Continuity of pleated sets}\label{subsec:continuity pleated sets}
We now discuss continuity properties of pleated sets associated to an acausal curve $\Lambda\subset\partial\mb{H}^{2,n}$ and maximal $\Lambda$-laminations $\lambda\in\mc{G}_\Lambda$. 

First notice that, by Lemma \ref{lem:acausal lifts}, the acausal curve $\Lambda$ is contained on the boundary of a properly convex set $\Omega_\Lambda \subset \hyp^{2,n}$ and has a well defined convex hull $\mc{CH}(\Lambda)$ independent of the choice of $\Omega_\Lambda$. The closure of $\Omega_\Lambda$, being simply connected, admits a lift to the 2-fold cover $\widehat{\hyp}^{2,n}$. In particular, we can identify a pleated set in $\mb{H}^{2,n}$, which is always contained in $\mc{CH}(\Lambda)\subset\Omega_\rho$, with its lift in $\widehat{\hyp}^{2,n}$.

We will deal with two topologies: On the one hand, as pleated sets $\widehat{S}_\lambda$ are closed subsets of $\widehat{\mb{H}}^{2,n}$, they are endowed with a natural Chabauty topology. On the other hand, if we fix a Poincaré model $\Psi:\mb{D}^2\times\mb{S}^n\to\widehat{\mb{H}}^{2,n}$, each pleated set $\widehat{S}_\lambda$ can be written as a graph of a (strictly) $1$-Lipschitz function $g_\lambda:\mb{D}^2\to\mb{S}^n$ (by Proposition \ref{pro:existence pleated sets} and Lemma \ref{lem:acausal graph}). Therefore, the collection of pleated sets can be endowed also with the topology of uniform convergence of the functions $g_\lambda$. Notice that convergence with respect to this topology implies Chabauty convergence.

\begin{pro}
	\label{pro:continuity pleated sets}
	Let ${\hat \Lambda}\subset\partial\widehat{\mb{H}}^{2,n}$ be an acausal curve. Then for every Poincaré model $\Psi:\mb{D}^2\times\mb{S}^n\to\widehat{\mb{H}}^{2,n}$, the map
	\[
	\begin{matrix}
		\mc{G}_\Lambda^m & \longrightarrow & {\rm Lip}_1(\mb{D}^2,\mb{S}^n) \\
		\lambda & \longmapsto & g_\lambda
	\end{matrix}
	\]
	is continuous with respect to the Chabauty topology on $\mc{G}_\Lambda^m$ and the uniform convergence on compact subsets on ${\rm Lip}_1(\mb{D}^2,\mb{S}^n)$. 
\end{pro}

\begin{proof}
	Let $\lambda_m$ be a sequence of maximal $\Lambda$-laminations converging to a maximal lamination $\lambda$ in the Chabauty topology. We denote by $\widehat{S}_m,\widehat{S}$ the corresponding pleated sets and by $g_m,g:\mb{D}^2\to\mb{S}^n$ the associated $1$-Lipschitz maps. 
	
	Our goal is to prove that $g_m\to g$ uniformly over all compact subsets of $\mathbb{D}^2$. Notice that, being $1$-Lipschitz, the maps $g_n$ converge uniformly on compact sets to a $1$-Lipschitz function $g':\mb{D}^2\to\mb{S}^n$ up to subsequences. If we show that $g'=g$, then the convergence $g_n\to g$ would follow.
	
	We now argue that:
	
	\begin{claim*}
		Each $x\in\widehat{S}$ is the limit of a sequence $x_m\in\widehat{S}_m$.
	\end{claim*}
	
	This will be enough to conclude: In fact, suppose that this is the case. Pick any $x\in\widehat{S}$ and select $x_m\in\widehat{S}_m$ as in the claim. We can express
	$$x_m = (y_m,g_m(y_m)), \quad x = (y,g(y)) ,$$
	for some $y_m, y \in \mathbb{D}^2$. By assumption the sequence $(x_m)_m$ converges to $x$, and hence $y_m\to y$ and $g_m(y_m)\to g(y)$. 
	On the other hand, since the sequence of functions $(g_m)_m$ is converging uniformly to $g'$ over all compact subsets of $\mathbb{D}^2$, we must have that $g_m(y_m)\to g'(y)$. We conclude that $g'(y)=g(y)$ and, by letting $x$ vary in $\widehat{S}$, that the functions $g$ and $g'$ coincide on $\mathbb{D}^2$.
	
	\begin{proof}[Proof of the claim]
		In the proof of the claim we distinguish whether $x$ belongs to a leaf or to a plaque, but the arguments are very similar.
		
		Consider a point $x$ on a leaf ${\hat \ell}$ of ${\hat \lambda}$. As $\lambda_m$ converges in the Chabauty topology to $\lambda$, the leaf $\ell\subset\lambda$ is the limit of a sequence of leaves $\ell_m\subset\lambda_m$ and, hence, also the geometric realization ${\hat \ell}$ is the Chabauty limit of the sequence of geometric realizations ${\hat \ell}_m$. Therefore, $x\in{\hat \ell}$ is the limit of a sequence of points $x_m\in{\hat \ell}_m$.
		
		Consider a point $x$ on a plaque $\widehat{\Delta}$ of $\widehat{S}-{\hat \lambda}$. As $\lambda_m$ converges in the Chabauty topology to $\lambda$, the plaque $\Delta$ is the limit of a sequence of plaques $\Delta_m$ of $\lambda_m$ and, hence, also the the geometric realization $\widehat{\Delta}$ is the Chabauty limit of the sequence of geometric realizations $\widehat{\Delta}_m$. Therefore $x\in\widehat{\Delta}$ is the limit of a sequence of points $x_m\in\widehat{\Delta}_m$.
	\end{proof}
	
	This concludes the proof of the proposition.
\end{proof}

\subsection{Bending locus}\label{subsec:bending locus}

As in the case of classical pleated surfaces in hyperbolic geometry, pleated sets $\widehat{S}_\lambda$ associated to maximal $\rho$-laminations $\lambda$ are not necessarily bent along all the leaves of ${\hat \lambda}$: Trivially, any pleated set that is invariant by the action of a \emph{Fuchsian} representation in $\SOtwon$ is in fact totally geodesic. In analogy with \cite{CEG}*{Definition~I.5.1.3} (see also Thurston \cite{ThNotes}*{\S~8.6}) we introduce the following notion:

\begin{dfn}[Bending Locus]
	Let $\rho:\Gamma\to\SOtwon$ be a maximal representation. Consider $\lambda$ a maximal $\rho$-lamination with geometric realization ${\hat \lambda}$, and denote by $\widehat{S}_\lambda$ the corresponding pleated set. A point $x\in\ell\subset{\hat \lambda}$ is in the \emph{bending locus of $\widehat{S}_\lambda$} if there is no (necessarily spacelike) geodesic segment $k$ entirely contained in $\widehat{S}_\lambda$ and such that ${\rm int}(k)\cap\ell=x$. 
\end{dfn}

We now prove the following:

\begin{pro}
	\label{pro:bending locus}
	The bending locus is a sublamination of ${\hat \lambda}$, and its complementary inside $\widehat{S}_\lambda$ is a union of $2$-dimensional totally geodesic spacelike regions.
\end{pro}

\begin{proof}
	We show that, if $x\in\ell\subset{\hat \lambda}$ is not in the bending locus, then there exists a neighborhood of $\ell$ inside $\widehat{S}_\lambda$ that is entirely contained in a spacelike plane and, therefore, its intersection with ${\hat \lambda}$ is not in the bending locus. This implies that the bending locus is closed and consists of a disjoint union of the leaves of ${\hat \lambda}$.
	
	Before developing the proof, we recall a general structural result (see Theorem I.4.2.8 in \cite{CEG}) for $\rho$-laminations: Every $\rho$-lamination $\lambda$ decomposes as a disjoint union of: 
	\begin{itemize}
		\item A finite number of minimal $\rho$-sublaminations $\lambda_j$, and
		\item A finite number of orbits $\rho(\Gamma)\ell$ of isolated leaves that are asymptotic in both directions to leaves of the minimal components $\lambda_j$.
	\end{itemize}
	Notice that the $\rho(\Gamma)$-orbit of every leaf of a minimal component $\lambda_j$ is dense inside $\lambda_j$. In particular, for every $\ell\in\lambda_j$ there exists a sequence of distinct elements $\gamma_n\in\Gamma$ such that $\ell_n : =\rho(\gamma_n)$ converges to $\ell$ in the Chabauty topology.
	
	Suppose now that $x\in\ell\subset{\hat \lambda}$ is not contained in the bending locus, and let $k\subset\widehat{S}_\lambda$ be a spacelike segment transverse to $\ell$ at $x$ and entirely contained in $\widehat{S}_\lambda$. According to the above decomposition, we have that either $\ell$ is isolated or it is contained in a minimal component of $\lambda$. 
	
	If $\ell$ is isolated, then there exists two distinct components $\Delta,\Delta'$ of $\widehat{S}_\lambda-{\hat \lambda}$ such that $\ell=\Delta\cap\Delta'$, and the geodesic segment $k$ intersects both $\Delta$ and $\Delta'$. In this case, we immediately conclude that the two triangles $\Delta,\Delta'$ must be contained in the same spacelike plane. In particular, the whole line $\ell$ is not contained in the bending locus.
	
	If $\ell$ is contained in a minimal component $\lambda_j$ of $\lambda$, then $\rho(\Gamma)\ell$ is dense in $\lambda_j$. In this case, there exist infinitely many pairwise distinct segments $k_n$ entirely contained in $\widehat{S}_\lambda$ that intersect $\ell$ transversely. To see this, let $\ell_n := \rho(\delta_n) \ell$ be a sequence of pairwise distinct translates of $\ell$ such that $\ell_n\to\ell$. If the endpoints of $\ell_n$ are sufficiently close to the endpoints of $\ell$, then $\ell_n$ must intersect $k$ transversely, and every transverse intersection $k\cap\ell_n$ can be translated back to $k_n\cap\ell$ by applying the isometry $\rho(\delta_n^{-1})$. 
	
	Consider now two distinct spacelike segments $k,k'$ entirely contained in $\widehat{S}_\lambda$ that intersect $\ell$ transversely. Every geodesic $\ell'\subset{\hat \lambda}$ with endpoints sufficiently close to those of $\ell$ must intersect both $k$ and $k'$. In particular, if $\Delta'$ and $\Delta''$ are triangles of $\widehat{S}_\lambda$ with edges $\ell',\ell''$ sufficiently close to $\ell$, then $k,k'$ intersect both ${\rm int}(\Delta')$ and ${\rm int}(\Delta'')$. This implies that $\Delta',\Delta''$ lie on the same spacelike plane. Hence, all triangles that have an edge sufficiently close to $\ell$ lie inside a common spacelike plane $H$. By density of triangles in $\widehat{S}_\lambda$, we conclude that a neighborhood of $\ell$ in $\widehat{S}_\lambda$ lies in $H$ and, therefore, \emph{every} point of the leaf $\ell$ does not lie in the bending locus.
\end{proof}

%%%

\section{Hyperbolic structures on pleated sets I}
\label{sec:cross ratio}

Let $\rho:\Gamma\to\SOtwon$ be a maximal representation, and let $\lambda$ be a maximal geodesic lamination with associated pleated set $\widehat{S}_\lambda\subset\mb{H}^{2,n}$. After investigating the causal and topological properties of $\widehat{S}_\lambda$, we now turn our attention on its geometric structure. Being pleated sets obtained as unions of spacelike geodesics and ideal triangles, a natural question that arises is whether the metrics on each totally geodesic region "patch nicely together", determining an intrinsic hyperbolic metric on $\Sigma$.

Inspired by the work of Bonahon \cite{Bo96} in the context of hyperbolic surfaces, we now intend to answer to this question by recording the relative position of the hyperbolic triangles that make up the pleated set $\widehat{S}_\lambda$ into a shear cocycle $\sigma_\lambda^\rho \in \mc{H}(\lambda;\mb{R})$ transverse to the lamination $\lambda$ (see Section \ref{subsec:shear coordinates}). Making use of Bonahon's characterization of hyperbolic shear cocycles in terms of lengths of measured laminations (see Theorem \ref{thm:thurston bonahon}), this will determine, for every maximal representation $\rho$ and for any maximal lamination $\lambda$, the intrinsic hyperbolic structure $X^\rho_\lambda \in \T$ of the pleated set $S_\lambda$. The construction of the shear cocycles $\sigma^\rho_\lambda$ and the investigation of their properties are going to be the main subject of the current and next sections. 

In fact, the process that we will outline applies in a wider generality than the one specifically needed for the study of pleated sets in $\hyp^{2,n}$. Indeed, the definition of the cocycle $\sigma_\lambda^\rho$ will rely only on certain analytic properties of the cross ratio $\beta^\rho$ naturally associated to the representation $\rho$, namely on its \emph{positivity} (see Definition \ref{def:positive cross ratio}) and \emph{local boundedness} (which we discuss below, see Definition \ref{def:locally bounded}). Examples of cross ratios satisfying these properties occur frequently in the literature about Higher Teichmüller Theories: This is for instance the case for Hitchin representations in ${\rm SO}_0(p,p+1)$ or $\Theta$-positive representations in ${\rm SO}_0(p,q)$ (see e.g. Beyrer and Pozzetti \cite{BP21} and Appendix \ref{other cross ratios}).

We can now describe our main result in this context:

\begin{thm}
	\label{shear of cross ratio improved} 
	Let $\beta : \partial \Gamma^{(4)} \to \R$ be a positive and locally bounded cross ratio. Then for every maximal lamination $\lambda$, the $\beta$-shear cocycle $\sigma^\beta_\lambda$ belongs to the closure of the cone $C(\lambda) \subset \mathcal{H}(\lambda;\R)$, that is 
	\[
	\omega_\lambda(\sigma^\beta_\lambda, \mu) \geq 0
	\]
	for every measured lamination $\mu$ with $\supp \mu \subseteq \lambda$. Moreover, if the cross ratio $\beta$ is strictly positive, then $\omega_\lambda(\sigma^\beta_\lambda, \mu) > 0$ for every non-trivial measured lamination $\mu$ as above, and consequently there exists a unique hyperbolic structure $Y = Y^\beta_\lambda \in \T$ such that $\sigma^\beta_\lambda = \sigma_\lambda^Y \in \mathcal{H}(\lambda;\R)$.
\end{thm}

We prove Theorem \ref{shear of cross ratio improved} in the case of a finite leaved maximal lamination in this section and in the case of a general lamination in the next section. 

We now briefly discuss the notion of locally bounded cross ratio. To this purpose, we need to introduce some notation. We select an identification of the universal cover $\widetilde{\Sigma}$ with $\hyp^2$ through the choice of a hyperbolic structure $X\in\T$. Moreover, given $\ell$ an oriented geodesic of $\hyp^2$, we denote by $\ell^+$ and $\ell^-$ the positive and negative endpoints of $\ell$ in $\partial \hyp^2 \cong \partial \Gamma$, respectively. If $\ell$ and $h$ are two disjoint oriented geodesics in $\hyp^2$, then we say that $\ell$ and $h$ are \emph{coherently oriented} if their endpoints satisfy $\ell^+ \leq h^+ < h^- \leq \ell^- < \ell^+$ with respect to some cyclic order on $\partial \Gamma$. With this notation, we now define:

\begin{dfn}[Locally Bounded] 
	\label{def:locally bounded}
	A cross ratio $\beta : \partial \Gamma^{(4)} \to \R$ is said to be \emph{locally bounded} if there exists a (and consequently for any) hyperbolic structure $X \in \T$ such that, for any constant $D > 0$ we can find $C, \alpha > 0$ such that
	\[
	\abs{\log \beta(h^+,\ell^+,\ell^-,h^-)} \leq C \abs{\log \beta^X(h^+,\ell^+,\ell^-,h^-)}^\alpha
	\]
	for any pair of coherently oriented geodesics $\ell , h$ in $\widetilde{\Sigma} \cong \hyp^2$ satisfying $0 < d_{\hyp^2}(\ell,h) \leq D$.
\end{dfn}

The term $\log \beta^X(h^+,\ell^+,\ell^-,h^-)$ appearing in the definition above has a precise geometrical interpretation in terms of $2$-dimensional hyperbolic geometry, as described by the following lemma:

\begin{lem} \label{lem:crossratio and distance}
	Let $X \in \T$. For any pair of coherently oriented disjoint geodesics $\ell, h$ in $\hyp^2 \cong \widetilde{\Sigma}$, the value $\beta^X(\ell^+,h^+,h^-,\ell^-)$ is strictly positive and satisfies
	\[
	\log \beta^X(h^+, \ell^+, \ell^-, h^-) = 2 \log \cosh \frac{d_{\hyp^2}(\ell,h)}{2} .
	\]
\end{lem}

As mentioned above, local boundedness is an analytic property of a cross ratio that will be crucially needed in the next section to guarantee convergence of the approximation process defining the shear cocycle in the case of a general lamination. In the proof of Theorem \ref{shear of cross ratio improved} in the case of finite leaved lamination, instead, such property is not used and the proofs are more elementary.

The cross ratios that we are mainly interested in are the natural cross ratios $\beta^\rho$ associated to maximal representations $\rho:\Gamma\to\SOtwon$. We start by introducing them and investigating their properties.

\subsection{Cross ratios of maximal representations} 
\label{sec:maximal cross ratios}

In what follows we describe a cross ratio $\beta^\rho$ on $\partial \Gamma$ naturally associated to $\rho : \Gamma \to \SOtwon$ and its limit map $\xi$. To this purpose, we start by defining a sign function on the set of $4$-tuples of points in $\partial \Gamma$ as follows. Given $\phi$ some fixed homeomorphism between $\partial \Gamma$ and $\rp^1$, we set
\[
\mathrm{Sgn}(u,v,w,z) : = 
\mathrm{sgn}\left( \frac{\phi(u) - \phi(w)}{\phi(u) - \phi(z)} \frac{\phi(v) - \phi(z)}{\phi(v) - \phi(w)} \right) ,
\]
for any $(u,v,w,z) \in \partial \Gamma^{(4)}$, where $\mathrm{sgn}(t) = + 1$ if $t > 0$, $\mathrm{sgn}(0) = 0$, and $\mathrm{sgn}(t) = + 1$ if $t < 0$. It is simple to check that the function $\mathrm{Sgn}$ is independent of the choice of the homeomorphism $\phi$, and that $\mathrm{Sgn}(u,v,w,z) = 0$ if and only if $u = w$ or $v = z$.

For any maximal representation $\rho : \Gamma \to \mathrm{SO}_0(2,n + 1)$ with associated acausal limit map $\xi : \partial \Gamma \to \partial \hyp^{2,n}$ we then define
\begin{equation} \label{eq:maximal crossratio}
	\beta^{\rho}(u,v,w,z) : = \mathrm{Sgn}(u,v,w,z) \left(\frac{\scal{\xi(u)}{\xi(w)}}{\scal{\xi(u)}{\xi(z)}} \frac{\scal{\xi(v)}{\xi(z)}}{\scal{\xi(v)}{\xi(w)}}\right)^{1/2} ,
\end{equation}
for any $(u,v,w,z) \in \partial \Gamma^{(4)}$, where $\scal{\bullet}{\bullet}$ denotes the scalar product $\scal{\bullet}{\bullet}_{2,n + 1}$, and we are implicitly selecting representatives in $\R^{2, n + 1}$ of the equivalence classes $\xi(y)$, for $y \in \{u,v,w,z\}$. By Theorem \ref{thm:maximal limit curve}, the scalar products involved in the definition above are all non-zero. Moreover, the quantity appearing under the square root does not depend on the chosen lightlike representatives of the equivalence classes $\xi(u),\xi(v),\xi(w),\xi(z)$, and it is always non-negative. Since $\rho$ preserves the scalar product $\scal{\bullet}{\bullet}_{2,n + 1}$ and the diagonal action of $\Gamma$ on $\partial \Gamma^{(4)}$ preserves $\mathrm{Sgn}$, the function $\beta^{\rho}$ is $\Gamma$-invariant. Finally, the H\"older continuity of $\beta^\rho$ at any point $(u,v,w,z) \in \partial \Gamma$ is a direct consequence of the H\"older continuity of the limit map $\xi$, which is guaranteed by Theorem \ref{thm:maximal limit curve}). It is straightforward to check that maps $\beta^\rho$ associated to $\SOtwon$-maximal representations as in \eqref{eq:maximal crossratio} satisfy the symmetries listed in \eqref{eq:crossshear}. 

We now prove positivity and local boundedness of the cross ratios $\beta^\rho$, properties that will be crucial for the construction of shear cocycles developed in Section \ref{sec:shear cocycles}. 

\begin{lem}\label{lem:strictly positive}
	For every maximal representation $\rho : \Gamma \to \SOtwon$, the cross ratio $\beta^\rho$ is strictly positive and satisfies relation \eqref{cross ratio one}.
\end{lem}

\begin{proof}
	Let $(u,v,w,z)$ be a $4$-tuple of distinct and cyclically ordered points in $\partial \Gamma$. Up to the action of $\SOtwon$, we can assume that $\xi(u) = e_2 + e_3$, $\xi(v) = - e_1 + e_3$, $\xi(w) = - e_2 + e_3$. Moreover, since $u,v,w,z$ are cyclically ordered, $\xi(z)$ can be expressed as $\xi(z) = \cos \vartheta \, e_1 + \sin \vartheta \, e_2 + x$, where $\vartheta \in (- \pi/2, \pi/2)$, and $x$ is some timelike vector of norm $-1$ orthogonal to the spacelike plane spanned by $e_1$ and $e_2$. Being $\xi(\partial \Gamma)$ a spacelike curve in $\hyp^{2,n}$, we see that $\vartheta$ and $x$ must satisfy
	\begin{equation} \label{eq:condition}
		-1 \leq \scal{e_3}{x} < - \abs{\sin \vartheta} .
	\end{equation}
	Observe that $\mathrm{Sgn}(u, v, w, z) = +1$. Therefore, by definition of $\beta^\rho$ and the normalization selected, we have
	\[
	\beta^\rho(u, v, w, z) = \left( \frac{2(\scal{e_3}{x} - \cos \vartheta)}{\scal{e_3}{x} + \sin \vartheta} \right)^{1/2} .
	\]
	From this identity, it is immediate to see that $\beta^\rho(u, v, w, z) > 1$ if and only if $\scal{e_3}{v} - \sin \vartheta < 2 \cos \vartheta$. This inequality is always satisfied: The left-hand side is negative by \eqref{eq:condition}, while the right-hand side is positive since $\vartheta \in (- \pi/2, \pi/2)$.
\end{proof}

\begin{lem} \label{lem: comparing crossratios}
	For any maximal representation $\rho : \Gamma \to \SOtwon$, the positive cross ratio $\beta^\rho$ is locally bounded. Furthermore, the constant $\alpha > 0$ satisfying the requirements of Definition \ref{def:locally bounded} can be chosen to be a H\"older exponent of the limit map $\xi : \partial\Gamma \to \partial\hyp^{2,n}$ associated to $\rho$.
\end{lem}

\begin{proof}
	Fix $D > 0$, and let $A = A_D$ be the subset of $\partial \Gamma^{(4)}$ given by
	\[
	\{ (h^+,\ell^+,\ell^-,h^-) \mid \text{$\ell, h$ coherently oriented, $0 < d_{\hyp^2}(\ell,h) \leq D$} \} .
	\]
	If $F : A \to \R$ denotes the function
	\[
	F(u,v,w,z) : = \frac{\abs{\log \beta^{\rho}(u,v,w,z)}}{ \abs{\log \beta^X(u,v,w,z)}^\alpha} ,
	\]
	then the statement is equivalent to $F$ being bounded (observe that $F$ is well defined on $A$ by Lemma \ref{lem:crossratio and distance}). Since $F$ is invariant with respect to the diagonal action of $\Gamma$, it induces a continuous function on the quotient space $A / \Gamma$. We introduce a convenient exhaustion by compact subsets of $A / \Gamma$. For any $d \in (0,D]$, let $A_d$ be the subset of $A$ given by
	\[
	\{ (h^+,\ell^+,\ell^-,h^-) \mid \text{$\ell, h$ coherently oriented, $d \leq d_{\hyp^2}(\ell,h) \leq D$} \} .
	\]
	Then it is immediate to see that the fundamental group $\Gamma$ acts cocompactly on $A_d$ for every $d > 0$. In particular $F$ admits an upper bound on $A_d$ for any $d > 0$. 
	
	Assume now that $F$ is not bounded over $A$. Then there exist sequences of coherently oriented geodesics $\ell_n, h_n$ in $\hyp^2$ such that $F(h_n^+,\ell_n^+,\ell_n^-,h_n^-)$ tends to $+ \infty$ as $n$ goes to $\infty$, and $d_{\hyp^2}(\ell_n,h_n) \leq D$. Since $A_d /\Gamma$ is compact for every $d > 0$, and the $4$-tuples $(h_n^+,\ell_n^+,\ell_n^-,h_n^-)$ are escaping every compact subset of $A / \Gamma$, we must have that $d_{\hyp^2}(\ell_n,h_n)$ tends to $0$ as $n$ goes to $\infty$. Up to the action of $\Gamma$, we can assume that the point of $\ell_n$ that minimizes the distance from $h_n$ is contained inside a fixed fundamental domain. Hence, the geodesics $\ell_n$ and $h_n$ converge to the same geodesic $\ell$, up to subsequences. Moreover, up to choosing a different identification between $\widetilde{\Sigma}$ and $\hyp^2$, we can assume that the points $\ell^\pm$ are different from $\infty \in \rp^1$. 
	
	By the properties of cross ratios and their continuity we observe that both terms $\beta^\rho(h_n^+, \ell_n^+, \ell_n^-,h_n^-), \beta^X(h_n^+, \ell_n^+, \ell_n^-,h_n^-)$ converge to $1$ as $n \to \infty$. The rest of the proof will be dedicated to the study of the order of convergence to $0$ of the logarithm of these terms, which will lead to a contradiction with $F$ being unbounded. 
	
	We start from the term involving $\beta^X$: By the symmetries of the standard hyperbolic cross ratio, we observe
	\begin{align*}
		\beta^X(h_n^+,\ell_n^+,\ell_n^-,h_n^-) & = 1 - \beta^X(h_n^+,\ell_n^-,\ell_n^+,h_n^-) \\
		& = 1 - \frac{h_n^+ - \ell_n^+}{h_n^+ - h_n^-} \frac{\ell_n^- - h_n^-}{\ell_n^- - \ell_n^+}
	\end{align*}
	Notice that the denominator $(h_n^+ - h_n^-)(\ell_n^- - \ell_n^+)$ converges to $(\ell^- - \ell^+)^2$, which is different from $0$. On the other hand, since $\ell_n, h_n \to \ell$, the factor $(h_n^+ - \ell_n^+)(\ell_n^- - h_n^-)$ is infinitesimal. Therefore we deduce that
	\begin{equation} \label{eq:limit cross ratio}
		\lim_{n \to \infty} \frac{\abs{\log \beta^X(h_n^+,\ell_n^+,\ell_n^-,h_n^-)}}{\abs{(h_n^+ - \ell_n^+)(\ell_n^- - h_n^-)}} = \frac{1}{(\ell^- - \ell^+)^2} =: M .
	\end{equation}
	
	Let now $\xi : \partial \Gamma \to \SOtwon$ denote the limit map associated to the representation $\rho$. In order to study the behavior of the cross ratios $\beta^{\rho}(h_n^+,\ell_n^+,\ell_n^-,h_n^-)$, it will be convenient to introduce representatives of the projective classes $\xi(\ell_n^\pm), \xi(h_n^\pm)$, by selecting some affine hyperplane $V$ in $\R^{2, n + 1}$ intersecting the projective classes $\xi(\ell^\pm)$ and pick representatives belonging to $V$. We will continue to denote with abuse these representatives by $\xi(\ell_n^\pm), \xi(h_n^\pm)$. Consider now
	\[
	\beta^{\rho}(h_n^+,\ell_n^+,\ell_n^-,h_n^-)^2 = \frac{\scal{\xi(h_n^+)}{\xi(\ell_n^-)}}{\scal{\xi(h_n^+)}{\xi(h_n^-)}} \frac{\scal{\xi(\ell_n^+)}{\xi(h_n^-)}}{\scal{\xi(\ell_n^+)}{\xi(\ell_n^-)}} .
	\]
	Since $\xi(\ell_n^\pm), \xi(h_n^\pm) \to \xi(\ell^\pm)$, respectively, the above quantity converges to $1$.
	A simple algebraic manipulation shows that
	\begin{align*}
		\beta^{\rho}(h_n^+,\ell_n^+,\ell_n^-,h_n^-)^2 - 1 & = \frac{\scal{\xi(h_n^+)}{\xi(\ell_n^-) - \xi(h_n^-)} \scal{\xi(\ell_n^+) - \xi(h_n^+)}{\xi(h_n^-)}}{\scal{\xi(h_n^+)}{\xi(h_n^-)} \scal{\xi(\ell_n^+)}{\xi(\ell_n^-)}} + \\
		& \qquad - \frac{\scal{\xi(\ell_n^+) - \xi(h_n^+)}{\xi(\ell_n^-) - \xi(h_n^-)} \scal{\xi(h_n^+)}{\xi(h_n^-)}}{\scal{\xi(h_n^+)}{\xi(h_n^-)} \scal{\xi(\ell_n^+)}{\xi(\ell_n^-)}}
	\end{align*}
	Let now $L > 0$ be a positive constant such that $\abs{\scal{u}{v}} \leq L \norm{u}_0 \norm{v}_0$, for some fixed Euclidean norm $\norm{\bullet}_0$ on $\R^{2,n + 1}$. We deduce that
	\[
	\abs{\beta^{\rho}(h_n^+,\ell_n^+,\ell_n^-,h_n^-)^2 - 1} \leq 2 L \norm*{\xi(h_n^+)}_0 \norm*{\xi(h_n^-)}_0 \frac{\norm*{\xi(\ell_n^+) - \xi(h_n^+)}_0 \norm*{\xi(\ell_n^-) - \xi(h_n^-)}_0}{\abs{\scal{\xi(h_n^+)}{\xi(h_n^-)} \scal{\xi(\ell_n^+)}{\xi(\ell_n^-)}}} .
	\]
	Since $\xi(\ell_n^\pm), \xi(h_n^\pm)$ are converging to $\xi(\ell^\pm)$, and $\xi(\ell^+) \neq \xi(\ell^-)$, we can find a constant $M' > 0$ such that
	\[
	\abs{\beta^{\rho}(h_n^+,\ell_n^+,\ell_n^-,h_n^-)^2 - 1} \leq M' \norm*{\xi(\ell_n^+) - \xi(h_n^+)}_0 \norm*{\xi(\ell_n^-) - \xi(h_n^-)}_0 .
	\]
	Moreover, being $\xi$ a H\"older continuous function with exponent $\alpha$, we conclude that for $n$ sufficiently large
	\begin{align}\label{eq:holder bound}
		\begin{split}
			\abs{\log\beta^{\rho}(h_n^+,\ell_n^+,\ell_n^-,h_n^-) } & = \frac{1}{2} \log(1 + (\beta^{\rho}(h_n^+,\ell_n^+,\ell_n^-,h_n^-)^2 - 1)) \\
			& \leq \frac{M'}{2} \norm*{\xi(\ell_n^+) - \xi(h_n^+)}_0 \norm*{\xi(\ell_n^-) - \xi(h_n^-)}_0 \\
			& \leq M'' \abs{(h_n^+ - \ell_n^+)(\ell_n^- - h_n^-)}^\alpha ,
		\end{split}
	\end{align}
	for some constant $M'' >0$ (here we are considering $\xi$ as a H\"older function from a neighborhood of $\ell^\pm$ in $\R \subset \rp^1 = \partial \hyp^2$ to $\hyp^{2,n} \subset V$ with their Euclidean metrics). Finally, combining this inequality with relation \eqref{eq:limit cross ratio}, we obtain that
	\[
	\limsup_{n \to \infty} \frac{\abs{\log \beta^{\rho}(u,v,w,z)}}{ \abs{\log \beta^X(u,v,w,z)}^\alpha} \leq \frac{M''}{M} ,
	\]
	which contradicts the fact that $F(h_n^+,\ell_n^+,\ell_n^-,h_n^-)$ diverges. We conclude that $F$ is bounded on $A$, and therefore that there exists a constant $C > 0$ satisfying the requirements.
\end{proof}

\begin{rmk} \label{rmk:uniform boundedness}
	The argument provided in the proof of Lemma \ref{lem: comparing crossratios} heavily relies on the H\"older continuity of the limit map $\xi : \partial\Gamma \to \partial\hyp^{2,n}$ associated to the representation $\rho$ (see in particular relation \eqref{eq:holder bound}). In fact, we can give a more detailed description on the dependence of the constants $C, \alpha > 0$ satisfying the local boundedness condition for $\beta^\rho$. To this purpose, let us introduce the following technical notion:
	
	\begin{dfn}[Uniform Family]\label{def:uniform family}
		Let $\{\rho_i : \partial \Gamma \to \SOtwon \}_{i \in I}$ be a family of maximal representations, with associated limit maps $\{\xi_i : \partial\Gamma \to \partial\hyp^{2,n}\}_{i \in I}$. Select arbitrarily a hyperbolic structure on $\Sigma$, and choose an identification of $\partial\Gamma$ with $\partial\hyp^2$. We say that the family $\{\rho_i\}_{i \in I}$ is \emph{uniform} if the limit maps $\{\xi_i\}_i$ are uniformly H\"older continuous, namely there exist Riemannian distances $d_\infty, d_\infty'$ on $\partial\hyp^2, \partial\hyp^{2,n}$, respectively, and constants $H, \alpha > 0$ such that
		\[
		d_\infty'(\xi_i(x), \xi_i(y)) \leq H d_\infty(x,y)^\alpha
		\]
		for every $i \in I$ and $x, y \in \partial\Gamma \cong \partial\hyp^2$.
	\end{dfn}
	
	Then, applying the same strategy described in the proof of Lemma \ref{lem: comparing crossratios} we can in fact obtain:
	
	\begin{lem} \label{lem: comparing crossratios pro}
		Let $\{\rho_i\}_{i \in I}$ be a uniform family of maximal representations. Then the cross ratios $\beta_i : = \beta^{\rho_i}$ are uniformly locally bounded, i.e. there exists a hyperbolic structure $X$ such that, for any choice of $D > 0$, we can find constants $C, \alpha > 0$, (in particular \emph{independent of $i \in I$}) satisfying
		\[
		\abs{\log \beta_i(h^+,\ell^+,\ell^-,h^-)} \leq C \abs{\log \beta^X(h^+,\ell^+,\ell^-,h^-)}^\alpha
		\]
		for any pair of coherently oriented geodesics $\ell , h$ in $\widetilde{\Sigma} \cong \hyp^2$ satisfying $0 < d_{\hyp^2}(\ell,h) \leq D$, and for any $i \in I$. Furthermore, the constant $\alpha$ can be chosen to be a common H\"older exponent of the family of limit maps $\{\xi_i\}_i$ of $\{\rho_i\}_i$.
	\end{lem}
\end{rmk}

We can therefore conclude that the cross ratios $\beta^\rho$ associated to maximal representations $\rho : \Gamma \to \textrm{SO}_0(2,n+1)$ satisfy the hypotheses of Theorem \ref{shear of cross ratio improved}.

\subsection{Outline of the construction}

We now move to the definition of shear cocycles associated to positive locally bounded cross ratios and maximal laminations. Throughout the rest of Section \ref{sec:cross ratio}, $\beta$ will always denote a cross ratio on $\partial \Gamma$, and $\lambda$ a maximal geodesic lamination on $\Sigma$.

Consider two distinct plaques $P, Q$ of $\lambda$. The shear $\sigma_\lambda^\beta(P,Q)$ between $P, Q$ will be defined following a careful approximation argument which depends on the fine properties of maximal geodesic laminations in hyperbolic surfaces. In order to describe the first steps of our construction, let us introduce some notation: We say that a plaque $R$ (or a leaf $\ell$) of $\lambda$ separates $P$ from $Q$ if $P$ and $Q$ are contained in distinct connected components of $\widetilde{\Sigma} - R$ (or $\widetilde{\Sigma} - \ell$). We denote by $\mathcal{P}_{P Q}$ the set of plaques of $\lambda$ that separates $P$ from $Q$.

In the remainder of the current section we will proceed as follows:

\begin{description}
	\item[\S~\ref{subsec:definition shear}] We start by recalling a simple process -- already described by Bonahon in \cite{Bo96} -- that, starting from a finite subset of plaques $\mathcal{P} \subseteq \mathcal{P}_{PQ}$, produces a finite lamination $\lambda_\mathcal{P}$ of $\widetilde{\Sigma}$ containing all the leaves of $\lambda$ that lie in the boundary of some plaque in $\mathcal{P}$. We introduce the elementary $\beta$-shear between two adjacent complementary regions of $\lambda_\mathcal{P}$, which naturally generalizes the classical definition in hyperbolic geometry. We then define the \emph{finite $\beta$-shear with respect to $\mathcal{P}$}, denoted by $\sigma^\beta_\mathcal{P}(P,Q)$, as the sum of the elementary shears between all adjacent complementary regions of $\lambda_\mathcal{P}$.
	\item[\S~\ref{subsec:shear finite case}] In this section we focus our attention on the notion of $\beta$-shears for finite leaved maximal laminations. We observe how the relations satisfied by finite $\beta$-shears (from Section~\ref{subsec:definition shear}) allow to define $\beta$-shear cocycles associated to finite leaved maximal lamination in a fairly elementary and natural way (see in particular Proposition~\ref{pro:shear finite case}). 
	\item[\S~\ref{subsec:shear and length I}] Lastly, we investigate the connections between $\beta$-shear cocycles associated to a finite leaved lamination, and the $\beta$-periods of its closed leaves (see Proposition \ref{pro:length and thurston sympl finite}). The bridge between these notions is provided by the \emph{Thurston symplectic form}, through which hyperbolic shear cocycles are fully characterized (see Theorem \ref{thm:thurston bonahon}).
\end{description}

\subsection{Finite shears between plaques} 
\label{subsec:definition shear}

We start by introducing the notion of finite $\beta$-shear between plaques of a maximal geodesic lamination. 

Let $X$ be a hyperbolic structure on $\Sigma$. Given any finite subset $\mathcal{P}$ of $\mathcal{P}_{P Q}$, we select an $X$-geodesic path $k$ in $(\widetilde{\Sigma}, \widetilde{X})$ joining two points in the interior of (the $\widetilde{X}$-geodesic realizations of) $P$ and $Q$, and we index the plaques $P_1, \dots, P_n$ in $\mathcal{P}$ according to their order along $k$, moving from $P$ to $Q$. It is immediate to check that the ordering is independent of the choice of the arc $k$. We set also $P_0 : = P$ and $P_{n + 1} : = Q$. For every $i$, let $\ell_i^P$ and $\ell_i^Q$ be the boundary leaves of $P_i$ that face $P$ and $Q$, respectively. If $S_i$ denotes the (possibly empty) region of $\widetilde{\Sigma}$ delimited by $\ell_i^Q$ and $\ell_{i+1}^P$, for every $i \in \{0, \dots, n\}$, we define $d_i$ to be the geodesic that joins the negative endpoints of the leaves $\ell_i^Q$ and $\ell_{i+1}^P$, as we orient them as boundary of the strip $S_i$ (if $\ell_i^Q$ and $\ell_{i+1}^P$ share one or two endpoints, then we take $d_i = \ell_i^Q$). For every $\mathcal{P} \subset \mathcal{P}_{P Q}$ as above, let now $\lambda_\mathcal{P}$ be the (finite) geodesic lamination of $\widetilde{\Sigma}$ given by
\[
\lambda_\mathcal{P} : = \{ \ell_0^Q, d_0, \ell_1^P, \ell_1^Q, d_1, \dots , \ell_n^P, \ell_n^Q, d_n, \ell_{n + 1}^P \} ,
\]
where the leaves are listed as we move from $P$ to $Q$. The complementary set of the lamination $\lambda_\mathcal{P}$ in $\widetilde{\Sigma}$ consists of two half-planes containing the plaques $P$ and $Q$, and a finite number of \emph{spikes}, i. e. regions bounded by two distinct asymptotic geodesics, that separate $P$ from $Q$. 

Consider now two adjacent complementary regions $R, R'$ of $\lambda_\mathcal{P}$. We denote by $\ell$ the leaf of $\lambda_\mathcal{P}$ shared by $R$ and $R'$, and we select arbitrarily an orientation on $\ell$. Let $u_l$ ($u_r$ resp.) be the ideal vertex in $(R \cup R') \cap \partial \Gamma$ that lies on the left (right resp.) of the geodesic $\ell$. If one of the regions $R, R'$ on the sides of $\ell$ coincides with a half-plane containing $P$ or $Q$, then we select $u_l$ or $u_r$ to be the vertex of the plaque $P$ or $Q$ different from $\ell^+$ and $\ell^-$. Then we set
\begin{equation} \label{eq:def_shear_adj}
	\sigma^\beta(R,R') : = \log \abs{\beta(\ell^+,\ell^-,u_l,u_r)} ,
\end{equation}
and we define the \emph{finite $\beta$-shear between $P$ and $Q$} relative to $\mathcal{P}$ to be
\[
\sigma^\beta_\mathcal{P}(P,Q) : = \sum_{i = 0}^m \sigma^\beta(R_i,R_{i + 1}) ,
\]
where $R_0, R_1, \dots, R_{m + 1}$ are the complementary regions of $\lambda_\mathcal{P}$ as we move from $P$ to $Q$. By the symmetry \eqref{eq:symmetry for shear} of the cross ratio $\beta$, each term $\sigma^\beta(R_i,R_{i + 1})$ does not depend on the choice of the orientation of the leaf separating $R_i$ and $R_{i + 1}$, and $\sigma^\beta(R_i,R_{i + 1}) = \sigma^\beta(R_{i + 1},R_i)$ for every $i$. Notice also that $\sigma^\beta_\mathcal{P}(P,Q) = \sigma^\beta_\mathcal{P}(Q,P)$.

\begin{rmk}\label{rmk:shear_is_shear}
	The definition of the cross ratio $\beta^\rho$ provided in Section \ref{sec:maximal cross ratios} is designed so that the shear $\sigma^\rho(T,T')$ between two adjacent ideal triangles (or spikes) coincides with the classical shear between their spacelike realizations $\hat{T}$ and $\hat{T}'$ inside $\hyp^{2,n}$ (i.e. if $T$ has ideal vertices $a,b,c \in \partial \Gamma$, then $\hat{T}$ is the spacelike triangle with ideal vertices $\xi(a), \xi(b), \xi(c) \in \partial \hyp^{2,n}$), measured with respect to the induced hyperbolic path metric on $\hat{T} \cup \hat{T}'$.
	
	In order to justify this assertion we need to introduce some notation. As we did previously, we denote by $\ell$ the geodesic shared by $R$ and $R'$ together with an arbitrary choice of orientation, and by $u_l$ and $u_r$ the ideal vertices in $(R \cup R') \cap \partial \Gamma$ that lie on the left and on the right of $\ell$, respectively. Since $\xi : \partial \Gamma \to \hyp^{2,n}$ is a spacelike curve by Theorem~\ref{thm:maximal limit curve}, there exist unique spacelike planes $H_l$ and $H_r$ in $\hyp^{2,n}$ whose boundary at infinity contain the triples $\xi(\ell^+), \xi(\ell^-), \xi(u_l)$ and $\xi(\ell^+), \xi(\ell^-), \xi(u_r)$, respectively. If $\hat{R}$ and $\hat{R}'$ denote the regions of $H_l$ and $H_r$ delimited by the spacelike geodesics corresponding to the boundary leaves of $R$ and $R'$, respectively, then the set $\hat{R} \cup \hat{R}'$ possesses a natural hyperbolic metric with geodesic boundary induced by the hyperbolic distances on $H_l$ and $H_r$. 
	
	Let now $\tilde{\xi}(\ell^\pm)$, $\tilde{\xi}(u_l)$, $\tilde{\xi}(u_r)$ be representatives of the projective classes $\xi(\ell^\pm)$, $\xi(u_l)$, $\xi(u_r)$, respectively, so that all their pairwise scalar products are negative (this is possible again by Theorem~\ref{thm:maximal limit curve}). The vectors $\tilde{\xi}(\ell^+)$ and $\tilde{\ell}(\ell^-)$ generate a $2$-plane $V$ in $\R^{2,n+1}$ of signature $(1,1)$. Moreover, the orthogonal projection of a vector $w \in \R^{2,n+1}$ onto $V$ can be expressed as
	\[
	p(w) = \frac{\scal{w}{\tilde{\xi}(u_-)}}{\scal{\tilde{\xi}(u_+)}{\tilde{\xi}(u_-)}} \tilde{\xi}(u_+) + \frac{\scal{w}{\tilde{\xi}(u_+)}}{\scal{\tilde{\xi}(u_+)}{\tilde{\xi}(u_-)}} \tilde{\xi}(u_-) .
	\]
	From here a simple computation shows that $\log \abs{\beta^{\rho}(\ell^+,\ell^-,u_l,u_r)}$ coincides with the signed distance between the projective classes of $p(\tilde{\xi}(u_l))$ and $p(\tilde{\xi}(u_r))$ along the oriented spacelike geodesic $[\xi(\ell^-), \xi(\ell^+)]$, which can be parametrized by
	\[
	\ell(t) = \frac{1}{\sqrt{- 2 \scal{\tilde{\xi}(u_+)}{\tilde{\xi}(u_-)}}} (e^t \tilde{\xi}(u_+) + e^{-t} \tilde{\xi}(u_-)) .
	\]
	
	On the other hand, the projection $p(\xi(u_l))$ can be characterized in terms of the hyperbolic metric of $\hat{R} \cup \hat{R}'$ as the unique point of the line $\ell = [\xi(\ell^-), \xi(\ell^+)]$ that is joined to the ideal vertex $\xi(u_l)$ by a geodesic ray in $H_l$ orthogonal to $\ell$, and similarly for $p(\xi(u_r))$ and $H_r$. Since the classical hyperbolic shear between two ideal triangles (or spikes) that share a boundary geodesic $h$ coincides with the signed distance between the projection of their ideal vertices different from $h^\pm$, we deduce that $\log \abs{\beta^{\rho}(u_+,u_-,u_l,u_r)}$ coincides with the classical notion of shear between the plaques $\hat{R}, \hat{R}'$.
\end{rmk}

We now highlight a few properties satisfied by finite $\beta$-shears. Since the proofs of these relations are elementary and only rely on the symmetries of the cross ratio $\beta$, we postpone them to Appendix \ref{shears and symmetries}. In what follows, we fix a maximal geodesic lamination $\lambda$, and we denote by $\lambda_c$ the sublamination of $\lambda$ consisting of the lifts of all simple closed geodesics contained in $\lambda$. Notice that $\lambda_c$ is non-empty for any finite leaved maximal lamination $\lambda$.

\subsubsection{Shear between plaques sharing a vertex} Let $P$ and $Q$ be two distinct plaques of $\lambda$ that share an ideal vertex $w \in \partial \Gamma$. We label the vertices of $P$ and $Q$ that are different from $w$ as $u_P, v_P$ and $u_Q$, $v_Q$, respectively, so that the leaves $[w,v_P]$ and $[w,v_Q]$ separate the interior of the plaque $P$ from the interior of the plaque $Q$. Then we have:

\begin{lem} \label{lem:asymptotic_plaques}
	For every finite subset $\mathcal{P} \subset \mathcal{P}_{P Q}$
	\[
	\sigma_{\mathcal{P}}^\beta(P,Q) = \log \abs{\beta(w,v_P, u_P, v_Q) \beta(w,v_Q,v_P,u_Q)} .
	\]
	In particular the shear between $P$ and $Q$ is independent of the selected family of plaques $\mathcal{P} \subset \mathcal{P}_{P Q}$.
\end{lem}

\subsubsection{Shear between plaques asymptotic to a closed leaf}

Consider now two plaques $P$ and $Q$ of $\lambda$ that are separated by exactly one component $\ell$ of $\lambda_c$. Select arbitrarily an orientation of $\ell$, and assume that the plaque $P$ has a vertex equal to $\ell^+$ and that lies on the left of $\ell$. Similarly, assume that $Q$ lies on the right of $\ell$ and that one of its vertices is equal to $\ell^-$. We denote by $x_P, y_P$ and $x_Q, y_Q$ the vertices of $P$ and $Q$ different from $\ell^+$ and $\ell^-$, respectively, so that $[y_P, \ell^+]$ and $[y_Q, \ell^-]$ are the boundary components of $P$ and $Q$ closest to $\ell$. Then we have:

\begin{lem}\label{lem:shear near closed leaves}
	For every finite subset $\mathcal{P} \subset \mathcal{P}_{P Q}$
	\[
	\sigma^\beta_\mathcal{P}(P,Q) = \log \abs{\beta(\ell^+, y_P, x_P, \ell^-) \beta(\ell^-, \ell^+, y_Q, y_P) \beta(\ell^-,y_Q, x_Q, \ell^+)} .
	\]
	In particular the shear between $P$ and $Q$ is independent of the selected family of plaques $\mathcal{P} \subset \mathcal{P}_{P Q}$.
\end{lem}

\subsection{Shear cocycles: Finite case} \label{subsec:shear finite case}

We now focus on the construction of $\beta$-shear cocycles $\sigma^\beta_\lambda$ associated to \emph{finite leaved} maximal geodesic laminations, and the investigation of their properties. Thanks to the relations described in Lemmas~\ref{lem:asymptotic_plaques} and~\ref{lem:shear near closed leaves}, it is possible to carry out the analysis of shear cocycles with respect to finite leaved laminations in a fairly elementary way, without any subtle approximation argument. 

Even though not generic, the convenience of examining the finite leaved case separately is twofold. On the one hand, it serves as a guideline and motivation for the analysis in the general case. On the other, the naturality of $\beta$-shear cocycles for finite leaved laminations, combined with the continuous dependence from Proposition~\ref{pro:continuity shear wrt lamination}, shows that the approximation process described in Section~\ref{subsec:shear general case} produces cocycles that are independent of their construction (see in particular Section~\ref{subsec:divergence radius function} and Lemma~\ref{lem:limits of shear}).

Until the end of the current section, $\lambda$ will denote a finite leaved maximal lamination of $\Sigma$. Recall that every leaf of a lamination of this form projects in $\Sigma$ either onto a simple closed geodesic, or onto a simple bi-infinite geodesic, and in the latter case each of its ends accumulates onto a (possibly common) simple closed geodesic. 

We start by outlining the definition of the $\beta$-shear cocycle relative to $\lambda$. Consider two plaques $P$ and $Q$ of $\lambda$, and denote by $\ell_P$ and $\ell_Q$ the boundary leaves of $P$ and $Q$, respectively, that separate the interior of $P$ from the interior of $Q$. Notice that the geodesics $\ell_P, \ell_Q$, lying in the boundary of a plaque of $\lambda$, project onto simple bi-infinite geodesics in $\Sigma$.
We then choose arbitrarily an oriented geodesic segment $k$ starting at a point in the interior of $P$ and reaching a point in the interior of $Q$. By compactness, there exist only finitely many leaves of $\lambda$ that intersect $k$ and that project onto simple closed geodesics of $\Sigma$. We label them as $\ell_1, \dots, \ell_n$, following the order in which we meet them moving along the segment $k$, and we orient each $\ell_i$ from right to left with respect to $k$. For any $i$, we now select plaques $P_i$ and $Q_i$ that lie on the left and on the right of $\ell_i$, respectively, and that have $\ell_i^+$ or $\ell_i^-$ as one of their vertices (if $P$ has a vertex equal to $\ell_1^+$ or $\ell_1^-$, then we choose $P_1 = P$, and similarly for $Q$, $Q_n$, and $\ell_n$). Since $Q_i$ and $P_{i + 1}$ are not separated by any lift of simple closed leaves, the set of plaques $\mathcal{P}_{Q_i P_{i+1}}$ is finite for any $i = 1, \dots, n - 1$. For the same reason we see that the sets $\mathcal{P}_{P P_1}$ and $\mathcal{P}_{Q_n Q}$ are finite. Finally, we set 
\[
\mathcal{P} : = \mathcal{P}_{P P_1} \cup \mathcal{P}_{Q_1 P_2} \cup \cdots \cup \mathcal{P}_{Q_{n-1} P_n} \cup \mathcal{P}_{Q_n Q} ,
\]
and we define the \emph{$\beta$-shear between $P$ and $Q$} to be
\[
\sigma^\beta_\lambda(P,Q) := \sigma^\beta_\mathcal{P}(P,Q) .
\]
It is not difficult to show that the quantity $\sigma^\beta_\lambda(P,Q)$ is independent of the collection of plaques $\mathcal{P}$ selected following the aforementioned procedure. To see this, it is in fact enough to check that $\sigma^\beta_\mathcal{P}(P,Q) = \sigma^\beta_{\mathcal{P}'}(P,Q)$ for any finite extension $\mathcal{P}' \supset \mathcal{P}$ obtained as above. Notice that any such $\mathcal{P}' \subset \mathcal{P}_{P Q}$ is of the form
\[
\mathcal{P}' = \mathcal{P}_{P P_1'} \cup \mathcal{P}_{Q_1' P_2'} \cup \cdots \cup \mathcal{P}_{Q_{n-1}' P_n'} \cup \mathcal{P}_{Q_n' Q} ,
\]
where $P_i'$ and $Q_i'$ are plaques that lie on the left and on the right of $\ell_i$, respectively. In fact, it is not restrictive to assume that $P_i' \neq P_i$ and $Q_i \neq Q_i'$, in which case both $P_i'$ and $Q_i'$ separate $P_i$ from $Q_i$. Observe also that both pairs $P_i, P_i'$ and $Q_i, Q_i'$ share one of the endpoints of $\ell_i$ as a vertex, and exactly one of the following hold: Either the plaques $P_i$ and $Q_i$ (and consequently $P_i'$ and $Q_i'$) have a common vertex, equal to $\ell_i^+$ or $\ell_i^-$, or $P_i$ and $Q_i$ do not share any vertex.
By Lemma~\ref{lem:asymptotic_plaques} in the former case, and Lemma~\ref{lem:shear near closed leaves} in the latter, we have
$$\sigma^\beta(P_i, Q_i) = \sigma^\beta(P_i, P_i') +  \sigma^\beta(P_i',Q_i') +  \sigma^\beta(Q_i',Q_i) $$ 
for every $i = 1, \dots, n$, which implies the equality between the finite $\beta$-shears computed with respect to the set of plaques $\mathcal{P}$ and $\mathcal{P}'$. 

To prove that $\sigma^\beta_\lambda$ is indeed a transverse H\"older cocycle (see Definition~\ref{def:holder cocycle}), we have to check symmetry, invariance, and additivity. The first two properties are straightforward and follow directly from the definition. The last property requires some care. The main observation is that the ordered family of leaves $\ell_1,\cdots,\ell_n$ that meet a geodesic arc $k$ joining $P$ to $Q$ and project to simple closed geodesics on $\Sigma$ does not depend on $k$ but only on the cyclic order of the endpoints of the leaves of the lamination. If $R$ separates $P$ from $Q$ and $\ell_1,\cdots,\ell_n$ and $\ell'_1,\cdots,\ell'_m$ are the lifts of the closed geodesics that separate $P,R$ and $R,Q$ respectively, then $\ell_1,\cdots,\ell_n,\ell'_1,\cdots,\ell'_m$ are the lifts of the closed geodesics that separate $P$ from $Q$. Using the definition $\sigma^\beta_\lambda(P,Q)=\sigma^\beta_{\mc{P}}(P,Q)$ and carefully selecting the finite family of plaques $\mc{P}\subset\mc{P}_{PR}\cup\mc{P}_{RQ}$ shows that $\sigma^\beta_\lambda(P,Q)=\sigma^\beta_\lambda(P,R)+\sigma^\beta_\lambda(R,Q)$.

We can summarize the above discussion in the following statement:

\begin{pro} \label{pro:shear finite case}
	Let $\beta$ be a cross ratio. Then for every finite leaved maximal lamination $\lambda$, the map $(P,Q) \mapsto \sigma_\lambda^\beta(P,Q)$ defines a H\"older cocycle $\sigma^\beta_\lambda \in \mathcal{H}(\lambda;\R)$ naturally associated to $\beta$ and $\lambda$.
\end{pro}

\subsection{Shears and length functions: Finite case}
\label{subsec:shear and length I}

We conclude our analysis of $\beta$-shear cocycles associated to finite leaved maximal laminations examining their relations with the periods of the cross ratio $\beta$ (see in particular Proposition \ref{pro:length and thurston sympl finite}). As already observed in the work of Bonahon \cite{Bo96}, length functions provide a complete characterization of the set of transverse H\"older cocycles in $\mathcal{H}(\lambda;\R)$ that arise as shear cocycles of hyperbolic structures on a closed surface with respect to the maximal lamination $\lambda$ (see in particular Theorem \ref{thm:thurston bonahon}). The connection between $\beta$-shears and $\beta$-periods rely on the properties of the \emph{Thurston symplectic form} on $\mathcal{H}(\lambda;\R)$, a skew-symmetric non-degenerate bilinear form, whose definition is recalled in Section \ref{subsubsec:thurston sympl}.

The combination of the analysis in the finite leaved case (developed in this section) together with the continuity results of $\beta$-shear cocycles (described in Section \ref{subsec:continuity shear} below) will eventually allow us to relate $\beta$-shear cocycles of strictly positive and locally bounded cross ratios to hyperbolic structures on $\Sigma$, as described in Theorem \ref{shear of cross ratio improved}.

\subsubsection{Thurston symplectic form} 
\label{subsubsec:thurston sympl} 

We start by briefly recalling the definition of the Thurston symplectic form $\omega_\lambda$ on the space of transverse H\"older cocycles $\mathcal{H}(\lambda;\R)$ to a maximal geodesic lamination $\lambda$. As described by Bonahon in \cite{Bo96}*{\S~3}, the symplectic form $\omega_\lambda$ can be intrinsically characterized in terms of the intersection form on the space of geometric currents supported on $\lambda$, in the sense of \cite{RS75}. However, for our purposes it will be more convenient to have an elementary (but less intrinsic) description of the Thurston symplectic form, through the choice of a train track $\tau$ carrying $\lambda$ and its induced isomorphism $\mathcal{H}(\lambda;\R) \cong \mathcal{W}(\tau;\R)$ (see Section~\ref{subsec:continuity shear}). 

In the following we briefly introduce the necessary notation. Given any switch $s$ of the train track $\tau$, we denote by $B_s$ the unique branch of $\tau$ whose vertical boundary contains $s$,  and we select arbitrarily lifts $\widetilde{B}_s$ and $\tilde{s} \subset \widetilde{B}_s$ of $B_s$ and $s$ to the universal cover of $\Sigma$. The switch $\tilde{s}$ is then adjacent to two other distinct branches $\widetilde{B}^s_l$ and $\widetilde{B}^s_r$ of $\tilde{\tau}$, with $\widetilde{B}^s_l$ and $\widetilde{B}^s_r$ lying on the left and on the right of $\tilde{s}$, respectively, if seen from $\widetilde{B}_s$ with respect to the orientation of $\widetilde{\Sigma}$. For any transverse H\"older cocycle $\alpha \in \mathcal{H}(\lambda;\R)$, we then denote by $\alpha_l^s$ and $\alpha_r^s$ the real weights associated by $\alpha$ to the branches $\widetilde{B}_l^s, \widetilde{B}_r^s$. Finally, the \emph{Thurston symplectic form} $\omega_\lambda$ applied to $\alpha, \beta \in \mathcal{H}(\lambda;\R)$ can be expressed as
\begin{equation}\label{eq:thurston sympl}
	\omega_\lambda(\alpha, \beta) = \frac{1}{2} \sum_s (\alpha^s_r \, \beta^s_l - \alpha^s_l \, \beta^s_r) ,
\end{equation}
where the sum is taken over all switches $s$ of $\tau$. 

As recalled in Section~\ref{subsec:shear coordinates}, the Thurston symplectic form provides a characterization of the set of transverse H\"older cocycles that can be obtained as shear cocycles of hyperbolic structures on $\Sigma$ (see Theorem~\ref{thm:thurston bonahon}, or \cite{Bo96}*{Theorem~20}). In addition, it is worth to mention that the Thurston symplectic form is also intimately related to the geometry of Teichm\"uller space, and in particular to its Weil-Petersson symplectic structure, as observed in \cite{SB01}. We refer to \cite{Bo96}*{\S~3} (see also \cite{Pap86}, \cite{PH92}*{\S~3.2}, \cite{SB01}) for a more detailed description of the Thurston symplectic form and its properties.

\subsubsection{Lengths} 

The relation between $\beta$-shear cocycles with respect to finite leaved laminations and $\beta$-periods relies on elementary arguments. The main ingredients are the combinatorial description of the Thurston symplectic form from relation \eqref{eq:thurston sympl}, and the following statement:

\begin{lem} \label{lem:shear and length}
	Let $\lambda$ be a finite leaved maximal lamination, and let $\gamma$ be a non-trivial element of $\Gamma$ whose axis $\tilde{\gamma}$ projects onto a closed leaf of $\lambda$. Consider a plaque $P$ of $\lambda$ that has one of the endpoints $\gamma^\pm$ of $\tilde{\gamma}$ as a vertex, and assume that it lies on the left of $\tilde{\gamma}$ . Then
	\[
	\sigma_\lambda^\beta(P, \gamma P) = \pm L_\beta(\gamma),
	\]
	with positive sign if $P$ has $\gamma^+$ as one of its vertices, and with negative sign otherwise.
\end{lem}

The proof of Lemma \ref{lem:shear and length} relies only on the symmetries of the cross ratio $\beta$, and we postpone it to Appendix \ref{shears and symmetries}.

We conclude the current section with the following result, which will play an important role in the proof of Theorem \ref{shear of cross ratio improved}, described in Section \ref{sec:shear cocycles}:

\begin{pro}\label{pro:length and thurston sympl finite}
	Let $\beta$ be a positive cross ratio. Then for every finite leaved maximal lamination $\lambda$ and for weighted multicurve $\mu$ with $\supp \mu \subseteq \lambda$, we have 
	$$L_\beta(\mu) = \omega_\lambda(\sigma^\beta_\lambda, \mu),$$ 
	where $\omega_\lambda$ denotes the Thurston symplectic form on the space of transverse H\"older cocycles $\mathcal{H}(\lambda;\R)$.
\end{pro}

\begin{proof}
	It is sufficient to consider the case in which $\mu$ consists of a single simple closed curve with weight $1$. Let $\gamma$ be an element of $\Gamma$ whose axis $\tilde{\gamma}$ projects onto a simple closed leaf of $\lambda$. We orient $\tilde{\gamma}$ so that it points towards the attracting fixed point $\gamma^+ \in \partial \Gamma$ and moves away from the repelling fixed point $\gamma^-$. We will denote by $\tilde{\gamma}^{-1}$ the axis of $\gamma$ endowed with the opposite orientation, and by $\abs{\gamma} \in \mathcal{C}$ the associated geodesic current.
	
	We now select a train track carrying $\lambda$. Being $\lambda$ a finite leaved lamination, there exist bi-infinite leaves of $\lambda$ that spiral around $\gamma$ both from its left and its right (with respect to the orientation of $\gamma$ and of the surface $\Sigma$). It is not restrictive to assume that $\lambda$ is carried by a train track $\tau$ obtained from a tubular neighborhood of $\gamma$ by adding two branches on its sides, and then properly extended away from $\gamma$. (Such a train track can be obtained by taking a sufficiently small metric neighborhood of $\lambda$ with respect to some hyperbolic structure $X$, and possibly by a small deformation to guarantee the trivalence of every switch.)
	
	In order to provide an explicit expression for the evaluation of the Thurston symplectic form $\omega_\lambda(\sigma^\beta_\lambda, \abs{\gamma})$ with respect to the train track $\tau$, we need to introduce some notation. Let $k$ be a tie of the train track $\tau$ that intersects $\gamma$. We select arbitrarily a lift $\tilde{k}$ of $k$ in $\widetilde{\Sigma}$ that crosses $\tilde{\gamma}$, and we denote by $P$ and $Q$ the plaques of $\lambda$ that contain the endpoints of $\tilde{k}$, so that $P$ lies on the left of $\tilde{\gamma}$ and $Q$ on its right. Both plaques have an ideal vertex equal to $\tilde{\gamma}^+$ or $\tilde{\gamma}^-$. We now introduce the following sign convention: We say that the left sign of $\tau$ with respect to $\gamma$, denoted by $\textrm{sgn}_l(\tau,\gamma)$, is equal to $+ 1$ if the plaque $P$ lying on the left of $\tilde{\gamma}$ has $\gamma^+$ as one of its ideal vertices, and we set it equal to $- 1$ otherwise. On the other hand, we define the right sign of $\tau$ with respect to $\gamma$ to be $\textrm{sgn}_r(\tau,\gamma) = + 1$ if the plaque $Q$ lying on the right of $\tilde{\gamma}$ has $\gamma^-$ as one of its vertices, and $-1$ otherwise. It is not difficult to see that the sign functions $\textrm{sgn}_l(\tau,\gamma)$ and $\textrm{sgn}_r(\tau,\gamma)$ depend only on the train track $\tau$ and the choice of the orientation of the curve $\gamma$: For instance, we can alternatively define $\textrm{sgn}_l(\tau,\gamma)$ to be $+1$ if the branch of $\tau$ that enters in the tubular neighborhood of $\gamma$ from its left follows the orientation of $\gamma$ and $-1$ otherwise; a similar description holds for $\textrm{sgn}_r(\tau,\gamma)$. 
	
	There are only finitely many possible configurations for the switches and branches of $\tau$ that intersected $s$. By applying relation \eqref{eq:thurston sympl} to each possible configuration, we obtain the expression:
	\[
	\omega_\lambda(\sigma^\beta_\lambda, \abs{\gamma}) = \frac{1}{2} \left( \mathrm{sgn}_l(\tau, \gamma) \, \sigma^\beta_\lambda(P, \gamma P) - \mathrm{sgn}_r(\tau, \gamma) \, \sigma^\beta_\lambda(Q, \gamma Q) \right) .
	\]
	Notice that by definition $\mathrm{sgn}_r(\tau,\gamma) = - \mathrm{sgn}_l(\tau,\gamma^{-1})$. Moreover, for any plaque $R$ of $\lambda$ we have 
	$$\sigma_\lambda^\beta(R,\gamma R) = \sigma_\lambda^\beta(\gamma^{-1} R, R) = \sigma_\lambda^\beta(R, \gamma^{-1} R) ,$$
	since $\sigma^\beta_\lambda$ is a transverse H\"older cocycle. In particular, the term $\omega_\lambda(\sigma^\beta_\lambda, \abs{\gamma})$ can be equivalently expressed as
	\[
	\omega_\lambda(\sigma^\beta_\lambda, \abs{\gamma}) = \frac{1}{2} \left( \mathrm{sgn}_l(\tau, \gamma) \, \sigma_\lambda^\beta(P, \gamma P) + \mathrm{sgn}_l(\tau, \gamma^{-1}) \, \sigma^\beta_\lambda(Q, \gamma^{-1} Q)\right) .
	\]
	Since $P$ lies on the left of $\tilde{\gamma}$ and $Q$ lies on the left of $\tilde{\gamma}^{-1}$, we can now apply Lemma~\ref{lem:shear and length} to both terms appearing above, obtaining
	\[
	\omega(\sigma^\beta_\lambda, \abs{\gamma}) = \frac{1}{2} \left( L_\beta(\gamma) + L_\beta(\gamma^{-1}) \right) = L_\beta(\gamma) ,
	\]
	which proves the desired identity.
\end{proof}

%%%

\section{Hyperbolic structures on pleated sets II}
\label{sec:shear cocycles}

This section is dedicated to the definition and the study of shear cocycles associated to positive and locally bounded cross ratios and general maximal geodesic laminations. We generalize the phenomena observed in Section \ref{subsec:shear and length I} for finite leaved maximal laminations, and we provide a proof of Theorem \ref{shear of cross ratio improved} in full generality.

The construction of the $\beta$-shear cocycle $\sigma_\lambda^\beta\in\mc{H}(\lambda;\mb{R})$ for a general maximal lamination $\lambda$ will require several auxiliary choices and a fine analysis of the convergence of finite $\beta$-shears. Nevertheless, we will observe that the resulting $\beta$-shear cocycles satisfy a series of natural properties:

\begin{enumerate}[(a)]
	\item{For every \emph{finite leaved} maximal lamination $\lambda$, the transverse H\"older cocycle $\sigma_\lambda^\beta\in\mc{H}(\lambda;\mb{R})$, obtained through the general process described in Section \ref{subsec:constructing cocycles}, coincides with the $\beta$-shear cocycle introduced in Section \ref{subsec:shear finite case} (see in particular Proposition \ref{pro:shear finite case}).}
	\item{The map $$\mc{GL} \ni \lambda \longmapsto \sigma_\lambda^\beta \in \mc{W}(\tau;\mb{R})$$ is continuous with respect to the Hausdorff topology on the space of maximal geodesic laminations. Here $\mc{W}(\tau;\mb{R})$ is the space of real weights of a suitable train track $\tau$ that carries $\lambda$.}
\end{enumerate}

Consequently, since maximal finite leaved laminations are dense in the entire set of maximal geodesic laminations (see e.g. \cite{CEG}*{Theorem~I.4.2.19}) and since the $\beta$-shear cocycles $\sigma_\lambda^\beta$ constructed in Section \ref{subsec:shear finite case} do not require any auxiliary choice, we can conclude that the transverse H\"older cocycle $\sigma_\lambda^\beta$ only depends on $\lambda$ and $\beta$, even in the case of a general maximal geodesic lamination.

\subsubsection*{Outline of the construction}

The $\beta$-shear $\sigma_\lambda^\beta(P,Q)$ between the plaques $P,Q$ will be defined as a limit of certain finite $\beta$-shears $\sigma_{\mc{P}_n}^\beta(P,Q)$ associated to a suitably chosen exhaustion $\mc{P}_n$ of the set of plaques $\mc{P}_{PQ}$ separating $P$ from $Q$. The choice of $\mc{P}_n$ depends on the geometry of $\lambda$ on a fine scale. More precisely, in order to select it, we will use a divergence radius function associated to the choice of a hyperbolic structure $X$, a train track $\tau$ that carries $\lambda$, and a geodesic arc $k$ joining $P$ to $Q$. 

We emphasize however that, as previously observed, the continuity properties of the construction (Proposition \ref{pro:continuity shear wrt lamination}) and the naturality in the case of finite leaved laminations (see Section \ref{subsec:shear finite case}) make the cocycle $\sigma_\lambda^\beta(P,Q)$ independent of all the auxiliary choices required for its definition. The rest of the section is structured as follows:

\begin{description}
	\item[\S~\ref{subsec:divergence radius function}] We dedicate this section to the description of \emph{divergence radius functions}, which were originally introduced by Bonahon in \cite{Bo96} to study the convergence of the shearing maps between hyperbolic surfaces. 
	\item[\S~\ref{subsec:shear general case}] In this section we give the general definition of the $\beta$-shear $\sigma^\beta_\lambda(P,Q)$: We deploy divergence radius functions to carefully select an exhausting sequence of finite sets of plaques $(\mathcal{P}_n)_n$ inside $\mathcal{P}_{P Q}$, whose associated finite shears converges. 
	\item[\S~\ref{subsec:continuity shear}] In this section we prove that $\beta$-shear cocycles $\sigma^\beta_\lambda$ satisfy a suitable notion of continuity with respect to the maximal lamination $\lambda$ (endowing the space of maximal geodesic laminations with the Chabauty topology). 
	\item[\S~\ref{subsec: shear and length}] We then study the relations between $\beta$-shear cocycles and $\beta$-periods associated to a positive and locally bounded cross ratio $\beta$, generalizing what observed in Section \ref{subsec:shear and length I} for finite leaved laminations (see in particular Proposition \ref{pro:length and thurston sympl}).
	\item[\S~\ref{subsec: proof of shear}] We conclude our analysis with the proof of Theorem \ref{shear of cross ratio improved}, combining the results from the previous sections with Bonahon's shear parametrization (see Theorem \ref{thm:thurston bonahon})
\end{description}

\subsection{Divergence radius functions}
\label{subsec:divergence radius function}

In order to define the $\beta$-shear between two plaques $P$ and $Q$ in the case of a general maximal lamination, we will need to determine an exhaustion $(\mathcal{P}_n)_n$ of the set of plaques separating $P$ from $Q$ whose associated sequence of finite shears $(\sigma_{\mathcal{P}_n}^\beta(P,Q))_n$ converges (compare with Section \ref{subsec:definition shear}). This part of our analysis requires some care, because of the (not particularly strong) control between finite $\beta$-shears associated to different collections of plaques (see Lemma~\ref{lem:finite_approx_shear}). In particular, we make use of the so-called \emph{divergence radius function} $r : \mathcal{P}_{P Q} \to \N$, associated to the choice of a trivalent train track carrying the lamination $\lambda$ (see Section \ref{subsubsec:train tracks} for the related terminology), a hyperbolic structure $X$ on $\Sigma$, and a ($X$-)geodesic path joining $P$ to $Q$ (see Bonahon-Dreyer \cite{BD17}, and Bonahon \cite{Bo96}*{\S~1}). Any such function depends on the choice of:
\begin{itemize}
	\item A hyperbolic structure $X$ on $\Sigma$.
	\item A (trivalent) train track $\tau$ inside $\Sigma$.
	\item A maximal geodesic lamination $\lambda$ (which will be identified with its $X$-geodesic realization in the universal cover of $(\Sigma,X)$) carried by $\tau$.
	\item Two distinct plaques $P$ and $Q$ of $\lambda$.
	\item A geodesic segment $k$ that joins a point in the interior of $P$ to a point in the interior of $Q$.
\end{itemize}
Once we fix these data, the associated \emph{divergence radius function}
\[
r = r_{X, \tau, \lambda,k} : \mathcal{P}_{P Q} \longrightarrow \N
\]
associates to every plaque $R$ that separates $P$ from $Q$ a natural number $r(R)$, which roughly measures the length of the geodesic arc $R \cap k$ in terms of the combinatorics of the fixed train track $\tau$ and the boundary leaves of $R$ that intersect $k$.

In order to be more precise, we need to introduce some notation. For any plaque $R \in \mathcal{P}_{P Q}$, let $\ell_R, \ell_R'$ be the boundary leaves of $R$ that intersect the arc $k$. If the geodesic segment $R \cap k$ is not entirely contained in the lift $\tilde{\tau}$ of the train track $\tau$ to the universal cover of $\Sigma$, then we set $r(R) = 0$. If this does not occur, then the intersection points between $k$ and the boundary leaves $\ell_R, \ell_R'$ lie in a common branch $\widetilde{B}_0$ of $\tilde{\tau}$. We now orient $\ell_R, \ell_R'$ so that they share their negative endpoint, and we denote by
\[
\dots, \widetilde{B}_2, \widetilde{B}_{-1}, \widetilde{B}_0, \widetilde{B}_1, \widetilde{B}_2, \dots
\]
the branches of $\tilde{\tau}$ that $\ell_R$ passes through, indexed in consecutive order according to the orientation of $\ell_R$. We then define $r(R) : = n + 1$, where $n$ is the largest natural number such that $\ell_R'$ passes through the branches $\widetilde{B}_m$ for every integer $m \in \{- n , - n + 1, \dots, n -1 , n\}$. Then we have:

\begin{lem}[{see \cite{BD17}*{Lemma~5.3}}]\label{lem:auxiliary function r}
	The divergence radius function $r : \mathcal{P}_{P Q} \to \N$ satisfies the following properties:
	\begin{enumerate}
		\item there exist constants $A, M > 0$ such that
		\[
		A^{-1} \, e^{- M^{-1} \, r(R)} \leq L_{\tilde{X}}(k \cap R) \leq A \, e^{- M \, r(R)} 
		\]
		for every $R \in \mathcal{P}_{P Q}$;
		\item there exists $N \in \N$ such that, for every $n \in \N$ the preimage $r^{-1}(n)$ contains at most $N$ plaques.
	\end{enumerate}
\end{lem}

Divergence radius functions were first introduced by Bonahon in \cite{Bo96}*{\S~1} to study the convergence of shear maps with respect to maximal geodesic laminations (see also Bonahon \cites{Bo97transv,Bo97geo}, and Bonahon and Dreyer \cite{BD17}). In our exposition, these functions will be useful to select exhaustions $(\mathcal{P}_n)_n$ of $\mathcal{P}_{P Q}$ by finite nested subsets whose associated finite $\beta$-shear $\sigma^\beta_{\mathcal{P}_n}(P,Q)$ converge. However, in certain steps of our analysis (see in particular Lemma~\ref{lem:limits of shear} and Proposition~\ref{pro:continuity shear wrt lamination}), it will useful to have a better understanding of the dependence of the functions $r$ and of the corresponding constants $A, M, N$ with respect to the choices of the lamination $\lambda$ carried by $\tau$, and the transverse path $k$. We summarize the necessary refinements of Lemma~\ref{lem:auxiliary function r} in the following statements. Fixed a hyperbolic structure $X$ on $\Sigma$, a train track $\tau$, a maximal lamination $\lambda$ carried by $\tau$, and two plaques $P$ and $Q$ of $\lambda$, we have:

\begin{lem} \label{lem:dependence of constants}
	For any geodesic arc $k$ joining the interiors of $P$ and $Q$, there exist constants $A, M, N > 0$ and a open neighborhood $U$ of $\lambda$ in the space of maximal laminations (endowed with the Chabauty topology) such that the following properties hold:
	\begin{itemize}
		\item Every maximal lamination $\lambda' \in U$ is carried by $\tau$.
		\item The $X$-geodesic path $k$ is transverse to (the $X$-geodesic realization of) every $\lambda' \in U$.
		\item For every $\lambda' \in U$ there exist distinct plaques $P', Q'$ of $\lambda'$ in $X$ that contain the endpoints of $k$.
		\item For any maximal lamination $\lambda' \in U$ and plaques $P', Q'$ as above, the associated divergence radius function $r' = r_{X, \tau, \lambda', k} : \mathcal{P}_{P' Q'} \to \N$ satisfies properties (1) and (2) in Lemma~\ref{lem:auxiliary function r} with constants $A, M, N > 0$ (which in particular are uniform in $\lambda' \in U$).
	\end{itemize}
\end{lem}

\begin{lem} \label{lem:dependence of r on metric}
	For any choice of $X$-geodesic paths $k$ and $k'$ with endpoints lying in (the geodesic realizations of) the plaques $P, Q$, the associated divergence radius functions $r = r_{k}, r' = r_{k'} : \mathcal{P}_{P Q} \to \N$ provided by Lemma~\ref{lem:auxiliary function r} are coarsely equivalent, i. e. there exist constants $H, K > 0$ such that
	\[
	H^{-1} \, r'(R) - K \leq r(R) \leq H \, r'(R) + K
	\]
	for every plaque $R \in \mathcal{P}_{P Q}$.
\end{lem}
We postpone the proofs of Lemmas~\ref{lem:auxiliary function r},~\ref{lem:dependence of constants}, and~\ref{lem:dependence of r on metric} to Appendix \ref{divergence appendix}.

\subsection{Shear cocycles: General case} \label{subsec:shear general case}

We now focus our attention on the construction of $\beta$-shear cocycles relative to a general maximal lamination $\lambda$. For the remainder of Section~\ref{sec:shear cocycles}, we will assume the cross ratio $\beta : \partial\Gamma^{(4)} \to \R$ to be \emph{locally bounded} (see Definitions~\ref{def:crossratio},~\ref{def:locally bounded}). Furthermore, we fix once and for all a hyperbolic structure $X$, and a train track $\tau$ carrying $\lambda$. 

We start our analysis with two elementary Lemmas: The first (Lemma \ref{lem:diagonal_exchange}) describes how the shear between two plaques changes under the operation of \emph{diagonal exchange} in the region separating $P$ from $Q$. The second (Lemma \ref{lem:finite_approx_shear}) provides a bound between finite $\beta$-shears computed with respect to two finite families of plaques $\mathcal{P}, \mathcal{P}' \subset \mathcal{P}_{P Q}$ with $\mathcal{P} \subseteq \mathcal{P}'$. The bound described by Lemma \ref{lem:finite_approx_shear} will be essential for the study of the approximation process needed to define $\sigma^\beta_\lambda$.

\subsubsection{Change of shear under diagonal exchange}

Let $P, Q$ be two plaques of $\lambda$ that share no ideal vertex. We denote by $\ell_P$ (resp. $\ell_Q$) the boundary leaf of $P$ (resp. $Q$) that separates the interior of $P$ from the interior of $Q$, and by $S$ the region of $\widetilde{\Sigma}$ bounded by $\ell_P$ and $\ell_Q$. Given a coherent orientation of $\ell_P$ and $\ell_Q$, we define $d$ and $d'$ to be the crossing geodesics $[\ell_P^+, \ell_Q^-]$ and $[\ell_P^-,\ell_Q^+]$, respectively. Finally, let $R, T$ (resp. $R', T'$) denote the complementary regions of $d$ (resp. $d'$) inside $S$. 

To simplify the notation, we set
\begin{align*}
	\sigma_d^\beta(P,Q) & : = \sigma^\beta(P,R) + \sigma^\beta(R,T) + \sigma^\beta(T,Q) , \\
	\sigma_{d'}^\beta(P,Q) & : = \sigma^\beta(P,R') + \sigma^\beta(R',T') + \sigma^\beta(T',Q) .
\end{align*}
Then we have:

\begin{lem}\label{lem:diagonal_exchange}
	The following relation holds:
	\[
	\abs{\sigma_d^\beta(P,Q) - \sigma_{d'}^\beta(P,Q)} = 2 \abs{\log \beta(\ell_P^+,\ell_Q^+,\ell_Q^-,\ell_P^-)} .
	\]
\end{lem}

As for Lemmas~\ref{lem:asymptotic_plaques} and~\ref{lem:shear near closed leaves}, the proof of Lemma~\ref{lem:diagonal_exchange} is an elementary consequence of the symmetries satisfied by the cross ratio $\beta$, and it will be described in Appendix \ref{shears and symmetries}. 

\subsubsection{Enlarging the finite set of plaques}

The next goal is to determine the behavior of the finite shear $\sigma^\beta_\mathcal{P}$ as we enlarge the finite family of plaques $\mathcal{P} \subset \mathcal{P}_{P Q}$. The statement that follows will play an essential role in the approximation process to determine $\sigma^\beta_{\lambda}(P,Q)$. Recall that, since $\beta$ is a locally bounded cross ratio (see Definition~\ref{def:locally bounded}), for $D = L_{\tilde{X}}(k) > 0$ (the length of $k$ in $(\widetilde{\Sigma}, \widetilde{X})$), we can find constants $C, \alpha > 0$ (depending on the fixed hyperbolic structure $X$, the cross ratio $\beta$, and $L_{\tilde{X}}(k)$) such that
\begin{equation} \label{eq:control crossratio}
	\abs{\log \beta(h{}^+,\ell^+,\ell^-,h{}^-)} \leq C \abs{\log \beta^X(h{}^+,\ell^+,\ell^-,h{}^-)}^\alpha
\end{equation}
for every pair of coherently oriented geodesics $\ell,h$ in $(\widetilde{\Sigma}, \widetilde{X})$ such that $0 < d_{\tilde{X}}(\ell,h) \leq L_{\tilde{X}}(k)$. We then have:

\begin{lem} \label{lem:finite_approx_shear}
	For any pair of finite subsets $\mathcal{P}, \mathcal{P}'$ of $\mathcal{P}_{P Q}$ satisfying $\mathcal{P} \subseteq \mathcal{P}'$, we have
	\[
	\abs{\sigma^\beta_{\mathcal{P}}(P,Q) - \sigma^\beta_{\mathcal{P}'}(P,Q)} \leq 2 C \, \abs{\mathcal{P}' - \mathcal{P}} \left(\sum_{d \subset k - \bigcup \mathcal{P}} L_{\tilde{X}}(d)^\alpha \right),
	\]
	where $\abs{\mathcal{P}' - \mathcal{P}}$ denotes the cardinality of the set $\mathcal{P}' - \mathcal{P}$, $d$ varies among the (countable) set of connected components of $k - \bigcup \mathcal{P}$, and $C, \alpha$ are the constants associated with $X$, $\beta$, $L_{\tilde{X}}(k)$ as above.
\end{lem}

\begin{proof}
	We first consider the case in which $\mathcal{P}' = \mathcal{P} \cup \{ R \}$. If 
	\[
	P = P_0, P_1, \dots , P_n, P_{n + 1} = Q
	\]
	are the plaques of $\mathcal{P}$, indexed as we encounter them along the arc $k$ from $P$ to $Q$, then the plaque $R$ will lie inside one of the components of $\widetilde{\Sigma} - \bigcup \mathcal{P}$ that separate $P_i$ from $P_{i + 1}$, for some $i$. We will denote by $S$ such a region.
	
	The laminations $\lambda_{\mathcal{P}}$ and $\lambda_{\mathcal{P}'}$ differ by a sequence of elementary moves, each of which either adds leaves to the lamination, or performs a diagonal exchange inside $S$. By Lemma~\ref{lem:asymptotic_plaques}, the shear between $P$ and $Q$ computed through the intermediate laminations $\lambda$ and $\lambda'$ does not change when $\lambda'$ is obtained from $\lambda$ by introducing new leaves. Therefore, it is sufficient to compute the change of the shear cocycle that occurs when a diagonal exchange is performed.
	
	Let $\sigma^\beta_\lambda$ and $\sigma^\beta_{\lambda'}$ be the shears associated with the plaques $P$ and $Q$ through the finite laminations $\lambda$ and $\lambda'$, respectively, which differ by a diagonal exchange in the region bounded by the leaves $\ell$ and $h$. We select orientations on $\ell$ and $h$ so that they are coherently oriented. By Lemma~\ref{lem:diagonal_exchange} we have
	\[
	\abs{\sigma^\beta_\lambda(P,Q) - \sigma^\beta_{\lambda'}(P,Q)} = 2 \abs{\log \beta(h{}^+,\ell^+,\ell^-,h{}^-)}
	\]
	Combining this equality with relation \eqref{eq:control crossratio} and Lemma~\ref{lem:crossratio and distance}, we deduce that 
	\[
	\abs{\sigma^\beta_\lambda(P,Q) - \sigma^\beta_{\lambda'}(P,Q)} \leq 2 C \, d_{\tilde{X}}(\ell,h)^\alpha \leq 2 C \, L_{\tilde{X}}(k \cap S)^\alpha ,
	\]
	where the last inequality holds since $k \cap S$ is a path that connects points lying on the leaves $\ell$ and $h$. When $\mathcal{P}$ and $\mathcal{P}'$ differ by a single plaque $R$, then by adding leaves and performing exactly one flip, we can move from the lamination $\lambda_\mathcal{P}$ to $\lambda_{\mathcal{P}'}$. If $\mathcal{P}'$ is obtained by adding to $\mathcal{P}$ $n_S$ plaques lying inside the same region $S$, then it is simple to check that $\lambda_\mathcal{P}$ and $\lambda_{\mathcal{P}'}$ differ by a suitable sequence of moves, exactly $n_S$ of which are diagonal exchanges. The difference in the shears $\sigma_{\mathcal{P}}^\beta(P,Q)$ and $\sigma_{\mathcal{P}'}^\beta(P,Q)$ can then be bounded by $2 C \, n_S \, L_{\tilde{X}}(k \cap S)^\alpha$, by the same argument outlined above. The desired statement follows by applying this process in any complementary region $S$ of $\widetilde{\Sigma} - \bigcup \mathcal{P}$, and noticing that $n_S \leq \abs{\mathcal{P}' - \mathcal{P}}$ for any $S$.
\end{proof}

\begin{rmk}
	Notice that the argument described above makes use of the local boundedness of $\beta$ only on pairs of leaves of the lamination $\lambda$. In particular, the machinery described in this section in fact applies to cross ratios $\beta$ that are \emph{$\lambda$-locally bounded}, i. e. that locally satisfy the control
	\[
	\abs{\log \beta(h{}^+,\ell^+,\ell^-,h{}^-)} \leq C \abs{\log \beta^X(h{}^+,\ell^+,\ell^-,h{}^-)}^\alpha
	\]
	for any pair of coherently oriented distinct leaves $\ell , h$ of the lamination $\lambda$.
\end{rmk}

\subsubsection{Constructing $\beta$-shear cocycles}
\label{subsec:constructing cocycles}

We are now ready to describe the approximation process for the $\beta$-shear cocycle $\sigma^\beta_\lambda$ in the case of a general maximal lamination. Throughout the current section, we denote by $P$ and $Q$ two distinct plaques of some fixed maximal lamination $\lambda$, and by $X$ an auxiliary hyperbolic structure on $\Sigma$. 

We start our construction by selecting a well behaved exhaustion $(\mathcal{P}_n)_n$ by nested finite subsets of $\mathcal{P}_{P Q}$ through the notion of divergence radius function. Concretely, let $k$ be a $X$-geodesic segment joining points in the interior of the plaques $P$ and $Q$, and let $r = r_{X, \tau, \lambda, k} : \mathcal{P}_{P Q} \to \N$ be the corresponding divergence radius function, defined as in Section~\ref{subsec:divergence radius function}. Then, for every $n \in \N$ we set
\[
\mathcal{P}_n : = \{ R \in \mathcal{P}_{P Q} \mid r(R) \leq n \}.
\]
Notice that by Lemma~\ref{lem:auxiliary function r} the cardinality of $\mathcal{P}_{n + 1} - \mathcal{P}_n$ is bounded above by a constant $N > 0$ independent of $n$, and the union $\bigcup_n \mathcal{P}_n$ is equal to $\mathcal{P}_{P Q}$. 

We are now ready to prove the first technical step of our construction:

\begin{lem} \label{lem:limits of shear}
	The series
	\[
	\sum_n \abs{\sigma^\beta_{\mathcal{P}_{n}}(P, Q) - \sigma^\beta_{\mathcal{P}_{n + 1}}(P, Q)} 
	\]
	converges, and in particular the limit
	\[
	\sigma^\beta_\lambda(P,Q) : = \lim_{n \to \infty} \sigma^\beta_{\mathcal{P}_{n}}(P, Q) 
	\]
	is finite. Moreover, the quantity $\sigma_\lambda^\beta(P,Q) \in \R$ is independent of the choice of the geodesic arc $k$ selected to construct the divergence radius function $r = r_{X, \tau, \lambda, k}$ and the set of plaques $(\mathcal{P}_n)_n$.
\end{lem}

\begin{proof}
	For simplicity, let $\sigma_n(P,Q)$ denote the quantity $\sigma^\beta_{\mathcal{P}_n}(P,Q)$. By Lemma~\ref{lem:finite_approx_shear} we have
	\begin{align*}
		\abs{\sigma_{n + 1}(P,Q) - \sigma_n(P,Q)} & \leq 2 C  \abs{\mathcal{P}_{n + 1} - \mathcal{P}_n} \left( \sum_{d \subset k - \bigcup \mathcal{P}_n} L_{\tilde{X}}(d)^\alpha \right) \\
		& \leq 2 C N \left( \sum_{d \subset k - \bigcup \mathcal{P}_n} L_{\tilde{X}}(d)^\alpha \right) .
	\end{align*}
	It is not restrictive to assume $\alpha < 1$, in which case we have
	\[
	\sum_{d \subset k - \bigcup \mathcal{P}_n} L_{\tilde{X}}(d)^\alpha \leq \sum_{R \in \mathcal{P}_{P Q} : r(R) \geq n + 1} L_{\tilde{X}}(k \cap R)^\alpha .
	\]
	Combining this estimate with the properties of the divergence radius function $r$ described Lemma~\ref{lem:auxiliary function r}, we obtain
	\begin{align}\label{eq:roba}
		\begin{split}
			\sum_{d \subset k - \bigcup \mathcal{P}_n} L_{\tilde{X}}(d)^\alpha & \leq \sum_{R \in \mathcal{P}_{P Q} : r(R) \geq n + 1} A^\alpha e^{- \alpha M r(R)} \\
			& \leq A^\alpha N \sum_{j > n} e^{- \alpha M j} \\
			& \leq \frac{A^\alpha N}{1 - e^{- \alpha M}} e^{- \alpha M (n + 1)} ,
		\end{split}
	\end{align}
	where $A, M, N > 0$ are the constants provided by Lemma~\ref{lem:auxiliary function r}. Therefore we deduce
	\[
	\sum_{n \in \N} \abs{\sigma_{n + 1}(P,Q) - \sigma_n(P,Q)} \leq  \frac{2 C A^\alpha N^2}{1 - e^{- \alpha M}} \sum_{n \in \N} e^{- \alpha M (n + 1)} < + \infty ,
	\]
	which concludes the proof of the first part of the statement.
	
	Let $(\mathcal{P}_n')_n$ be the sequence of plaques associated with a different choice of geodesic segment $k'$, and hence divergence radius function $r'$ as in Lemma~\ref{lem:auxiliary function r}. By Lemma~\ref{lem:dependence of r on metric}, there exist two natural numbers $l, m$ such that
	\[
	\mathcal{P}_n \subseteq \mathcal{P}_{ln + m}' , \quad \mathcal{P}_{n}' \subseteq \mathcal{P}_{ln + m}
	\]
	for every $n \in \N$. Moreover, by property (2) of Lemma~\ref{lem:auxiliary function r}, there exists a constant $N > 0$ such that the cardinality of the sets $\mathcal{P}_{l (l n + m) + m} - \mathcal{P}_{n}$ is bounded above by $N (l^2 - 1) n + l m + m)$ for every $n \in \N$. The same function of $n$ in particular provides an upper bound of the cardinality of the set $\mathcal{P}_{ln + m}' - \mathcal{P}_n \subseteq \mathcal{P}_{l (l n + m) + m} - \mathcal{P}_{n}$. Applying Lemma~\ref{lem:finite_approx_shear} and relation \eqref{eq:roba}, we deduce
	\begin{align*}
		| \sigma^\beta_{\mathcal{P}_n}(P,Q) - \sigma^\beta_{\mathcal{P}_{ln + m}'}(P,Q) | & \leq 2 C N (l^2 - 1) n + l m + m) \left( \sum_{d \subset k - \bigcup \mathcal{P}_n} L_{\tilde{X}}(d)^\alpha \right) \\
		& \leq \frac{2 A^\alpha C N^2}{1 - e^{- \alpha M}} (l^2 - 1) n + l m + m) e^{- \alpha M (n + 1)} .
	\end{align*}
	This proves in particular that the difference between $\sigma^\beta_{\mathcal{P}_n}(P,Q)$ and $\sigma^\beta_{\mathcal{P}_{ln + m}'}(P,Q)$ tends to $0$ as $n \to \infty$, and therefore we conclude
	\[
	\lim_{n \to \infty} \sigma^\beta_{\mathcal{P}_n}(P,Q) = \lim_{n \to \infty} \sigma^\beta_{\mathcal{P}_{n}'}(P,Q) ,
	\]
	as desired.
\end{proof}

\begin{rmk} \label{rmk: uniform cauchyness}
	Fix a locally bounded cross ratio $\beta$, a hyperbolic structure $X$, a train track $\tau$ carrying $\lambda$, and a geodesic arc $k$ transverse to $\lambda$. The estimates appearing in the proof of Lemma~\ref{lem:limits of shear} show that, given two distinct plaques $P$ and $Q$ of $\lambda$, there exist constants $C' = C'(C,\alpha,A, M, N), M' = M'(\alpha, M) > 0$ such that for every $n \in \N$
	\[
	\abs{\sigma^\beta_{\lambda}(P,Q) - \sigma^\beta_{\mathcal{P}_n}(P,Q)} \leq C' e^{- M' n} ,
	\]
	where the constants $A, M, N > 0$, provided by Lemma~\ref{lem:auxiliary function r}, depend only the structure $X$, the train track $\tau$ carrying $\lambda$, and the path $k$ and $\alpha, C > 0$ are provided by the local boundedness of $\beta$ (see Definition~\ref{def:locally bounded}), with the choice of $D = L_{\tilde{X}}(k)$.
	
	By Lemma~\ref{lem:dependence of constants} we can then find a neighborhood $U$ of $\lambda$ in the space of maximal geodesic laminations such that, for every $\lambda' \in U$, the finite $\beta$-shears $\sigma_{\mathcal{P}_n'}^\beta(P',Q')$ associated with the arc $k$ and the corresponding divergence radius function $r' = r_{X, \tau, \lambda',k} : \mathcal{P}_{P' Q'} \to \N$ converge to $\sigma_{\lambda'}^\beta(P', Q')$ and satisfy
	\[
	\abs{\sigma^\beta_{\lambda'}(P',Q') - \sigma^\beta_{\mathcal{P}_n'}(P',Q')} \leq C' e^{- M' n} ,
	\]
	with uniform constants $C', M' > 0$ with respect to $\lambda' \in U$ (compare with the notation of Lemma~\ref{lem:dependence of constants}). For future reference (see in particular Proposition~\ref{pro:continuity shear wrt lamination}), we notice that the constants $C', M' > 0$ also satisfy
	\begin{equation}\label{eq:stima_placche}
		\sum_{d \subset k - \bigcup \mathcal{P}_n} L_{\tilde{X}}(d)^\alpha \leq \frac{C' e^{- M' n}}{2 C N} .
	\end{equation}
	(Compare with relation \eqref{eq:roba}.)
\end{rmk}

We finally define the \emph{$\beta$-shear relative to $\lambda$ between the plaques $P$ and $Q$} to be
\[
\sigma^\beta_\lambda(P, Q) : = \lim_{n \to \infty} \sigma^\beta_{\mathcal{P}_n}(P, Q) ,
\]
where $(\mathcal{P}_n)_n$ is the exhausting sequence of $\mathcal{P}_{P Q}$ associated with the divergence radius function $r = r_{X, \tau, \lambda, k} : \mathcal{P}_{P Q} \to \N$, for some choice of a $X$-geodesic path $k$ joining $P$ and $Q$. By Lemma~\ref{lem:limits of shear}, the value $\sigma^\beta_\lambda(P, Q)$ is independent of the choice of $k$. We are now ready to conclude the construction of $\beta$-shear cocycles:

\begin{pro}\label{pro:shear is a cocycle}
	The map $(P,Q) \mapsto \sigma^\beta_{\lambda}(P,Q)$, constructed following the process described above, is a H\"older cocycle transverse to $\lambda$.
\end{pro}

\begin{proof}
	All the properties are simple consequences of the definition of finite $\beta$-shears from Section~\ref{subsec:definition shear}, and of the independence of the quantity $\sigma_{\lambda}^\beta$ from the selected geodesic path $k$ and the associated divergence radius $r : \mathcal{P}_{P Q} \to \N$, as established by Lemma~\ref{lem:limits of shear}.
	
	To prove property (1) from Definition~\ref{def:holder cocycle}, it suffices to select the same path $k$ (and hence same divergence radius function $r$) to approximate both $\sigma^\beta_\lambda(P,Q)$ and $\sigma_\lambda^\beta(Q,P)$. Indeed, by the symmetries of the cross ratio $\beta$ we have $\sigma^\beta_{\mathcal{P}_n}(P, Q) = \sigma^\beta_{\mathcal{P}_n}(Q, P)$ for every $n \in \N$. 
	
	To see property (2), let $k$ be a path connecting the plaques $P$ and $Q$, with associated function $r : \mathcal{P}_{P Q} \to \R$, and let $R \in \mathcal{P}_{P Q}$. We select a subarc $k'$ of $k$ that connects $P$ to $R$, and we set $k'' = \overline{k - k'}$. Observe that the restriction of $r = r_{k}$ to the set of plaques $\mathcal{P}_{P R}$ coincides with the divergence radius function $r'$ associated to $k'$. The same holds for the restriction of $r$ to $\mathcal{P}_{R Q}$ and the path $k''$. Therefore, the divergence radius functions $r'$ and $r''$ associated to $k'$ and $k''$ determine sequences of finite collections of plaques $(\mathcal{P}_n')_n$ and $(\mathcal{P}_n'')_n$, respectively, satisfying
	\[
	\lim_{n \to \infty} \sigma^\beta_{\mathcal{P}_n'}(P,R) = \sigma^\beta_{\lambda}(P,R) , \qquad \lim_{n \to \infty} \sigma^\beta_{\mathcal{P}_n''}(R, Q) = \sigma^\beta_{\lambda}(R, Q) .
	\]
	Moreover, if $(\mathcal{P}_n)_n$ denotes the exhaustion of $\mathcal{P}_{P Q}$ associated to $k$ and $r$, then by construction
	\[
	\mathcal{P}_n = \mathcal{P}_n' \cup \{ R \} \cup \mathcal{P}_n''
	\]
	for every $n \geq r(R)$. Moreover the finite $\beta$-shears satisfy
	\[
	\sigma^\beta_{\mathcal{P}_n}(P,Q) = \sigma^\beta_{\mathcal{P}_n'}(P,R) + \sigma^\beta_{\mathcal{P}_n''}(R,Q)
	\]
	again for every $n \geq r(R)$. By taking the limit as $n \to \infty$, we obtain the additivity property described in property (2) of Definition~\ref{def:holder cocycle}.
	
	Finally, to show property (3), let $\gamma \in \Gamma$ and select $\gamma(k)$ as a path joining the interiors of the plaques $\gamma P$ to $\gamma Q$. The associated divergence radius function coincides with $r \circ \gamma^{-1} : \mathcal{P}_{\gamma P \, \gamma Q} \to \N$, where $r : \mathcal{P}_{P Q} \to \N$ is the divergence radius function of $k$. If $(\mathcal{P}_n)_n$ denotes the sequence of finite family of plaques associated with $k$ and $r$, then $\gamma(k)$ and $r \circ \gamma^{-1}$ have associated sequence $(\gamma \mathcal{P}_n)_n$. Moreover, being $\beta$ $\Gamma$-invariant, we have
	\[
	\sigma^\beta_{\mathcal{P}_n}(P, Q) = \sigma^\beta_{\gamma \mathcal{P}_n}(\gamma P, \gamma Q) 
	\]
	for every $n \in \N$. The identity $\sigma^\beta_\lambda(P, Q) = \sigma^\beta_\lambda(\gamma P, \gamma Q)$ then follows by taking the limit as $n \to \infty$.
\end{proof}

\subsection{Continuity of shear cocycles}
\label{subsec:continuity shear}

We now study the continuity properties of the map $$\mc{GL} \ni \lambda \longmapsto \sigma_\lambda^\beta \in \mc{H}(\lambda;\mb{R}).$$ 

As recalled in Section \ref{subsubsec:real weights}, the choice of a train track $\tau$ that carries a maximal lamination $\lambda$ determines natural identifications between its associated system of real weights $\mathcal{W}(\tau;\R)$ and the space of H\"older cocycles $\mc{H}(\lambda';\mb{R})$ transverse to any lamination $\lambda'$ carried by $\tau$. In particular, there exists a sufficiently small neighborhood $U$ of $\lambda$ inside $\mathcal{GL}$ for which the map
$$U \ni \lambda' \longmapsto \sigma_{\lambda'}^\beta \in \mc{W}(\tau;\mb{R})$$ 
is well defined. Within this framework, it makes sense to ask ourselves whether the map $\lambda' \mapsto \sigma_{\lambda'}^\beta$ is continuous. The next statement answers affirmatively to this question:

\begin{pro} \label{pro:continuity shear wrt lamination}
	Let $(\lambda_m)_m$ be a sequence of maximal geodesic laminations converging to $\lambda$ in the Chabauty topology. Given $\tau$ a train track that carries $\lambda$, we identify $\mathcal{H}(\lambda;\R)$ and $\mathcal{H}(\lambda_m;\R)$ with $\mathcal{W}(\tau;\R)$, the space of real weights of $\tau$ (for $m$ sufficiently large). Then
	\[
	\lim_{m \to \infty} \sigma^\beta_{\lambda_m} = \sigma^\beta_\lambda \in \mathcal{W}(\tau;\R) .
	\] 
\end{pro}

\begin{proof}
	If $k$ is a tie of the lift of the train track $\tau$ in $(\widetilde{\Sigma}, \widetilde{X})$, then the endpoints of $k$ lie in the interior of two plaques $P, Q$ of $\lambda$. Moreover, since $\lambda_m$ converges to $\lambda$ in the Chabauty topology, there exists a $m_0 \in \N$ such that for every $m > m_0$ the endpoints of $k$ lie in the interior of two plaques $P^{(m)}, Q^{(m)}$ of $\lambda_m$. Then the statement is equivalent to show that, for any $k$ as above
	\[
	\lim_{m \to \infty} \sigma^\beta_{\lambda_m}(P^{(m)}, Q^{(m)}) = \sigma^\beta_\lambda(P,Q) .
	\]
	
	Let $\mathcal{P}_{k}$ (resp. $\mathcal{P}^{(m)}_{k}$) denote the set of plaques of $\lambda$ (resp. $\lambda_m$) that separate $P$ from $Q$ (resp. $P^{(m)}$ from $Q^{(m)}$). If $k$ is the geodesic arc joining the endpoints of $k$, then Lemma~\ref{lem:auxiliary function r} provide us functions 
	\[
	r : \mathcal{P}_{k} \to \N, \quad r_m : \mathcal{P}_{k}^{(m)} \to \N 
	\]
	satisfying properties (1), (2) with respect to constants $A, M, N > 0$ that are \emph{independent of $m$}, and defined in terms of the same train track $\tau$ and arc $k$ (see in particular Remark~\ref{rmk: uniform cauchyness}). To simplify the notation, for every $n \in \N$ and $m > m_0$ we set
	\begin{align*}
		\mathcal{P}_n : = & \{ R \in \mathcal{P}_{k} \mid r(R) \leq n \} , \\
		\mathcal{P}_n^{(m)} : = & \{ R \in \mathcal{P}^{(m)}_{k} \mid r_m(R) \leq n \} ,
	\end{align*}
	and 
	\begin{align*}
		\sigma : = &\,\, \sigma^\beta_{\lambda}(P,Q) , & & \sigma^{(m)} : = \sigma^\beta_{\lambda_m}(P^{(m)}, Q^{(m)}) ,  \\
		\sigma_n : = &\,\, \sigma^\beta_{\mathcal{P}_n}(P,Q) , & & \sigma^{(m)}_n : = \sigma^\beta_{\mathcal{P}^{(m)}_n}(P^{(m)}, Q^{(m)}) .
	\end{align*}
	
	Let now $N, C', M' > 0$ be positive constants satisfying the requirements of Lemma~\ref{lem:auxiliary function r} and Remark~\ref{rmk: uniform cauchyness}. We will prove the desired assertion by showing that
	\begin{equation} \label{eq:bound for convergence}
		\limsup_{m \to \infty} \abs*{\sigma^{(m)} - \sigma} \leq (2 + n) C' e^{- M' n} 
	\end{equation}
	for every $n \in \N$. Since the left-hand side of the inequality is independent of $n$, and the right-hand side converges to $0$ as $n \to \infty$, the assertion will follow.
	
	We will divide the proof of relation \eqref{eq:bound for convergence} into smaller steps. In order to describe them, we need to introduce some notation. For any $R \in \mathcal{P}_{n}$, we choose arbitrarily a point $x_R$ in the interior of $R$. Since $\lambda_m \to \lambda$, we can find a $m_1 > 0$ sufficiently such that, for any $m > m_1$, there exists a unique plaque $R^{(m)} \in \mathcal{P}_{k}^{(m)}$ whose interior contains $x_R$. Being $\mathcal{P}_{n}$ finite, up to selecting a larger $m_1$ we can assume that this holds for every plaque $R \in \mathcal{P}_n$. We then introduce the sets
	\[
	\mathcal{Q}_{n}^{(m)} : = \{ R^{(m)} \in \mathcal{P}^{(m)}_{k} \mid R \in \mathcal{P}_{n} \} ,
	\]
	for $m > m_1$. Finally, we set
	\[
	\hat{\sigma}_{n}^{(m)} : = \sigma^\beta_{\mathcal{Q}^{(m)}_{n}}(P^{(m)}, Q^{(m)}) .
	\]

	\noindent\textbf{Step 1.} For every $n \in \N$ there exists a $m_2 \geq m_1$ such that $\mathcal{Q}^{(m)}_{n} \subseteq \mathcal{P}^{(m)}_{n}$ for all $m > m_2$. Moreover
	\[
	\abs{\hat{\sigma}_{n}^{(m)} - \sigma_{n}^{(m)}} \leq 2 C N n \left( \sum_{d \subset k - \bigcup \mathcal{Q}_{n}^{(m)}} L_{\tilde{X}}(d)^\alpha \right) ,
	\]
	where  $C, \alpha, N > 0$ are the constants appearing in Lemmas~\ref{lem:auxiliary function r} and~\ref{lem:finite_approx_shear} (see also Remark~\ref{rmk: uniform cauchyness}). 
	
	\begin{proof}[Proof of Step 1]
		We will show that $\mathcal{Q}^{(m)}_n \subseteq \mathcal{P}^{(m)}_n$ for every $m$ sufficiently large. The second part of the assertion will follow by applying Lemma~\ref{lem:finite_approx_shear} and noticing that
		\[
		\abs{\mathcal{P}^{(m)}_n - \mathcal{Q}^{(m)}_n} \leq \abs{\mathcal{P}^{(m)}_n} \leq N n
		\]
		by Lemma~\ref{lem:auxiliary function r}.
		
		Let $R \in \mathcal{P}_n$, and denote by $\ell_R, h_R \subset \lambda$ the boundary leaves of $R$ that cross the tie $k$. Similarly, let $\ell_R^{(m)}, h_R^{(m)} \subset \lambda_m$ be the boundary leaves of $R^{(m)} \in \mathcal{Q}^{(m)}_n$ that cross $k$. It is enough to prove that $\lim_{m \to \infty} r_m(R^{(m)}) = n$ for every $R \in \mathcal{P}_n$. 
		
		Since the laminations $\lambda_m$ converge to $\lambda$ and the plaques $R^{(m)}$ contain a fixed point $x_R$ in the interior of the plaque $R$, the leaves $\ell_R^{(m)}, h_R^{(m)}$ converge in the Hausdorff topology to $\ell_R, h_R$ as $m \to \infty$ (up to relabeling). Recalling the definition of the divergence radius functions $r, r_m$ from Section~\ref{subsec:divergence radius function}, the condition $r(R) = n$ is equivalent to say that the leaves $\ell_R, h_R$ cross $n + 1$ common branches of $\tilde{\tau}$ in both directions (as we count starting from the branch containing the tie $k$) before taking different paths at some switch of $\tilde{\tau}$. Since the boundary leaves of $R^{(m)}$ that meet $k$ converge to the boundary leaves $\ell_R$ and $h_R$, we can find a sufficiently large $m_2 \geq m_1$ such that $\ell_R^{(m)}$ passes through the same $n + 1$ branches of $\tilde{\tau}$ as $\ell_R$ in both directions, and similarly for $h_R^{(m)}$, for all $m \geq m_2$. This implies in particular that $r_m(R^{(m)}) = n$. Since $\mathcal{P}_n$ is finite, up to enlarging $m_2$ we can assume that this holds for every $R \in \mathcal{P}_n$, as desired.
	\end{proof}
	
	\noindent\textbf{Step 2.} For every $n \in \N$ we have
	\[
	\lim_{m \to \infty} \sum_{d \subset k - \bigcup \mathcal{Q}_{n}^{(m)}} L_{\tilde{X}}(d)^\alpha = \sum_{d \subset k - \bigcup \mathcal{P}_{n}} L_{\tilde{X}}(d)^\alpha \leq \frac{C' e^{- M' n}}{2 C N} .
	\]
	
	\begin{proof}[Proof of Step 2]
		As observed in the proof of the previous step, the boundary leaves of $R^{(m)}$ converge to the boundary leaves of $R$ with respect to the Chabauty topology for every $R \in \mathcal{P}_n$. In particular, each subarc $k \cap R$ of $k$ is equal to the limit of the subarcs $(k \cap R^{(m)})_m$. Since $\mathcal{P}_n$ is a finite collection of plaques, the set
		\[
		\{d \mid \text{$d$ connected component of $k - \bigcup \mathcal{Q}^{(m)}_n$}\}
		\]
		is finite, and the length of each of its components converges to the length of the corresponding component of $k - \bigcup \mathcal{P}_n$. This implies the equality appearing in the statement. The upper bound of the limit follows from Remark~\ref{rmk: uniform cauchyness}, and more specifically relation \eqref{eq:stima_placche}.
	\end{proof}
	
	\noindent\textbf{Step 3.} For every $n \in \N$ we have $\lim_{m \to \infty} \hat{\sigma}_{n}^{(m)} = \sigma_{n}$.
	
	\begin{proof}[Proof of Step 3]
		We denote as above by $\ell_R, h_R$ (resp. $\ell_R^{(m)}, h_R^{(m)}$) the leaves of $R \in \mathcal{P}_n$ (resp. $R^{(m)} \in \mathcal{Q}^{(m)}_n$) that cross $k$, and we orient them from right to left as we follow the geodesic arc $k$, moving from $P$ to $Q$ (resp. $P^{(m)}$ and $Q^{(m)}$). By what observed above we have
		\[
		\lim_{m \to \infty} (\ell_R^{(m)})^\pm = \ell_R^\pm , \quad \lim_{m \to \infty} (h_R^{(m)})^\pm = h_R^\pm \in \partial \Gamma .
		\]
		By definition, the quantity $\sigma_n = \sigma^\beta_{\mathcal{P}_n}(P,Q)$ is a finite sum of elementary shears, defined as in relation \eqref{eq:def_shear_adj}, where the points at infinity $u_+, u_-,u_l, u_r$ belong to the set $\{ \ell_R^\pm, h_R^\pm \mid R \in \mathcal{P}_n \}$, and similarly for $\sigma_n^{(m)}$ and the set $\{ (\ell_R^{(m)})^\pm, (h_R^{(m)})^\pm  \mid R \in \mathcal{P}_n \}$. From the construction it follows that the finite laminations $\lambda_{\mathcal{Q}^{(m)}_n}$, defined as in Section~\ref{subsec:definition shear}, converge to the lamination $\lambda_{\mathcal{P}_n}$ as $m \to \infty$. By the continuity of the cross ratio $\beta$, it is now immediate to see that the finite sum of shears $\sigma_n^{(m)}$ converge to $\sigma_n$ as $m \to \infty$.
	\end{proof}
	
	\noindent\textbf{Step 4.} For every $n \in \N$ and $m > m_0$
	\begin{align*}
		\abs*{\sigma - \sigma_{n}} & \leq C' e^{- M' n} , \\
		\abs*{\sigma^{(m)} - \sigma_{n}^{(m)}} & \leq C' e^{ - M' n} .
	\end{align*}

	\begin{proof}[Proof of Step 4]
		This is an immediate consequence of Remark~\ref{rmk: uniform cauchyness} and the definition of $\sigma, \sigma_n, \sigma^{(m)}, \sigma_n^{(m)}$.
	\end{proof}

	We now have all the ingredients to conclude our argument. First we observe that
	\begin{align*}
		\abs*{\sigma^{(m)} - \sigma} & \leq \abs*{\sigma^{(m)} - \sigma_{n}^{(m)}} + \abs*{\sigma_{n}^{(m)} - \hat{\sigma}_{n}^{(m)}} + \abs*{\hat{\sigma}_{n}^{(m)} - \sigma_{n}} + \abs{\sigma_{n} - \sigma} .
	\end{align*}
	By Steps 1 and 2 we have
	\[
	\limsup_{m \to \infty} \abs*{\sigma_{n}^{(m)} - \hat{\sigma}_{n}^{(m)}} \leq n \, C' e^{- M' n} .
	\]
	Relation \eqref{eq:bound for convergence} now follows by combining the above inequalities with Steps 3 and 4:
	\begin{align*}
		\limsup_{m \to \infty} \abs*{\sigma^{(m)} - \sigma} & \leq 2 C' e^{- M' n} + \limsup_{m \to \infty} \abs*{\sigma_{n}^{(m)} - \hat{\sigma}_{n}^{(m)}} \\
		& \leq (2 + n) C' e^{- M' n} .
	\end{align*}
	
	This concludes the proof of Proposition~\ref{pro:continuity shear wrt lamination}.
\end{proof}

As mentioned in the introduction of the section, the continuity of $\beta$-cocycles that we just established, combined with the elementary description of the $\beta$-shear cocycles associated to maximal finite leaved laminations (see Section \ref{subsec:shear finite case}) implies that $\sigma_\beta^\lambda$ only depends on the cross ratio $\beta$ and the maximal lamination $\lambda$ and not on any of the auxiliary choices made to define it.

To conclude the investigation of the continuous dependence of shear cocycles, we notice that the analysis described in the proof of Proposition \ref{pro:continuity shear wrt lamination} allows to recover the following result, originally due to Bonahon:

\begin{cor}[Bonahon]\label{cor:uniform convergence shear}
	Let $(X_m)_m$ be a sequence of hyperbolic structures converging to $X \in \T$, and let $(\lambda_m)_m$ be a sequence of maximal geodesic laminations converging to $\lambda \in \mathcal{GL}$ in the Chabauty topology. Select a train track $\tau$ that carries $\lambda$, and denote by $\sigma_m \in \mathcal{W}(\tau;\R) \cong \mathcal{H}(\lambda_m;\R)$ (resp. $\sigma$) the system of real weights associated to the shear coordinates of $X_m$ (resp. $X$) with respect to $\lambda_m$ (resp. $\lambda$). Then
	\[
	\lim_{m \to \infty} \sigma_m = \sigma \in \mathcal{W}(\tau;\R) .
	\]
	In other words, the shear parametrizations $\varphi_{\lambda_m} : \T \to \mathcal{H}(\lambda_m;\R) \cong \mathcal{W}(\tau;\R)$ converge uniformly over all compact subsets of Teichm\"uller space to $\varphi_{\lambda} : \T \to \mathcal{H}(\lambda;\R) \cong \mathcal{W}(\tau;\R)$.
\end{cor}

\begin{proof}
	This statement appeared already in a very similar form in the work of Bonahon, see in particular the proof of \cite{B98}*{Lemma~13}. For completeness, we provide an alternative proof, which fits well with the techniques developed throughout our exposition.
	
	By well-known facts, uniformly quasi-conformal (normalized, i.e. fixing $0, 1, \infty$) homeomorphisms of $S^1$ are uniformly H\"older continuous (see e.g. \cite{A06}*{Theorem~III.2}). If $(X_m)_m \subset \T$ is a precompact family of hyperbolic structures on the fixed closed surface $\Sigma$, then the corresponding limit maps $\xi_m : \partial\Gamma \to \partial \hyp^2$ are uniformly quasi-conformal (see e.g. \cite{M68}), since the surfaces $X_m$ are uniformly biLipschitz equivalent. Hence, by Lemma \ref{lem: comparing crossratios pro}, the family of cross ratios $\{\beta_m = \beta^{\rho_m}\}_{m}$ is uniformly locally bounded (see Definition \ref{def:uniform family}).
	
	We now notice that the constants $C', M' > 0$ involved in the convergence argument of Proposition \ref{pro:continuity shear wrt lamination} (see in particular relation \eqref{eq:bound for convergence}) depend on the quantities $C, \alpha > 0$ appearing in the locally boundedness condition of $\beta$, and several arbitrary choices (an auxiliary hyperbolic structure $X$, a train track $\tau$, and divergence radius functions) that are \emph{independent} of the cross ratio $\beta$ (compare in particular with Remark \ref{rmk: uniform cauchyness}). This implies in particular that the convergence of the shear cocycles is uniform on family of cross ratios that are uniformly locally bounded. Since this is the case for a precompact family of hyperbolic structures $(X_m)_m$, the desired assertion follows.
\end{proof}

\subsection{Shears and length functions: General case}\label{subsec: shear and length} 

Now that we have established the continuous dependence of $\beta$-shear cocycles $\sigma^\beta_\lambda$ in the maximal lamination $\lambda$, we can easily generalize the relation between shear cocycles and $\beta$-periods observed in Section \ref{subsec:shear and length I} for finite leaved laminations to any maximal geodesic lamination. More precisely:

\begin{pro}\label{pro:length and thurston sympl}
	Let $\beta$ be a positive and locally bounded cross ratio. Then for every maximal lamination $\lambda$ and for every measured lamination $\mu$ with $\supp \mu \subseteq \lambda$, we have 
	$$L_\beta(\mu) = \omega_\lambda(\sigma^\beta_\lambda, \mu),$$ 
	where $\omega_\lambda$ denotes the Thurston symplectic form on the space of transverse H\"older cocycles $\mathcal{H}(\lambda;\R)$, and $L_\beta$ is the length function introduced in Section~\ref{subsec:cross ratios definitions}.
\end{pro}

\begin{proof} When $\lambda$ is a finite leaved maximal lamination, then the statement is equivalent to Proposition~\ref{pro:length and thurston sympl finite}. Consider now a general maximal lamination $\lambda$ and a measured lamination $\mu$ with support contained in $\lambda$. Without loss of generality we can assume that $\mu$ is minimal, so it can be approximated in $\mathcal{ML}$ by a sequence of weighted simple closed curves $(a_n \gamma_n)_n$. Moreover, following the procedure described by Canary, Epstein, and Green in \cite{CEG}*{Theorem~I.4.2.14}, for every $n$ we can extend the curve $\gamma_n$ to a finite leaved lamination $\lambda_n$ so that, up to subsequence, $\lambda_n$ converges in the Chabauty topology to $\lambda$. 
	
	Select now a train track $\tau$ that carries $\lambda$. Since the laminations $\lambda_n$ are converging to $\lambda$ in the Chabauty topology, the train track $\tau$ carries $\lambda_n$ for $n$ sufficiently large. In particular, we can identify the spaces of transverse H\"older cocycles $\mathcal{H}(\lambda;\R)$ and $\mathcal{H}(\lambda_n;\R)$ with the space of real weights $\mathcal{W}(\tau;\R)$. Notice that the isomorphisms $\mathcal{H}(\lambda_n;\R) \cong \mathcal{W}(\tau;\R) \cong \mathcal{H}(\lambda;\R)$ are linear symplectomorphisms with respect to the associated Thurston symplectic forms and the algebraic intersection pairing $\omega_\tau$ on $\mathcal{W}(\tau;\R)$, in light of the description provided in Section~\ref{subsubsec:thurston sympl}. By Proposition~\ref{pro:length and thurston sympl finite} we have 
	\begin{equation}\label{eq:lenghts finite case}
		a_n L_\beta(\gamma_n) = \omega_\tau(\sigma^\beta_{\lambda_n}, a_n \gamma_n)
	\end{equation}
	for every $n$ sufficiently large (we are identifying with abuse the cocycles $\sigma^\beta_{\lambda_n}, a_n \gamma_n \in \mathcal{H}(\lambda_n;\R)$ with their image inside $\mathcal{W}(\tau;\R)$). Now, by Theorem~\ref{thm:current of cross-ratio} the left-hand side $a_n L_\beta(\gamma_n)$ is continuous in $a_n \gamma_n$, and hence converges to $L_\beta(\mu)$, while the right-hand side converges to $\omega_\lambda(\sigma^\beta_\lambda, \mu)$ by Proposition~\ref{pro:continuity shear wrt lamination}. By taking the limit as $n \to \infty$ of relation \eqref{eq:lenghts finite case} we obtain the statement for the lamination $\lambda$ and the minimal measured lamination $\mu$.
\end{proof}

\subsection{The proof of Theorem \ref{shear of cross ratio improved}}\label{subsec: proof of shear}

We finally have all the elements to prove the main result of the section:

\begin{mthm}{\ref{shear of cross ratio improved}}
	Let $\beta : \partial \Gamma^{(4)} \to \R$ be a positive and locally bounded cross ratio. Then for every maximal lamination $\lambda$, the $\beta$-shear cocycle $\sigma^\beta_\lambda$ belongs to the closure of the cone $C(\lambda) \subset \mathcal{H}(\lambda;\R)$, that is 
	$$\omega_\lambda(\sigma^\beta_\lambda, \mu) \geq 0$$ 
	for every measured lamination $\mu$ with $\supp \mu \subseteq \lambda$. Moreover, if the cross ratio $\beta$ is strictly positive, then $\omega_\lambda(\sigma^\beta_\lambda, \mu) > 0$ for every non-trivial measured lamination $\mu$ as above, and consequently there exists a unique hyperbolic structure $Y = Y^\beta_\lambda \in \T$ such that $\sigma^\beta_\lambda = \sigma_\lambda^Y \in \mathcal{H}(\lambda;\R)$.
\end{mthm}

\begin{proof}
	By Theorem~\ref{thm:current of cross-ratio}, every positive cross ratio has an associated Liouville current $\mathscr{L}_\beta$, and the corresponding $\beta$-length $L_\beta(\bullet) = i(\mathscr{L}_\beta, \bullet)$ is a non-negative function on the space of geodesic currents. Hence the first part of the assertion follows directly from Proposition~\ref{pro:length and thurston sympl}.
	
	As observed in Lemma~\ref{lem: positive lengths}, if the cross ratio $\beta$ is strictly positive, then $L_\beta(c) > 0$ for any non-trivial geodesic current $c$. Therefore, combining Proposition~\ref{pro:length and thurston sympl} with Theorem~\ref{thm:thurston bonahon}, we deduce that for every maximal geodesic lamination $\lambda$ there exists a unique hyperbolic structure $Y$ such that $\sigma^\beta_\lambda = \sigma^Y_\lambda$, as desired.
\end{proof}

\begin{rmk}
	We point out to the reader that the work of Burger, Iozzi, Parreau, and Pozzetti on geodesic currents (see in particular \cite{BIPP21}*{Theorems~1.3, 1.7, Corollary~1.9}) can be deployed to investigate in detail the set of measured laminations $\mu$ with trivial $\beta$-length, by examining the geometric decomposition of the Liouville current $\mathscr{L}_\beta$. This in turn determines the set of maximal geodesic laminations $\lambda$ for which the associated $\beta$-shear cocycle lies in the boundary $\partial C(\lambda)$ of the shear parametrization from Theorem \ref{thm:thurston bonahon}.
\end{rmk}

%%%

\section{Geometry of pleated surfaces}
\label{sec:geometry}

In this section we prove the main structural result about the geometry of pleated surfaces, that is, Theorem \ref{geometry pleated surfaces h2n}. 

We start with the definition of pleated surfaces in the context of maximal representations in $\SOtwon$:

\begin{dfn}[Pleated Surface]\label{def:pleated surface}
	Let $\rho:\Gamma\to\SOtwon$ be a maximal representation. A {\em pleated surface} for $\rho$ realizing the maximal lamination $\lambda\in\mc{GL}$ consists of the following data:
	\begin{enumerate}
		\item{The pleated set $S_\lambda=\widehat{S}_\lambda/\rho(\Gamma)$.}
		\item{The hyperbolic surface $X_\lambda\in\T$ whose shear cocycle with respect to $\lambda$ is equal to the intrinsic shear cocycle constructed in Section \ref{sec:shear cocycles}, i.e. $$\sigma_\lambda^{X_\lambda}=\sigma_\lambda^\rho\in\mc{H}(\lambda;\mb{R}).$$}
		\item{A homeomorphism $f:S_\lambda\to X_\lambda$ that is totally geodesic on every leaf of $\lambda$ and every plaque of $S_\lambda-\lambda$, and that is $1$-Lipschitz with respect to the intrinsic pseudo-metric (see Section \ref{subsec:acausal and poincare}) and the hyperbolic metric.}
	\end{enumerate}
	
	We call a homeomorphism $f:S_\lambda\to X_\lambda$ satisfying the requirements in (3) a {\em $1$-Lipschitz developing map} of the pleated set $S_\lambda$.
\end{dfn}

The geometric structures of a pleated surface, namely,
\begin{itemize}
	\item{the hyperbolic metric of $X_\lambda$,}
	\item{the pseudo metric of $S_\lambda$,}
	\item{the natural length space structure of $S_\lambda$ (see Definition \ref{def:length space})}
\end{itemize}
are all linked by the 1-Lipschitz developing map $f:S_\lambda\to X_\lambda$. In the first part of the section we study more in detail the general properties of such maps. In particular, we show that: The developing map $f$ is totally geodesic outside of its bending locus (Lemma \ref{lem:totally geodesic outside bending}), it contracts lengths of paths (Lemma \ref{lem:rectifiable to rectifiable}), and the  length spectrum of the hyperbolic surface $X_\lambda$ is dominated by the length spectrum of $\rho$, with strict inequality on every curve that intersects the bending locus (Proposition \ref{pro:strict inequality}).

In the second part of the section, we establish the existence of developing maps for pleated sets associated to maximal laminations $\lambda\in\mc{GL}$. To this purpose, we first analyze explicitly the case of finite leaved maximal laminations and prove the existence of $1$-Lipschitz developing maps, as given in Theorem \ref{geometry pleated surfaces h2n}, in that setting. The proofs here are completely elementary (see Propositions \ref{pro:developing finite} and \ref{pro:intrinsic shear finite case}). Then, we exploit continuity properties of pleated surfaces to deduce the existence of a $1$-Lipschitz developing map in the general case (see Proposition \ref{pro:developing general}).

\subsection{Developing maps and metric properties}

By definition, a $1$-Lipschitz developing map of a pleated set $S_\lambda$ sends every complementary region of the maximal lamination $\lambda$ into a spacelike ideal triangle. In fact, we can be more precise:

\begin{lem}
	\label{lem:totally geodesic outside bending}
	Let $\rho:\Gamma\to\SOtwon$ be a maximal representation, and let $S_\lambda$ be the pleated set associated to a maximal $\rho$-lamination $\lambda$. Then every $1$-Lipschitz developing map $f:S_\lambda\to X_\lambda$ is totally geodesic on the complement of the bending locus of $S_\lambda$.
\end{lem}

\begin{proof}
	Consider a component $W$ of the complement of the bending locus of $S_\lambda$. By Proposition \ref{pro:bending locus}, the restriction of the pseudo-metric to $W$ is a hyperbolic metric. As the restriction of $f:S_\lambda\to X_\lambda$ to $f:W\to f(W)$ is a $1$-Lipschitz map between hyperbolic surfaces of the same (finite) area, we conclude that $f:W\to f(W)$ is an isometry (compare with Thurston \cite{T86}*{Proposition~2.1}).
\end{proof}

In general, it is always possible to define on the pleated set $\widehat{S}_\lambda$ a natural {\em length space} structure: Recall that a Lipschitz function is differentiable almost everywhere.

\begin{dfn}[Regular Path]
	\label{def:length space}    
	A {\em (weakly) regular path} is a map $\gamma:I=[a,b]\to\mb{H}^{2,n}$ such that: 
	\begin{itemize}
		\item There exists a (and hence for any) \emph{Riemannian} distance $d : \hyp^{2,n} \times \hyp^{2,n} \to \R$ with respect to which $\gamma$ is Lipschitz.
		\item The tangent vector ${\dot \gamma}(t)$ is spacelike (or lightlike) for almost every $t \in I$ (at which $\dot \gamma$ is defined).
	\end{itemize}
	The length of a weakly regular path is 
	\[
	L(\gamma):=\int_I{\sqrt{\langle{\dot \gamma}(t),{\dot \gamma}(t)\rangle} dt} ,
	\]
	where $\scal{\bullet}{\bullet} = \scal{\bullet}{\bullet}_{2,n+1}$. The Lipschitz property implies that the length $L(\gamma)$ is always finite.
\end{dfn}

If $\lambda$ is a maximal $\rho$-lamination and $\widehat{S}_\lambda\subset\mb{H}^{2,n}$ is the associated pleated set, we say that a path $\gamma:I\to\widehat{S}_\lambda$ is (weakly) regular if it is (weakly) regular as a path in $\mb{H}^{2,n}$. Furthermore, a regular path $\gamma:I\to\widehat{S}_\lambda$ is said to be \emph{transverse to ${\hat \lambda}$} if $\gamma^{-1}({\hat \lambda})$ has Lebesgue measure zero. 

It is not difficult to check that every pair of points $x,y\in\widehat{S}_\lambda$ can be joined by a weakly regular path using the representation of $\widehat{S}_\lambda$ as the graph of a (strictly) $1$-Lipschitz function $g_\lambda:\mb{D}^2\to\mb{S}^n$ in a Poincaré model $\mb{D}^2\times\mb{S}^n$ of $\widehat{\mb{H}}^{2,n}$. In fact, for any Lipschitz path $\alpha:I\to\mb{D}^2$ joining the projections $\pi(x),\pi(y)$, the graph parametrization $t \mapsto (\alpha(t), g_\lambda(\alpha(t)))$ satisfies the requirements. When $\lambda$ is a finite leaved maximal lamination, it is possible to join any two points $x,y\in\widehat{S}_\lambda$ with a regular path that intersects the lamination in countably many points (and, hence, transversely).

We now have all the elements to study the behavior of $1$-Lipschitz developing maps with respect to the length of weakly regular paths:

\begin{lem}
	\label{lem:rectifiable to rectifiable}
	Let $\rho:\Gamma\to\SOtwon$ be a maximal representation, and let $S_\lambda$ be the pleated set associated to a maximal $\rho$-lamination $\lambda$. Then every $1$-Lipschitz developing map $f:S_\lambda\to X_\lambda$ sends weakly regular paths $\gamma:I\to S_\lambda$ to Lipschitz (hence rectifiable) paths $f\gamma:I\to X_\lambda$ of smaller length $L(\gamma)\ge L(f\gamma)$.
\end{lem}

\begin{proof}
	By Lemma \ref{lem:acausal lifts}, the map $f$ admits a lift ${\hat f}:\widehat{S}_\lambda\subset\widehat{\mb{H}}^{2,n}\to\mb{H}^2$. We start by proving the following: 
	
	\begin{claim}{\it 1}
		The map ${\hat f}$ sends weakly regular paths to rectifiable paths. 
	\end{claim}
	
	\begin{proof}[Proof of the claim]
		It is convenient to work in a Poincaré model $\mb{D}^2\times\mb{S}^n$ of $\widehat{\mb{H}}^{2,n}$ and represent $\widehat{S}_\lambda$ as a graph of a $1$-Lipschitz function $g:\mb{D}^2\to\mb{S}^n$. Let us denote by $u:\mb{D}^2\to\widehat{S}_\lambda$ the graph map $u(x)=(x,g(x))$ and by $h:\mb{D}^2\to\mb{H}^2$ the composition $h=\hat{f}u$. 
		
		By Lemma \ref{lem:projection}, the map $u$ is $1$-Lipschitz with respect to the hyperbolic metric on $\mb{D}^2$ and the pseudo-metric on $\widehat{S}_\lambda$. As the developing map ${\hat f}$ is $1$-Lipschitz with respect to the intrinsic pseudo-metric on $\widehat{S}_\lambda$ and the hyperbolic metric on $\mb{H}^2$, we conclude that $h : \mathbb{D}^2 \to \mathbb{H}^2$ is $1$-Lipschitz with respect to the hyperbolic metric on both the source and the target. 
		
		Let now $\gamma:I\to\widehat{S}_\lambda$ be a weakly regular path, and select a Riemannian distance $\hat{d}$ on $\widehat{\hyp}^{2,n}$. Since the projection $\pi: \widehat{\mb{H}}^{2,n}\to\mb{D}^2$ is locally Lipschitz with respect to $\hat{d}$ and the hyperbolic distance on $\mathbb{D}^2$,  we can write $\gamma$ as a composition $\gamma=u\alpha$, where $\alpha:I\to\mb{D}^2$ is the Lipschitz path obtained by composing $\gamma$ with the projection $\pi$. As $\hat{f}\gamma=h\alpha$, we deduce that $\hat{f}\gamma$ is a Lipschitz path. 
	\end{proof}
	
	We now prove that $\hat{f}$ contracts the length of $\gamma$, namely $L(\gamma)\ge L(\hat{f}\gamma)$. Recall that for every $t,s\in I$ the points $\gamma(t),\gamma(s)\in\widehat{\mb{H}}^{2,n}$ are joined by a spacelike segment $[\gamma(s),\gamma(t)]$, whose length satisfies $\cosh(L[\gamma(s),\gamma(t)])=-\langle\gamma(s),\gamma(t)\rangle$. 
	
	\begin{claim}{\it 2}
		Let $t\in I$ be a point of differentiability of $\gamma$. Then we have 
		\[
		\sqrt{\scal{{\dot \gamma}(t)}{{\dot \gamma}(t)}}=\lim_{\ep\to 0}{\frac{d_{\mb{H}^{2,n}}(\gamma(t+\ep),\gamma(t))}{\ep}}.
		\]
	\end{claim}
	
	\begin{proof}[Proof of the claim]
		We consider $\gamma:I\to\widehat{\mb{H}}^{2,n}$ as a path with values in $\mb{R}^{2,n+1}$, so that we can write
		\[
		{\dot \gamma}(t)=\lim_{\ep\to 0}{\frac{\gamma(t+\ep)-\gamma(t)}{\ep}}.
		\]
		Notice that
		\begin{align*}
			\Vert\gamma(t+\ep)-\gamma(t)\Vert^2 &=\langle\gamma(t+\ep)-\gamma(t),\gamma(t+\ep)-\gamma(t)\rangle \\
			&=\Vert\gamma(t+\ep)\Vert^2+\Vert\gamma(t)\Vert^2-2\langle\gamma(t+\ep),\gamma(t)\rangle\\
			&=-2+2\cosh\left(d_{\mb{H}^{2,n}}(\gamma(t+\ep),\gamma(t))\right), 
		\end{align*}
		where $\norm{\bullet}^2 := \scal{\bullet}{\bullet}$. Since $\cosh(x)=1+\frac{x^2}{2}+o(x)$ and $d_{\mb{H}^{2,n}}(\gamma(t+\ep),\gamma(t))\to 0$ as $\ep\to 0$ (see e.g. Lemma~\ref{lem:same topology}), we have that
		$$-2+2\cosh(d_{\mb{H}^{2,n}}(\gamma(t+\ep),\gamma(t)))\approx d_{\mb{H}^{2,n}}(\gamma(t+\ep),\gamma(t))^2$$ 
		as $\ep$ goes to $0$. Therefore, we conclude that
		\[
		\sqrt{\scal{{\dot \gamma}(t)}{{\dot \gamma}(t)}}=\lim_{\ep\to 0}{\frac{\sqrt{\Vert\gamma(t+\ep)-\gamma(t)\Vert^2}}{\ep}}=\lim_{\ep\to 0}{\frac{d_{\mb{H}^{2,n}}(\gamma(t+\ep),\gamma(t))}{\ep}}.
		\]
	\end{proof}
	
	Notice that the above claim applies also to the differentiability points of the curve $\hat{f}\gamma:I\to\mb{H}^2$, that is
	\[
	\sqrt{\scal{(\hat{f} \gamma)\dot{}(t)}{(\hat{f} \gamma)\dot{}(t)}}=\lim_{\ep\to 0}{\frac{d_{\mb{H}^2}(\hat{f}\gamma(t+\ep),\hat{f}\gamma(t))}{\ep}}.
	\]
	Also observe that, by the first claim, the curve $\hat{f}\gamma$ is Lipschitz and, hence, differentiable almost everywhere. 
	
	We are now ready to conclude: The length of $\gamma$ is given by
	\[
	L(\gamma)=\int_I{\sqrt{\scal{{\dot \gamma}(t)}{{\dot \gamma}(t)}} \,\mathrm{d}t}.
	\] 
	Recall that $\gamma$ can be expressed as $\gamma=u\alpha$, where $\alpha:I\to\mb{D}^2$ is a Lipschitz curve and $u:\mb{D}^2\to\widehat{S}_\lambda$ is the graph map, which is $1$-Lipschitz with respect to the hyperbolic metric on $\mathbb{D}^2$ and the pseudo-metric on $\widehat{S}_\lambda$. In particular, we have 
	\[
	\frac{d_{\mb{H}^{2,n}}(\gamma(t+\ep),\gamma(t))}{\ep}\le \frac{d_{\mb{H}^2}(\alpha(t+\ep),\alpha(t))}{\ep} 
	\]
	and, consequently, the functions $t \mapsto d_{\mb{H}^{2,n}}(\gamma(t+\ep),\gamma(t))/\ep$ are bounded uniformly in $\epsilon$ almost everywhere. By applying Claim 2 and Lebesgue's dominated convergence theorem, we have
	\[
	\int_I{\sqrt{\scal{{\dot \gamma}(t)}{{\dot \gamma}(t)}} \,\mathrm{d}t}=\lim_{\ep\to 0}\int_I{\frac{d_{\mb{H}^{2,n}}(\gamma(t+\ep),\gamma(t))}{\ep} \,\mathrm{d}t}.
	\]
	Being $\hat{f}$ a $1$-Lipschitz map, we have $d_{\mb{H}^{2,n}}(\gamma(t+\ep),\gamma(t))\ge d_{\mb{H}^2}(\hat{f}\gamma(t+\ep),\hat{f}\gamma(t))$. Hence,
	\[
	\lim_{\ep\to 0}\int_I{\frac{d_{\mb{H}^{2,n}}(\gamma(t+\ep),\gamma(t))}{\ep} \,\mathrm{d}t}\ge\limsup_{\ep\to 0}{\int_I{\frac{d_{\mb{H}^2}(\hat{f}\gamma(t+\ep),\hat{f}\gamma(t))}{\ep} \,\mathrm{d}t}}.
	\]
	Again, by Claim 2 and the dominated convergence theorem (recall that $\hat{f}\gamma$ is Lipschitz by the first claim, so the maps $t \mapsto d_{\mb{H}^2}(\hat{f}\gamma(t+\ep),\hat{f}\gamma(t))/\ep$ are bounded uniformly in $\epsilon$), we have
	\[
	\limsup_{\ep\to 0}{\int_I{\frac{d_{\mb{H}^2}(\hat{f}\gamma(t+\ep),\hat{f}\gamma(t))}{\ep} \,\mathrm{d}t}}=\int_I{\sqrt{\scal{(\hat{f}\gamma)\dot{}(t)}{(\hat{f}\gamma)\dot{}(t)}} \,\mathrm{d}t} = L(\hat{f} \gamma).
	\]
	In conclusion, we showed $L(\gamma) \geq L(\hat{f} \gamma)$, as desired.
\end{proof}

Finally, we observe that the length spectrum of any $\rho$-equivariant pleated surface is dominated by the length spectrum $\rho$:

\begin{pro}
	\label{pro:strict inequality}
	Let $\rho:\Gamma\to\SOtwon$ be a maximal representation. Consider a pleated surface $S_\lambda$ associated to $\rho$ and a maximal lamination $\lambda$, together with a $1$-Lipschitz developing map $f:S_\lambda\to X_\lambda$. Then, for every $\gamma\in\Gamma - \{1\}$ we have 
	\[
	L_{X_\lambda}(\gamma) \leq L_\rho(\gamma),
	\]
	where the strict inequality holds if and only if $\gamma$ intersects essentially the bending locus of $S_\lambda$. 
\end{pro}

\begin{proof}
	We proceed as in \cite{CTT19}*{Proposition~3.38}.
	
	If $\gamma$ does not intersect essentially the bending locus then the invariant geodesic $\ell$ of $\rho(\gamma)$ is contained in $\widehat{S}_\lambda$. By Lemma \ref{lem:totally geodesic outside bending}, the $1$-Lipschitz developing map ${\hat f}:\widehat{S}_\lambda\to\mb{H}^2$ is an isometry on the complement of the bending locus. Therefore, we have $L_\rho(\gamma)=L_{X_\lambda}(\gamma)$.
	
	Assume that $\gamma$ intersects essentially the bending locus of $S_\lambda$. 
	
	Let $\ell$ be the axis of $\rho(\gamma)$. We first observe that $\ell$ is not contained in $\widehat{S}_\lambda$: If this was the case, then, as $\gamma$ intersects the bending locus essentially, the geodesic $\ell$ must also intersect some bending line $\ell'\subset{\hat \lambda}$. However, this would contradict the fact that $\ell'$ is a bending line by the definition of bending locus. Therefore $\ell$ is not contained in $\widehat{S}_\lambda$. 
	
	As $\ell$ is not contained in $\widehat{S}_\lambda$, it can be connected to $\widehat{S}_\lambda$ by a timelike geodesic and we can take the timelike geodesic $[x,p]$ of maximal length $\ell[x,p]=T>0$ joining a point $x\in\ell$ to a point $p\in\widehat{S}_\lambda$. We parametrize $\ell$ as $\ell(t)=\cosh(t)x+\sinh(t)w$ with $w$ a spacelike vector orthogonal to $x$ and we write $p:=\cos(T)x+\sin(T)v$ with $v$ timelike and orthogonal to $x,w$ (as $[x,p]$ maximizes the timelike distance between $\ell$ and $\widehat{S}_\lambda$). The isometry $\rho(\gamma)$ acts on $\ell$ by translating points by $L=L_\rho(\gamma)$ and acts on ${\rm Span}\{x,v\}^\perp$ by an isometry $A$. Thus, we have $\rho(\gamma)p=\cos(T)(\cosh(L)x+\sinh(L)w)+\sin(T)Av$.
	
	\begin{align*}
		\cosh\left(d_{\mb{H}^{2,n}}(p,\rho(\gamma)^np)\right) &=-\langle p,\rho(\gamma)^np\rangle\\ 
		&=\cos(T)^2\cosh(nL)-\sin(T)^2\langle v,A^nv\rangle.
	\end{align*}
	
	Since the developing map ${\hat f}:\widehat{S}_\lambda\to\mb{H}^2$ is $1$-Lipschitz and equivariant, we get
	\[
	d_{\mb{H}^2}(\hat{f}(p),\rho_{X_\lambda}(\gamma)^n \hat{f}(p))\le d_{\mb{H}^{2,n}}(p,\rho(\gamma)^np).
	\]
	
	Furthermore, for a hyperbolic isometry $\rho_{X_\lambda}(\gamma)^n$, we have that the minimal displacement coincides with the translation length so that 
	\[
	nL_{X_\lambda}(\gamma)=L_{X_\lambda}(\gamma^n)\le d_{\mb{H}^2}(f(p),\rho_{X_\lambda}(\gamma)^nf(p)).
	\]
	
	Putting together the previous inequalities we get
	\[
	\cosh\left(nL_{X_\lambda}(\gamma))\right)\le\cos(T)^2\cosh(nL)-\sin(T)^2\langle v,A^nv\rangle.
	\]
	
	Since the spectral radius of $A$ is strictly smaller than $e^L$, we can choose $n$ sufficiently large so that $\vert\langle v,A^nv\rangle\vert<\cosh(nL)$ (see for example \cite{CTT19}*{Corollary~2.6} and \cite{BPS19}). For this value of $n$ we get
	\begin{align*}
		\cosh\left(nL_{X_\lambda}(\gamma))\right) &\le\cos(T)^2\cosh(nL)-\sin(T)^2\langle v,A^nv\rangle\\
		&<\cos(T)^2\cosh(nL)+\sin(T)^2\cosh(nL)=\cosh(nL).
	\end{align*}
	Which implies $L_{X_\lambda}(\gamma)<L$.
\end{proof}

\subsection{Finite leaved maximal laminations}
We can now focus on the existence of pleated surfaces, as announced in Theorem \ref{geometry pleated surfaces h2n}. We start from the case of finite leaved laminations:

\begin{pro}
	\label{pro:developing finite}
	Let $\rho:\Gamma\to\SOtwon$ be a maximal representation and let $\lambda$ be a finite leaved maximal $\rho$-lamination of $\Sigma$. If ${\widehat{S}_\lambda}$ denotes pleated set associated to $\lambda$, then there exists a homeomorphism $\hat{f}:\widehat{S}_\lambda\to\mb{H}^2$ with the following properties:
	\begin{enumerate}[(i)]
		\item{It is totally geodesic on every leaf and plaque.}
		\item{It is is $1$-Lipschitz, that is, $d_{\widehat{\mb{H}}^{2,n}}(x,y)\ge d_{\mb{H}^2}(\hat{f}(x),\hat{f}(y))$ for every $x,y\in\widehat{S}_\lambda$.}
		\item{There exists a holonomy representation $\rho_X : \Gamma \to \psl$ of some hyperbolic structure $X \in \T$ such that $\hat{f}$ is $(\rho_X, \rho)$-equivariant, i.e. $$\hat{f} \rho(\gamma)(x) = \rho_X(\gamma) \hat{f}(x)$$for every $\gamma \in \Gamma$, $x \in \widehat{S}_\lambda$.}
	\end{enumerate} 
\end{pro}

\begin{proof}
	Let $\lambda_c = \gamma_1\sqcup\cdots\sqcup\gamma_k\subset\lambda$ denote the collection of the closed leaves of $\lambda$. Each leaf $\ell$ of $\lambda-\lambda_c$ is an isolated point of $\lambda \subset \mathcal{G}$ and its geometric realization $\hat{\ell}$ in $\widehat{\hyp}^{2,n}$ is adjacent to two distinct spacelike ideal triangles $\Delta,\Delta'$, that is, $\ell=\Delta\cap \Delta'$. In particular, every component $S_j$ of $S-\lambda_c$ has an intrinsic (possibly incomplete) hyperbolic metric, with $S = S_\lambda$. The abstract metric completion $X_j$ of $S_j$ is a (a priori possibly non-compact) hyperbolic surface with totally geodesic boundary. 
	
	Let $\widehat{S}_j$ denote the preimage of $S_j$ in $\widehat{\hyp}^{2,n}$, and let $U_j \to X_j$ be the universal cover of $X_j$, where $U_j$ is a closed convex domain of $\hyp^2$ with totally geodesic boundary. The natural inclusion $S_j \subset X_j$ lifts to an isometric embedding $\hat{f}_j : \widehat{S}_j \to U_j$, which is unique up to post-composition by a deck transformation of the covering $U_j \to X_j$. We start with the following assertion:
	
	\begin{claim}{\it 1}
		For every $x,y\in\widehat{S}_j$ we have
		\[
		d_{\widehat{\mb{H}}^{2,n}}(x,y)\ge d_{U_j}(\hat{f}_j(x),\hat{f}_j(y)) .
		\]
	\end{claim}
	
	\begin{proof}[Proof of claim] 
		Consider two distinct points $x,y\in\widehat{S}_j$. We select a spacelike plane $H$ containing $x,y$, and we denote by $T$ the timelike sphere in $\widehat{\hyp}^{2,n}$ that is orthogonal to $H$ at $x$. The choice of $H$ and $T$ determines a Poincaré model $\Psi:\mb{D}^2\times\mb{S}^n\to\widehat{\mb{H}}^{2,n}$ such that $\Psi(\mb{D}^2\times\{v\})=H$ and $\Psi(0,v)=x$, for some fixed $v \in \mathbb{S}^n$.
		
		By Lemma \ref{lem:planes and lines}, the projection of any leaf $\hat{\ell}\subset{\hat \lambda}$ to $\mb{D}^2$ intersects (transversely) every diameter of $\mb{D}^2$ at most once. As a consequence, if we represent $\widehat{S}_\lambda$ as a graph of a $1$-Lipschitz function $g:\mb{D}^2\to\mb{S}^n$, the geodesic segment $[x,y]\subset\mb{D}^2$ (which, by the choice of the Poincaré model, is a subsegment of a diameter in $\mb{D}^2$) lifts on a curve $\gamma$ joining $x,y\in\widehat{S}_j$ that stays inside $\widehat{S}_j$. Indeed, if this was not the case, we could find a leaf $\hat{\ell}\subset\partial\widehat{S}_j\subset{\hat \lambda}_c$ whose  projection onto $\mb{D}^2$ crosses $[x,y]$ at least twice. 
		
		Using the explicit expression of the metric given by Proposition \ref{pro:poincare H2n}, we can now conclude that $L(\hat{f}_j\gamma)\le L(\alpha)$ with equality if and only if $[x,y]\subset\widehat{S}_j$. To see this, we select a parametrization of the geodesic segment $\alpha:I\to[x,y] \subset \mb{D}^2$, and write $\gamma(t)=(\alpha(t),g(\alpha(t)))$. We also denote by $\norm{\bullet}_{\mb{S}^n}$ (resp. $\norm{\bullet}_{\hyp^2}$ and $\norm{\bullet}_0$) the norm associated to the spherical metric of $\mb{S}^n$ (resp. hyperbolic and Euclidean metrics on $\mb{D}^2$). Being $\hat{f}_j$ an isometric embedding, we have $L(\hat{f}_j\gamma) = L(\gamma)$. On the other hand, we observe
		\begin{align*}
			L(\gamma) &=\int_I \sqrt{\langle\dot{\gamma}(t),\dot{\gamma}(t)\rangle} \, \mathrm{d}t\\
			&=\int_I \sqrt{\Vert \dot{\alpha}(t)\Vert_{\mb{H}^2}^2-\frac{4}{(1 - \Vert\alpha(t)\Vert^2_0)^2}\Vert{(g\alpha)\dot{}(t)}\Vert^2_{\mb{S}^n}} \, \mathrm{d}t\\
			&\le\int_I{\Vert{\dot{\alpha}(t)}\Vert_{\mb{H}^2} \, \mathrm{d}t}=L_{\mb{H}^2}(\alpha).
		\end{align*}
		Therefore, we conclude that
		\[
		d_{U_j}(\hat{f}_j(x), \hat{f}_j(y)) \leq L(\hat{f}_j \gamma) = L(\gamma) \leq L_{\hyp^2}(\alpha) .
		\]
		
		The desired statement now follows by observing that the length of $\alpha$ with respect to the hyperbolic metric of $\mathbb{D}^2$ is equal to the pseudo-distance $d_{\widehat{\mb{H}}^{2,n}}(x,y)$, by the choice of the Poincaré model $\Psi$.
	\end{proof}
	
	This proves in particular that the map $\hat{f}_j : \widehat{S}_j \to U_j$ is uniformly continuous, in the sense of Definition \ref{def:uniformly continuous}. Hence, by Lemma \ref{lem:extension and convergence}, $\hat{f}_j$ extends uniquely to a map (that we continue to denote with abuse by ${\hat f}_j$) from the closure of $\widehat{S}_j$ inside $\widehat{S} = \widehat{S}_\lambda$, obtained from $\widehat{S}_j$ by adding the leaves of $\hat{\lambda}_c$ that are adjacent to $\widehat{S}_j$, into $U_j$. By construction and continuity, the map ${\hat f}_j$ is $(\rho_{X_j}, \rho_j)$-equivariant, where $\rho_j$ denotes the restriction of $\rho$ to $\pi_1(S_j)$, and $\rho_{X_j}$ is a holonomy representation of the hyperbolic surface $X_j$ that preserves $U_j \subset \hyp^2$. 
	
	\begin{claim}{\it 2} 
		The extension ${\hat f}_j : \widehat{S}_j \cup \partial\widehat{S}_j \to U_j$ satisfies the following properties:
		\begin{itemize}
			\item{It is a $(\rho_{X_j},\rho_j)$-equivariant homeomorphism.} 
			\item{It maps every leaf $\hat{\ell}\subset\partial\widehat{S}_j$ isometrically into a connected component of the geodesic boundary of $U_j$.}
			\item{We have $L({\hat f}_j \gamma)=L(\gamma)$ for every regular path $\gamma:I\to\widehat{S}_j\cup\partial\widehat{S}_j$ that intersects ${\hat \lambda}$ in countably many points.}
		\end{itemize}
	\end{claim}
	
	\begin{proof}[Proof of the claim] 
		Let $\hat{\ell}$ be a boundary leaf of $\partial\widehat{S}_j$ and let $[x,y]\subset\hat{\ell}$ be a finite subsegment. Since $\lambda$ is a finite leaved maximal lamination and $\hat{\ell}$ projects in $S$ onto a simple closed curve, there exists a sequence of leaves $(\hat{\ell}_n)_n \subset\widehat{S}_j$ that converges to $\hat{\ell}$. Thus, we can approximate $[x,y]$ with a sequence of segments $[x_n,y_n]\subset\hat{\ell}_n$ for which $d_{U_j}({\hat f}_j(x_n),{\hat f}_j(y_n))=d_{\mb{H}^{2,n}}(x_n,y_n)$. By continuity, we conclude that $d_{U_j}({\hat f}_j(x),{\hat f}_j(y))=d_{\mb{H}^{2,n}}(x,y)$. Thus, ${\hat f}_j$ maps each boundary leaf $\hat{\ell}\subset\partial\widehat{S}_j$ to a boundary leaf of $\partial{U}_j$ in a totally geodesic way. 
		
		If $\gamma:I\to\widehat{S}_j$ is a regular path that intersects ${\hat \lambda}$ in countably many points then $\ell(\gamma)=\ell(\gamma-{\hat \lambda})$. As ${\hat f}_j$ is totally geodesic on each component of $\widehat{S}_j-{\hat \lambda}$, it follows that ${\hat f}_j\gamma$ is rectifiable and $L({\hat f}_j\gamma)=L(\gamma)$.
		
		Notice that $\hat{f}_j$ restricts to a $\pi_1(S_j)$-equivariant homeomorphism between $\widehat{S}_j$ and the interior of $U_j$ and sends every boundary component of $\widehat{S}_j$ isometrically into a boundary component of $U_j$. Moreover, the surface $S_j \cup \partial S_j$ is a compact orientable surface with boundary of negative Euler characteristic. In particular, distinct connected components of $\partial S_j$ are not freely homotopic to each other and, hence, the map $\hat{f}_j$ must send distinct leaves in $\partial \widehat{S}_j$ into distinct geodesics in $\partial U_j$. This implies that the map $f_j : S_j \cup \partial S_j \to X_j$ is a bijective continuous function from a compact space to a Hausdorff one, and hence a homeomorphism.  
	\end{proof}
	
	Let $\gamma\subset\lambda_c$ be a leaf adjacent to the components $S_i,S_j$ (possibly equal), and denote by $\alpha_i\subset\partial X_i,\alpha_j\subset\partial X_j$ the boundary components corresponding to $\gamma$ in the abstract closures. There is a unique way to glue the completions $X_i,X_j$ along $\alpha_i,\alpha_j$ so that the identifications with $\gamma$ agree. Thus, after gluing all the completions $X_j$ along their boundary components as prescribed by the leaves of $\lambda_c$, we get a hyperbolic surface 
	\[
	X_\lambda=\bigsqcup_{S_k\textrm{ component of }S_\lambda-\lambda_c}{X_k}\left/\alpha_i\sim\alpha_j\right.
	\]
	and a homeomorphism $f:S_\lambda\to X_\lambda$ which is isometric on every leaf of $\lambda$ and plaque of $S_\lambda-\lambda$. 
	
	Lift $f$ to a map ${\hat f}:\widehat{S}_\lambda\to\mb{H}^2$. 
	
	\begin{claim}{\it 3}
		The map ${\hat f}$ sends regular paths that intersect ${\hat \lambda}_c$ in countably many points to rectifiable paths of the same length and we have 
		\[
		d_{\widehat{\mb{H}}^{2,n}}(x,y)\ge d_{\mb{H}^2}({\hat f}(x),{\hat f}(y))
		\]
		for every $x,y\in\widehat{S}_\lambda$.
	\end{claim}
	
	\begin{proof}[Proof of the claim]
		The proof is similar to the one of the first claim. 
		
		By continuity and density, it is enough to restrict our attention to $x,y\in\widehat{S}-{\hat \lambda}$. As in the first claim, let $H$ be a spacelike hyperplane containing $x,y$ and let $T$ be a timelike sphere orthogonal to $H$ at $x$. This choice corresponds to a Poincaré model $\Psi:\mb{D}^2\times\mb{S}^n\to\widehat{\mb{H}}^{2,n}$ such that $\Psi(\mb{D}^2\times\{v\})=H$ and $\Psi(0,v)=x$, for some fixed $v \in \mb{S}^n$. 
		
		Let $\ell\subset{\hat \lambda}$ be a leaf. By Lemma \ref{lem:planes and lines}, the projection of $\ell$ to $\mb{D}^2$ intersects (transversely) every diameter of $\mb{D}^2$ at most once. As a consequence, the geodesic segment $[x,y]\subset\mb{D}^2$ intersects the projections of the leaves of ${\hat \lambda}_c$ in finitely many points $p_1,\cdots,p_k\in[x,y]$. For simplicity, set also $p_0:=x$ and $p_{k+1}:=y$. Let $\alpha_j,\alpha$ be the lifts of the subsegments $[p_j,p_{j+1}],[x,y]$ inside $\widehat{S}_\lambda$ through the graph parametrization on $\mathbb{D}^2$. By construction, the path ${\hat f}(\alpha)\subset\mb{H}^2$ is the concatenation of the paths ${\hat f}(\alpha_j)$. By the above discussion, since $\alpha_j$ is entirely contained in (the closure of) a component $\widehat{S}_i$, the path ${\hat f}(\alpha_j)$ is rectifiable and has length $L({\hat f}(\alpha_j))=L(\alpha_j)$. Therefore, ${\hat f}(\alpha)$ is also a rectifiable path of length $L({\hat f}(\alpha))=L(\alpha)$ and joins ${\hat f}(x)$ to ${\hat f}(y)$. Thus, 
		\[
		d_{\mb{H}^2}({\hat f}(x),{\hat f}(y))\le L({\hat f}(\alpha))=L(\alpha)=d_{\widehat{\mb{H}}^{2,n}}(x,y).
		\]
		This concludes the proof of the claim.
	\end{proof}
	
	This concludes the proof of the proposition.
\end{proof}

The only missing piece in the finite leaved setting is the equivalence of the intrinsic hyperbolic structure $X_\lambda$ provided by Proposition \ref{pro:developing finite} and the one given by the $\rho$-intrinsic shear cocycle $\sigma^\rho_{\lambda}$, associated to the $\rho$-invariant cross ratio $\beta^\rho$ and the maximal lamination $\lambda$ by Theorem \ref{shear of cross ratio}.

\begin{pro} 
	\label{pro:intrinsic shear finite case}
	Let $\lambda$ be a finite maximal lamination, and let $S_\lambda$ be the pleated set of $M$ associated to $\lambda$. Then the $\rho$-shear cocycle $\sigma^\rho_{\lambda}$ coincides with the shear coordinates of the hyperbolic metric $X = X_\lambda$ described in Proposition \ref{pro:developing finite}. In other words, we have $\sigma^\rho_{\lambda} = \sigma^X_{\lambda}$.
\end{pro}

\begin{proof}
	Recall that the shear coordinates $\sigma^X_\lambda$ of a hyperbolic structure $X \in \T$ with respect to a maximal lamination $\lambda$, as well as the $\rho$-shear cocycle associated by Theorem \ref{shear of cross ratio} to the cross ratio $\rho$ and $\lambda$, are H\"older cocycles transverse to $\lambda$. In particular, by additivity of $\sigma^\rho_\lambda, \sigma^X_\lambda \in \mathcal{H}(\lambda; \R)$, it is enough to show that $\sigma^\rho_\lambda(P,Q) = \sigma^X_\lambda(P,Q)$ when $P$ and $Q$ are separated by at most one component $\tilde{\gamma}$ of $\lambda_c$, the set of leaves of $\lambda$ that project onto simple closed geodesics in $\Sigma$. 
	
	If no component of $\lambda_c$ separates $P$ from $Q$, then there exists a finite collection of plaques $P = P_0, P_1, \dots, P_n, P_{n + 1} = Q$ such that $P_i$ and $P_{i + 1}$ are adjacent for every $i$. Again by additivity, it is sufficient to check that $\sigma^X_\lambda(P_i,P_{i + 1}) = \sigma^\rho_{\lambda}(P_i,P_{i + 1})$, and this follows from what we observed in Remark \ref{rmk:shear_is_shear}.
	
	Therefore it is enough to consider the case in which $P$ and $Q$ are separated by exactly one component of $\lambda_c$. By what we just proved, we can further reduce the discussion to the case in which both plaques $P$ and $Q$ have exactly one ideal vertex equal to one of the endpoints $\tilde{\gamma}^\pm$ of $\tilde{\gamma}$. Up to relabeling the plaques and change orientation of $\tilde{\gamma}$, we can assume that $P$ lies on the left of $\tilde{\gamma}$ and has one vertex equal to $\tilde{\gamma}^+$. We denote by $x_P, y_P$ the vertices of $P$ different from $\tilde{\gamma}^\pm$, so that the leaf $[y_P, \tilde{\gamma}^+]$ separates the interior of $P$ from $\tilde{\gamma}$. If $z_Q$ denotes the vertex of $Q$ that coincides with one of the endpoints of $\tilde{\gamma}$, then we label the other vertices of $Q$ as $x_Q, y_Q$, so that $[y_Q, z_Q]$ is the boundary component of $Q$ that separates the interior of $Q$ from $\tilde{\gamma}$.
	
	As usual, we denote by $\mathcal{P}_{P Q}$ the set of plaques of $\lambda$ that separate $P$ from $Q$. By Lemmas \ref{lem:asymptotic_plaques} and \ref{lem:shear near closed leaves}, for any finite collection $\mathcal{P} \subset \mathcal{P}_{P Q}$ we have
	\begin{align}
		\begin{split}
			\sigma^\rho_\mathcal{P}(P,Q) & = \log\abs{\beta^{\rho}(\tilde{\gamma}^+, y_P, x_P, y_Q) \beta^{\rho}(\tilde{\gamma}^+, y_Q, y_P, x_Q)} , \\
			\sigma^X_\mathcal{P}(P,Q) & = \log\abs{\beta^X(\tilde{\gamma}^+, y_P, x_P, y_Q) \beta^X(\tilde{\gamma}^+, y_Q, y_P, x_Q)} ,
		\end{split}
	\end{align}
	if $z_Q = \tilde{\gamma}^+$, and 
	\begin{align}
		\begin{split}
			\sigma^\rho_\mathcal{P}(P,Q) & = \log\abs{\beta^{\rho}(\tilde{\gamma}^+, y_P, x_P, \tilde{\gamma}^-) \beta^{\rho}(\tilde{\gamma}^+, \tilde{\gamma}^-, y_P, y_Q) \beta^{\rho}(y_Q, \tilde{\gamma}^-, \tilde{\gamma}^+, x_Q)} , \\
			\sigma^X_\mathcal{P}(P,Q) & = \log\abs{\beta^X(\tilde{\gamma}^+, y_P, x_P, \tilde{\gamma}^-) \beta^X(\tilde{\gamma}^+, \tilde{\gamma}^-, y_P, y_Q) \beta^X(y_Q, \tilde{\gamma}^-, \tilde{\gamma}^+, x_Q)} ,
		\end{split}
	\end{align}
	if $z_Q = \tilde{\gamma}^-$. In particular $\sigma^\rho_\lambda(P,Q) = \sigma^\rho_\mathcal{P}(P,Q)$ and $\sigma^X_\lambda(P,Q) = \sigma^X_\mathcal{P}(P,Q)$ are independent of the choice of $\mathcal{P} \subset \mathcal{P}_{P Q}$.
	
	Select now an identification between the universal cover of $\Sigma$ and $\hyp^2$ compatible with the intrinsic hyperbolic structure $X=X_\lambda\in\T$. Then the classical shear $\sigma^X_\lambda(P,Q)$ can be characterized as follows (see e. g. \cite{Bo96}): 
	
	\begin{fact*}
		Let $\alpha = \alpha(s)$ be a unit speed parametrization of the geodesic $[\tilde{\gamma}^-, \tilde{\gamma}^+]$ in $\hyp^2$ pointing towards $\tilde{\gamma}^+$, and denote by $\hat{v}_P, \hat{v}_Q \in \hyp^{2}$ the projections of the ideal vertices $x_P, x_Q$ onto the spacelike geodesics $[\tilde{\gamma}^+, y_P], [z_Q, y_P]$, respectively. In addition we set $\alpha(s_P)$ (resp. $\alpha(s_Q)$) to be the intersection point between $[\tilde{\gamma}^-, \tilde{\gamma}^+]$ and the horocycle of $\hyp^2$ based at $\tilde{\gamma}^+$ (resp. $z_Q$) that passes through $\hat{v}_P$ (resp. $\hat{v}_Q$). Then $s_Q - s_P = \sigma^X_\lambda(P,Q)$.
	\end{fact*}
	
	On the other hand, we will prove that a similar description holds for the shear $\sigma^\rho_\lambda$:
	
	\begin{claim*} 
		Let $\ell = \ell(t)$ be a unit speed parametrization of the spacelike geodesic $[\xi(\tilde{\gamma}^+), \xi(\tilde{\gamma}^-)]$ pointing towards $\xi(\tilde{\gamma}^+)$, and denote by $v_P, v_Q \in \hyp^{2,n}$ the projections of the ideal vertices $\xi(x_P), \xi(x_Q)$ onto the spacelike geodesics $[\xi(\tilde{\gamma}^+), \xi(y_P)], [\xi(z_Q), \xi(y_P)]$, respectively. In addition we set $\ell(t_P)$ (resp. $\ell(t_Q)$) to be the intersection point between $\ell$ and the horosphere of $\hyp^{2,n}$ based at $\xi(\tilde{\gamma}^+)$ (resp. $\xi(z_Q)$) that passes through $v_P$ (resp. $v_Q$). Then $t_Q - t_P = \sigma^\rho_\lambda(P,Q)$.
	\end{claim*}
	
	Assuming that such characterization holds true, we can finally prove the statement. By Lemma \ref{lem:acausal lifts}, there exists an acausal lift $\xi: \partial \Gamma \to \widehat{\Lambda} \subset \partial \widehat{\hyp}^{2,n}$ of the limit map of $\rho$ inside $\partial \widehat{\hyp}^{2,n}$. Notice that $\scal{\tilde{\xi}(x)}{\tilde{\xi}(y)} < 0$ for any distinct $x, y \in \partial\Gamma$ and for any choice of representatives $\tilde{\xi}(x), \tilde{\xi}(y)$ of $\xi(x), \xi(y) \in \partial\widehat{\hyp}^{2,n}$. Let $\widehat{S}$ denote the lift of the pleated set $S$ realizing $\lambda$ to $\widehat{\hyp}^{2,n}$. 
	
	Recall from Section \ref{subset:pseudo riemann} that any horosphere $O$ based at $\xi(\tilde{\gamma}^+) \in \partial\widehat{\hyp}^{2,n}$ (or $\xi(z_Q)$) intersects every plaque that has an ideal vertex equal to $\xi(\tilde{\gamma}^+)$ (or $\xi(z_Q)$) into a horocycle. It follows that the curve $\partial O \cap \widehat{S}$ is a horocycle based at $\tilde{\gamma}^+$ (or $z_Q$) with respect to the intrinsic hyperbolic metric of $\widehat{S}$. On the other hand, the points $v_P$ and $v_Q$ are uniquely determined by the intrinsic hyperbolic structures of the plaques $P$ and $Q$, so the horocycles $\partial O \cap \widehat{S}$ pass through the point of $\widehat{S}$ corresponding to $\hat{v}_P$ (or $\hat{v}_Q$). This implies that $s_Q - s_P = t_Q - t_P$, and therefore $\sigma^X_\lambda(P,Q) = \sigma^\rho_\lambda(P,Q)$, which was what we were left to prove. 
	
	\begin{proof}[Proof of the claim]
		The projections $v_P, v_Q \in \widehat{\hyp}^{2,n}$ satisfy
		\begin{align*}
			v_P = & \sqrt{- \frac{\scal{\tilde{\xi}(y_P)}{\tilde{\xi}(\tilde{\gamma}^+)}}{2 \scal{\tilde{\xi}(x_P)}{\tilde{\xi}(\tilde{\gamma}^+)} \scal{\tilde{\xi}(x_P)}{\tilde{\xi}(y_P)}}} \left(\frac{\scal{\tilde{\xi}(x_P)}{\tilde{\xi}(\tilde{\gamma}^+)}}{\scal{\tilde{\xi}(y_P)}{\tilde{\xi}(\tilde{\gamma}^+)}} \tilde{\xi}(y_P) + \frac{\scal{\tilde{\xi}(x_P)}{\tilde{\xi}(y_P)}}{\scal{\tilde{\xi}(y_P)}{\tilde{\xi}(\tilde{\gamma}^+)}} \tilde{\xi}(\tilde{\gamma}^+)\right), \\
			v_Q = & \sqrt{- \frac{\scal{\tilde{\xi}(y_Q)}{\tilde{\xi}(z_Q)}}{2 \scal{\tilde{\xi}(x_Q)}{\tilde{\xi}(z_Q)} \scal{\tilde{\xi}(x_Q)}{\tilde{\xi}(y_Q)}}} \left(\frac{\scal{\tilde{\xi}(x_Q)}{\tilde{\xi}(z_Q)}}{\scal{\tilde{\xi}(y_Q)}{\tilde{\xi}(z_Q)}} \tilde{\xi}(y_Q) + \frac{\scal{\tilde{\xi}(x_Q)}{\tilde{\xi}(y_Q)}}{\scal{\tilde{\xi}(y_Q)}{\tilde{\xi}(z_Q)}} \tilde{\xi}(z_Q)\right) ,
		\end{align*}
		where $\tilde{\xi}(x)$ denotes a representative of the projective class $\xi(x) \in \partial\hyp^{2,n} \subset \rp^{n + 2}$ (compare with Remark \ref{rmk:shear_is_shear}). Consider now the parametrization of the leaf $[\xi(\tilde{\gamma}^+), \xi(\tilde{\gamma}^-)]$ given by
		\[
		\ell(t) = \frac{1}{\sqrt{- 2 \scal{\tilde{\xi}(\tilde{\gamma}^+)}{\tilde{\xi}(\tilde{\gamma}^-)}}} (e^t \tilde{\xi}(\tilde{\gamma}^+) + e^{-t} \tilde{\xi}(\tilde{\gamma}^-))
		\]
		
		The horosphere based at $\xi(\tilde{\gamma}^+)$ that passes through $v_P$ intersects the spacelike geodesic $[\xi(\tilde{\gamma}^+), \xi(\tilde{\gamma}^-)]$ at $\ell(t_P)$,	where
		\[
		e^{t_P} = \sqrt{\frac{\scal{\tilde{\xi}(\tilde{\gamma}^+)}{\tilde{\xi}(\tilde{\gamma}^-)} \scal{\tilde{\xi}(x_P)}{\tilde{\xi}(y_P)}}{\scal{\tilde{\xi}(x_P)}{\tilde{\xi}(\tilde{\gamma}^+)} \scal{\tilde{\xi}(y_P)}{\tilde{\xi}(\tilde{\gamma}^+)}}}.
		\]
		
		Similarly, the horosphere based at $\xi(z_Q)$ that passes through $v_Q$ intersects the spacelike geodesic $[\xi(\tilde{\gamma}^+), \xi(\tilde{\gamma}^-)]$ at $\ell(t_Q)$,	where
		\[
		e^{\pm t_Q} = \sqrt{\frac{\scal{\tilde{\xi}(\tilde{\gamma}^+)}{\tilde{\xi}(\tilde{\gamma}^-)} \scal{\tilde{\xi}(x_Q)}{\tilde{\xi}(y_Q)}}{\scal{\tilde{\xi}(x_Q)}{\tilde{\xi}(\tilde{\gamma}^\pm)} \scal{\tilde{\xi}(y_Q)}{\tilde{\xi}(\tilde{\gamma}^\pm)}}} 
		\]
		if $z_Q = \tilde{\gamma}^\pm$, respectively. In particular we have
		\begin{align*}
			t_Q - t_P & = \frac{1}{2} \log \frac{ \scal{\tilde{\xi}(x_Q)}{\tilde{\xi}(y_Q)} \scal{\tilde{\xi}(x_P)}{\tilde{\xi}(\tilde{\gamma}^+)} \scal{\tilde{\xi}(y_P)}{\tilde{\xi}(\tilde{\gamma}^+)}}{\scal{\tilde{\xi}(x_Q)}{\tilde{\xi}(\tilde{\gamma}^+)} \scal{\tilde{\xi}(y_Q)}{\tilde{\xi}(\tilde{\gamma}^+)} \scal{\tilde{\xi}(x_P)}{\tilde{\xi}(y_P)}} \\
			& = \log\abs{\beta^{\rho}(\tilde{\gamma}^+, y_P, x_P, y_Q) \beta^{\rho}(\tilde{\gamma}^+, y_Q, y_P, x_Q)} \\
			& = \sigma^\rho_\mathcal{P}(P,Q)
		\end{align*}
		if $z_Q = \tilde{\gamma}^+$, and
		\begin{align*}
			t_Q - t_P & = \frac{1}{2} \log \frac{\scal{\tilde{\xi}(x_Q)}{\tilde{\xi}(\tilde{\gamma}^-)} \scal{\tilde{\xi}(y_Q)}{\tilde{\xi}(\tilde{\gamma}^-)}}{\scal{\tilde{\xi}(\tilde{\gamma}^+)}{\tilde{\xi}(\tilde{\gamma}^-)} \scal{\tilde{\xi}(x_Q)}{\tilde{\xi}(y_Q)}} \frac{\scal{\tilde{\xi}(x_P)}{\tilde{\xi}(\tilde{\gamma}^+)} \scal{\tilde{\xi}(y_P)}{\tilde{\xi}(\tilde{\gamma}^+)}}{\scal{\tilde{\xi}(\tilde{\gamma}^+)}{\tilde{\xi}(\tilde{\gamma}^-)} \scal{\tilde{\xi}(x_P)}{\tilde{\xi}(y_P)}} \\
			& = \log\abs{\beta^{\rho}(\tilde{\gamma}^+, y_P, x_P, \tilde{\gamma}^-) \beta^{\rho}(\tilde{\gamma}^+, \tilde{\gamma}^-, y_P, y_Q) \beta^{\rho}(y_Q, \tilde{\gamma}^-, \tilde{\gamma}^+, x_Q)} \\
			& = \sigma^\rho_\mathcal{P}(P,Q)
		\end{align*}
		if $z_Q = \tilde{\gamma}^-$.
	\end{proof}
	This concludes the proof of Proposition \ref{pro:intrinsic shear finite case}.
\end{proof}

The combination of Propositions \ref{pro:strict inequality}, \ref{pro:developing finite}, and \ref{pro:intrinsic shear finite case} concludes the proof of the existence part of Theorem \ref{geometry pleated surfaces h2n} in the case of finite leaved maximal laminations. The uniqueness is addressed in the case of a general maximal lamination in Proposition \ref{pro:developing general} below.

\subsection{General maximal laminations}
We now extend the result from finite leaved laminations to the general case using the continuity of the construction.

\begin{pro}
	\label{pro:developing general}
	Let $\rho:\Gamma \to \SOtwon$ be a maximal representation. Let $\lambda$ be a maximal lamination with associated pleated set $S_\lambda=\widehat{S}_\lambda/\rho(\Gamma)$. Let $X_\lambda\in\T$ be the hyperbolic surface whose shear coordinates with respect to the lamination $\lambda$ agree with the intrinsic shear cocycle $\sigma^\rho_{\lambda}\in\mathcal{H}(\lambda;\R)$. Then there exists a unique developing map $f:S_\lambda\to X_\lambda$ which is $1$-Lipschitz with respect to the intrinsic pseudo-metric on $S_\lambda$ and the hyperbolic metric on $X_\lambda$.
\end{pro}

\begin{proof}
	Let $\lambda_m$ be a sequence of finite leaved maximal laminations that converges to $\lambda$ in the Hausdorff topology. By Propositions \ref{pro:existence pleated sets}, \ref{pro:developing finite}, and \ref{pro:intrinsic shear finite case} for every $m$ we can find a pleated set $S_m=\widehat{S}_m/\rho(\Gamma)$, a hyperbolic surface $X_m\in\T$ with shear cocycle $\sigma_m=\sigma_{\lambda_m}^\rho\in\mc{H}(\lambda_m;\mb{R})$ and a developing map $f_m:S_m \to X_m$ which is $1$-Lipschitz with respect to the pseudo-metric and the hyperbolic metric.
	
	Let $S=\widehat{S}/\rho(\Gamma)$ be the pleated set associated to $\lambda$. By Proposition \ref{pro:continuity shear wrt lamination} the cocycles $\sigma_m=\sigma^\rho_{\lambda_m}$ converge to the cocycle $\sigma=\sigma^\rho_{\lambda}$ naturally associated with the lamination $\lambda$. Moreover, the family of hyperbolic structures $(X_m)_m$ is bounded in $\T$, being their length spectra bounded by the length spectrum of $\rho$ by Proposition \ref{pro:strict inequality}. In particular, up to subsequences, there exists a hyperbolic structure $Y \in \T$ such that $X_m \to Y \in \T$. On the one hand, the shear coordinates $\sigma_m = \sigma_{\lambda_m}^{X_m}$ must converge to $\sigma_\lambda^Y$, in light of Corollary \ref{cor:uniform convergence shear}. On the other hand, by Proposition \ref{pro:continuity shear wrt lamination}, the shears $\sigma_m = \sigma_{\lambda_m}^\rho$ converge to $\sigma_{\lambda}^\rho$ = $\sigma_{\lambda}^X$. Hence we conclude that $X = Y$ and $\sigma_{\lambda}^\rho = \sigma_\lambda^X$.
	
	In order to obtain convergence of developing maps, we will work in the Poincaré model $\Psi:\mb{D}^2\times\mb{S}^n\to\widehat{\mb{H}}^{2,n}$ associated to the choice of an orthogonal splitting $\mb{R}^{2,n+1}=E\oplus F$ where $E$ is a $(2,0)$-plane. We write $\widehat{S}_m,\widehat{S}$ as graphs of $1$-Lipschitz functions $g_m,g:\mb{D}^2\to\mb{S}^n$, that is, they are the images of the functions $u_m,u$ defined by 
	\[
	\begin{matrix}
		u_m,u: & \mathbb{D}^2 & \longrightarrow &\widehat{S}_m,\widehat{S}\subset\widehat{\mb{H}}^{2,n}\\
		& x & \longmapsto & \Psi(x,g_m(x)),\Psi(x,g(x)).
	\end{matrix}
	\]
	
	By Proposition \ref{pro:continuity pleated sets}, we have $g_m\to g$ and $u_m\to u$ uniformly on compact subsets of $\mb{D}^2$. Furthermore, by property (3) of Lemma \ref{lem:projection}, the maps $u_m,u$ are $1$-Lipschitz with respect to the hyperbolic distance of $\mathbb{D}^2$ and the pseudo-distance $d_{\hyp^{2,n}}$, that is, they satisfy
	\begin{equation}\label{eq:lipschitz}
		d_{\hyp^{2,n}}(u(x),u(y)),d_{\hyp^{2,n}}(u_m(x),u_m(y)) \leq d_{\hyp^2}(x,y)
	\end{equation}
	for every $x,y\in\mathbb{D}^2$. 
	
	Let $\hat{f}_m:\widehat{S}_m\to\widehat{X}_m$ denote the lifts of the developing maps $f_m$ to the universal covers. Fix $x_0 \in \mathbb{D}^2$ so that $u(x_0)$ lies in the interior of a plaque of $\widehat{S}$. This implies in particular that $u_m(x_0) \not\in\hat{\lambda}_m$ for $m$ sufficiently large. Choose now identifications $\widehat{X}_m\simeq\hyp^2$ so that the sequence $(\hat{f}_m u_m(x_0))_m$ converges to some $y_0\in\hyp^2$. 
	
	By Proposition \ref{pro:existence pleated sets} and relation \eqref{eq:lipschitz}, the maps $h_m:=\hat{f}_mu_m:\mathbb{D}^2\to\hyp^2$ are $1$-Lipschitz with respect to the hyperbolic metrics of both domain and codomain, and $h_m(x_0) \to y_0$ as $m$ goes to $\infty$. By Ascoli-Arzelà, up to subsequences, we have that $h_m$ converges uniformly on compact sets to a $1$-Lipschitz map $h:\mathbb{D}^2\to\hyp^2$ with $h(x_0)=y_0$. Finally, we set $\hat{f}:=h\pi:\widehat{S}\to\hyp^2$, where $\pi:\hyp^{2,n}\to\mathbb{D}^2$ is the projection determined by the map $\Psi$.
	
	Notice that each ${\hat f}_m$ is $(\rho_{X_m},\rho)$-equivariant where $\rho_{X_m}$ is the holonomy associated to the chosen identification ${\hat X}_m\simeq\mb{H}^2$. The sequence of holonomies $\rho_{X_m}$ converges to $\rho_X$, a representative of the holonomy of the hyperbolic surface $X$: As $X_n\to X$ in Teichmüller space $\T$, we only have to check that the sequence is precompact. This follows from the fact that 
	\[
	\rho_{X_m}(\gamma)h_m(x_0)=\rho_{X_m}(\gamma)\hat{f}_m u_m(x_0)=h_m\pi(\rho(\gamma)u_m(x_0))\to h\pi(\rho(\gamma)u(x_0)).
	\]
	As a consequence, we deduce that ${\hat f}$ is $(\rho_X,\rho)$-equivariant: Take $x\in\widehat{S}$ and select $x_m\in\widehat{S}_m$ that converge to $x$. Then
	\begin{align*}
		\hat{f}(\rho(\gamma)x) &=h\pi(\rho(\gamma)\lim_{m\to\infty}x_m)\\
		&=\lim_{m\to\infty}h_m\pi(\rho(\gamma)x_m)\\
		&=\lim_{m \to \infty}\rho_{X_m}(\gamma)h_m\pi (x_m)\\
		&=\rho_X(\gamma)\hat{f}(x),
	\end{align*}
	where in the second equality we used the uniform convergence of the maps $h_m$. 
	
	We now show that ${\hat f}$ is the lift of a $1$-Lipschitz developing map $f:S\to X$. In order to do so, we have to prove that:
	\begin{itemize}
		\item{$\hat{f}$ is injective.}
		\item{$\hat{f}$ is totally geodesic on each leaf of ${\hat \lambda}$ and plaque of $\widehat{S}-{\hat \lambda}$.}
	\end{itemize}
	This will be enough to conclude the proof of the existence of the developing map.
	
	Let now $P$ be a plaque of $\widehat{S}$, and consider a sequence of plaques $P_m \subset \widehat{S}_m$ that converges to $P$. By hypothesis $\hat{f}_m = h_m \pi$ is an isometric embedding on $P_m$, and therefore the same holds for the restriction of $\hat{f} = h \pi$ on $P$. In the same way we see that $\hat{f}(\ell)$ is a parametrized geodesic for every leaf $\ell$ of $\hat{\lambda}$.
	
	Since distinct plaques of $\widehat{S}_m$ are sent by $\hat{f}_m$ into ideal triangles of $\hyp^2$ with disjoint interiors, the same property is verified by $\hat{f}$ and the plaques of $\widehat{S}$. In particular the map $h$ restricts to a homeomorphism between $\mathbb{D}^2 \setminus \pi(\hat{\lambda})$ and $\hyp^2 \setminus \hat{f}(\hat{\lambda})$. In addition, if $P, Q, R$ are plaques of $\widehat{S}$ and $R$ separates $P$ from $Q$, then $\widehat{f}(R)$ separates $\hat{f}(P)$ from $\hat{f}(Q)$. From here it is simple to see that $\hat{f}$ is in fact globally injective.
	
	For the uniqueness part of the statement, assume that $f, f' : S \to X$ are two $1$-Lipschitz developing maps of the same pleated set $S = S_\lambda$. The composition $\phi : = f' \circ f^{-1} : X \to X$ is a continuous homeomorphism isotopic to the identity that sends every leaf and every plaque of a fixed maximal lamination $\lambda$ into itself in a totally geodesic way. It follows that there exists a lift $\tilde{\phi} : \hyp^2 \to \hyp^2$ of $\phi$ that extends to the identity on $\partial \hyp^2$, and that sends plaques and leaves of $\lambda$ into plaques and leaves of $\lambda$. This implies in particular that $\tilde{\phi}$ coincides with the identity on every plaque of $\lambda$. By continuity, we conclude that $\tilde{\phi} = \textrm{id}_{\hyp^2}$, and hence that $f' = f$.
\end{proof}

%%%

\section{Teichmüller geometry and length spectra}
\label{sec:teichmuller}

In this section we relate the geometry of maximal representations to the geometry of Teichmüller space and use Teichmüller geometry to study the length spectrum of maximal representations. Our main goal is the proof of Theorem \ref{structure pleated} from the introduction.

\subsection{Length distortion and dominated set}
Let us briefly describe the picture.

By Theorem \ref{geometry pleated surfaces h2n} and Proposition \ref{pro:strict inequality}, we know that the length spectrum $L_\rho(\bullet)$ of a maximal representation $\rho:\Gamma\to\SOtwon$ dominates the length spectrum $L_{S_\lambda}(\bullet)$ of every $\rho$-equivariant pleated surface $S_\lambda$. Furthermore, we have characterized those curves $\gamma\in\Gamma$ for which the strict inequality $L_{S_\lambda}(\gamma)<L_\rho(\gamma)$ holds: They are precisely the ones that do not intersect essentially the bending locus of $S_\lambda$. 

Thus, we can consider the dominated set of $\rho$ which is the following space:

\begin{dfn}[Dominated Set]
	The {\em dominated set} of the maximal representation $\rho$ is the subset of Teichmüller space $\T$ defined by
	\[
	\mc{P}_\rho:=\{Z\in\T\left|\;L_Z(\bullet)\le L_\rho(\bullet)\right.\}
	\]
	where $L_Z,L_\rho$ are the length spectra of $Z,\rho$.
	
	Similarly, the {\em simply dominated set} is
	\[
	\mc{P}^{{\rm simple}}_\rho:=\{Z\in\T\left|\;L^{{\rm simple}}_Z(\bullet)\le L^{{\rm simple}}_\rho(\bullet)\right.\}
	\]
	where $L^{{\rm simple}}_Z\le L^{{\rm simple}}_\rho$ are the simple length spectra of $Z,\rho$. Clearly $\mc{P}_\rho\subset\mc{P}^{{\rm simple}}_\rho$.
\end{dfn}

Let us stress the fact that the set $\mc{P}_\rho$ is non-empty as it contains the hyperbolic structures $X_\lambda$ of all pleated surfaces $S_\lambda$ associated to maximal laminations $\lambda$, but it always has more structure: By work of Bestvina, Bromberg, Fujiwara, and Souto \cite{BBFS13}, and Théret \cite{The14} on convexity of length functions in shear coordinates (see also \cite{MV} for a different approach), the dominated set $\mc{P}_\rho$ is convex with respect to shear paths. By results of Wolpert \cite{W87}, \cite{W04} on convexity of length functions along Weil-Petersson geodesics, it is also convex with respect to the Weil-Petersson metric.

We will analyze more carefully the structure of the dominated set. In order to do so, let us introduce the following useful auxiliary function which is an analogue of the Thurston's distortion function \cite{T86} for hyperbolic surfaces: 

\begin{dfn}[Maximal Length Distortion]
	The \emph{maximal length distortion} $\kappa:\T\to(0,\infty)$ is the function defined by
	\[
	\kappa(Z):=\sup_{c\in\mc{C}-\{0\}}{\frac{L_\rho(c)}{L_Z(c)}}.
	\]
	Similarly, we also define
	\[
	\kappa^{{\rm simple}}(Z):=\sup_{\mu\in\mc{ML}-\{0\}}{\frac{L_\rho(\mu)}{L_Z(\mu)}}.
	\]
\end{dfn}

As both $L_\rho,L_Z$ are continuous homogeneous positive functions on the space of geodesic currents $\mc{C}$, their ratio $\kappa$ descends to a continuous positive function on the projectivization $\mb{P}\mc{C}$. Since the the projectivization $\mb{P}\mc{C}$ is compact, the supremum $\kappa(Z)$ is a maximum $\kappa(Z)=L_\rho({\bar c})/L_Z({\bar c})$, which is achieved at some current ${\bar c}\in\mc{C}$.

In the first part of the section, we use the maximal length distortion to characterize interior points $Z\in{\rm int}(\mc{P}_\rho)$ as those points for which $\kappa(Z)<1$ (see Lemma \ref{lem:interior points}). In other words, those are exactly the points that are strictly dominated by $\rho$. Thus, Theorem \ref{thm:ctt} is equivalent to ${\rm int}(\mc{P}_\rho)\neq\emptyset$ in this setting. 

Using strict convexity of length functions along Weil-Petersson geodesics, one shows that ${\rm int}(\mc{P}_\rho)\neq\emptyset$ provided that $\mc{P}_\rho$ contains at least two distinct points. If $\rho$ is not Fuchsian, such points can be produced by considering pleated surfaces associated to maximal extensions of two intersecting simple closed curves $\alpha,\beta$. This is the content of Proposition \ref{pro:non empty}. 

For convenience of the reader, we recall the definition of Fuchsian representation

\begin{dfn}[Fuchsian Representation]
	A maximal representation $\rho:\Gamma\to\SOtwon$ is \emph{Fuchsian} if it preserves a spacelike plane $H\subset\mb{H}^{2,n}$.
\end{dfn}

In the second part of the section, we consider points on the boundary $Z\in\partial\mc{P}_\rho$ and exterior points $Z\in\T-\mc{P}_\rho$. An immediate observation is that the pleated surfaces $S_\lambda$ all lie on $\partial\mc{P}_\rho$. Indeed, since they satisfy $L_{S_\lambda}(\mu)=L_\rho(\mu)$ for every measured lamination $\mu\in\mc{ML}$ whose support is contained in $\lambda$, we must have $\kappa(S_\lambda)=1$. We show that for every $Z$ outside ${\rm int}(\mc{P}_\rho)$, the maximum $\kappa(Z)\ge 1$ is realized by some measured lamination (see Proposition \ref{pro:outside max is lamination}). The proof of this fact follows arguments of Thurston \cite{T86} on the existence of maximally stretched laminations between two hyperbolic surfaces. 

As a consequence we deduce that $\mc{P}_\rho$ coincides with the simply dominated set $\mc{P}^{{\rm simple}}_\rho$ (see Corollary \ref{cor:boundary points}). In fact, on the one hand, we have $\mc{P}_\rho\subset\mc{P}_\rho^{{\rm simple}}$ directly from the definition. On the other hand, from the above discussion we get $\partial\mc{P}_\rho\subset\partial\mc{P}_\rho^{{\rm simple}}$. As both subsets are topological disks, a topological argument shows that equality holds.

\subsection{Structure of the dominated set}
We start our analysis of the dominated set by characterizing interior points.

\begin{lem}
	\label{lem:interior points}
	A point $Z\in\mc{P}_\rho$ lies in the interior ${\rm int}(\mc{P}_\rho)$ if and only if we have $\kappa(Z)<1$.
\end{lem} 

\begin{proof}
	Suppose that $\kappa = \kappa(Z)<1$. We can find a small neighborhood $U$ of $Z\in\T$ such that every $X \in U$ is $K$-biLipschitz homeomorphic to $Z$, with $K = 1/\kappa$. In particular, we have $L_X(\bullet)/K<L_Z(\bullet)<KL_X(\bullet)$ for any $X \in U$. We deduce that for every surface in $X\in U$ we have $L_X/L_\rho\le KL_Z/L_\rho\le K\kappa=1$, that is, $X\in\mc{P}_\rho$.  
	
	Vice versa, if $Z\in{\rm int}(\mc{P}_\rho)$, then $Z$ is the midpoint of a WP geodesic $[Z',Z'']$ entirely contained in ${\rm int}(\mc{P}_\rho)$. Let $c\in\mc{C}$ be a geodesic current such that $\kappa=\frac{L_Z(c)}{L_\rho(c)}$. By strict convexity of length functions along Weil-Petersson geodesics (see in particular Wolpert \cite{W04}*{\S~3}), we have $L_Z(c)<(L_{Z'}(c)+L_{Z''}(c))/2\le L_\rho(c)$. Therefore $\kappa(Z)<1$. 
\end{proof}

We remark that exactly the same argument also shows that a point $Z\in\mc{P}_\rho^{{\rm simple}}$ lies in the interior ${\rm int}(\mc{P}_\rho^{{\rm simple}})$ if and only if we have $\kappa^{{\rm simple}}(Z)<1$.

We now show that ${\rm int}(\mc{P}_\rho)$ is never empty when $\rho$ is not Fuchsian.

\begin{pro}
	\label{pro:non empty}
	If $\rho$ is not Fuchsian then ${\rm int}(\mc{P}_\rho)\neq\emptyset$.
\end{pro} 

\begin{proof}
	We prove the statement in two steps: First we show that if $\mc{P}_\rho$ contains two distinct points then ${\rm int}(\mc{P}_\rho)\neq\emptyset$. Then we show that if $\rho$ is not Fuchsian then $\mc{P}_\rho$ contains at least two points.
	
	The first step only uses the Weil-Petersson geometry of Teichmüller space: Let $X,Y\in\mc{P}_\rho$ be distinct points. Let $Z$ be their Weil-Petersson midpoint. We show that $Z$ is an interior point: By Lemma \ref{lem:interior points} this is equivalent to $\kappa(Z)=\sup_{\gamma\in\Gamma}\left\{L_Z(\gamma)/L_\rho(\gamma)\right\}<1$. Let $c\in\mc{C}$ be a geodesic current that achieves $\kappa=L_Z(c)/L_\rho(c)$. By results of Wolpert \cite{W04}*{\S~3}, the length of a geodesic current is strictly convex along a Weil-Petersson geodesic. Hence $L_Z(c)<(L_X(c)+L_Y(c))/2\le L_\rho(c)$. Therefore $\kappa(Z)<1$.
	
	The second step, instead, relies on the pseudo-Riemannian geometry of $\mb{H}^{2,n}$. Let $\alpha$ and $\beta$ be intersecting essential simple closed curves. Extend $\alpha,\beta$ to two finite leaved maximal laminations $\mu,\nu$ of $\Sigma$ by adding finitely many leaves spiraling around $\alpha,\beta$. Let $S_\mu,S_\nu\subset M$ be the pleated surfaces realizing $\mu,\nu$ for $\rho$. Denote by $X_\mu,X_\nu$ their intrinsic hyperbolic structures. Note that $L_{X_\mu}(\alpha)=L_\rho(\alpha)$ and $L_{X_\nu}(\beta)=L_\rho(\beta)$. 
	
	Since $\rho$ is not Fuchsian, the bending loci of $S_\mu$ and $S_\nu$ are both non-empty and, by Proposition \ref{pro:bending locus}, they are sublaminations of $\mu$ and $\nu$. By construction, any non-trivial sublamination of $\mu,\nu$ contains $\alpha$ and $\beta$ as every leaf of $\mu-\alpha,\nu-\beta$ spirals around $\alpha,\beta$. Therefore, the bending loci of $S_\mu,S_\nu$ contain $\alpha,\beta$, respectively. As $\alpha,\beta$ are intersecting, we conclude, by Proposition \ref{pro:strict inequality}, that $L_{X_\mu}(\beta)<L_\rho(\beta)$ and $L_{X_\nu}(\alpha)<L_\rho(\alpha)$. Hence $X_\mu,X_\nu$ are different hyperbolic surfaces. 
\end{proof}

From Lemma \ref{lem:interior points} and Proposition \ref{pro:non empty} we deduce the following result of Collier, Tholozan, and Toulisse \cite{CTT19}

\begin{thm}
	\label{thm:ctt}
	Let $\rho$ be a maximal representation of a surface group into $\SOtwon$. Then either $\rho$ is Fuchsian or there exists $k>1$ and a Fuchsian representation $\sigma$ such that $L_\rho\ge k L_\sigma$
\end{thm}

\subsection{Simple length spectrum}
We now analyze $\kappa(Z)$ for points outside $\mc{P}_\rho$. Our aim is to prove the second part of Theorem \ref{structure pleated}.

\begin{pro}
	\label{pro:outside max is lamination}
	For every $Z\in\T-{\rm int}(P_\rho)$, the maximum $\kappa(Z)$ is achieved by some measured lamination $\mu\in\mc{ML}$.
\end{pro}

\begin{proof}
	Let us first consider $Z\in\T-\mc{P}_\rho$. Following an argument of Thurston \cite{T86}, we show that 
	
	\begin{claim*} 
		$\kappa(Z)=\kappa^{{\rm simple}}(Z):=\sup_{\gamma\text{ simple }}{L_Z(\gamma)/L_\rho(\gamma)}$.
	\end{claim*}
	
	Once we know that $\kappa(Z)$ can be computed by restricting to simple closed curves, it immediately follows that the maximum is achieved at a measured lamination $\lambda\in\mc{ML}$.
	
	\begin{proof}[Proof of the claim]
		In order to prove the claim, we show that if $\gamma$ is not simple and we have $L_Z(\gamma)/L_\rho(\gamma)>1$, then there is a shorter curve $\alpha$ (with respect to $Z$) such that $L_Z(\alpha)/L_\rho(\alpha) > L_Z(\gamma)/L_\rho(\gamma)$. 
		
		As $\gamma$ is not simple, it describes an immersed figure 8 inside $Z$. Let $P\to Z$ be the covering corresponding to the immersed figure 8. The convex core $\mc{CC}(P)$ of the surface $P$ is a pair of pants with geodesic boundary curves $\alpha_1,\alpha_2,\alpha_3$. The idea is to consider the pleated surface $S$ realizing the curves $\alpha_j$ for the maximal representation given by the restriction of $\rho$ to the subgroup corresponding to $\pi_1(P)$. 
		
		Since we treated in detail the construction of pleated surfaces only in the case of a closed surface, we will not directly consider the restriction of $\rho$ to $\pi_1(P)$, but rather we will reduce to the closed surface case by passing to a suitable finite index subgroup of $\Gamma$. We have the following: By a result of Scott \cites{S78,S78corr} there is an intermediate covering $P\to Z'\to Z$ such that $Z'\to Z$ is finite and the projection $P\to Z'$ induces an embedding on the convex core $\mc{CC}(P)$ into $Z'$. Let $\Gamma':=\pi_1(Z')<\pi_1(Z)=\Gamma$ be the subgroup corresponding to the covering. Let $S$ be a pleated surface realizing a finite leaved maximal lamination of $Z'$ containing the curves $\alpha_j$ for the maximal representation given by the restriction of $\rho$ to $\Gamma'$ (the restriction of a maximal representation to a finite index subgroup is maximal as well). Let $P'\to S$ be the covering corresponding to $\pi_1(P)$.
		
		We have $L_\rho(\alpha_j)=L_{S}(\alpha_j)=L_{P'}(\alpha_j)$ for $j\le 3$ by construction and $L_\rho(\gamma)\ge L_{S}(\gamma)=L_{P'}(\gamma)$ by the Lipschitz properties of pleated surfaces. As a consequence, we  get
		\[
		L_Z(\gamma)/L_\rho(\gamma) = L_P(\gamma)/L_\rho(\gamma) \le L_P(\gamma)/L_{P'}(\gamma).
		\]
		On the other hand, by \cite{T86}*{Lemma~3.4} and the hypothesis $L_P(\gamma)/L_{P'}(\gamma) \geq L_Z(\gamma)/L_\rho(\gamma) > 1$, we have that 
		\[
		L_P(\gamma)/L_{P'}(\gamma)\le\max_{j\le 3}\{L_P(\alpha_j)/L_{P'}(\alpha_j)\}) = \max_{j\le 3}\{L_Z(\alpha_j)/L_\rho(\alpha_j)\}) .
		\]
		Furthermore, as $P \to Z$ is a hyperbolic pair of pants, we have 
		$$L_Z(\gamma) = L_{P}(\gamma) > L_{P}(\alpha_j) = L_Z(\alpha_j)$$ 
		for every $j \leq 3$, which yields the conclusion.
	\end{proof}
	
	Lastly, we take care of boundary points $Z\in\partial\mc{P}_\rho$: Let $Z_n$ be a sequence of points outside $P_\rho$ converging to $Z$. By the previous steps, we can associate to each of them a measured lamination $\mu_n\in\mc{ML}$ such that $L_{Z_n}(\mu_n)/L_\rho(\mu_n)=\kappa(Z_n)>1$. Up to subsequence and rescaling, we can assume that the sequence of measured laminations $\mu_n$ converges to some $\mu\in\mc{ML}$. By continuity of length functions, we have $L_{Z_n}(\mu_n)/L_\rho(\mu_n)\to L_Z(\mu)/L_\rho(\mu)\ge 1$. As $Z\in\mc{P}_\rho$, we also have the opposite inequality so we conclude that equality holds and $L_Z(\mu)/L_\rho(\mu)=1=\kappa(Z)$.
\end{proof}

From Proposition \ref{pro:outside max is lamination}, we deduce the following

\begin{cor}
	\label{cor:boundary points}
	We have $\mc{P}_\rho=\mc{P}_\rho^{{\rm simple}}$. 
\end{cor}

\begin{proof}
	Observe that, directly from the definitions, we always have $\mc{P}_\rho\subset\mc{P}_\rho^{{\rm simple}}$.
	
	Also notice that both sets are topological convex disks with non-empty interior. 
	
	If we knew that $\partial\mc{P}_\rho\subset\partial\mc{P}_\rho^{{\rm simple}}$, then the claim would follow from a topological argument based on the following:
	
	\begin{claim*}
		Let $D,D'\subset\mb{R}^n$ be topological $n$-disks such that $D\subset D'$ and $\partial D\subset\partial D'$. Then $D=D'$.
	\end{claim*}
	
	\begin{proof}[Proof of the claim]
		Consider the map $j_*:H_n(D,\partial D)\to H_n(D',\partial D')$ induced by the proper inclusion $j:(D,\partial D)\to(D',\partial D')$. We now show that $j$ is degree one, that is, $j_*$ is an isomorphism. By well-known consequences, we deduce that $j$ is surjective which implies the claim. 
		
		The computation of the degree can be done as follows: Let $\star\in{\rm int}(D)\subset{\rm int}(D')$ be any interior point. As $D-\star,D'-\star$ deformation retract to $\partial D,\partial D'$, we have that the degree $n$ relative homology groups are isomorphic to the local homology groups $H_n(D,\partial D)=H_n(D,D-\star)$ and $H_n(D',\partial D')=H_n(D',D'-\star)$. By the excision theorem, if $U\subset{\rm int}(D)$ is a small ball around $\star$, then $H(D,D-\star)=H_n(U,U-\star)$ and $H(D',D'-\star)=H_n(U,U-\star)$. As $j$ restricts to the identity $U\to U$, we conclude that $j_*$ is an isomorphism.
	\end{proof}
	
	Hence, it is sufficient to show that $\partial\mc{P}_\rho\subset\partial\mc{P}_\rho^{{\rm simple}}$. Consider $Z\in\partial\mc{P}_\rho$, by Lemma \ref{lem:interior points}, we have $\kappa(Z)=1$. Furthermore, by Proposition \ref{pro:outside max is lamination}, the maximum is realized by a measured lamination $\mu\in\mc{ML}$. Therefore $\kappa^{{\rm simple}}(Z)=1$ since weighted simple closed curves are dense in $\mc{ML}$, and, hence, $Z\in\partial\mc{P}_\rho^{{\rm simple}}$, as interior points of $\mc{P}_\rho^{{\rm simple}}$ are the ones for which $\kappa^{{\rm simple}}(Z)<1$.
\end{proof}

%%%

\section{Fibered photon structures}
\label{sec:photons}

As shown by Guichard and Wienhard \cite{GW12}, maximal representations $\rho:\Gamma\to\SOtwon$ parametrize deformations of photon structures, a class of geometric structures in the sense of Thurston (see \cite{ThNotes}*{Chapter~3}), on certain closed manifolds $E$. 

\begin{dfn}[Photon Structure]
	A {\em photon} of $\mb{R}^{2,n+1}$ is an isotropic 2-plane. We denote by ${\rm Pho}^{2,n}$ the space of photons in $\mb{R}^{2,n+1}$. The group $\SOtwon$ acts transitively on the homogeneous space ${\rm Pho}^{2,n}$ with non-compact stabilizer. We call a $(\SOtwon,{\rm Pho}^{2,n})$-structure on a manifold $M$ a {\em photon structure}.
\end{dfn}

\subsubsection*{The space of photons}
Collier, Tholozan, and Toulisse proved in \cite{CTT19}*{Lemma~4.8} that the space of photons ${\rm Pho}^{2,n}$ is homeomorphic to the Stiefel manifold $\mathcal{S}_2(\R^{n+1})$ of (Euclidean) orthonormal $2$-frames of $\R^{n+1}$. In fact, the homeomorphism ${\rm Pho}^{2,n} \cong \mathcal{S}_2(\R^{n+1})$ has a very simple geometric interpretation: If $E$ is a fixed positive definite $2$-plane of $\R^{2,n+1}$, then the orthogonal projection $\pi_E : \R^{2,n+1} \to E$ restricts to a linear isomorphism on every photon $F$ of $\R^{2,n+1}$. In particular, every photon $F \subset \R^{2,n+1} = E \oplus E^\perp$ coincides with the graph of a unique linear isometric embedding
$$t_F : (E, \scal{\bullet}{\bullet}|_E) \to (E^\perp, - \scal{\bullet}{\bullet}|_{E^\perp}) ,$$
which is uniquely determined by the image of a fixed orthonormal basis $e_1, e_2$ of $E$. The homeomorphism ${\rm Pho}^{2,n} \cong \mathcal{S}_2(\R^{n+1})$ is then given by
\[
\begin{matrix}
	{\rm Pho}^{2,n} & \longrightarrow & \mathcal{S}_2(\R^{n+1}) \\
	F & \longmapsto & (t_F(e_1), t_F(e_2)) .
\end{matrix}
\]

Consequently, the space of photons ${\rm Pho}^{2,n}$ is homeomorphic to $\mb{S}^1 \sqcup \mb{S}^1$ if $ n = 1$, to $\rp^3$ if $n = 2$, and it is simply connected for all $n > 2$. Notice in particular that the manifold ${\rm Pho}^{2,n}$ is orientable for every $n \geq 1$.

\subsubsection*{Guichard-Wienhard's domains of discontinuity}

The construction of Guichard and Wienhard is the following: The maximal representation $\rho$ has a natural domain of discontinuity $\Omega_\rho\subset{\rm Pho}^{2,n}$ obtained by removing from the space of photons the closed subset
\[
K_\rho:=\{F\in{\rm Pho}^{2,n}\left|\ell\subset F\text{ for some isotropic line $\ell\in\Lambda_\rho$}\right.\}.
\]

The group $\rho(\Gamma)$ acts properly discontinuously, freely, and cocompactly on $\Omega_\rho$ so that the quotient $E_\rho:=\Omega_\rho/\rho(\Gamma)$ is a closed manifold endowed with a photon structure. By the Ehresmann-Thurston principle \cite{ThNotes}, the topology of $E_\rho$ does not change as we vary $\rho$ continuously.

\subsubsection*{Fibered photon structures}

Collier, Tholozan, and Toulisse \cite{CTT19} have shown that the manifold $E_\rho$ has a natural description as a ${\rm Pho}^{2,n-1}$-bundle $E_\rho\to S$ in a way compatible with the geometric structure, that is, in such a way that the fibers are also geometric. 

\begin{dfn}[{Fibered Photon Structure, \cite{CTT19}}]\label{def:fibered}
	Let $\pi : E\to \Sigma$ be a fiber bundle over the surface $\Sigma$ with characteristic fiber $\mathrm{Pho}^{2,n-1}$, and let $\tilde{\pi}:\widetilde{E}\to\widetilde{\Sigma}$ be the pull-back bundle through the universal covering map $\widetilde{\Sigma} \to \Sigma$. We say that a function $\widetilde{E} \to \mathrm{Pho}^{2,n}$ is \emph{fibered} if it maps every fiber $\tilde{\pi}^{-1}(x)$ homeomorphically onto $\mathrm{Pho}(\iota(x)^\perp) \subset \mathrm{Pho}^{2,n}$ for some $\iota(x) \in \hyp^{2,n}$.
	
	Given a representation $\rho : \Gamma = \pi_1(\Sigma) \to \SOtwon$, a photon structure on $E$ with holonomy $\rho \circ \pi_* : \pi_1(E) \to \SOtwon$ is \emph{fibered} if its developing map $\delta : \widetilde{E} \to \mathrm{Pho}^{2,n}$ is fibered.
\end{dfn}

\begin{rmk}[{\cite{CTT19}*{Remark~4.10}}]
	Even if the manifold $\widetilde{E}$ may be not simply connected (for $n = 1, 2$), its developing map factors through $\widetilde{E}$, since its holonomy is of the form $\rho \circ \pi_*$ for some $\rho : \Gamma \to \SOtwon$.
\end{rmk}

In this section we consider the point of view of fibered photon structures $E\to\Sigma$ associated to maximal representations. We use pleated surfaces to give a geometric decomposition of $E\to\Sigma$, namely triangles and lines of photons which we now introduce. 

\begin{dfn}[Triangles and Lines of Photons]
	For every ideal spacelike triangle $\Delta\subset\mb{H}^{2,n}$ and spacelike geodesic $\ell\subset\mb{H}^{2,n}$, we define a {\em triangle of photons} $E(\Delta)\subset{\rm Pho}^{2,n}$ and a {\em line of photons} $E(\ell)\subset{\rm Pho}^{2,n}$ as the subsets consisting of those photons that are orthogonal to some point $x\in\Delta$ and $x\in\ell$, respectively. Triangles and lines of photons $E(\Delta),E(\ell)$ are naturally fiber bundles over $\Delta,\ell$, where the fiber over the point $x\in\Delta$ is the space ${\rm Pho}(x^\perp) \cong \mathrm{Pho}^{2,n-1}$.
\end{dfn}

Triangles of photons $E(\Delta)$ are codimension 0 submanifolds of ${\rm Pho}^{2,n}$ with boundary. The boundary $\partial E(\Delta)$ consists of three components which are smooth submanifolds of ${\rm Pho}^{2,n}$. Each boundary component is a line of photons. Notice that lines of photons carry an action of the subgroup
$$({\rm SO}(1,1)\times{\rm SO}(1,n)) \cap \SOtwon,$$
which is compatible with the fibration $E(\ell)\to\ell$.

\begin{dfn}[Ideal Boundary]
	Both triangles and lines of photons have a natural notion of ideal boundary. Boundary components correspond to isotropic lines and have the following form: For every isotropic line $[a]\in\partial\mb{H}^{2,n}$, we consider the subspace 
	\[
	E(a):={\rm Pho}(a^\perp)=\{F\in{\rm Pho}^{2,n}\left|a\subset F\right.\}.
	\]
	If $\ell\subset\mb{H}^{2,n}$ is a spacelike geodesic with endpoints at infinity $a,b\in\partial\mb{H}^{2,n}$, then the ideal boundary of $E(\ell)$ is given by $E(a)\cup E(b)$. The subset $E(a)\cup E(\ell)\cup E(b)$ is the closure of $E(\ell)$ in ${\rm Pho}^{2,n}$. 
	
	Similarly, if $\Delta\subset\mb{H}^{2,n}$ is a spacelike ideal triangle with vertices $a,b,c\in\partial\mb{H}^{2,n}$, then the ideal boundary of $E(\Delta)$ is equal to $E(a)\cup E(b)\cup E(c)$. The subset $E(\Delta)\cup E(a)\cup E(b)\cup E(c)$ is the closure of $E(\Delta)$ in ${\rm Pho}^{2,n}$.
\end{dfn}

After having proved the geometric decomposition, we will explain, conversely, how to explicitly construct photon structures that fiber over hyperbolic surfaces by assembling together triangles of photons. The process is completely analogous to the construction of hyperbolic surfaces by gluing ideal triangles. The holonomy of such photon structures corresponds to maximal representations $\rho:\Gamma\to\SOtwon$; the hyperbolic surface $S$, which is the base of the fibering, corresponds to a pleated surface for $\rho$; the gluing parameters of the triangles of photons correspond to the bending of the pleated surface.

The goal of the section is to develop this picture in detail. 

\subsection{A geometric decomposition}

We have the following geometric decomposition of the Guichard-Wienhard domain of discontinuity $\Omega_\rho\subset{\rm Pho}^{2,n}$:

\begin{pro}
	\label{pro:decomposition}
	Let $\rho:\Gamma\to\SOtwon$ be a maximal representation with Guichard-Wienhard domain of discontinuity $\Omega_\rho\subset{\rm Pho}^{2,n}$. Denote by $\widetilde{\Sigma}$ the universal covering of $\Sigma$, and consider a $\rho$-equivariant embedding $\iota:\widetilde{\Sigma}\to\mb{H}^{2,n}$ with acausal image $\iota(\widetilde{\Sigma})$. Then we have:
	\begin{itemize}
		\item{The closure of $\iota(\widetilde{\Sigma})$ inside $\mb{H}^{2,n}\cup\partial\mb{H}^{2,n}$ is equal to $\iota(\widetilde{\Sigma})\cup\Lambda_\rho$, where $\Lambda_\rho$ denotes the limit set of $\rho$.}
		\item{the domain $\Omega_\rho$ is foliated by the subsets $\{{\rm Pho}(\iota(x)^\perp) \mid x \in \widetilde{\Sigma}\}$, and the map $\Omega_\rho\to\widetilde{\Sigma}$, which associates to a point $y\in\Omega_\rho$ the unique leaf ${\rm Pho}(\iota(x)^\perp)$ that contains it, is an equivariant fibration.}
	\end{itemize} 
\end{pro} 

In \cite{CTT19} these properties are proved for smooth equivariant spacelike embeddings (see Lemma 3.23, Lemma 4.11, and Theorem 5.3 of \cite{CTT19}). Here we slightly generalize their results in a purely topological setting, which is necessary when dealing with pleated surfaces.

\begin{proof}
	Let us first prove the first point.
	
	We lift $\iota$ to an acausal embedding ${\hat \iota}:\widetilde{\Sigma}\to\widehat{\mb{H}}^{2,n}$. We will work in different Poincaré models of $\widehat{\mb{H}}^{2,n}$, for now we fix an arbitrary one $\Psi:\mb{D}^2\times\mb{S}^n\to\widehat{\mb{H}}^{2,n}$ and denote by $\pi:\widehat{\mb{H}}^{2,n}\to\mb{D}^2$ the associated projection.
	
	As $\widehat{S}:={\hat \iota}(\widetilde{\Sigma})$ is acausal, it can be represented as the graph of a $1$-Lipschitz function $g:\pi(\widehat{S})\subset\mb{D}^2\to\mb{S}^n$ by Lemma \ref{lem:acausal graph}. Since the map $\pi{\hat \iota}:\widetilde{\Sigma}\to\mb{D}^2$ is injective, the set $\pi(\widehat{S})\subset\mb{D}^2$ is simply connected and open, by invariance of domain. Let $D$ denote the projection $\pi(\widehat{S})$ and let $\overline{D}$ be its closure inside $\overline{\mb{D}}{}^2 = \mb{D}^2\cup\partial\mb{D}^2$. As $g$ is $1$-Lipschitz, it continuously extends to a $1$-Lipschitz function ${\bar g}:\overline{D}\subset\overline{\mb{D}}{}^2\to\mb{S}^n$. We deduce that the closure $\widehat{S}\cup\partial\widehat{S}$ of $\widehat{S}$ inside $\widehat{\mb{H}}^{2,n}\cup\partial\widehat{\mb{H}}^{2,n}$ is the graph of ${\bar g}$. 
	
	We start by showing:
	
	\begin{claim}{1}
		We have $D=\mb{D}^2$.
	\end{claim}
	
	\begin{proof}[Proof of the claim] 
		As $\rho(\Gamma)$ acts cocompactly on $\widehat{S}$, we can find a compact fundamental domain $R\subset\widehat{S}$. Let $U\subset\widehat{S}$ be an open neighborhood of $R$ in $\widehat{S}$ with compact closure. By compactness, there exists $\varepsilon>0$ such that $d_{\mb{H}^{2,n}}(x,y)\ge \varepsilon$ for every $x\in R$ and $y\in\partial U$. As $\rho(\Gamma)$ preserves $\widehat{S}$ and its pseudo metric, we deduce that every point $x\in\widehat{S}$ has an open neighborhood $U_x\subset\widehat{S}$ such that $d_{\mb{H}^{2,n}}(x,\partial U_x)\ge \varepsilon$.
		
		Recall that $\mb{D}^2$ is endowed with a hyperbolic metric. As $\widehat{S}$ is acausal, by Lemma \ref{lem:projection} we have that $d_{\mb{H}^2}(\pi(x),\pi(y))\ge d_{\mb{H}^{2,n}}(x,y)$ for every $x,y\in\widehat{S}$. In particular, for every $x\in\widehat{S}$ we have 
		\[
		d_{\mb{H}^2}(\pi(x),\pi(\partial U_x))\ge d_{\mb{H}^{2,n}}(x,\partial U_x)\ge \varepsilon.
		\] 
		Since $\pi{\hat \iota}:\widetilde{\Sigma}\to\mb{D}^2$ is an injective map, by invariance of domain, it is also open. Therefore, for every $x\in\widehat{S}$ the set $\pi(U_x)$ is an open neighborhood of $\pi(x)$. Furthermore, by the above discussion, $\pi(U_x)$ contains the hyperbolic metric ball of radius $\varepsilon$ centered at $\pi(x)$. We are now ready to conclude: The projection $D=\pi(\widehat{S})$ is a subset of $\mb{D}^2$ with the property that its hyperbolic $\varepsilon$-neighborhood is still contained in $D$. This is possible only if $D=\mb{D}^2$.
	\end{proof}
	
	Using the dynamical properties of $\rho$ we now show that $\partial\widehat{S}={\hat \Lambda}_\rho$, where $\hat{\Lambda}_\rho$ denotes a lift of $\Lambda_\rho$ to $\partial\widehat{\hyp}^{2,n}$.
	
	First, recall that for every $\gamma\in\Gamma$ the element $\rho(\gamma)$ preserves a spacelike geodesic $[a,b]$, on which it acts by translations by $L>0$, and its orthogonal subspace ${\rm Span}\{a,b\}^\perp$, on which it acts with (generalized) largest eigenvalue $\nu$ with $\abs{\nu} < e^L$ (see \cite{BPS19} or \cite{CTT19}*{Corollary~2.6}). Up to replacing $\gamma$ with $\gamma^{-1}$, we can assume that $$\rho(\gamma)a=e^La, \quad \rho(\gamma)b=e^{-L}b,$$
	for some $L, \nu$ satisfying $L> \max\{1, \log\abs{\nu}\}$. 
	
	Fix now $\gamma\in\Gamma$ with invariant axis $[a,b]$. Every $x\in\widehat{\mb{H}}^{2,n}$ can be written as $x=\alpha a+\beta b+u$, with $\alpha,\beta\in\mb{R}$ and $u\in V={\rm Span}\{a,b\}^\perp$. 
	
	\begin{claim}{2}
		There exists a point $x\in\widehat{S}$ that can be written as $x=\alpha a+\beta b+u$, with either $\alpha\neq 0$ or $\beta\neq 0$.
	\end{claim}
	
	\begin{proof}[Proof of the claim] 
		Suppose that this is not the case, then $\widehat{S}\subset V$. Let $e\in{\rm Span}\{a,b\}$ be a spacelike vector. As $V$ has signature $(1,n)$, there exists $e'\in V$ spacelike. Consider the Poincaré model $\Psi:\mb{D}^2\times\mb{S}^n\to\widehat{\mb{H}}^{2,n}$ associated to the orthogonal splitting $\mb{R}^{2,n+1}=E\oplus E^\perp$ where $E={\rm Span}\{e,e'\}$. If a point $x=\Psi(u,v)$ lies on $V$, then
		\begin{align*}
			0 &=\langle\Psi(u,v),e\rangle\\
			&=\langle\frac{2}{1-\Vert u\Vert^2}u+\frac{1+\Vert u\Vert^2}{1-\Vert u\Vert^2}v,e\rangle\\
			&=\frac{2}{1-\Vert u\Vert^2}\langle u,e\rangle .
		\end{align*}
		Since $\pi(\Psi(u,v)) = u$, we deduce that the projection of $V$ to $\mb{D}^2$ lies on the line $\langle\bullet,e\rangle=0$ of $E$. Being $\widehat{S}$ a acausal subset of $\widehat{\hyp}^{2,n}$, the projection $\pi_E(\widehat{S})$ is an open subset of $\mb{D}^2$ and, hence, there there exists a point $x\in\widehat{S}$ which is not contained in $V$.
	\end{proof}
	
	Suppose that there is a point $x=\alpha a+\beta b+u\in\widehat{S}$ with $\alpha\neq 0$. Then we claim that $a$ lies in $\partial\widehat{S}$. To see this, first observe that, by $\rho(\Gamma)$ invariance of $\widehat{S}$, we have
	$$\rho(\gamma)^mx=\alpha e^{mL}a+\beta e^{-mL}b+\rho(\gamma)^nu\in\widehat{S}.$$
	As the largest (generalized) eigenvalue of the restriction of $\rho(\gamma)$ to $V={\rm Span}\{a,b\}^\perp$ is smaller than $e^L$, the sequence $[\alpha e^{mL}a+\beta e^{-mL}b+\rho(\gamma)^nu]$ converges to $[a]$ in the sphere of rays $\mb{R}^{2,n+1}-\{0\}/y\sim\lambda^2 y$. Thus $[a]\in\partial\widehat{S}$. Similarly we see that, if there is a point $x=\alpha a+\beta b+u\in\widehat{S}$ with $\beta\neq 0$, then $b$ lies in $\partial\widehat{S}$.
	
	By $\rho(\Gamma)$-invariance, the orbit $\rho(\Gamma)a$ is contained in $\hat{\Lambda}_\rho\cap\partial\widehat{S}$ (for some lift $\hat{\Lambda}_\rho$ of $\Lambda_\rho$) and is dense inside $\hat{\Lambda}_\rho$. Therefore we conclude that $\hat{\Lambda}_\rho\subset\partial\widehat{S}$. As $\hat{\Lambda}_\rho$ and $\partial\widehat{S}$ are both graphs of functions $\partial\mb{D}^2\to\mb{S}^n$, we conclude that $\hat{\Lambda}_\rho=\partial\widehat{S}$. This concludes the proof of the first point.
	
	For the second point we need to prove the following three properties: 
	\begin{enumerate}
		\item{For every $x\in\widetilde{\Sigma}$, the space ${\rm Pho}(\iota(x)^\perp)$ is contained in $\Omega_\rho$.}
		\item{If $p\in\Omega_\rho$, then $p\in {\rm Pho}(\iota(x)^\perp)$ for some $x\in\widetilde{\Sigma}$.}
		\item{If $x,y\in\widetilde{\Sigma}$ are distinct points, then ${\rm Pho}(\iota(x)^\perp),{\rm Pho}(\iota(y)^\perp)$ are disjoint.}
	\end{enumerate}
	
	Together, the properties imply that $\Omega_\rho$ is foliated by ${\rm Pho}(\iota(x)^\perp)$, as $x$ varies in $\widetilde{\Sigma}$, and is equipped with an equivariant map $\Omega_\rho\to\widetilde{\Sigma}$.
	
	\begin{property}{\it (1)}
		The first property follows from the following fact:
		
		\begin{claim}{\it 3}
			$\mb{P}(\iota(x)^\perp)\cap\Lambda_\rho=\emptyset$ for every $x\in\widetilde{\Sigma}$. 
		\end{claim}
		
		\begin{proof}[Proof of the claim]
			If $a\in\mb{P}(\iota(x)^\perp)\cap\partial\mb{H}^{2,n}$, then the 2-plane ${\rm Span}\{a,\iota(x)\}$ is lightlike, that is, $a,\iota(x)$ are joined by a lightlike geodesic. Let $\mb{D}^2\times\mb{S}^n$ be a Poincaré model where $a,\iota(x)$ correspond respectively to $(p,{\bar g}(p))$ and $(o,{\bar g}(o))$ where $p\in\partial\mb{D}^2$ and $o\in\mb{D}^2$ is the center. By Lemma \ref{lem:projection}, since $[a,\iota(x)]$ is lightlike, we have $d_{\mb{S}^n}({\bar g}(o),{\bar g}(p))=d_{\mb{S}^2}(o,p)$. As ${\bar g}$ is $1$-Lipschitz we must have $d_{\mb{S}^n}({\bar g}(o),{\bar g}(t))=d_{\mb{S}^2}(o,t)$ for every $t$ on the radial segment $[o,p]\subset\mb{D}^2$ which is a minimal geodesic for the hemispherical metric on $\mb{D}^2$. However, by Lemma \ref{lem:projection}, this implies that $(o,g(o))$ and $(t,g(t))$ are connected by a lightlike geodesic. This contradicts the fact that $\widehat{S}$, the graph of $g$, is acausal.
		\end{proof}
		
		Recall that $\Omega_\rho\subset{\rm Pho}^{2,n}$ is the complement of the set 
		\[
		K_\rho:=\{F\in{\rm Pho}^{2,n}\left|\ell\subset F\text{ for some isotropic line $\ell\in\Lambda_\rho$}\right.\}.
		\]
		If ${\rm Pho}(\iota(x)^\perp)\cap K\neq\emptyset$, then there is a photon $F$ orthogonal to $\iota(x)$ containing an isotropic line $a\in\Lambda_\rho$. In particular,  $a\subset \iota(x)^\perp$ which cannot happen by Claim 3.
	\end{property}
	
	\begin{property}{\it (3)} 
		The last property follows from the fact that $\iota(x),\iota(y)$ are joined by a spacelike segment: Suppose that there is a photon $F$ that is simultaneously orthogonal to $\iota(x)$ and $\iota(y)$. Then it is orthogonal to the 2-plane ${\rm Span}\{\iota(x),\iota(y)\}$ which has signature $(1,1)$ as $\iota(x),\iota(y)$ are joined by a spacelike segment. However, the orthogonal of such plane, having signature $(1,n)$, cannot contain photons. 
	\end{property}
	
	\begin{property}{\it (2)}
		The second property follows from the fact that every timelike sphere intersects $\iota(\widetilde{\Sigma})$ exactly once and $\Lambda_\rho=\partial \iota(\widetilde{\Sigma})$. 
		
		Let $F\in{\rm Pho}^{2,n}$ be a photon. The orthogonal $F^\perp$ is non-positive definite and can be approximated by negative definite $(n+1)$-planes $F_n$. Each such plane intersects $\iota(\widetilde{\Sigma})$ exactly once in a point $\iota(x_n)$. Thus $F^\perp$ either intersects $\iota(\widetilde{\Sigma})$ in some point $\iota(x)$ or it intersects $\Lambda_\rho$. In the first case, $F\subset \iota(x)^\perp$, that is $F\in{\rm Pho}(\iota(x)^\perp)$, and, moreover, by Property (3), $F^\perp$ intersects $\iota(\widetilde{\Sigma})$ in exactly the point $\iota(x)$. In the second case, $a\subset F^\perp$ for some isotropic line $a\in\Lambda_\rho$ which implies $a\subset F$ and, hence, $F\in K_\rho$.
	\end{property}
	
	Note that, as a byproduct of the proof, we can describe explicitly the projection $\Omega\to \widehat{S} = \hat{u}(\widetilde{\Sigma})$ as $F\to\mb{P}(F^\perp)\cap \widehat{S}$. This shows in particular the continuity of the corresponding map $\Omega\to \widetilde{\Sigma}$.
\end{proof}

Since every maximal representation admits equivariant pleated surfaces (and hence equivariant acausal embeddings $\widetilde{\Sigma} \to \hyp^{2,n}$), Proposition \ref{pro:decomposition} and Theorem \ref{geometry pleated surfaces h2n} imply that every maximal representation is the holonomy of some fibered photon structure $E \to \Sigma$.

We now have all the elements to prove our first result on fibered photon structures, namely Proposition \ref{decomposition domain}:

\begin{proof}[Proof of Proposition \ref{decomposition domain}]
	Let $S_\lambda$ be the pleated set associated to some maximal lamination $\lambda$ of $\Sigma$. Recall that, by Proposition \ref{pro:existence pleated sets}, its lift $\widehat{S}_\lambda \subset \hyp^{2,n}$ is acausal and is equivariantly homeomorphic to $\widetilde{\Sigma}$ via a homeomorphism $\iota : \widetilde{\Sigma} \to \widehat{S}_\lambda \subset \hyp^{2,n}$. Since $\widehat{S}_\lambda$ can be decomposed as the disjoint union of spacelike ideal triangles and spacelike geodesics, by Proposition \ref{pro:decomposition} the acausal embedding $\iota$ determines a decomposition of the Guichard-Wienhard's domain of discontinuity of $\rho$ into lines and triangles of photons, associated to the leaves and the plaques of the maximal lamination $\lambda$, respectively.
\end{proof}

\begin{rmk}[Connected Components of Maximal Representations]\label{rmk:connected comp}
	Let $E \to \Sigma$ be a fibered photon structure with maximal holonomy $\rho$ associated to some $\rho$-equivariant spacelike embedding $\iota : \widetilde{\Sigma} \to \hyp^{2,n}$. We denote by $V_{\rho, \iota} \to \Sigma$ the vector bundle obtained through the following process: Let $\widetilde{V}_\iota \to \widetilde{\Sigma}$ be the bundle with total space
	\[
	\widetilde{V}_\iota : = \{ (x, v) \in \widetilde{\Sigma} \times \R^{2,n+1} \mid v \in \iota(x)^\perp \subset \R^{2,n+1} \}
	\]
	and bundle map given by the projection onto the first component $(x, v) \mapsto x$. The representation $\rho$ determines a natural action of $\Gamma$ on $\widetilde{V}_\iota$, given by $\gamma(x,v) : = (\gamma x, \rho(\gamma) v)$, for any $\gamma \in \Gamma$ and $(x,v) \in \widetilde{V}_\iota$. The bundle $V_{\rho, \iota} \to \Sigma$ is then obtained by considering $V_{\rho, \iota} : = \widetilde{V}_\iota/\rho(\Gamma)$, together with the projection induced by the universal covering map $\widetilde{\Sigma} \to \Sigma$.
	
	Collier, Tholozan, and Toulisse observed in \cite{CTT19}*{\S~2.5} that, by the work of Ramanathan \cite{R75} and Oliveira \cites{O11,O19}, for every $n > 2$ the connected component of the space of maximal representations $\mathcal{R}_{\mathrm{max}}$ containing $\rho$ is determined by the first and second Stiefel-Whitney classes $w_i(V_{\rho,\iota}) \in H^i(\Sigma;\Z/2\Z)$ of its associated vector bundle $V_{\rho, \iota} \to \Sigma$ and, vice versa, for every choice of classes $c_i \in H^i(\Sigma;\Z/2\Z)$, $i = 1, 2$, the set $\mathcal{R}_{\mathrm{max}}^{c_1,c_2}$ of maximal representations $\rho$ that verify $w_i(V_{\rho,\iota}) = c_i$ constitute a connected component of $\mathcal{R}_{\mathrm{max}}$.
	
	The analysis of the connected components of the representation variety for $n = 2$ is more subtle and requires a quite sophisticated analysis in work of Collier, Tholozan, and Toulisse \cite{CTT19}. For this reason, we prefer to focus in the remainder of the section on the case $n > 2$, whenever the topology of the fibered photon structures is discussed. However, we will emphasize in each statement whether the hypothesis $n > 2$ is in fact required.
\end{rmk}

\subsection{Gluing triangles of photons}\label{subsec:gluing triangles}
We start with a simple computation: Let $\Delta$ be an ideal spacelike triangle and let $\ell$ be a spacelike geodesic. Consider an isometry $\phi\in\SOtwon$ such that $\phi(E(\Delta))=E(\Delta)$ or $\phi(E(\ell))=E(\ell)$. As $\phi$ induces an orientation-preserving homeomorphism of ${\rm Pho}^{2,n}$, it extends to a homeomorphism of the closures of $E(\Delta)$ and $E(\ell)$. In particular, $\phi$ must permute the ideal vertices of $E(\Delta)$ or $E(\ell)$. 

We denote by ${\rm PStab}_{\SOtwon}(E(\Delta))$ and ${\rm PStab}_{\SOtwon}(E(\ell))$ the elements that stabilize $E(\Delta)$ and $E(\ell)$ without permuting their ideal vertices. Observe that if $\phi(E(a))=E(a)$ for some isotropic line $a\in\partial\mb{H}^{2,n}$, then $\phi(a)=a$ in $\partial\mb{H}^{2,n}$. As a consequence, we have the following 

\begin{lem}
	\label{lem:stabilizer}
	We have 
	\begin{itemize}
		\item{${\rm PStab}_{\SOtwon}(E(\Delta))={\rm PStab}_{\SOtwon}(\Delta)$.}
		\item{${\rm PStab}_{\SOtwon}(E(\ell))={\rm PStab}_{\SOtwon}(\ell)$.}
		\item{${\rm Stab}_{\SOtwon}(E(a))={\rm Stab}_{\SOtwon}(a)$.}
	\end{itemize}
\end{lem}

Both stabilizers of a spacelike triangle $\Delta$ and of a spacelike geodesic $\ell$ have two connected components. To see this, assume that $\Delta$ and $\ell$ lie in a common spacelike $2$-plane $H$ of $\hyp^{2,n}$, and let $W$ (resp. $L \subset W$) denote the subspace of $\R^{2,n+1}$ of signature $(2,1)$ (resp. $(1,1)$) that projects onto $H$ (resp. $\ell$). Let also $r_W \in \SOtwon$ be an isometric involution that restricts to $- \mathrm{id}$ on $W$, and to an orthogonal reflection on $W^\perp$. Then:

\begin{itemize}
	\item The stabilizer of $\Delta \subset \hyp^{2,n}$ decomposes as
	$${\rm PStab}_{\SOtwon}(\Delta)=(\{\mathrm{id}_W\} \times {\rm SO}(W^\perp))\sqcup r_{W}(\{\mathrm{id}_W\} \times {\rm SO}(W^\perp)).$$
	\item The stabilizer of $\ell \subset \hyp^{2,n}$ decomposes as
	$${\rm PStab}_{\SOtwon}(\ell)=({\rm SO}_0(L)\times{\rm SO}_0(L^\perp)) \sqcup r_{W}({\rm SO}_0(L)\times{\rm SO}_0(L^\perp)) .$$ 
\end{itemize}

\subsubsection{Building pants of photons}\label{subsubsec:building pants}

We now describe a process to glue pairs of triangles of photons along their boundary to form a so-called \emph{pants of photons}. To this purpose, we select arbitrarily an orientation on $H$, and consider two ideal triangles $\Delta$ and $\Delta'$ with cyclically ordered vertices $a, b, c \in \partial H$ and $c', b', a' \in \partial H$, respectively, endowed with the orientation induced by the one of $H$. We denote by $\ell_a$, $\ell_b$, $\ell_c$ (resp. $\ell_{a'}$, $\ell_{b'}$, $\ell_{c'}$) the edges of $\Delta$ (resp. $\Delta'$) opposite to the vertices $a, b, c$ (resp. $a', b', c'$), respectively. We also orient the edges of $\Delta$ and $\Delta'$ according to the boundary orientations of $\partial \Delta$ and $\partial \Delta'$.

Recall that, for any $u \in \{a,b,c\}$ (resp. $u' \in \{a',b',c'\}$), the orthogonal projection of the isotropic line $u \in \partial \hyp^{2,n}$ (resp. $u'$) onto the spacelike geodesic $\ell_u$ (resp. $\ell_{u'}$) provides a basepoint $x_u \in \ell_u$ (resp. $x_{u'} \in \ell_{u'}$) naturally associated to the ideal triangle $\Delta$ (resp. $\Delta'$).

Let $s_{a a'}^B, s_{b b'}^B, s_{c c'}^B \in \mathrm{SO}_0(W)$ be the unique isometries that send the edges $\ell_{a'}, \ell_{b'}, \ell_{c'}$ of $\Delta'$ and their basepoints $x_{a'}, x_{b'}, x_{c'}$ onto the edges $\ell_a, \ell_b, \ell_c$ of $\Delta$ and their basepoints $x_a, x_b, x_c$, respectively, so that the restriction $s_{u u'}^B : \ell_{u'} \to \ell_{u}$ is an orientation reversing isometry for any $u \in \{a,b,c\}$. We then denote by $s_{u u'} \in \SOtwon$ the isometry extending $s_{u u'}^B$ to $\R^{2,n+1}$ that restricts to the identity on $W^\perp$, for any $u \in \{a,b,c\}$. By construction, the ideal triangles $\Delta$ and $s_{u u'}(\Delta')$ are adjacent along $\ell_u = s_{u u'}(\ell_{u'})$ and they have hyperbolic shear $\sigma(\Delta, s_{u u'}(\Delta'))$ equal to $0$ (compare with Remark \ref{rmk:shear_is_shear}).

In order to combine the triangles of photons $E(\Delta), E(\Delta')$ to construct pants of photons, we now describe the admissible gluing maps between $\partial E(\Delta)$ and $\partial E(\Delta')$. For every edge $\ell_u$ of $\Delta$, we start by choosing an orientation preserving isometry $\psi_u^B:\ell_u\to\ell_u$, and select $\psi_u\in{\rm PStab}(E(\ell_u))$ covering $\psi_u^B$. We then set
$$\phi^B_u : = \psi^B_u s^B_{u u'}, \quad \phi_u : = \psi_u s_{u u'}$$
for any $u \in \{a,b,c\}$. Finally, we define
\[
S : = (\Delta \sqcup \Delta')/\phi^B_a \cup \phi^B_b \cup \phi^B_c ,
\]
where $x \in \ell_u \subset \Delta$ identifies to $x' \in \ell_{u'} \subset \Delta'$ if and only if $x = \phi^B_{u u'}(x')$, and similarly
\[
E := (E(\Delta) \sqcup E(\Delta'))/\phi_a \cup \phi_b \cup \phi_c ,
\]
where $F \in E(\ell_u) \subset E(\Delta)$ identifies to $F' \in E(\ell_{u'}) \subset E(\Delta')$ if and only if $F = \phi_{u u'}(F')$. Observe that the construction provides:
\begin{itemize}
	\item A (possibly incomplete) hyperbolic structure on $S$, inherited by the hyperbolic metric on $\Delta, \Delta'$ and the isometric gluing maps $\phi_a^B,  \phi_b^B, \phi_c^B$.
	\item A photon structure on $E$, with respect to which a $(\SOtwon,\mathrm{Pho}^{2,n})$-local chart around $F \in E(\ell_u) \subset E(\Delta)$ is obtained by juxtaposing $E(\Delta)$ and $\phi_u(E(\Delta'))$.
	\item A natural fiber bundle projection $E \to S$ with characteristic fiber $\mathrm{Pho}^{2,n-1}$.
\end{itemize}
It is simple to check that the construction determines a fibered photon structure $E \to S$ as in Definition \ref{def:fibered}.

Observe that, if we choose elements $\psi_\Delta \in {\rm PStab}(E(\Delta))$, $\psi_{\Delta'} \in {\rm PStab}(E(\Delta'))$ and we change the gluing maps $\phi_u$ with $\psi_\Delta\phi_u\psi_{\Delta'}^{-1}$ for all $u\in\{a,b,c\}$, then the resulting fibered photon structure $E' \to S$ is isomorphic to $E \to S$, namely there exists a fiber-preserving $(\SOtwon,\mathrm{Pho}^{2,n})$-isomorphism $E \to E'$ covering the identity map $\mathrm{id}_S$ induced by $\psi_\Delta\cup\psi_{\Delta'}$. Thus, the space of parameters for the gluing maps is
\[
{\rm PStab}_{\SOtwon}(\ell_a)\times{\rm PStab}_{\SOtwon}(\ell_b)\times{\rm PStab}_{\SOtwon}(\ell_c)
\]
modulo the diagonal action by left and right multiplications of
\[
{\rm PStab}(\Delta)\times{\rm PStab}(\Delta').
\]

\subsubsection{Boundary completions}

If the hyperbolic shears between the ideal triangles in $S$ satisfy
$$\sigma(\Delta,\phi^B_u (\phi^B_v)^{-1}(\Delta)) = \sigma(\Delta, \phi^B_{u}(\Delta')) + \sigma(\phi^B_u(\Delta'), \phi^B_u (\phi^B_v)^{-1}(\Delta)) \neq 0$$
for every pair of distinct vertices $u, v \in \{a,b,c\}$, the hyperbolic surface is incomplete and its metric completion $S'$ is a hyperbolic pair of pants with three totally geodesic boundary components. In fact, the length of the boundary component $\gamma_u$ adjacent to the vertex $u$ of $\Delta$ is given by
\[
L(\gamma_u) = \abs{\sigma(\Delta,\phi^B_w (\phi^B_v)^{-1}(\Delta))}
\]
where $u, v, w$ are cyclically ordered vertices of $\Delta$ (compare with Lemma \ref{lem:shear and length}). We now wish to study the completion $E'$ of the photon manifold $E$. In particular, we give conditions on the isometries $\psi_u \in \mathrm{PStab}(\ell_u)$ under which the fibered photon structure $E$ admits a completion $E'$, which is a fibered photon structure with photon boundary that naturally fibers over $S'$. 

First, let us compute the holonomies $\rho_a,\rho_b,\rho_c$ of the boundary components $\gamma_a, \gamma_b, \gamma_c$ adjacent to the vertices $a,b,c$ of $\Delta \subset S$ and oriented as boundary of $S$: A direct computation from the definition of the photon structure on $E$ shows that
\[
\rho_u = \phi_w\phi_v^{-1} = \psi_w s_{w w'} s_{v v'}^{-1} \psi_v^{-1}
\]
where the $u,v,w$ are in cyclic order. It is not difficult to see that the composition $s_{w w'} s_{v v'}^{-1}$ coincides with $\nu^2_{v w} \in \SOtwon$, where $\nu_{v w}$ is the unique unipotent isometry of $\R^{2,n+1}$ that restricts to the identity on $W^\perp$, and that acts on $\partial H \subset \partial\hyp^{2,n}$ by fixing $u$ and sending $w$ to $v$.

\begin{dfn}[Loxodromic Isometry]
	An isometry $\phi\in\SOtwon$ is \emph{loxodromic} if: It admits an invariant spacelike line $\ell=[a,b]$, and the isotropic lines $a$ and $b$ are equal to the generalized eigenspaces of the eigenvalues of $\phi$ with largest and smallest absolute value, respectively. In such case, we say that $a$ and $b$ are the \emph{attracting} and \emph{repelling fixed points} of $\phi$, respectively.
\end{dfn}

Any loxodromic element $\phi$ with attracting and repelling fixed points $a,b$ has north-south dynamics both on $\hyp^{2,n}$ and ${\rm Pho}^{2,n}$ in the following sense:

\begin{lem}
	Let $\phi \in \SOtwon$ be a loxodromic isometry. Then:
	\begin{itemize}
		\item{$\phi^m\to a$ uniformly on all compact subsets of $\hyp^{2,n}-\mathbb{P}(b^\perp)$.} 
		\item{$\phi^{-m}\to b$ uniformly on all compact subsets of $\hyp^{2,n}-\mathbb{P}(a^\perp)$.}
		\item{$\phi^m\to E(a)$ uniformly on all compact subsets of ${\rm Pho}^{2,n}-E(b)$.} 
		\item{$\phi^{-m}\to E(b)$ uniformly on all compact subsets of ${\rm Pho}^{2,n}-E(a)$.}
	\end{itemize}
\end{lem}

\begin{proof}
	The first and second properties follow from the fact that $\phi$ is a bi-proximal element of $\mathrm{PGL}(d,\R)$ (see e.g. \cite{C21}*{\S~4.17}). 
	
	For the third assertion, let $F$ be a photon outside $E(b)$. First, observe that $\mathrm{Pho}^{2,n}$ is a closed subset of the Grassmannian of $2$-planes of $\R^{2,n+1}$, and hence compact. In particular, for any divergent sequence $(m_k)_k \subset \N$, we can find a subsequence of $(\phi^{m_k}(F))_k$ that converges to some $F' \in \mathrm{Pho}^{2,n}$. 
	
	We divide the analysis in three cases, depending on whether (1) $F \subset b^\perp$, (2) $\dim F \cap b^\perp = 1$, or (3) $F \cap b^\perp = \{0\}$. In fact, the first case occurs only if $F \in E(b)$, and hence can be excluded. To see this, notice that the hyperplane $b^\perp \subset \R^{2,n+1}$ has signature $(i_+,i_-,i_0) = (1,n,1)$. Since no subspace of signature $(1,n)$ can contain a photon, the only possibility for $F$ to be contained in $b^\perp$ is that $b \in F$, and hence $F \in E(b)$. On the other hand, if we are in case (2) or (3), then $F' \in E(a)$, since we can find a non-zero vector $v$ of $F$ such that to $\phi^n([v]) \to a$, by the first assertion. The uniform convergence of $\phi^n \to E(a)$ can then be deduced from the control of the eigenvalues of $\phi$.
\end{proof}

In particular, $\phi$ acts properly discontinuously and freely on ${\rm Pho}^{2,n}-(E(b)\cup E(a))$.

\begin{dfn}[Fibered Photon Structure with Geodesic Boundary]
	A {\em half space} of ${\rm Pho}^{2,n}$ is a subset $E(W) \subset \mathrm{Pho}^{2,n}$ of the form
	\[
	E(W)=\{V\in{\rm Pho}^{2,n}\left|V\perp x\text{ \rm for some }x\in W\right.\} ,
	\]
	where $W\subset H$ is a half plane in a spacelike plane $H\subset\mb{H}^{2,n}$. A {\em photon structure with totally geodesic boundary} on a manifold with boundary $E'$ is a maximal atlas of charts with values in a half space of ${\rm Pho}^{2,n}$, whose change of charts are restrictions of transformations in $\SOtwon$.
	
	Let now $\Sigma'$ be an orientable compact surface with boundary and let $\pi : E'\to \Sigma'$ be a fiber bundle with characteristic fiber $\mathrm{Pho}^{2,n-1}$. We denote by $\tilde{\pi} : \widetilde{E}'\to\widetilde{\Sigma}'$ the pull-back bundle of $\pi$ via the universal covering map $\widetilde{\Sigma}'\to \Sigma'$. A photon structure with totally geodesic boundary on $E'$ with holonomy $\rho \circ \pi_* : \pi_1(\Sigma') \to \SOtwon$ and developing map $\delta:\widetilde{E}'\to{\rm Pho}^{2,n}$ is called \emph{fibered} if for any $x \in \widetilde{\Sigma}'$ there exists some $\iota(x)\in\mb{H}^{2,n}$ such that $\delta(\tilde{\pi}^{-1}(x))={\rm Pho}(\iota(x)^\perp)$.
\end{dfn}

We have the following:

\begin{lem}
	\label{lem:completion}
	Let $a,b,c$ be the vertices of $\Delta \subset S$. Suppose that the holonomies $\rho_u:= \phi_w\phi_v^{-1} \in{\rm Stab}(u)$, for $\{u,v,w\}=\{a,b,c\}$ and $u,v,w$ cyclically ordered, are all loxodromic, and denote their invariant lines by $\ell(\rho_u)$. Then there is a completion $E\subset E'$ which is a fibered photon structure with totally geodesic boundary over the metric completion $S'$, whose boundary component adjacent to $v$ is equal to $E(\ell(\rho_u))/\rho_u$.
\end{lem}

\begin{proof}
	Let $\ell(\rho_u)$ be the invariant spacelike line of $\rho_u$. Notice that, as $\rho_u$ is loxodromic and leaves invariant $E(u)$, there exists a $t \in \partial \hyp^{2,n}$ such that $\ell(\rho_u)=[u,t]$. Let $E(\ell(\rho_u))$ be the corresponding $\rho_u$-invariant line of photons. Note that the action $\rho_u\curvearrowright E(\ell(\rho_u))$ is properly discontinuous and free. We denote by
	\[
	E(\ell(\rho_u))/\rho_u\to\ell(\rho_u)/\rho_u
	\]
	the corresponding quotient bundle. Our aim is to describe local charts for a fibered photon structure with geodesic boundary on the space
	\[
	E\cup(E(\ell(\rho_u))/\rho_u)\to S\cup(\ell(\rho_u)/\rho_u).
	\]
	
	Let $\widetilde{E}\to\widetilde{S}$ be the pull-back bundle to the universal covering $\widetilde{S}\to S$, and choose a lift of the vertex $u$ (which we continue to denote with abuse by $u$). Let $\widetilde{S}_u$ be the fan of triangles with an ideal vertex in $u$, and let $\widetilde{E}_u$ be the fan of all triangles of photons of $\widetilde{E}$ with ideal vertex $E(u)$. 
	
	Consider the restriction to $\widetilde{S}_u$ of the map $\iota:\widetilde{S}\to\mb{H}^{2,n}$ associated to the fibered photon structure $E \to S$. We have $\widetilde{S}_u=\bigcup_{j\in\mb{Z}}{\Delta_j}$, where $(\Delta_j)_j=(\Delta(u,v_{j-1},v_j))_j$ is the collection of spacelike ideal triangles in $\hyp^{2,n}$ that share the ideal vertex $u$. Notice that the cyclic order of the vertices of $\widetilde{S}_u$ on $\partial\widetilde{S}$ is
	\[
	u<\cdots<v_{j-1}<v_j<v_{j+1}<\cdots<u.
	\]
	
	We start with the following observation: 
	
	\begin{claim}{1}
		For every $j\in\mb{Z}$, the subset $\iota(\Delta_j\cup\Delta_{j+1}) \subset \hyp^{2,n}$ is acausal.
	\end{claim}
	
	\begin{proof}[Proof of the claim]
		Notice that, from the construction outlined in Section~\ref{subsubsec:building pants}, it is not restrictive to assume that $\Delta_0, \Delta_1 \subset \widetilde{S}_u$ coincide with the spacelike triangles $$\Delta = \Delta(u,v,w), \quad\phi_w^B(\Delta') = \Delta(\phi^B_{w}(w'),v, u) \subset H,$$
		respectively, and that $\iota(\Delta_0) \subset H \subset \hyp^{2,n}$. Moreover, if $\gamma_u \in \pi_1(S)$ denotes the deck transformation that preserves the ideal vertex $u \in \partial\widetilde{S}_u$ (oriented according to the boundary orientation of $\partial S'$), then the union $\Delta_j\cup\Delta_{j+1}$ is either equal to $\gamma_u^h(\Delta_0\cup\Delta_1)$ if $j=2h \in \Z$, or $\gamma_u^{h+1}(\Delta_{-1}\cup\Delta_0)$ if $j=2h+1$, so it is enough to consider $j=0,-1$. As the two cases are completely analogous to each other, we explain in detail only the case $j=0$. 
		
		Since the set $\Delta_0 \cup \Delta_1 \subset \widetilde{S}_u$ is contractible, the restriction of $\iota$ to $\Delta_0 \cup \Delta_1$ admits a lift $\hat{\iota} : \Delta_0 \cup \Delta_1 \to \widehat{\hyp}^{2,n}$. We will show that $\hat{\iota}(\Delta_0 \cup \Delta_1)$ is an acausal subset of $\widehat{\hyp}^{2,n}$. Let $\hat{u}, \hat{v}, \hat{w} \in \partial \widehat{\hyp}^{2,n}$ be the vertices of $\hat{\iota}(\Delta_0)$, and let $\hat{u}', \hat{v}', \hat{w}'$ be the vertices of the lift of $\Delta' \subset H$ lying on the same spacelike plane as $\hat{\iota}(\Delta_0)$ (recall from Section \ref{subsubsec:building pants} that $\Delta$ and $\Delta'$ lie in a common spacelike plane $H$ of $\hyp^{2,n}$). 
		
		The vertex $\hat{z} \in \partial \widehat{\hyp}^{2,n}$ of $\hat{\iota}(\Delta_1)$ different from $\hat{u}, \hat{v}$ is projectively equivalent to $\phi_w(\hat{w}')$, and it coincides with the unique lift of $\phi_w(w') \in \partial\hyp^{2,n}$ for which both $\scal{\hat{z}}{\hat{u}}$, $\scal{\hat{z}}{\hat{v}}$ are negative. In order to express $\hat{z}$ in terms of $\phi_w$ and $\hat{w}'$, we need to distinguish two cases, depending on the connected component of $\mathrm{PStab}(\ell_w)$ containing $\psi_w = \phi_w s_{w w'}^{-1}$:
		\begin{enumerate}
			\item If $\psi_w$ belongs to the connected component of the identity in $\mathrm{PStab}(\ell_w)$, then $\hat{z} = \phi_w(\hat{w}') \in \partial\widehat{\hyp}^{2,n}$.
			\item If $\psi_w$ does not belong to the connected component of the identity, then $\hat{z} = - \phi_w(\hat{w}')$.
		\end{enumerate}
		
		Within this setting, the set $\hat{\iota}(\Delta_0 \cup \Delta_1)$
		is acausal if and only if any subtriple of the set of isotropic rays $\{\hat{u}, \hat{v}, \hat{w},\hat{z}\}$ generates a subspace of $\R^{2,n+1}$ of signature $(2,1)$. In fact, by the choices we made, it is enough to prove that $\scal{\hat{w}}{\hat{z}} < 0$ (compare with Lemma \ref{lem:spacelike lines}). From here, the desired statement can be reduced to an explicit computation in $\R^{2,n+1}$, which we briefly summarize. 
		
		Up to the action of $\SOtwon$, we can assume to be in the following setting:
		\begin{align*}
			\hat{u} & = e_1 + e_3, & \hat{v} & = -e_1 + e_3,\\
			\hat{w} & = e_2 + e_3, & s_{w w'}(\hat{w}') & = - e_2 + e_3,
		\end{align*}
		where $(e_i)_i$ is the standard basis of $\R^{2,n+1}$. Since the isometry $\psi_u$ preserves the orthogonal decomposition $L \oplus L^\perp$, with $L = \mathrm{Span}\{\hat{u}, \hat{v}\} = \mathrm{Span}\{e_1,e_3\}$, we have
		\begin{equation}\label{eq:cazzo merda}
			\scal{\hat{w}}{\phi_w(\hat{w}')} = \scal{e_2 + e_3}{\psi_w(-e_2 + e_3)} = - \scal{e_2}{\psi_w(e_2)} + \scal{e_3}{\psi_w(e_3)}
		\end{equation}
		
		Now one sees that, if $\psi_w$ satisfies (1), then $\psi_w(e_2)$ lies in the same connected component of $L^\perp \cap \widehat{\hyp}^{2,n}$ (which is isometric to two copies of $-\hyp^n$, i.e. the hyperbolic $n$-space with the metric $- g_{{\hyp}^n}$) as $e_2$, and $\psi_w(e_3)$ lies in the same connected component of $L \cap \widehat{\hyp}^{2,n}$ as $e_3$ (which is isometric to two copies of $+\hyp^1$, i.e. the hyperbolic $1$-space with its Riemannian metric). From here we deduce that $\scal{e_2}{\psi_w(e_2)} \geq 1$ and $\scal{e_3}{\psi_w(e_3)} \leq -1$ (see Lemma \ref{lem:spacelike segments linear}), which implies by relation \ref{eq:cazzo merda} that $\scal{\hat{w}}{\hat{z}} = \scal{\hat{w}}{\phi_w(\hat{w}')} < 0$.
		
		On the other hand, if $\psi_w$ satisfies (2), then $\psi_w$ exchanges the components of both $L^\perp \cap \widehat{\hyp}^{2,n}$ and $L \cap \widehat{\hyp}^{2,n}$ (see beginning of Section \ref{subsec:gluing triangles}). By the same argument as above, we deduce that $\scal{e_2}{\psi_w(e_2)} \leq -1$ and $\scal{e_3}{\psi_w(e_3)} \geq 1$. Since $\scal{\hat{w}}{\hat{z}} = - \scal{\hat{w}}{\phi_w(\hat{w}')}$, the assertion follows again from relation \ref{eq:cazzo merda}.
	\end{proof}
	
	We now promote the claim to the following:
	
	\begin{claim}{2} 
		The map $\iota:\widetilde{S}_v\to\mb{H}^{2,n}$ is an acausal embedding.
	\end{claim}
	
	\begin{proof}[Proof of the claim]
		The proof is a simpler version of the one of Proposition \ref{pro:existence pleated sets}.
		
		We lift $\iota:\widetilde{S}\to\mb{H}^{2,n}$ to the two fold cover $\widehat{\mb{H}}^{2,n}\to\mb{H}^{2,n}$ and work in a Poincaré model $\mb{D}^2\times\mb{S}^n$ of $\widehat{\mb{H}}^{2,n}$. Observe that $\widetilde{S}_v=\bigcup_{j\in\mb{Z}}{\Delta_j}$ where $\Delta_j=\Delta(v,u_{j-1},u_j)$. Notice that the cyclic order of the vertices of $\widetilde{S}_v$ on $\partial\widetilde{S}$ is
		\[
		v<\cdots<u_{j-1}<u_j<u_{j+1}<\cdots<v.
		\]
		
		Let $\pi:\widehat{\mb{H}}^{2,n}\cup\partial\widehat{\mb{H}}^{2,n}\to\mb{D}^2\cup\partial\mb{D}^2$ be the natural projection. Consider two consecutive triangles $\Delta_j=\Delta(v,u_{j-1},u_j),\Delta_{j+1}=\Delta(v,u_j,u_{j+1})$ intersecting along the geodesic $[v,u_j]$. By the previous claim, $\iota(\Delta_j\cup\Delta_{j+1})$ is acausal so the projections $\pi(v),\pi(u_{j-1}),\pi(u_j),\pi(u_{j+1})$ of the vertices $v,u_{j-1},u_j,u_{j+1}$ to $\partial\mb{D}^2$ appear in this exact cyclic order on $\partial\mb{D}^2$. We deduce that the projections of the vertices $\pi(u_j)$ appear in the same order 
		\[
		\pi(v)<\cdots<\pi(u_{j-1})<\pi(u_j)<\pi(u_{j+1})<\cdots<\pi(v)
		\]
		as they appear on $\partial\widetilde{S}$. As a consequence, the restriction of $\iota$ to the union $\lambda_v=\bigcup_{j\in\mb{Z}}{\partial\Delta_j}$ of the sides of the triangles $\Delta_j$ is an acausal embedding: For every $\ell,\ell'\subset\lambda_v$ we have that the endpoints of $\ell,\ell'$ are in disjoint position. 
		
		We immediately deduce that $\iota({\rm int}(\Delta_j))\cap \iota({\rm int}(\Delta_i))=\emptyset$ for all $j\neq i$, which says that $\iota$ is an embedding: We already know that this is the case when $|i-j|=1$. Assume $|i-j|>1$. Note that $\pi \iota$ is an embedding on both $\Delta_j,\Delta_i$ and the images coincide with the topological disks bounded by the closures in $\mb{D}^2\cup\partial\mb{D}^2$ of $\pi \iota(\partial\Delta_j),\pi \iota(\partial\Delta_i)$. Those curves are disjoint and not nested. The conclusion follows. 
		
		The argument provided above shows in fact that the restriction $\pi \iota : \widetilde{S}_v \to \mathbb{D}^2$ is injective, and $\pi \iota(\widetilde{S}_v)$ is an open subset of $\mathbb{D}^2$. At this point, checking that $\iota(\widetilde{S}_v)$ is acausal is simple. First observe that there is no timelike geodesic joining distinct points of $\widetilde{S}_v$: If this was not the case, then we could find a Poincaré model of $\widehat{\hyp}^{2,n}$ with respect to which the projection $\pi \iota$ is not injective, contradicting what we just observed. We deduce in particular that $\iota(\widetilde{S}_v)$ is an achronal subset of $\widehat{\hyp}^{2,n}$, and hence it can be represented as the graph of some $1$-Lipschitz function $\pi \iota(\widetilde{S}_v) \to \mathbb{S}^n$.

		On the other hand, suppose that $\alpha$ is a lightlike geodesic connecting two points of $\iota(\widetilde{S}_v)$. The path $\alpha$ cannot join two points on the same leaf, since $\lambda_v$ is an acausal set. Since $\iota(\widetilde{S}_v)$ is the graph of a $1$-Lipschitz function on the open set $\pi \iota(\widetilde{S}_v)$, Lemma \ref{lem:projection} implies that there exists a subsegment of $\alpha$ that is entirely contained in one of the triangles $\iota(\Delta_j)$. This clearly contradicts the fact that $\iota(\Delta_j)$ is a spacelike triangle. 
	\end{proof}
	
	A consequence of the claim is that the restriction of the developing map $\delta:\widetilde{E}\to{\rm Pho}^{2,n}$ to $\widetilde{E}_v$ is an embedding. 
	
	Notice that the images $\delta(\widetilde{E}_v)$ are $\rho_v$-invariant and contained in ${\rm Pho}^{2,n}-E(\ell(\rho_v))$: In fact, if $E(\ell(\rho_v))$ intersects the image of one of the triangles of photons $E(\Delta')$ in $\widetilde{E}_v$ under the developing map, then we must have $$\rho_v\delta(E(\Delta'))\cap\delta(E(\Delta'))\neq\emptyset ,$$
	since both $\delta(E(\Delta'))$ and $E(\ell(\rho_v))$ have a vertex in $E(v)$. But $\rho_v$ moves every triangle $\delta(E(\Delta'))$ off itself.
	
	Furthermore, by the loxodromic assumption on $\rho_v$, the $\rho_v$-orbit of every triangle in $\delta(\widetilde{E}_v)$ accumulates to $E(\ell(\rho_v))$ either in the forward or backward direction and to $E(v)$ or $E(t)$ in the opposite one. We deduce that the union $\delta(\widetilde{E}_v)\cup E(\ell(\rho_v))\subset{\rm Pho}^{2,n}$ is a $\rho_v$-invariant submanifold with boundary. This provides local charts for $E\cup (E(\ell(\rho_v))/\rho_v)$ at points in $E(\ell(\rho_v))/\rho_v$.
\end{proof}

\subsubsection{Computing Stiefel-Whitney classes}\label{subsubsec:stiefel}

We now provide a description of the first Stiefel-Whitney class of the vector bundle $V' = V_{\iota, \rho}' \to S'$, associated to the completion $E' \to S'$ as described in Remark \ref{rmk:connected comp}, in terms of the holonomy of the boundary $\partial S'$. If $\iota : \widetilde{S} \to \hyp^{2,n}$ denotes the equivariant immersion associated to $E' \to S'$, then $V'\to S'$ is obtained as the quotient by $\rho(\Gamma)$ of the vector bundle $\widetilde{V}' \to \widetilde{S}'$, whose fiber over $x \in \widetilde{S}'$ is equal to $\iota(x)^\perp \subset \R^{2,n+1}$. Notice that, since the vector bundle $V' \to S'$ has dimension $2 + n$, for every $n \geq 1$ its isomorphism classes is uniquely determined by the Stiefel-Whitney class $w_1(V') \in H^1(S';\Z/2\Z)$. 

Recall that the cohomology groups satisfy
$$H^i(S';\Z/2\Z) = \mathrm{Hom}(H_i(S';\Z/2\Z),\Z/2\Z).$$
In particular, in order to determine the class $w_1(V')$ it is enough to describe its evaluation on each $[\alpha]\in H_1(S';\mb{Z}/2\mb{Z})$. Let $\alpha:S^1\to S'$ be a loop representing $[\alpha]$, and consider the pull-back bundle $\alpha^*V'\to S^1$. Then $w_1(V')[\alpha]\in\mb{Z}/2\mb{Z}$ is equal to $0$ if $\alpha^*V$ is orientable, and is equal to $1$ otherwise.

Notice that, if $\gamma_a,\gamma_b,\gamma_c$ are the boundary curves of $S'$ corresponding to the vertices $a,b,c$, then the homomorphism $w_1(V')$ is uniquely determined by any two of $w_1(V')[\gamma_a]$, $w_1(V')[\gamma_b]$, $w_1(V')[\gamma_c]$, since any pair of distinct classes among $[\gamma_a]$, $[\gamma_b]$, $[\gamma_c]$ freely generates $H_1(S';\mb{Z}/2\mb{Z})$.

Let $\rho_a=\phi_c\phi_b^{-1}$ be the holonomy around $\gamma_a$. Recall that by construction $\rho_a$ is a loxodromic element. Loxodromic elements in $\SOtwon$ are divided into two connected components $\mc{L}^+,\mc{L}^-$: If $\phi\in\SOtwon$ is a loxodromic isometry with invariant spacelike line $\ell$ given by the projectivized of a subspace $L \subset \R^{2,n+1}$, then $\phi$ can either preserve or exchange the connected components of $L\cap\widehat{\mb{H}}^{2,n}$. This feature distinguishes the two connected components of ${\rm SO}(L)$ and the two connected components of loxodromic elements in $\SOtwon$. In the first case, when $\phi\in{\rm SO}_0(L)$, the bundle $V(\ell)/\phi$ is the trivial bundle over $\ell/\phi$. In the second case, when $\phi\not\in{\rm SO}_0(L)$, the bundle $V(\ell)/\phi\to\ell/\phi$ is the unique non-orientable bundle over $\ell/\phi \cong S^1$ of dimension $n+2$. Equivalently, the connected component that contains $\phi$ can be distinguished by the sign of its leading eigenvalue $l_1(\phi)$, i.e. $\phi \in \mathcal{L}^+$ if $ l_1(\phi)>0$, and $\phi \in \mathcal{L}^-$ otherwise.

In light of this fact, let us compute $w_1(V')[\gamma_a]$: We already know that one of the eigenvectors of $\rho_a$ is equal to the isotropic ray $a \in \partial \widehat{\hyp}^{2,n}$, so we can read off the connected component from $\rho_a(a)$. To conclude, we deduce that
\[
w_1(V')[\gamma_a]= \begin{cases}
	0 & \text{if $\rho_a(a) = a \in \partial \widehat{\hyp}^{2,n}$,} \\
	1 & \text{if $\rho_a(a) = - a \in \partial \widehat{\hyp}^{2,n}$.}
\end{cases}
\]

\subsubsection{Classifying pants of photons}

We can now summarize our analysis in a concise statement.  

\begin{mthm}{\ref{pants of photons}}
	\label{thm:classification pair of pants}
	Let $\Delta =\Delta(a,b,c), \Delta' = \Delta(c',b',a') \subset\mb{H}^{2,n}$ be two ideal spacelike triangles lying on a common spacelike plane $H=\mb{P}(W)\cap\mb{H}^{2,n}$ with $\partial\Delta=\ell_a\cup\ell_b\cup\ell_c, \partial\Delta'=\ell_{a'}\cup\ell_{b'}\cup\ell_{c'},$ where $\ell_u, \ell_{u'}$ are the sides opposite to the vertices $u\in\{a,b,c\}, u' \in \{a',b',c'\}$, respectively. For every $u\in\{a,b,c\}$, let $s_{uu'}\in\SOtwon$ be the unique isometry of $\R^{2,n+1}$ that restricts to the identity on $W^\perp$ and to the element of ${\rm SO}_0(W)$ that maps $\ell_{u'}$ to $\ell_u$ so that the shear between the adjacent triangles $\Delta$ and $s_{uu'}\Delta'$ is zero. For every equivalence class of triples
	\[
	\phi\in\left\{
	[\phi_a,\phi_b,\phi_c]\in\left(\prod_{u\in\{a,b,c\}}{{\rm PStab}(\ell_u)s_{uu'}}\right)\left/{\rm PStab}(\Delta)\times{\rm PStab}(\Delta')\right.
	\right\}
	\] 
	there is a fibered photon structure
	\[
	E=E(\Delta)\cup_\phi E(\Delta')
	\]
	fibering over a (possibly incomplete) hyperbolic pair of pants
	\[
	S=\Delta\cup_\phi\Delta'
	\]
	such that the holonomy around the peripheral simple closed curve $\gamma_u$ surrounding the puncture of $S$ corresponding to the vertex $u\in\{a,b,c\}$ is given by
	\[
	\rho_u=\phi_w\phi_v^{-1}.
	\]
	
	If for every $u\in\{a,b,c\}$ the holonomy $\rho_u$ is loxodromic, then $S,E$ are respectively the interior of a hyperbolic pair of pants $S'$ with totally geodesic boundary and the interior of a fibered photon structure $E'$ with totally geodesic boundary fibering over $S'$. The fibration $E'\to S'$ extends $E\to S$. For every $n \geq 1$, the topology of $E'$ is determined by the first Stiefel-Whitney class $w_1(V_{E})\in H^1(S,\mb{Z}/2\mb{Z})$ of the underlying vector bundle $V_E \to S$. The class $w_1(V_E)$ can be computed as follows: Let $\gamma_a,\gamma_b,\gamma_c\subset S$ be the peripheral curves corresponding to the vertices $a,b,c$ respectively. Then
	\[
	w_1(V_E)[\gamma_u]=\left\{
	\begin{array}{l l}
		0 &\text{if $\rho_u\in\mc{L}^+$},\\
		1 &\text{if $\rho_u\in\mc{L}^-$},\\
	\end{array}
	\right. 
	\]
	where $\mathcal{L}^+$ (resp. $\mathcal{L}^-$) is the set of loxodromic elements of $\SOtwon$ whose eigenvalue with largest absolute value is positive (resp. negative).
\end{mthm}

\subsection{Gluing pants of photons}\label{subsec:gluing pants}

Let $E'_j\to S'_j$ be $2g-2$ oriented pants of photons with totally geodesic boundary. We label by $\gamma_{a_j},\gamma_{b_j},\gamma_{c_j}$ the boundary components of $S'_j$ (with the orientation induced by $S_j'$ on its boundary $\partial S_j'$) and by $E(\gamma_{a_j}),E(\gamma_{b_j}),E(\gamma_{c_j})$ the corresponding boundary components of $E'_j$.

For every oriented pants of photons $E'_j\to S'_j$, we fix a developing map $\delta_j:\widetilde{E}_j'\to{\rm Pho}^{2,n}$ and its corresponding holonomy. We denote by $\rho_{a_j},\rho_{b_j},\rho_{c_j}$ the holonomies of the boundary curves $\gamma_{a_j},\gamma_{b_j},\gamma_{c_j}$, and by $\ell(\rho_{a_j}),\ell(\rho_{b_j}),\ell(\rho_{c_j})$ their invariant spacelike geodesics.

Let $\zeta$ be an orientation reversing pairing of the boundary components of the fibered pairs of pants $E'_j\to S'_j$. In order to perform a geometric gluing of the blocks $E'_j\to S'_j$ that implements the pairing $\zeta$, some compatibility conditions must be fulfilled: Every time that we have an identification of a boundary component of $S'_i$, labeled by $u_i$, with some boundary component of $S'_j$, labeled by $v_j$, the holonomies $\rho_{u_i},\rho_{v_j}$ must be conjugate in $\SOtwon$. 

If this happens, then we choose, for every pair of boundary components $E(\gamma_{u_i})$ and $E(\gamma_{v_j})$ that are paired by $\zeta$, an arbitrary initial orientation reversing identification $c_{v_j  u_i}:E(\gamma_{u_i})\to E(\gamma_{v_j})$ as bundles over $\gamma_{u_i},\gamma_{v_j}$, which is induced by an element $c_{v_j u_i}\in\SOtwon$ such that $\rho_{v_j}=c_{v_j u_i}\rho_{u_i}c_{v_j u_i}^{-1}$. 

All other admissible gluing maps will be of the form $\zeta_{v_j u_i}:=\eta_{v_j}c_{v_j u_i}\eta_{u_i}$ where $\eta_{u_i}\in{\rm PStab}(\ell(\rho_{u_i}))$ and $\eta_{v_j}\in{\rm PStab}(\ell(\rho_{v_j}))$ are isometries that commute with $\rho_{u_i}$ and $\rho_{v_j}$, respectively. The restriction $\zeta^B_{v_j u_i}:\ell(\rho_{u_i})\to\ell(\rho_{v_j})$ induces an orientation reversing isometry between the boundary components $\gamma_{u_i},\gamma_{v_j}$.

Thus we can form: A hyperbolic structure over a closed surface 
\[
S:=\bigcup_{k\le 2g-2}{S'_k}\left/\bigcup_{(u_i, v_j)\in \zeta}{\zeta^B_{v_j u_i}}\right.,
\]
a photon structure over a closed manifold
\[
E:=\bigcup_{k\le 2g-2}{E'_k}\left/\bigcup_{(u_i, v_j)\in \zeta}{\zeta_{v_j u_i}}\right.,
\]
where a $(\SOtwon,{\rm Pho}^{2,n})$-local chart around a point $x\in E(\gamma_{v_j})=E(\gamma_{u_i})$ is obtained by juxtaposing $\delta_j(\widetilde{E}_j)$ and $\zeta_{v_j u_i}\delta_i(\widetilde{E}_i)$. As before, since gluing and fibering are compatible, we also get a fiber bundle projection $E\to S$ with geometric fibers.

The following, which is analogous to \cite{CTT19}*{Proposition~3.13}, shows that the holonomy $\rho:\pi_1(S)\to\SOtwon$ of $E\to S$ is maximal.

\begin{lem}
	\label{lem:spacelike implies maximal}
	Let $\rho:\Gamma\to\SOtwon$ be a representation. Suppose that there exists a $\rho$-equivariant \emph{locally acausal} embedding $\iota:\widetilde{\Sigma}\to\mb{H}^{2,n}$, meaning that every point $x\in\widetilde{\Sigma}$ has a neighborhood $U$ such that $\iota|_U:U\to\mb{H}^{2,n}$ is an embedding with acausal image. Then $\rho$ is maximal.
\end{lem}

\begin{proof}
	We can lift $\iota$ to a locally acausal embedding $\hat{\iota}:\widetilde{\Sigma}\to\widehat{\mb{H}}^{2,n}$. By assumption, every point $x\in\widetilde{\Sigma}$ has a neighborhood $U$ such that $\hat{\iota}|_U$ is an embedding with acausal image. In particular, by Lemma \ref{lem:acausal graph}, we can represent $\hat{\iota}(U)$ in a Poincaré model $\Psi :\mb{D}^2\times\mb{S}^n \to \widehat{\hyp}^{2,n}$ as the graph of a strictly $1$-Lipschitz function $g:\pi \hat{\iota}(U)\subset\mb{D}^2\to\mb{S}^n$, where $\pi : \widehat{\hyp}^{2,n} \to \mathbb{D}^2$ denotes the projection associated to $\Psi$.
	
	As $\Gamma$ acts cocompactly on $\widetilde{\Sigma}$ and $\hat{\iota}$ is $\rho$-equivariant, we can choose the neighborhoods $U$ in a uniform way. In order to do so we proceed as follows: We endow $\widetilde{\Sigma}$ with a $\Gamma$-invariant hyperbolic metric obtained by pulling back an arbitrary hyperbolic metric on $\Sigma$. We cover $\Sigma$ with the projections of the neighborhoods $U$ and find a Lebesgue number $r>0$ for the open covering which is smaller than the injectivity radius of $\Sigma$. With these choices, the restriction of $\hat{\iota}$ to $B(x,r)$ is an embedding with acausal image for every $x\in\widetilde{\Sigma}$.
	
	Observe that we can choose a continuous $\rho(\Gamma)$-invariant family of orthogonal splittings $P_x\oplus N_x$ of $\mb{R}^{2,n+1}$ where $P_x$ is an oriented (future oriented if $n=1$) spacelike 2-plane: Let ${\rm Gr}_{(2,0)}^{\mathrm{or}}(\mb{R}^{2,n})$ be the Grassmannian of oriented spacelike $2$-planes in $\mb{R}^{2,n}$. Topologically, we have an identification
	\[
	{\rm Gr}^\mathrm{or}_{(2,0)}(\mb{R}^{2,n})={\rm SO}_0(2,n)/{\rm SO}(2)\times{\rm SO}(n),
	\] 
	where the right hand side is the symmetric space of ${\rm SO}_0(2,n)$. In particular, we observe that the Grassmannian ${\rm Gr}_{(2,0)}^{\mathrm{or}}(\mb{R}^{2,n})$ is contractible. Let $G\to\widehat{\mb{H}}^{2,n}$ be the bundle with fiber ${\rm Gr}^{\mathrm{or}}_{(2,0)}(x^\perp)$ over the point $x\in\widehat{\mb{H}}^{2,n}$ and let $\hat{\iota}^*G\to\widetilde{\Sigma}$ be the corresponding pull-back bundle. As the fiber is contractible, we can always find a $\Gamma$-invariant global section. Such a section corresponds to the desired continuous family of orthogonal splittings. 
	
	Consider now the corresponding $\rho(\Gamma)$-invariant plane bundle $P\to\widetilde{\Sigma}$, whose fiber over $x\in\widetilde{\Sigma}$ is given by the spacelike plane $P_x$, and its associated unit circle bundle $P^1\subset P$, whose fiber over $x$ is equal to $\mb{S}^1_x\subset P_x$.
	
	We now define a $\Gamma$-equivariant isomorphism between $P^1 \to \widetilde{\Sigma}$ and the unit tangent bundle $T^1\widetilde{\Sigma} \to \widetilde{\Sigma}$: As a concrete model of $T^1\widetilde{\Sigma}$, we exploit the $\Gamma$-invariant hyperbolic metric obtained by pulling back a hyperbolic metric on $\Sigma$ and, using the exponential map, we identify $T^1\widetilde{\Sigma}\to\widetilde{\Sigma}$ with  
	\[
	B^1:=\{(x,y)\in\widetilde{\Sigma}\times\widetilde{\Sigma}\left|d(x,y)=r\right\}\to\widetilde{\Sigma} ,
	\]
	where the bundle projection is the projection to the first factor, and the fiber over $x$ is the unit circle $B^1_x$ around $x$ in $\widetilde{\Sigma}$. In what follows, we show that $B^1 \to \widetilde{\Sigma}$ is equivariantly isomorphic to $P^1 \to \widetilde{\Sigma}$. Since the Toledo invariant of the representation $\rho$ coincides with the Euler class of the circle bundle $P^1/\rho(\Gamma) \to \Sigma$ (see in particular Collier, Tholozan, and Toulisse \cite{CTT19}*{\S~2.1}), this will allow us to conclude that the representation $\rho$ is maximal, as desired.
	
	For every $y\in B^1_x$, let $\xi_x(y)$ be the endpoint at infinity of the spacelike geodesic ray issuing from $\hat{\iota}(x)$ and passing through $\hat{\iota}(y)$. Explicitly, if $t_x(y)\in T^1\widehat{\mb{H}}^{2,n}$ is the direction of such ray, then 
	\[
	\xi_x(y)=[\hat{\iota}(x)+t_x(y)]\in\partial\widehat{\mb{H}}^{2,n}.
	\] 
	
	Notice that $\xi_x(B^1_x)\subset\partial\widehat{\mb{H}}^{2,n}$ is a loop freely homotopic to the ideal boundary of some fixed spacelike plane $\partial H \subset \partial\widehat{\mb{H}}^{2,n}$: This can be seen in the Poincaré disk model $\widehat{\mb{H}}^{2,n}\cong\mb{D}^2\times\mb{S}^n$ associated to the splitting $\mb{R}^{2,n+1}=P_x\oplus N_x$, where the circle $\hat{\iota}(B^1_x)$ is the graph of a topological circle around the origin of $\mathbb{D}^2$, and $\xi_x(B^1_x)$ is obtained by projecting radially such circle to $\partial\mb{D}^2$ and then mapping it to $\partial\widehat{\mb{H}}^{2,n}$ via the graph map. 
	
	We now exhibit an explicit degree one map $\phi_x:B^1_x\to\mb{S}^1_x$. If $[a]\in\partial\widehat{\mb{H}}^{2,n}$ is an isotropic ray, then we can represent it uniquely as $u_x([a])+v_x([a])$, with $u_x([a])\in P_x$ and $v_x([a])\in N_x$ vectors of norm $1$ and $-1$, respectively. Explicitly, 
	\[
	u_x([a]):=\pi_{P_x}(a)/\sqrt{\langle\pi_{P_x}(a), \pi_{P_x}(a)\rangle},\quad v_x([a]):= \pi_{N_x}(a)/\sqrt{-\langle\pi_{N_x}(a),\pi_{N_x}(a)\rangle},
	\]
	where $\pi_{P_x},\pi_{N_x}:\mb{R}^{2,n+1}\to P_x,N_x$ are the orthogonal projections and $a$ is some fixed representative of $[a]$ in the isotropic cone. We denote by $\mb{S}^1_x,\mb{S}^n_x$ the unit spheres of $P_x,N_x$ and by $u_x,v_x:\partial\widehat{\mb{H}}^{2,n}\to\mb{S}^1_x,\mb{S}^n_x$ the two projections maps defined above. Observe that $u_x$ satisfies 
	\[
	u_{\gamma x}(\rho(\gamma)[a])=\rho(\gamma)u_x([a])
	\]
	for every $[a] \in \partial \widehat{\hyp}^{2,n}$ and $\gamma \in \Gamma$. In fact, $\rho(\gamma)$ maps isometrically $P_x\oplus N_x$ to $P_{\gamma x}\oplus N_{\gamma x}$. We now have all the elements to define the map $\phi_x : B^1_x \to \mathbb{S}^1_x$: For any $y\in B^1_x$, we split $\xi_x(y)=[\hat{\iota}(x)+t_x(y)]$ as the sum $u_x(\xi_x(y))+v_x(\xi_x(y))$, and we set $\phi_x(y):=u_x(\xi_x(y))$.
	
	To conclude, we define a map $\Phi:B^1\to P^1$ by setting
	\[
	\Phi(x,y):=(\hat{\iota}(x),\phi_x(y))
	\]
	for any $(x,y) \in B^1$. By construction, the function $\Phi$ is a $\rho$-equivariant continuous bundle map that covers the identity of $\widetilde{\Sigma}$, and that has degree one on every fiber. 
	
	Since the circle bundle $B^1/\Gamma \to \Sigma$ has Euler number of absolute value $2 \abs{\chi(\Sigma)}$, the same holds for $P^1/\rho(\Gamma) \to \Sigma$, and hence the representation $\rho$ is maximal.
\end{proof}

\subsection{Topology of the gluing}
We conclude with a brief discussion of the topology of the gluing. To this purpose, we only have to compute the first and second Stiefel-Whitney classes of the vector bundle $V\to S$ naturally associated to $E\to S$, which distinguish the connected components of the space of maximal representations in $\SOtwon$ for every $n > 2$ (compare with Remark \ref{rmk:connected comp}).

Recall that $S=\bigcup_j{S_j'}$ where each $S_j'$ is a hyperbolic pair of pants with totally geodesic boundary and let $G$ be the dual graph associated to the gluing $\zeta$, having a vertex per each pair of pants $S_j'$, and an edge between the vertices corresponding to $S_i'$ and $S_j'$ (with possibly $i = j$) whenever there is a boundary component of $S_i'$ that is glued to one of $S_j'$. We may consider $G$ as embedded in $S$ so that each vertex lies in the interior of the corresponding pair of pants and each edge intersects the corresponding curve exactly once. 

The first Stiefel-Whitney class is a homomorphism $w_1(V):H_1(S;\mb{Z}/2\mb{Z})\to\mb{Z}/2\mb{Z}$. The homology group $H_1(S;\mb{Z}/2\mb{Z})$ is generated by the classes of the boundary curves of the pair of pants $S_j'$ and by the simple cycles of $G$ so it is enough to compute $w_1(V)$ on them. We already explained in Theorem \ref{pants of photons} how to compute the value of $w_1(V)$ on each of the classes coming from $\partial S_j'$ in terms of the gluing data. The computation for simple cycles is similar as it requires only to determine whether the holonomy around the cycle, which is loxodromic since the representation is maximal (see Proposition \ref{maximal holonomy}), belongs to $\mc{L}^+$ or $\mc{L}^-$. A precise formula requires a careful bookkeeping of the gluing choices and we will not pursue it here.

The second Stiefel-Whitney class $w_2(V)$ can be computed as follows: Choose for every pair of identified boundary components $E(\gamma_{u_i}),E(\gamma_{v_j})$ a pair of $(n+1)$-frames $\sigma_{u_i},\sigma_{v_j}$ of the underlying vector bundles $V\to\gamma_{u_i},V\to\gamma_{v_j}$ that are identified under the gluing map $\zeta_{u_iv_j}$. The Stiefel-Whitney number $w_2(V)[S]$, that uniquely determines $w_2(V)$, can be computed as the sum of the relative Stiefel-Whitney numbers $w_2(V_j,\sigma_j)[S_j,\partial S_j]$ that are the obstructions to extend the $(n+1)$-frames defined over $\partial S_j$ to $(n+1)$-frames over $S_j$.

%%%

\appendix

\section{Other cross ratios}
\label{other cross ratios}

There are multiple non-equivalent definitions of cross ratios in the literature. 

For the reader's convenience, we summarize the relations between Definition~\ref{def:crossratio} and the notions of cross ratios studied by Ledrappier \cite{ledrappier}, Hamenst{\"a}dt \cites{H97,H99}, and Labourie \cite{Lab08}. If $\beta = \beta(u,v,w,z)$ satisfies Definition \ref{def:crossratio}, then:
\begin{itemize}
	\item The function $(u,v,w,z) \mapsto \log \abs{\beta(u,v,w,z)}$ is a Ledrappier's cross ratio (compare with \cite{ledrappier}*{D{\'e}finition~1.f}, \cite{MZ19}*{Definition~2.4}). 
	\item The function $(u,v,w,z) \mapsto \abs{\beta(u,v,w,z)}$ is a Hamenst{\"a}dt's cross ratio (compare with \cites{H97,H99}).
	\item If $\beta$ further satisfies 
	\begin{equation} \label{cross ratio one}
		\beta(u,v,w,z) = 1 \Leftrightarrow \text{$u=v$ or $w = z$}, 
	\end{equation}
	then the function $B(u,v,w,z) : = \beta(u,w,v,z)$ is a Labourie's cross ratio (see \cite{Lab08}*{Definition~3.2.1}).
\end{itemize}

For the sake of completeness, even if we do not investigate in detail the properties of such cross ratios in this paper, we briefly discuss other examples of positive and locally bounded cross ratios from the literature strictly related to representations in ${\rm SO}(p,q)$ and pseudo-hyperbolic spaces $\mb{H}^{p,q}$. They come from respectively:

\begin{itemize}
	\item{Hitchin representations $\rho:\Gamma\to{\rm SO}(p,p+1)$.}
	\item{More generally, $\Theta$-positive Anosov representations $\rho:\Gamma\to{\rm SO}(p,q)$.}
\end{itemize}

In both cases (strict) positivity comes from transversality of the boundary maps (as explained in \cite{BP21}) and local boundedness comes from their H\"older regularity (following the same strategy of Lemma~\ref{lem: comparing crossratios}).

As studied by Martone and Zhang in \cite{MZ19} there are other natural classes of positive cross ratios arising from the study of Anosov representations. The ones that we mentioned above close to the setting of our interest and can have a more direct link with similar pleated surface constructions in $\mb{H}^{p,q}$.

\section{Shears and symmetries of cross ratios} \label{shears and symmetries}

The current appendix is dedicated to the proofs of the relations satisfied by cross ratios and their associated shears, which were deployed throughout Section \ref{sec:shear cocycles}. We start by proving the following elementary relation:

\begin{lem} \label{eq della madonna}
	Let $\beta$ be a cross ratio. Then for every $6$-tuple of pairwise distinct points $a, b, c, d, e, x \in \partial \Gamma$ we have
	\[
	|\beta(a,b,c,d) \beta(a,d,b,e)| = \abs{\beta(a,b,c,x) \beta(a,x,b,d) \beta(a,d,x,e)}
	\]
\end{lem}

\begin{proof}
	It is sufficient to apply the symmetries of the cross ratio $\beta$ in \eqref{eq:crossshear} as follows
	\begin{align*}
		|\beta(a,b,c,d) \beta(a,d,b,e)| & = \abs{\beta(a,b,c,x) \beta(a,b,x,d) \beta(a,d,b,e)} \\
		& = \abs{\beta(a,b,c,x) \beta(a,b,x,d) \beta(a,d,b,x) \beta(a,d,x,e)} \\
		& = \abs{\beta(a,b,c,x) \beta(a,x,b,d) \beta(a,d,x,e)} ,
	\end{align*}
	where we used in the order twice the fourth relation and once the fifth relation from \eqref{eq:crossshear}. By applying $\log$ to both members we obtain relation \eqref{eq:add_leaf}. 
\end{proof}

Making use of the relation described in Lemma \ref{eq della madonna}, we can now provide a proof of the properties satisfied by finite $\beta$-shears and described by Lemmas \ref{lem:asymptotic_plaques} and \ref{lem:shear near closed leaves}:

\begin{proof}[Proof of Lemma \ref{lem:asymptotic_plaques}]
	Let $S$ denote the (closure of the) connected component of $\widetilde{\Sigma} \setminus \{P, Q\}$ that separates $P$ from $Q$. Observe that the right-hand side of the statement can be expressed as $\sigma^\beta(P,S) + \sigma^\beta(S,Q)$. Consider now any geodesic $g$ lying in the interior of $S$ with endpoints $w$ and $x$, and denote by $R$ and $R'$ the complementary regions of $g$ inside $S$ adjacent to $P$ and $Q$, respectively. We claim that the following equality holds:
	\begin{equation} \label{eq:add_leaf}
		\sigma^\beta(P,S) + \sigma^\beta(S,Q) = \sigma^\beta(P,R) + \sigma^\beta(R,R') + \sigma^\beta(R',Q) .
	\end{equation}
	This is in fact a simple consequence of Lemma \ref{eq della madonna}. To see this, observe that, by definition of the finite shear $\sigma^\beta$, the left-hand side coincides with
	\[
	\log|\beta(w,v_P, u_P, v_Q) \beta(w,v_Q,v_P,u_Q)| ,
	\]
	while the right-hand side is equal to
	\[
	\log\abs{\beta(w,v_P, u_P, x) \beta(w,x, v_P, v_Q) \beta(w,v_Q,x,u_Q)} .
	\]
	Therefore relation \eqref{eq:add_leaf} follows from Lemma \ref{eq della madonna} applied to the $6$-tuple $a=w$, $b=v_P$, $c=u_P$, $d=v_Q$, $e=u_Q$, and $x=x$.
	
	The relation appearing in the statement can now be deduced simply by applying relation \eqref{eq:add_leaf} enough times: at the $k$-th step we introduce inside the region $S$ a leaf $\ell_k$ lying in the boundary of some plaque in $\mathcal{P}$, obtaining a finite lamination $\lambda_k = \lambda_{k-1} \cup \{\ell_k\}$. Relation \eqref{eq:add_leaf} then allows us to split the sum of the shears between the complementary regions of $\lambda_{k - 1}$ as the sum of the shear of the complementary regions of $\lambda_k$. In a finite number of steps we obtain that $\sigma^\beta_\mathcal{P}(P,Q)$ coincides with $\sigma^\beta(P,S) + \sigma^\beta(S,Q)$, as desired (observe that, in the notation of \S \ref{subsec:definition shear}, the lamination $\tilde{\lambda}_\mathcal{P}$ does not contain any geodesic of type $d_i$, since every spike has ideal vertex equal to $w$ under our assumptions).
\end{proof}

\begin{proof}[Proof of Lemma \ref{lem:shear near closed leaves}]
	Among all the elements of $\mathcal{P}$ that lie on the left (resp. on the right) of $g$, we denote by $P'$ (resp. $Q'$) the plaque that is closest to $g$. Let $x_P', y_P'$ (resp. $x_Q', y_Q'$) be the vertices of $P'$ (resp. $Q'$) different from $g^+$ (resp. $g^-$), so that $[y_p', g^+]$ (resp. $[y_Q', g^-]$) is the boundary component of $P'$ (resp. $Q'$) closest to $g$.
	
	By following the process outlined in \S \ref{subsec:definition shear}, we see that the shear $\sigma_\mathcal{P}(P,Q)$ satisfies
	\[
	\sigma^\beta_\mathcal{P}(P,Q) = \sigma_\lambda^\beta(P,R_P) + \sigma^\beta(R_P, R_Q) + \sigma_\lambda^\beta(R_Q, Q) ,
	\]
	where $R_P$ and $R_Q$ denote the plaques of $\tilde{\lambda}_\mathcal{P}$ with vertices $g^+, g^-, y_P'$ and $g^+, g^-, y_Q'$, respectively. By Lemma \ref{lem:asymptotic_plaques}, the shear $\sigma_\lambda^\beta(P,R_P)$ is independent of the set of plaques that separate $P$ and $R_P$ inside $\mathcal{P}$, since $P$ and $R_P$ share the ideal vertex $g^+$. The exact same argument applies for $\sigma_\lambda^\beta(R_Q, Q)$. Furthermore we have
	\begin{align*}
		\sigma_\lambda^\beta(P,R_P) & = \log\abs{\beta(g^+,y_P, x_P, y_P') \beta(g^+,y_P',y_P,g^-)} , \\
		\sigma_\lambda^\beta(R_Q,Q) & = \log\abs{\beta(g^-,y_Q, x_Q, y_Q') \beta(g^-,y_Q',y_Q,g^+)} .
	\end{align*}
	On the other hand, the plaques $R_P$ and $R_Q$ share the boundary component $[g^+,g^-]$ and their shear satisfies
	\[
	\sigma^\beta(R_P, R_Q) = \log\abs{\beta(g^+,g^-, y_P', y_Q')} .
	\]
	By applying Lemma \ref{eq della madonna} to the $6$-tuple $a=g^+$, $b=y_P$, $c=x_P$, $d=g^-$, $e=y_Q'$, and $x = y_P'$, we obtain
	\[
	\sigma_\lambda^\beta(P,R_P) + \sigma_\lambda^\beta(R_P,R_Q) = \log\abs{\beta(g^+,y_P,x_P,g^-) \beta(g^+,g^-,y_P,y_Q')} .
	\]
	Combining this identity with the expression for $\sigma_\lambda^\beta(R_Q,Q)$ we deduce
	\begin{align*}
		\sigma_\lambda^\beta(&P,R_P) + \sigma_\lambda^\beta(R_P,R_Q) + \sigma_\lambda^\beta(R_Q,Q) \\ 
		& = \log\abs{\beta(g^+,y_P,x_P,g^-) \beta(g^+,g^-,y_P,y_Q') \beta(g^-,y_Q, x_Q, y_Q') \beta(g^-,y_Q',y_Q,g^+)} \\
		& = \log\abs{\beta(g^+,y_P,x_P,g^-) \beta(g^-,g^+,y_Q',y_P) \beta(g^-,y_Q, x_Q, y_Q') \beta(g^-,y_Q',y_Q,g^+)} \\
		& = \log\abs{\beta(g^+,y_P,x_P,g^-) \beta(g^-,g^+,y_Q,y_P) \beta(g^-,y_Q, x_Q, g^+)}
	\end{align*}
	where in the second equality we applied relation \eqref{eq:symmetry for shear}, and in the last line we applied again Lemma \ref{eq della madonna} to the $6$-tuple $a=g^+$, $b=y_P$, $c=x_P$, $d=g^-$, $e=y_Q'$, and $x = y_P'$. This concludes the proof of the statement.
\end{proof}

We now provide a proof of Lemma \ref{lem:diagonal_exchange}, which again follows easily from the symmetries of cross ratios:

\begin{proof}[Proof of Lemma \ref{lem:diagonal_exchange}]
	Let $u$ denote the vertex of $P$ that is not an endpoint of $\ell_P$, and by $v$ the vertex of $Q$ that is not an endpoint of $\ell_Q$. Then the left-hand side of the equation can be expressed as
	\[
	\abs{\log \abs{\frac{\beta(\ell_P^+,\ell_P^-,u,\ell_Q^-) \beta(\ell_P^+, \ell_Q^-,\ell_P^-,\ell_Q^+) \beta(\ell_Q^+,\ell_Q^-,\ell_P^+,v)}{\beta(\ell_P^+,\ell_P^-,u,\ell_Q^+) \beta(\ell_Q^+, \ell_P^-,\ell_P^+,\ell_Q^-) \beta(\ell_Q^+,\ell_Q^-,\ell_P^-,v)}} }
	\]
	Applying the third symmetry in \eqref{eq:crossshear}, we obtain the identities
	\begin{align*}
		\abs{\beta(\ell_P^+,\ell_P^-,u,\ell_Q^-)} & = \abs{\beta(\ell_P^+,\ell_P^-,u,\ell_Q^+) \beta(\ell_P^+,\ell_P^-,\ell_Q^+,\ell_Q^-)} , \\
		\abs{\beta(\ell_Q^+,\ell_Q^-,\ell_P^-,v)} & = \abs{\beta(\ell_Q^+,\ell_Q^-,\ell_P^-,\ell_P^+) \beta(\ell_Q^+,\ell_Q^-,\ell_P^+,v)} .
	\end{align*}
	By replacing these terms in the expression above we obtain
	\begin{align*}
		\abs{\sigma_d^\beta(P,Q) - \sigma_{d'}^\beta(P,Q)} & = \abs{\log \abs{ \frac{\beta(\ell_P^+,\ell_P^-,\ell_Q^+,\ell_Q^-) \beta(\ell_P^+, \ell_Q^-,\ell_P^-,\ell_Q^+)}{\beta(\ell_Q^+, \ell_P^-,\ell_P^+,\ell_Q^-) \beta(\ell_Q^+,\ell_Q^-,\ell_P^-,\ell_P^+)} }} \\
		& = \abs{\log \abs{\beta(\ell_P^+,\ell_P^-,\ell_Q^+,\ell_Q^-)^2 \beta(\ell_P^+, \ell_Q^-,\ell_P^-,\ell_Q^+)^2 }}
	\end{align*}
	where in the last equality we made use of \eqref{eq:inverse crossratio} and \eqref{eq:symmetry for shear}. The desired expression then follows by applying the fourth relation in \eqref{eq:crossshear} and \eqref{eq:inverse crossratio}. (Notice that $\beta(g^+,h^+,h^-,g^-) > 1$ for any pair of coherently oriented geodesics $g, h$ that share no endpoint.)
\end{proof}

We are now left with the proof of Lemma \ref{lem:shear and length}, which directly relates $\beta$-periods and $\beta$-shears:

\begin{proof}[Proof of Lemma \ref{lem:shear and length}]
	Let $x, y, \gamma^\pm \in \partial \Gamma$ be the vertices of $P$ in counterclockwise order along $\partial \Gamma$. By Lemma \ref{lem:asymptotic_plaques}, we have
	\begin{align*}
		\sigma_\lambda^\beta(P, \gamma P) & = \log \abs{\beta(\gamma^\pm, x, y, \gamma y) \beta(\gamma^\pm, \gamma y, x, \gamma x)} .
	\end{align*}
	The proof of the relation appearing in the statement now reduces to a careful applications of the symmetries of the cross ratio $\beta$ (see in particular \eqref{eq:crossshear}, \eqref{eq:inverse crossratio}). In what follows, we express the chain of equalities that leads to the proof, reporting on the right the relations that are applying (the symbol (\ref{eq:crossshear}.$n$) refers to the $n$-th symmetry of $\beta$ appearing in \eqref{eq:crossshear}):
	\begin{align*}
		|\beta(\gamma^\pm, x, y, \gamma y) \beta(\gamma^\pm, & \,\gamma y, x, \gamma x)| \\
		& = \abs{\beta(\gamma^\pm, x, \gamma x, \gamma y) \beta(\gamma^\pm, x, y, \gamma x) \beta(\gamma^\pm, \gamma y, x, \gamma x) } \tag{\ref{eq:crossshear}.$4$} \\
		& = \abs{\beta(\gamma^\pm, \gamma x, x, \gamma y) \beta(\gamma^\pm, x, y, \gamma x)} \tag{\ref{eq:crossshear}.$5$} \\
		& = \abs{\beta(\gamma^\pm, x, \gamma^{-1} x, y) \beta(\gamma^\pm, x, y, \gamma x)} \tag{$\Gamma$-inv.} \\
		& = \abs{\beta(\gamma^\pm, x, \gamma^{-1} x, \gamma x)} \tag{\ref{eq:crossshear}.$4$} \\
		& = \abs{\beta(\gamma^\pm, x, \gamma^{\mp}, \gamma x) \beta(\gamma^\pm, \gamma x, x, \gamma^{\mp})}  \tag{\ref{eq:crossshear}.$4$} \\
		& = \abs{\beta(\gamma^\pm, \gamma^\mp, x, \gamma x)} \tag{\ref{eq:crossshear}.$5$} \\
		& = \abs{\beta(\gamma^+, \gamma^-, x, \gamma x)}^{\pm 1} \tag{\ref{eq:inverse crossratio}} .
	\end{align*}
	Taking the logarithm of this relation we obtain the identity $\sigma^\beta(P, \gamma P) = \pm L_\beta(\gamma)$, as desired.
\end{proof}

\section{On divergence radius functions}\label{divergence appendix}

In our construction of $\beta$-shear cocycles, we made use of a series of technical properties satisfied by divergence radius functions, described in Lemmas \ref{lem:auxiliary function r}, \ref{lem:dependence of constants}, and \ref{lem:dependence of r on metric}. We remark that the statement of Lemma \ref{lem:auxiliary function r} already appeared in the work of Bonahon and Dreyer \cite{BD17}. The underlying strategy of proof is essentially the same as the one described by Bonahon in \cite{Bo96}*{Lemmas~3,~5}. However, since the work \cite{Bo96} uses a definition of divergence radius function that is weaker than the one we introduced in Section \ref{subsec:divergence radius function}, we describe how to adapt the argument accordingly. The strategy of proof is in fact particularly useful to understand the dependence of the constants, as asserted in Lemma \ref{lem:dependence of constants}.

We start by fixing some hyperbolic metric on $\Sigma$ and a train track $\tau$ that carries a maximal lamination $\lambda$. Furthermore, we introduce the following terminology: If $B$ is a branch of $\tau$, we define the \emph{width} of $B$ (with respect to the chosen metric) to be the distance between the components of the horizontal boundary of $\tilde{b}$, for some lift $\tilde{B}$ of $B$ in $\widetilde{\Sigma}$. Similarly, the \emph{length} of $B$ is defined as the distance between the components of the vertical boundary of $\tilde{B}$, for some lift $\tilde{B}$ of $B$.

\begin{proof}[Proof of Lemma~\ref{lem:auxiliary function r}]
	We start by selecting suitable constants $M, A_0, \theta > 0$, which depends exclusively on the train track $\tau$ and the fixed hyperbolic structure $X$:
	\begin{itemize}
		\item We select $M < 1$ so that every branch of the train track $\tau$ has length within $M$ and $M^{-1}$.
		\item We let $A_0 > 1$ be such that every component of the vertical boundary of $\tau$ (compare with the terminology introduced in Section \ref{subsubsec:train tracks}) has endpoints at distance $\geq A_0^{-1}$, and such that every branch of $\tau$ has width $\leq A_0$. 
		\item We choose $\theta \in (0,\pi/2)$ a lower bound for the intersection angle between the geodesic arc $k$ and the leaves of the lamination $\lambda$. 
	\end{itemize}
	Consider now the following situation: Let $\ell$ and $\ell'$ be two distinct asymptotic geodesics in $(\widetilde{\Sigma}, \widetilde{X})$, and let $u$ be their common endpoint in $\partial \Gamma$. Consider a geodesic segment $k'$ joining a point $p \in \ell$ to a point of $\ell'$, and assume that the angles between $k'$ and $\ell$, $\ell'$ satisfy
	\[
	\theta \leq \abs{\angle(k',\ell)}, \abs{\angle(k', \ell)} \leq \pi - \theta .
	\]
	Finally, select a parametrization by arc length of the geodesic $\ell = \ell(t)$ such that $\ell(t)$ tends to $u$ as $t \to - \infty$ and $\ell(0) = p$, and assume that there exists some positive $t > 0$ for which $\ell(t)$ satisfies
	\[
	A_0^{-1} \leq d_{\tilde{X}}(\ell(t),\ell') \leq A_0 .
	\]
	A simple computation in the upper half plane model of $\hyp^2$ then shows that there exists a constant $A > 0$, which depends only on $A_0$ and $\theta$, such that
	\begin{equation}\label{eq:comparison auxiliary}
		A^{-1} e^{- t} \leq L_{\tilde{X}}(k') \leq A e^{-t} .
	\end{equation}

	We now have all the technical ingredients for the proof the desired statement: First recall the definition of the divergence radius function $r : \mathcal{P}_{P Q} \to \N$ outlined in Section \ref{subsec:divergence radius function}, select any plaque $R \in \mathcal{P}_{P Q}$, and denote by $s = s_R$ the switch of the lift of the train track $\tau$ that separates the branches $\widetilde{B}_{r(R) - 1}$ and $\widetilde{B}_{r(R)}$ (see Section \ref{subsec:divergence radius function} for the necessary terminology). By definition of the divergence radius function $r$, the boundary leaves $\ell_R$ and $\ell_R'$ of $R$ that separate $P$ from $Q$ fellow travel along the branches $\widetilde{B}_n$ for all $n < r(R)$, and then take different turns at the switch $s$. Indeed, while the leaf $\ell_R$ crosses $s$ to then enter in the branch $\widetilde{B}_{r(R)}$, the leaf $\ell_R'$ passes through the unique branch of $\tilde{\tau}$ adjacent to $s$ and different from $\widetilde{B}_{r(R) - 1}$ and $\widetilde{B}_{r(R)}$. 
	
	Now, if $\gamma_R$ denotes the subsegment of $\ell_R$ that joins $k \cap R$ to the switch $s$ of the train track $\tilde{\tau}$, then by the choice of $M$ we have
	\begin{equation} \label{eq:branches}
		M \, r(R) \leq L_{\tilde{X}}(\gamma_R) \leq M^{-1} \, r(R)
	\end{equation}
	whenever $r(R) > 1$.  Moreover, if we travel along the geodesic $\ell_R$ at distance $\ell(\gamma_R)$ towards the positive direction of $\ell_R$, the geodesics $\ell_R$ and $\ell_R'$ are at distance $d(s \cap \ell_R, s \cap \ell_R' ) \in (A_0^{-1},A_0)$ by our initial choices. (Notice that the switch $s$ contains exactly one connected component of the vertical boundary of $\tilde{\tau}$, whose endpoints are at distance between $A_0^{-1}$ and $A_0$.) We then are in right setting to apply relation \eqref{eq:comparison auxiliary} to $k' = k \cap R$, $\ell = \ell_R$, $\ell' = \ell_R'$ and $t = L_{\tilde{X}}(\gamma_R)$: consequently we conclude that
	$$
	A^{-1} \, e^{- L_{\tilde{X}}(\gamma_R)} \leq L_{\tilde{X}}(k \cap R) \leq A \, e^{- L_{\tilde{X}}(\gamma_R)}
	$$
	Combining this comparison with relation \eqref{eq:branches}, we obtain the control appearing in property (1) of Lemma \ref{lem:auxiliary function r} for all $r(R) > 1$. Now, up to enlarging the multiplicative constant $A > 0$ to obtain a bound from above of the diameter of every complementary region of $\tau$ in $X$, we can then make sure that $\textit{(1)}$ holds for every $R \in \mathcal{P}_{P,Q}$.
	
	The proof of the second bound appearing in (2) is a simple generalization of \cite{Bo96}*{Lemma~4}: in his work Bonahon showed that, if $k_0$ is a geodesic arc transverse to $\tilde{\lambda}$ that projects onto an \emph{embedded} arc in $\Sigma$, then the number of plaques $R \in \mathcal{P}_{P Q}$ satisfying $r_{k_0}(R) = n$ is bounded above by an explicit function $N_0 = N_0(\Sigma)$ that depends only on the topology of $\Sigma$. For a general geodesic arc $k$, we can argue as follows: there exists a natural number $m$ such that the arc $k$ can be subdivided into $m$ subsegments $(k_i)_i$ with disjoint interiors and such that every $k_i$ projects onto an embedded geodesic arc in $\Sigma$. Then the cardinality of $r_k^{-1}(n)$ is bounded above by $N : = m N_0(\Sigma)$. Observe also that, if $\varepsilon_0$ is equal to the injectivity radius of $X$, then $m \leq \ell(k) / \varepsilon_0$.
\end{proof}

From the proof provided above, and in particular from the definition of the constants $A,M,N > 0$, Lemma \ref{lem:dependence of constants} easily follows:

\begin{proof}[Proof of Lemma \ref{lem:dependence of constants}]
	We fix a hyperbolic structure $X$ on $\Sigma$, and we select a train track $\tau$ that carries $\lambda$ and a $X$-geodesic ark $k$ joining the interiors of the plaques $P$ and $Q$. We denote by $M,A_0, A, \theta > 0$ the constants introduced in the proof of Lemma \ref{lem:auxiliary function r}. Up to selecting a smaller $\theta  > 0$, we can find a small neighborhood $U$ of $\lambda$ inside $\mathcal{GL}$ satisfying the following conditions:
	\begin{itemize}
		\item Every $\lambda' \in U$ is carried by $\tau$.
		\item For every $\lambda' \in U$, the geodesic segment $k$ is transverse to $\lambda'$ and $\theta > 0$ is a uniform lower bound of the intersection angle between $k$ and $\lambda'$.
		\item the endpoints of $k$ lie in the interior of two distinct plaques $P', Q'$ of $\lambda'$, for every $\lambda' \in U$.
	\end{itemize}
	The constants $M, A_0 > 0$ depends only on the train track $\tau$ (and the hyperbolic structure $X$), and $A$ is determined by $A_0$ and $\theta$. In particular, $A$ and $M$ satisfy relations \eqref{eq:comparison auxiliary}, \eqref{eq:branches} for any divergence radius function $r' = r_{X, \tau, \lambda', k} : \mathcal{P}_{P' Q'} \to \N$ associated to a lamination $\lambda' \in U$ and the path $k$. Relations \eqref{eq:comparison auxiliary}, \eqref{eq:branches} in turn imply property (1) for all such divergence radius functions $r'$.
	Finally, it is immediate from the explicit description of the constant $N > 0$ satisfying property (2) provided in the proof of Lemma \ref{lem:auxiliary function r} that we can assume $N$ to be uniform in $\lambda' \in U$.
\end{proof}

The only technical statement left to prove is Lemma \ref{lem:dependence of r on metric}. For its proof, we will make use of an elementary lemma of planar hyperbolic geometry. In order to recall its statement, we need to introduce some notation.

If $X \in \T$ is a hyperbolic structure and $(\widetilde{\Sigma},\widetilde{X}) \cong \hyp^2$ is a fixed identification between the universal cover of $\Sigma$ and the hyperbolic plane determined by $X$, then we select $d_\infty$ a fixed Riemannian distance on $\partial \Gamma \cong \partial \hyp^2$. The choice of the metric $d_\infty$ determines a distance (which we will continue to denote with abuse by $d_\infty$) on the space of oriented geodesics of $\widetilde{\Sigma}$, by setting
\[
d_\infty(g,h) : = d_\infty(g^+,h^+) + d_\infty(g^-,h^-)
\]
for any pair of oriented geodesics $g$ and $h$. Then we have:

\begin{lem} \label{lem:comparison_arc_distance at infinity}
	Let $\lambda$ be a maximal geodesic lamination on $\Sigma$, and let $P$ and $Q$ be two distinct plaques of $\lambda$. For any geodesic segment $k$ joining two points in the interior of $P$ and $Q$, respectively, we can find a constant $C=C(k) > 0$ such that, for every plaque $R \in \mathcal{P}_{P Q}$ 
	\[
	C^{-1} \, d_\infty(\ell_R, h_R) \leq L_{\tilde{X}}(k \cap R) \leq C \, d_\infty(\ell_R, h_R) ,
	\]
	where $\ell_R, \ell_R'$ denote the boundary leaves of $R$ that separate $P$ from $Q$.
\end{lem}

We are now ready to prove Lemma \ref{lem:dependence of r on metric}:

\begin{proof}[Proof of Lemma \ref{lem:dependence of r on metric}]
	By property \textit{(1)} of Lemma \ref{lem:auxiliary function r}, there exist positive constants $A, A', M, M' > 0$ such that
	\begin{align*}
		A^{-1} e^{- M^{-1} r(R)} \leq L_{\tilde{X}}(k \cap R) \leq A e^{- M r(R)} \\
		(A')^{-1} e^{- (M')^{-1} r'(R)} \leq L_{\tilde{X}}(k' \cap R) \leq A' e^{- M' r'(R)} 
	\end{align*}
	for every $R \in \mathcal{P}_{P Q}$. On the other hand, by Lemma \ref{lem:comparison_arc_distance at infinity}, there exist constants $S, T > 0$ such that for every $R \in \mathcal{P}_{P Q}$ we have
	\begin{gather*}
		W^{-1} d_\infty(\ell_R, h_R) \leq L_{\tilde{X}}(k \cap R) \leq W \, d_\infty(\ell_R, h_R) , \\
		(W')^{-1} d_\infty(\ell_R, h_R) \leq L_{\tilde{X}}(k' \cap R) \leq W' \, d_\infty(\ell_R, h_R) .
	\end{gather*}
	By combining the inequalities above, we obtain
	\begin{align*}
		e^{M r(R)} & \leq \frac{A}{L_{\tilde{X}}(k \cap R)} \\
		& \leq \frac{A W}{d_\infty(\ell_R, h_R)} \\
		& \leq \frac{A W W'}{L_{\tilde{X}}(k' \cap R)} \\
		& \leq A A' W W' e^{M' r'(R)} ,
	\end{align*}
	which implies the upper bound appearing in the statement with suitable choices of $H, K > 0$. By exchanging the roles of $r$ and $r'$ in the argument above we determine the existence of the lower bound.
\end{proof}

\emergencystretch=1em

\bibliography{biblio.bib}

\end{document}